\documentclass{amsart}

\def\abb{{\mathbb{A}}}

\def\cbb{{\mathbb{C}}}

\def\fbb{{\mathbb{F}}}

\def\lbb{{\mathbb{L}}}

\def\nbb{{\mathbb{N}}}

\def\qbb{{\mathbb{Q}}}
\def\rbb{{\mathbb{R}}}

\def\zbb{{\mathbb{Z}}}

\newcommand{\bvec}[1]{\mbox{\boldmath $#1$}}
\def\abf{{\mathbf{A}}}

\def\sbf{{\mathbf{S}}}

\def\acal{{\mathcal{A}}}
\def\bcal{{\mathcal{B}}}
\def\ccal{{\mathcal{C}}}
\def\dcal{{\mathcal{D}}}
\def\ecal{{\mathcal{E}}}
\def\fcal{{\mathcal{F}}}

\def\hcal{{\mathcal{H}}}
\def\ical{{\mathcal{I}}}
\def\jcal{{\mathcal{J}}}
\def\kcal{{\mathcal{K}}}

\def\mcal{{\mathcal{M}}}
\def\ncal{{\mathcal{N}}}

\def\pcal{{\mathcal{P}}}

\def\scal{{\mathcal{S}}}
\def\tcal{{\mathcal{T}}}
\def\ucal{{\mathcal{U}}}
\def\vcal{{\mathcal{V}}}
\def\wcal{{\mathcal{W}}}

\def\sd{{\rm sd}}
\def\rmone{{\rm (1)}}
\def\rmtwo{{\rm (2)}}
\def\rmthree{{\rm (3)}}
\def\rmfour{{\rm (4)}}

\def\rmi{{\rm (i)}}
\def\rmii{{\rm (ii)}}
\def\rmiii{{\rm (iii)}}
\def\rmiv{{\rm (iv)}}

\def\czero{{\mathcal{C}^0}}
\def\cinf{{\mathcal{C}^\infty}}

\def\sdcal{{{S}^\dcal}}
\def\sub{{\mathrm{sub}}}
\def\upsub{{{\mathrm{SUB}}}}
\def\top{{\mathrm{top}}}
\def\Hom{{\mathrm{Hom}}}

\def\cfra{{\mathfrak{C}}}
\def\xfra{{\mathfrak{X}}}
\def\ofrack{{\mathfrak{o}}}
\newcommand{\xoverright}[1]{\xrightarrow[\ \ \ \ \ ]{#1}}

\newcommand{\xhookoverright}[2]{\lhook\hspace{-.#2mm}\xrightarrow[\ \ \ \ \ ]{#1}}

\newcommand{\equivalence}[2]{#1$\Leftrightarrow$#2}
\def\wczero{{\wcal_{\czero}}}
\newcommand{\absno}[1]{\left|#1\right|}
\newcommand{\cdarrow}[1]{\arrow[#1]}

\newcommand*{\longhookrightarrow}{\ensuremath{\lhook\joinrel\relbar\joinrel\rightarrow}}
\newcommand\phantomarrow[2]{%
\setbox0=\hbox{$\displaystyle #1\to$}%
\hbox to \wd0{%
	$#2\mapstochar
	\cleaders\hbox{$\mkern-1mu\relbar\mkern-3mu$}\hfill
	\mkern-7mu\rightarrow$}%
\,}

\def\colim{{\mathrm{colim}}}
\newcommand{\colimunder}[1]{\underset{#1}{\colim}}

\def\ho{{\mathrm{Ho}}}
\def\cel{{c\ell}}

\let\protect\relax
{\catcode`\|=\active
	\xdef\Braket{\protect\expandafter\noexpand\csname Braket \endcsname}
	\expandafter\gdef\csname Braket \endcsname#1{\begingroup
		\ifx\SavedDoubleVert\relax
		\let\SavedDoubleVert\|\let\|\BraDoubleVert
		\fi
		\mathcode`\|32768\let|\BraVert
		\left\langle{#1}\right\rangle\endgroup}
}
\def\BraVert{\@ifnextchar|{\|\@gobble}
	{\egroup\,\mid@vertical\,\bgroup}}
\def\BraDoubleVert{\egroup\,\mid@dblvertical\,\bgroup}
\let\SavedDoubleVert\relax

{\catcode`\|=\active
	\xdef\set{\protect\expandafter\noexpand\csname set \endcsname}
	\expandafter\gdef\csname set \endcsname#1{\mathinner
		{\lbrace\,{\mathcode`\|32768\let|\midvert #1}\,\rbrace}}
	\xdef\Set{\protect\expandafter\noexpand\csname Set \endcsname}
	\expandafter\gdef\csname Set \endcsname#1{\left\{%
		\ifx\SavedDoubleVert\relax \let\SavedDoubleVert\|\fi
		\:{\let\|\SetDoubleVert
			\mathcode`\|32768\let|\SetVert
			#1}\:\right\}}
}
\def\midvert{\egroup\mid\bgroup}
\def\SetVert{\@ifnextchar|{\|\@gobble}
	{\egroup\;\mid@vertical\;\bgroup}}
\def\SetDoubleVert{\egroup\;\mid@dblvertical\;\bgroup}

\begingroup
\edef\@tempa{\meaning\middle}
\edef\@tempb{\string\middle}
\expandafter \endgroup \ifx\@tempa\@tempb
\def\mid@vertical{\middle|}
\def\mid@dblvertical{\middle\SavedDoubleVert}
\else
\def\mid@vertical{\mskip1mu\vrule\mskip1mu}
\def\mid@dblvertical{\mskip1mu\vrule\mskip2.5mu\vrule\mskip1mu}
\fi

\usepackage{amsmath,amsfonts,amssymb,amsthm}
\usepackage{enumerate}
\usepackage{titletoc}
\usepackage{tikz}
\usetikzlibrary{backgrounds}
\usepackage{mathtools}
\usetikzlibrary{shapes.geometric, arrows, calc, positioning, decorations.markings, patterns}
\usepackage{graphicx}
\usepackage[all]{xy}
\usepackage{braket}
\usepackage{tikz-cd,calc}
\usepackage{ulem}
\usepackage{etoolbox}
\usepackage{hyperref}
\usepackage{tikz}
\usepackage{fancyhdr}

\tikzstyle{startstop} = [rectangle, minimum width=2.5cm, minimum height=0.5cm,text centered, draw=black]

\tikzstyle{arrow} = [thick,->,>=stealth]

\tikzset{middlearrow/.style={
		decoration={markings,
			mark= at position 0.5 with {\arrow{#1}} ,
		},
		postaction={decorate}
	}
}

\numberwithin{equation}{section}

\newtheorem{thm}{Theorem}[section]

\newtheorem{cor}[thm]{Corollary}
\newtheorem{rem}[thm]{Remark}
\newtheorem{prop}[thm]{Proposition}
\newtheorem{lem}[thm]{Lemma}

\newtheorem{axiom}{Axiom}
\newtheorem{exa}[thm]{Example}

\theoremstyle{definition}
\newtheorem{defn}[thm]{Definition}

\theoremstyle{remark}

\usepackage[foot]{amsaddr}

\makeatletter
\renewcommand{\email}[2][]{%
	\ifx\emails\@empty\relax\else{\g@addto@macro\emails{,\space}}\fi%
	\@ifnotempty{#1}{\g@addto@macro\emails{\textrm{(#1)}\space}}%
	\g@addto@macro\emails{#2}%
}
\makeatother

\makeindex

\begin{document}


\title{Smooth Homotopy of Infinite-Dimensional $C^{\infty}$-Manifolds}


\author{Hiroshi Kihara}
\address{Center for Mathematical Sciences, University of Aizu, Tsuruga, Ikki-machi, Aizu-Wakamatsu City, Fukushima, 965-8580, Japan}
\email{(kihara@u-aizu.ac.jp)}



\subjclass[2010]{Primary 58B05; Secondary 58A40,18G55}

\keywords{Smooth homotopy, $C^\infty$-manifolds, convenient calculus, diffeological spaces, model category.}


\begin{abstract}
	In this paper, we use homotopical algebra (or abstract homotopical methods) to study smooth homotopical problems of infinite-dimensional $C^{\infty}$-manifolds in convenient calculus. More precisely, we discuss the smoothing of maps, sections, principal bundles, and gauge transformations.\par
	We first introduce the notion of hereditary $C^\infty$-paracompactness along with the semiclassicality condition on a $C^\infty$-manifold, which enables us to use local convexity in local arguments. Then, we prove that for $C^\infty$-manifolds $M$ and $N$, the smooth singular complex of the diffeological space $C^\infty(M,N)$ is weakly equivalent to the ordinary singular complex of the topological space $\czero(M,N)$ under the hereditary $C^\infty$-paracompactness and semiclassicality conditions on $M$. We next generalize this result to sections of fiber bundles over a $C^\infty$-manifold $M$ under the same conditions on $M$. Further, we establish the Dwyer-Kan equivalence between the simplicial groupoid of smooth principal $G$-bundles over $M$ and that of continuous principal $G$-bundles over $M$ for a Lie group $G$ and a $C^\infty$-manifold $M$ under the same conditions on $M$, encoding the smoothing results for principal bundles and gauge transformations.\par
	For the proofs, we fully faithfully embed the category $C^{\infty}$ of $C^{\infty}$-manifolds into the category $\dcal$ of diffeological spaces and develop the smooth homotopy theory of diffeological spaces via a homotopical algebraic study of the model category $\dcal$ and the model category $\czero$ of arc-generated spaces, also known as $\Delta$-generated spaces. Then, the hereditary $C^\infty$-paracompactness and semiclassicality conditions on $M$ imply that $M$ has the smooth homotopy type of a cofibrant object in $\dcal$. This result can be regarded as a smooth refinement of the results of Milnor, Palais, and Heisey, which give sufficient conditions under which an infinite-dimensional topological manifold has the homotopy type of a $CW$-complex. We also show that most of the important $C^\infty$-manifolds introduced and studied by Kriegl, Michor, and their coauthors are hereditarily $C^\infty$-paracompact and semiclassical, and hence, results can be applied to them.
	\newpage
\end{abstract}

\maketitle
\setcounter{tocdepth}{2}
\tableofcontents



\section{Introduction}
This paper aims to develop a smooth homotopy theory of diffeological spaces and apply it to global analysis on infinite-dimensional $C^{\infty}$-manifolds by embedding $C^\infty$-manifolds fully faithfully into the category of diffeological spaces.
\par\indent
In Section 1.1, we formulate fundamental problems on $C^{\infty}$-manifolds, and in Section 1.2, we provide answers to these problems, as the main results on $C^{\infty}$-manifolds. In Section 1.3, we outline the results obtained herein on smooth homotopy for diffeological spaces, to which most of this paper is devoted, and explain how they yield the results in Section 1.2.
\subsection{Fundamental problems on $C^{\infty}$-manifolds}
Fr\"{o}licher, Kriegl, and Michor \cite{KM} established the foundation of infinite-dimensional calculus, which is called convenient calculus and is regarded as a prime candidate for the final theory of infinite-dimensional calculus. However, it has no efficient approach for solving one of the most critical problems: to investigate how many smooth maps exist between the given infinite-dimensional $C^{\infty}$-manifolds $M$ and $N$. Since the study of continuous maps between $M$ and $N$ is done by topological homotopy theory (or algebraic topology), we formulate the problem as follows:
\begin{itemize}
	\item[$(\rm a)$] When do the smooth homotopy classes of smooth maps between $M$ and $N$ bijectively correspond to the continuous homotopy classes of continuous maps ?
\end{itemize}
The following two problems are also important; Problem $(\rm b)$ is a generalization of Problem $(\rm a)$, and Problem $(\rm c)$ is closely related to Problems $(\rm a)$ and $(\rm b)$.
\begin{itemize}
	\item[$(\rm b)$] Let $p: E \longrightarrow M$ be a smooth fiber bundle. When do the vertical smooth homotopy classes of smooth sections of $E$ bijectively correspond to the vertical continuous homotopy classes of continuous sections ?
	\item[$(\rm c)$] Let $G$ be a Lie group. When do the isomorphism classes of smooth principal $G$-bundles over $M$ bijectively correspond to those of continuous principal $G$-bundles over $M$? Let $\pi: P \longrightarrow M$ be a smooth principal $G$-bundle. When do the isotopy classes of smooth gauge transformations of $P$ bijectively correspond to those of continuous gauge transformations of $P$?
\end{itemize}

If all $C^\infty$-manifolds involved are finite-dimensional, then the correspondences in Problems (a) and (b) are always bijective by the Steenrod approximation theorem (\cite[Section 6.7]{Steenrod}), which is one of the most basic results in differential topology. In more general settings, Problems $\rm (a)$, $(\rm b)$, and $(\rm c)$ were addressed in \cite{KMSC}, \cite{Wockel09}, and \cite{MW, Wockel05}, respectively. Roughly speaking, as answers to these questions, it has been shown that the correspondences in the questions are bijective, provided that $M$ is finite-dimensional. Precisely, M\"{u}ller and Wockel \cite{MW, Wockel05, Wockel09} did not work in convenient calculus but in Keller's $C^{\infty}_{c}$-theory; moreover, to prove the smoothing result of continuous gauge transformations, Wockel imposed even compactness condition on $M$ along with additional conditions on $G$ \cite[Proposition 1.20]{Wockel05}. However, no essential answer is known in the case where $M$ is infinite-dimensional since the existing approaches are essentially based on the finite dimensionality (or local compactness) of $M$ (cf. \cite[Section 1]{Wockel09} and \cite[Introduction]{MW}).
\par\indent
In the rest of this subsection, we more precisely formulate the fundamental problems mentioned above; we actually address the higher homotopical versions of Problems $(\rm a)$-$(\rm c)$ by observing that the involved relevant categories and functors can be enriched over the category $\scal$ of simplicial sets.
\par\indent
Throughout this paper, $C^{\infty}$-manifolds are ones in the sense of \cite[Section 27]{KM} unless stated otherwise, and $C^{\infty}$ denotes the category of (separated) $C^{\infty}$-manifolds (see Section 2.2). Lie groups are defined as groups in $C^{\infty}$ (\cite[p. 75]{Mac}). The underlying topological space $\widetilde{M}$ of a $C^{\infty}$-manifold $M$ is defined as the set $M$ endowed with the final topology for the smooth curves (\cite[27.4]{KM}). Then, we have the underlying topological space functor $\widetilde{\cdot}: C^{\infty} \longrightarrow \czero$, where $\czero$ is the category of arc-generated spaces and continuous maps (see Section 2.1).

Since the category $C^{\infty}$ is not closed under various categorical constructions, we fully faithfully embed $C^{\infty}$ into the category $\dcal$ of diffeological spaces (see Section 2.2). Recall that the underlying topological space $\widetilde{X}$ of a diffeological space $X$ is defined to be the set $X$ endowed with the final topology for the diffeology $D_X$. Then, we can see that the fully faithful embedding $C^\infty \longhookrightarrow \dcal$ and the underlying topological space functors for $C^\infty$ and $\dcal$ form the commutative diagram
\[
\begin{tikzcd}
C^{\infty} \arrow[hook]{rr} \arrow[swap]{rd}{\widetilde{\cdot}} & & \arrow{ld}{\widetilde{\cdot}} \dcal \\
 & \czero &
\end{tikzcd}
\]
consisting of functors that preserve finite products (Proposition \ref{dmfd}). Because of this fact and the convenient properties of $\dcal$ and $\czero$ (Section 2.1), we use the categories $\dcal$ and $\czero$ as convenient categories of smooth spaces and topological spaces, respectively.

Now, we shall discuss the $\scal$-enrichment of categories embedded into $\dcal$ and $\czero$. In \cite{origin}, we constructed the standard simplices $\Delta^{p}$ $(p \geq 0)$, which satisfy several good properties, and used them to introduce the adjoint pair
\[
|\ |_{\dcal} : \scal \rightleftarrows \dcal : S^{\dcal}
\]
of the realization and singular functors; this is a diffeological analog of the adjoint pair
\[
|\ | : \scal \rightleftarrows \czero : S
\]
of the topological realization and singular functors (see Section 2.3). Let $\mcal$ denote one of the categories $\dcal$ and $\czero$ and set
\[
(|\ |_{\mcal}, S^{\mcal}) =
\begin{cases*}
(|\ |_{\dcal}, S^{\dcal}) & \text{for $\mcal = \dcal$,} \\
(|\ |, S) & \text{for $\mcal = \czero$}.
\end{cases*}
\]
Since $\mcal$ is cartesian closed, $\mcal$ is itself an $\mcal$-category (\cite[Proposition 6.2.6]{Borceux}) and hence an $\scal$-category via $S^{\mcal}$. Further, every category endowed with a faithful functor into $\mcal$ has the canonical $\mcal$-category and hence $\scal$-category structure (see Section 7.1), and the underlying topological space functor $\widetilde{\cdot}:\dcal \longrightarrow \czero$ and its restrictions to subcategories are $\scal$-functors (see Lemma \ref{enrich} and Section 7.1). Using these facts, we formulate higher homotopical versions of Problems (a)-(c).\par
Since $\widetilde{\cdot}:\dcal \longrightarrow \czero$ is an $\scal$-functor and $C^\infty$ is an $\scal$-full subcategory of $\dcal$ via the fully faithful embedding $C^\infty\longhookrightarrow \dcal$, we have the $\scal$-functor
\[
	\widetilde{\cdot}:C^\infty \longrightarrow \czero,
\]
which induces the natural inclusion of function complexes
\[
	\widetilde{\cdot}:S^\dcal C^\infty(M,N)=S^\dcal \dcal(M,N) \longhookrightarrow S\czero(\widetilde{M},\widetilde{N}).
\]
Furthermore, since $\pi_{0} (\widetilde{\cdot})$ is just the natural map from the smooth homotopy classes of smooth maps to the continuous homotopy classes of continuous maps (see Section \ref{4.2}), the following is a precise and higher homotopical formulation of Problem $(\rm a)$.\vspace{0.2cm} \\
\textbf{Problem A.} Under what condition is the map $\widetilde{\cdot}: S^{\dcal}C^{\infty}(M, N) \longhookrightarrow S\czero(\widetilde{M}, \widetilde{N})$ a weak equivalence of simplicial sets ?\vspace{0.2cm}\\

Since the overcategory $\mcal / X$ has the canonical faithful functor into $\mcal$ for any $X \in \mcal$, $\mcal / X$ has the canonical $\mcal$-category and hence $\scal$-category structure. The hom-object $\mcal / X (D, E)$ of the $\mcal$-category $\mcal/X$ is often denoted by $\Gamma(X, E)$ in the case where $D = X$. Since $\widetilde{\cdot}:\dcal/X \longrightarrow \czero/\widetilde{X}$ is an $\scal$-functor for any $X\in \dcal$ and $C^\infty/M$ is an $\scal$-full subcategory of $\dcal/M$ via the fully faithful embedding $C^\infty/M \longhookrightarrow \dcal/M$, we have the $\scal$-functor
\[
	\widetilde{\cdot}: C^\infty/M \longrightarrow \czero/\widetilde{M},
\]
which induces the natural inclusion of function complexes
\[
	\widetilde{\cdot}:S^\dcal \Gamma(M,E) \longhookrightarrow S\Gamma(\widetilde{M},\widetilde{E}).
\]
Furthermore, since $\pi_{0} (\widetilde{\cdot})$ is just the natural map from the vertical smooth homotopy classes of smooth sections to the vertical continuous homotopy classes of continuous sections (Section \ref{6.3}), the following is a precise and higher homotopical formulation of Problem $(\rm b)$.\vspace{0.2cm}\\
\textbf{Problem B.} Let $p : E \longrightarrow B$ be a smooth fiber bundle of $C^{\infty}$-manifolds. Under what condition is the map $\widetilde{\cdot}: S^{\dcal}\Gamma (M, E) \longhookrightarrow S\Gamma(\widetilde{M}, \widetilde{E})$ a weak equivalence of simplicial sets ?\vspace{0.2cm}\\

For an object $X$ of $\mcal$ and a group $G$ in $\mcal$, $\mathsf{P}\mcal G / X$ denotes the category of principal $G$-bundles over $X$ in $\mcal$ (see Definition \ref{bdle} for the notions of smooth and continuous principal bundles referred to here). Since $\mathsf{P} \mcal G / X$ has the canonical faithful functor into $\mcal$, $\mathsf{P} \mcal G / X$ is an $\mcal$-category and hence an $\scal$-category. Moreover, it is shown that the underlying topological space functor defines the simplicial functor
\[
\widetilde{\cdot}: \mathsf{P}\dcal G / X \longrightarrow \mathsf{P} \czero \widetilde{G} / \widetilde{X}
\]
(see Lemma \ref{bdleforget} and Remark \ref{bdlerem}). If $G$ is a Lie group, then the category $\mathsf{P} C^{\infty} G/ M$ of principal $G$-bundles over $M$ in $C^{\infty}$ is regarded as a $\dcal$-category and hence an $\scal$-category via the isomorphism $\mathsf{P} C^\infty G/M \cong \mathsf{P}\dcal G / M$ (see Remark \ref{4properties}(3)). Therefore, we have the simplicial functor
\[
\widetilde{\cdot} : \mathsf{P} C^{\infty} G / M \longrightarrow \mathsf{P}\czero \widetilde{G} / \widetilde{M}.
\]
We thus consider the following problem. See Section 7.4 for the definition of a Dwyer-Kan equivalence between simplicial categories.\vspace{0.2cm}\\
\textbf{Problem C} Let $G$ be a Lie group. Under what condition is the simplicial functor
\[
\widetilde{\cdot} : \mathsf{P} C^{\infty} G / M \longrightarrow \mathsf{P} \czero \widetilde{G} / \widetilde{M}
\]
a Dwyer-Kan equivalence? \vspace{0.2cm}

Now, we explain that Problem ${\rm C}$ is a precise formulation of a higher homotopical version of Problem $(\rm c)$. Observing that $\mathsf{P}\mcal H/X$ is an $\mcal$-groupoid, we define the gauge group ${\rm Gau}_\mcal (P)$ of a principal $H$-bundle $\pi:P\longrightarrow X$ in $\mcal$ to be the automorphism group ${\rm Aut}_{\mathsf{P}\mcal H/X}(P)$ (see Section 7.2). If $\mathsf{P}\mcal H/X$ is regarded as an $\scal$-groupoid (or simplicial groupoid) via $S^\mcal$, then the automorphism group ${\rm Aut}_{\mathsf{P}\mcal H/X}(P)$ is just $S^\mcal{\rm Gau}_\mcal(P)$. Thus, that the simplicial functor $\widetilde{\cdot} : \mathsf{P} C^{\infty} G / M = \mathsf{P}\dcal G/M \longrightarrow \mathsf{P}\czero \widetilde{G} / \widetilde{M}$ between simplicial groupoids is a Dwyer-Kan equivalence is equivalent to the following:
\begin{itemize}
	\item Every continuous principal $\widetilde{G}$-bundle over $\widetilde{M}$ admits a unique (up to isomorphism) smooth principal $G$-bundle structure over $M$.
	\item For any smooth principal $G$-bundle $P$ over $M$, the natural inclusion
	\[
	S^\dcal{\rm Gau}_{C^\infty}(P) = S^{\dcal}{\rm Gau}_{\dcal}(P) \longhookrightarrow S {\rm Gau}_{\czero}(\widetilde{P})
	\]
	is a weak equivalence of simplicial sets.
\end{itemize}
(See Section 7.4).
\subsection{Main results on $C^{\infty}$-manifolds}\label{1.2}
In this subsection, we state the main results on $C^{\infty}$-manifolds, which constitute answers to Problems A, B, and C. To this end, we begin by introducing the basic notions on $C^{\infty}$-manifolds.
\par\indent
A $C^{\infty}$-manifold $M$ is called {\sl hereditarily $C^{\infty}$-paracompact} if any open submanifold of $M$ is $C^{\infty}$-paracompact (see Definition \ref{partition}). A $C^{\infty}$-manifold $M$ is called {\sl semiclassical} if $M$ admits an atlas $\{ (U_{\alpha}, u_{\alpha}) \}_{\alpha \in A}$ such that $u_{\alpha} (U_{\alpha})$ and $u_{\alpha}(U_{\alpha} \cap U_{\beta})$ are open in the model vector space $E_{\alpha}$ with respect to the locally convex topology for any $\alpha, \beta \in A$. (See Section 11.2 for the notion of a classical $C^\infty$-manifold.)\par\indent
The following theorem answers Problem A.
\begin{thm}\label{mapsmoothing}
	Let $M$ and $N$ be $C^{\infty}$-manifolds. If $M$ is hereditarily $C^{\infty}$-paracompact and semiclassical, then the natural inclusion
	\[
	\sdcal C^{\infty}(M, N) \longhookrightarrow S \czero (\widetilde{M}, \widetilde{N})
	\]
	is a weak equivalence of simplicial sets.
\end{thm}
The following theorem, which is a smoothing result of continuous sections, answers Problem B, generalizing Theorem \ref{mapsmoothing}.
\begin{thm}\label{sectionsmoothing}
	Let $p: E \longrightarrow M$ be a smooth fiber bundle of $C^{\infty}$-manifolds. If $M$ is hereditarily $C^{\infty}$-paracompact and semiclassical, then the natural inclusion
	\[
	S^{\dcal}\Gamma (M, E) \longhookrightarrow S\Gamma (\widetilde{M}, \widetilde{E})
	\]
	is a weak equivalence of simplicial sets.
\end{thm}
Lastly, we provide an answer to Problem C.
\begin{thm}\label{DK}
	Let $M$ be a $C^{\infty}$-manifold and $G$ a Lie group. If $M$ is hereditarily $C^{\infty}$-paracompact and semiclassical, then the functor
	\[
	\widetilde{\cdot}: \mathsf{P} C^{\infty}G / M \longrightarrow \mathsf{P} \czero \widetilde{G} / \widetilde{M}
	\]
	is a Dwyer-Kan equivalence of simplicial groupoids.
\end{thm}
Kriegl, Michor, and their coauthors introduced and studied many $C^\infty$-manifolds which play vital roles in global analysis and differential geometry (\cite[Chapter IX and Section 47]{KM}). In this paper, we show that most of those $C^\infty$-manifolds are hereditarily $C^\infty$-paracompact and semiclassical (see Section 11.4); therefore, Theorems \ref{mapsmoothing}-\ref{DK} apply to them.
\begin{rem}\label{smparacpt}
	In spite of the importance of $C^\infty$-paracompactness in global analysis (\cite[Chapter III]{KM}), few $C^\infty$-paracompact $C^\infty$-manifolds are known except finite dimensional $C^\infty$-manifolds and several kinds of convenient vector spaces. We thus develop the basics of hereditary $C^\infty$-paracompactness (Appendix C and Section 11.3), and then establish the hereditary $C^\infty$-paracompactness of important $C^\infty$-manifolds such as the $C^\infty$-manifolds $\cfra^\infty(M,N)$ and $C^\omega(M,N)$ of maps and the $C^\infty$-manifolds $B(M,N)$ and $B^\omega(M,N)$ of submanifolds (see Section 11.4); our approach also applies to paracompact Hilbert manifolds and direct limit manifolds of finite dimensional $C^\infty$-manifolds. Note also that a $C^\infty$-manifold $X$ is hereditarily $C^\infty$-paracompact if and only if every submanifold of $X$ is $C^\infty$-paracompact (Corollary \ref{submfd}).
\end{rem}
\subsection{Smooth homotopy theory of diffeological spaces}
Since $C^{\infty}$ is a full subcategory of $\dcal$ (Proposition \ref{dmfd}) and $\dcal$ is a model category which has various convenient properties, we develop smooth homotopy theory of diffeological spaces and prove Theorems \ref{mapsmoothing}-\ref{DK} by applying the theory to $C^{\infty}$-manifolds. Precisely, in this subsection, we outline the smooth homotopical results on diffeological spaces, and explain how Theorems \ref{mapsmoothing}-\ref{DK} are deduced from them.\par
Our smooth homotopy theory primarily concerns simplicial categorical properties of objects and morphisms of $\dcal$ and $\czero$ (see Problems A-C and Theorems \ref{mapsmoothing}-\ref{DK}). However, the model categorical study of $\dcal$ and $\czero$ plays a crucial role in many aspects of the theory. Thus, we begin by outlining the model categorical relations between $\scal$, $\dcal$, and $\czero$.\par
In \cite{origin}, a compactly generated model structure on the category $\dcal$ of diffeological spaces was constructed. In the study, the adjoint pair
$$
|\ |_\dcal : \scal \rightleftarrows \dcal : S^\dcal
$$
of the realization and singular functors, and the adjoint pair
$$
\widetilde{\cdot} : \dcal \rightleftarrows \czero: R
$$
of the underlying topological space functor and its right adjoint were introduced. The composite of these adjoint pairs is just the adjoint pair
$$
|\ | : \scal \rightleftarrows \ccal^0 : S
$$
of the topological realization and singular functors (see Section 2.3).
\par\indent
In the following theorem, we show that the model categories $\scal$, $\dcal$, and $\ccal^0$ are Quillen equivalent via $(|\ |_\dcal, S^\dcal)$ and $(\tilde{\cdot}, R)$. See \cite{Hi, MP} for model categories and Quillen equivalences, \cite{K, GJ, MP} for the model structure of $\scal$, and Remark \ref{arc} for the model structure of $\ccal^0$.
\begin{thm}\label{Quillenequiv}
	\begin{itemize}
		\item[$(1)$]
		$|\ |_{\dcal}: \scal \rightleftarrows \dcal: S^{\dcal}$ is a pair of Quillen equivalences.
		\item[$(2)$]
		$\tilde{\cdot}: \dcal \rightleftarrows \ccal^{0}: R$ is a pair of Quillen equivalences.
	\end{itemize}
\end{thm}
Theorem \ref{Quillenequiv} is the starting point of our homotopical study of diffeological spaces, which shows that the Quillen homotopy categories of $\dcal$ and $\ccal^0$ are equivalent via the total derived functors of $\tilde{\cdot}$ and $R$; however, it does not guarantee that $S^\dcal \dcal (X,Y)$ is homotopically equivalent to $S\ccal^0 (\tilde{X}, \tilde{Y})$ for $X,Y \in \dcal$. In fact, the natural inclusion $S^\dcal(X,Y)\hookrightarrow S\czero(\widetilde{X},\widetilde{Y})$ need not be a weak equivalence (see \cite[Examples 3.12 and 3.20]{CW} and Appendix A). Thus, to establish a smoothing theorem for continuous maps, we must give a sufficient condition under which
\[
\widetilde{\cdot{}} : S^{\dcal} \dcal (X, Y) \longhookrightarrow S\czero (\widetilde{X}, \widetilde{Y})
\]
is a weak equivalence in $\scal$. For this, we introduce and study two subclasses of $\dcal$.

One is the subclass $\wcal_{\dcal}$ defined by
\[
\wcal_{\dcal} = \{A \in \dcal \ | \ A\ \text{has the $\dcal$-homotopy type of a cofibrant diffeological space} \},
\]
where the notion of a $\dcal$-homotopy type (or smooth homotopy type) is defined using the unit interval $I\ (\cong \Delta^{1})$ in the obvious manner (Section 2.3 and \cite[Section 2.4]{origin}). The class $\wcal_{\dcal}$ is a diffeological version of Milnor's class $\wcal$ of topological spaces that have the homotopy type of a $CW$-complex \cite{Mi}. 
\par\indent
The other is the subclass $\vcal_{\dcal}$ defined by
\[
\vcal_\dcal = \bigl \{ A \in \dcal\ |\ id : A \longrightarrow R\widetilde{A}\ \text{ is a weak equivalence in }\ \dcal \bigr \};
\]
where the map $id: A \longrightarrow R\widetilde{A}$ is the unit of the adjoint pair $(\, \widetilde{\cdot}, R)$, which is set-theoretically the identity map (see Remark \ref{convenrem}(2)).
\par\indent
For the description of the following corollary, we also introduce the subclass $\wcal_{\czero}$ of $\czero$ defined by
$$
\wcal_{\ccal^0} = \{X \in \ccal^0 \ | \ X \ \text{has the homotopy type of a cofibrant arc-generated space} \}.
$$
It should be noted that $\wcal_{\czero}$ is just the intersection $\wcal\, \cap\, \czero$ (see Remark \ref{arc} and \cite[Section 2.2]{origin}). Also note that the singular complex $S^{\dcal}A$ is a subcomplex of $S\widetilde{A}$ (Lemma \ref{prism}(1)) and let $\pi^{\dcal}_{p}(A, a)$ denote the smooth homotopy group of a pointed diffeological space $(A, a)$ (\cite{CW}, \cite{IZ}).
\begin{cor}\label{W} \begin{itemize} \item[$(1)$] If $A$ is in $\wcal_{\dcal}$, then $\widetilde{A}$ is in $\wcal_{\czero}$.
		\item[$(2)$] If $A$ is in $\wcal_\dcal$, then $A$ is in $\vcal_\dcal$.
		\item[$(3)$] The following conditions are equivalent:
		\begin{itemize}
			\item[$\mathrm{(i)}$] $A$ is in $\vcal_{\dcal}$;
			\item[$\mathrm{(ii)}$] $\text{The inclusion } S^\dcal A \longhookrightarrow S\widetilde{A} \text{ is a weak equivalence in } \scal$;
			\item[$\mathrm{(iii)}$] The natural homomorphism
			$$
			\pi^{\dcal}_{p}(A, a) \longrightarrow \pi_{p}(\widetilde{A}, a)
			$$
			is an isomorphism for any $a \in A$ and any $p \geq 0$.
		\end{itemize}
	\end{itemize}
\end{cor}
To understand why the class $\wcal_{\dcal}$ is needed, see Remark \ref{WD}. See also Appendix A for diffeological spaces which are not in $\wcal_{\dcal}$ and $\vcal_{\dcal}$.

We can now state the smoothing theorem for continuous maps.

\begin{thm}\label{dmapsmoothing}
	Let $A$ and $X$ be diffeological spaces. If $A$ is in $\wcal_{\dcal}$ and $X$ is in $\vcal_{\dcal}$, then the natural inclusion
	$$
	\sdcal\dcal(A,X)\longhookrightarrow S\czero(\tilde{A},\tilde{X})
	$$
	is a weak equivalence in $\scal$.
\end{thm}
Theorem \ref{dmapsmoothing} generalizes to the following smoothing theorem for continuous sections. See Section 5 for the notion of a $\mathcal{D}$-numerable $F$-bundle in $\dcal$.
\begin{thm}\label{dsectionsmoothing}
	Let $p:E\longrightarrow X$ be a $\mathcal{D}$-numerable $F$-bundle in $\mathcal{D}$. If $X$ is in $\mathcal{W}_\mathcal{D}$ and $F$ is in $\mathcal{V}_\mathcal{D}$, then the natural inclusion
	\[
	S^\mathcal{D}\Gamma(X,E)\longhookrightarrow S\Gamma(\widetilde{X},\widetilde{E})
	\]
	is a weak equivalence in $\scal$.
\end{thm}
Next, we state a smoothing theorem for principal bundles and gauge transformations. Let $\mathcal{M}$ denote one of the categories $\mathcal{D}$ and $\czero$. $(\mathsf{P}\mathcal{M} G/X)_{\rm num}$ denotes the full subcategory of $\mathsf{P}\mathcal{M} G/X$ consisting of $\mathcal{M}$-numerable principal $G$-bundles over $X$ in $\mathcal{M}$ (Definition \ref{bdle}), which has the canonical $\mathcal{M}$-full subcategory and hence $\mathcal{S}$-full subcategory structure of $\mathsf{P}\mathcal{M} G/X$. For a diffeological space $X$ and a diffeological group $G$, the $\mathcal{S}$-functor $\widetilde{\cdot}:\mathsf{P}\mathcal{D} G/X \longrightarrow \mathsf{P}\czero \widetilde{G}/\widetilde{X}$ (see Section 1.1) obviously restricts to the $\mathcal{S}$-functor
\[
\widetilde{\cdot}:(\mathsf{P}\mathcal{D} G/X)_{\rm num} \longrightarrow (\mathsf{P}\czero \widetilde{G} / \widetilde{X})_{\rm num}.
\]
\begin{thm}\label{dDK}
	Let $X$ be a diffeological space and $G$ a diffeological group. If $X$ is in $\wcal_{\dcal}$ and $G$ is in $\vcal_{\dcal}$, then the functor
	\[
	\widetilde{\cdot} : (\mathsf{P} \dcal G / X)_{\rm num} \longrightarrow (\mathsf{P} \czero \widetilde{G} / \widetilde{X})_{\rm num}
	\]
	is a Dwyer-Kan equivalence of simplicial groupoids.
\end{thm}
To deduce Theorems \ref{mapsmoothing}-\ref{DK} from Theorems \ref{dmapsmoothing}-\ref{dDK}, we must establish the following two theorems. See Definition \ref{partition} for a $\dcal$-numerable covering and Definition \ref{loccontr} for a locally contractible diffeological space.
\begin{thm}\label{hcofibrancy}
	Let $X$ be a diffeological space and $U = \{U_{\alpha} \}_{\alpha \in A}$ a $\mathcal{D}$-numerable covering of $X$ by subsets. If $U_{\sigma}\ (:= \underset{\alpha \in \sigma}{\cap} U_{\alpha}) $ is in $\wcal_{\dcal}$ for any nonempty finite $\sigma \subset A$, then $X$ is in $\wcal_{\dcal}$.
\end{thm}

\begin{thm}\label{V}
	Every locally contractible diffeological space is in $\vcal_{\dcal}$.
\end{thm}

Theorems \ref{hcofibrancy}-\ref{V} imply the following important facts:
\begin{itemize}
	\item Every hereditarily $C^\infty$-paracompact, semiclassical $C^\infty$-manifold is in $\wcal_{\dcal}$ (Theorem \ref{mfdhcofibrancy}).
	\item Every $C^\infty$-manifold is in $\vcal_{\dcal}$ (Theorem \ref{mfdV}).
\end{itemize}
By these facts, Theorems \ref{dmapsmoothing}-\ref{dDK} easily imply Theorems \ref{mapsmoothing}-\ref{DK} (see Section 11.1).\par
The smooth homotopical resutls outlined here are also applicable to the theory of orbifolds \cite{IKZ}, which we shall discuss in a subsequent study in this direction.
The following remark relates to the relationship between the smooth homotopy theory discussed in this paper and topological homotopy theory.
\begin{rem}\label{MPH}
	\begin{itemize}
		\item[$\rmone$] Theorem \ref{hcofibrancy} is a diffeological version of a theorem of tom Dieck \cite[Theorem 4]{tom}.
		\item[$\rmtwo$] Theorem \ref{mfdhcofibrancy} mentioned above and its variant (Proposition \ref{countable}) can be regarded as smooth refinements of  the results of Palais, Heisey, and Milnor on the homotopy types of infinite-dimensional topological manifolds (see Remark \ref{topmfd}).
		\item[$\rmthree$] Since the underlying topological space $\widetilde{M}$ of a  hereditarily $C^\infty$-paracompact, semiclassical $C^{\infty}$-manifold $M$ has the homotopy type of a $CW$-complex (Theorem \ref{mfdhcofibrancy} and Corollary \ref{W}(1)), Theorem \ref{mapsmoothing} enables us to study the homotopy of $S^{\dcal}C^{\infty}(M, N)$ via algebraic topological methods.
	\end{itemize}
\end{rem}
We end this subsection with the following remark on the relation between the simplicial category structure and the model structure on $\mcal=\dcal,\czero$.
\begin{rem}\label{svsm}
	The categories $\dcal$ and $\czero$ have simplicial category structures and model structures, both of which define the homotopical notions on objects and morphisms (see Section 4.3).\par
	As mentioned in the beginning of this subsection, we are primarily concerned with simplicial categorical properties of objects and morphisms of $\dcal$ and $\czero$. However, the methods and results of model category theory are needed in many aspects of our smooth homotopy theory. Since $\czero$ and its overcategory $\czero/X$ are simplicial model categories (Corollaries \ref{notenrich} and \ref{notenrich/X}), we can do homotopy theory in these categories without considering the differences between simplicial categorical notions and model categorical notions (see Propositions \ref{cfctcpx} and \ref{cfctcpx/X} and their proofs). However, since neither $\dcal$ nor its overcategory $\dcal/X$ is a simplicial model category (Corollaries \ref{notenrich}, \ref{notenrich/X} and Remark \ref{inevitable}), we need subtle arguments comparing the function complexes with the homotopy function complexes for these categories (see Theorems \ref{dfctcpx} and \ref{dfctcpx/X} and their proofs).
\end{rem}
\subsection{Notation and terminology}
We follow the notation and terminology of \cite{origin}, much of which is reviewed in Section 2 of this paper. Here, we recall some of the most important ones while introducing new notations.\par
The most important categories are the following:
\begin{itemize}
	\item[] The category $\dcal$ of diffeological spaces (Section 2.1),
	\item[] The category $\czero$ of arc-generated spaces (Section 2.1),
	\item[] The category $\scal$ of simplicial sets.
\end{itemize}
In this paper, $\mcal$ denotes one of the categories $\dcal$ and $\czero$. $|\ |_{\mcal} : \scal \rightleftarrows \mcal : S^{\mcal}$ denotes the adjoint pair of the realization and singular functors for $\mcal$; explicitly,
\[
(|\ |_{\mcal}, S^{\mcal}) = 
\begin{cases*}
(|\ |_{\dcal}, S^{\dcal}) & \text{for $\mcal = \dcal$,} \\
(|\ |, S) & \text{for $\mcal = \czero$}
\end{cases*}
\]
(see Section 2.3).\par
The categories $\dcal$ and $\czero$ have the obvious underlying set functors. Thus, we use the terminology of \cite[Sections 8.7-8.8]{FK}, which relates to initial and final structures with respect to the underlying set functor, to deal with the categories $\dcal$ and $\czero$ simultaneously (see Proposition \ref{conven} and Remarks \ref{convenrem}(1) and \ref{suitable}).\par
For a concrete category $\xfra$, the following notation is used in this paper:
\begin{itemize}
	\item[(1)] For $X\in \xfra$, $1_X$ (or $1$) denotes the identity morphism of $X$.
	\item[(2)] For $X_1$, $X_2 \in \xfra$ having the same underlying set, $id : X_1 \longrightarrow X_2$ is the set-theoretic identity map or the morphism whose underlying set-theoretic map is the identity map.
\end{itemize}\par
Our basic reference for enriched category theory is \cite[Chapter 6]{Borceux}. An $\scal$-category and an $\scal$-functor are often called a simplicial category and a simplicial functor respectively. However, by a diffeological category, we mean a category object in $\dcal$, not a $\dcal$-category (see Section 9.1).\par
Our basic references for model category theory are \cite{Hi,MP}. For a model category $\ecal$, the symbols $\ecal_{c}$, $\ecal_{f}$, and $\ecal_{cf}$ denote the full subcategories of $\ecal$ consisting of cofibrant objects, fibrant objects, and fibrant-cofibrant objects, respectively.\par
The symbol $\emptyset$ (resp. $\ast$) is used to denote an initial (resp. terminal) object of a category.\par
For $\czero$-homotopies $h$, $k:X\times I \longrightarrow Y$ with $h(\cdot, 1)= k(\cdot,0)$, the (vertical) composite of $h$ and $k$ is defined in the obvious manner (see \cite[p. 273]{Mac}). On the other hand, we always compose two $\dcal$-homotopies $h$, $k:X\times I \longrightarrow Y$ with $h(\cdot,1)=k(\cdot,0)$ smoothly using a cut-off function (\cite[Definition 3.1]{CW}).\par
We end this subsection with the topological notation. For a subset $S$ of a topological space $X$, $\bar{S}$ and $S^\circ$ denote the closure and the interior of $S$ respectively. For a continuous function $f:X \longrightarrow \mathbb{R}$, we define $\mathrm{carr}\ f$ and $\mathrm{supp}\ f$ by
\[
\mathrm{carr}\ f = \{x\in X |\ f(x)\neq 0 \} \ \mathrm{and}\ \mathrm{supp}\ f = \overline{ \mathrm{carr}\ f}
\]
(cf. \cite[p. 165]{KM}).

\subsection{Organization of the paper}
This paper is organized as follows.
\par\indent
In Section 2, we recall the basic notions and results of diffeological spaces and arc-generated spaces from \cite{origin} and observe that the category $C^{\infty}$ can be fully faithfully embedded into $\dcal$. We also see that $|\ |_{\dcal}: \scal \rightleftarrows \dcal : S^{\dcal}$ and $\widetilde{\cdot} : \dcal \rightleftarrows \czero : R$ are Quillen pairs.

Sections 3 and 4 contain the main ideas of the smooth homotopy theory of diffeological spaces. In Section 3, we prove Theorem \ref{Quillenequiv} and Corollary \ref{W}. In Section 4, we study the simplicial category structures on $\czero$ and $\dcal$. Then, we compare function complexes and homotopy function complexes for $\mcal = \czero, \dcal$, which enables us to prove Theorem \ref{dmapsmoothing} by replacing the function complexes with the homotopy function complexes. Since $\dcal$ is not a simplicial model category (Corollary \ref{notenrich}(2)), the arguments for $\dcal$, which use Theorem \ref{Quillenequiv}, are much more subtle than those for $\czero$.
\par\indent
Sections 5-7 contain applications and generalizations of the results established in Section 4. In Section 5, we recall the basic facts on the classification theory of (smooth and continuous) principal bundles and use Theorem \ref{dmapsmoothing} to prove that under the assumption as in Theorem \ref{dDK}, the functor $\widetilde{\cdot}: (\mathsf{P}\dcal G/ X)_{\rm num} \longrightarrow (\mathsf{P}\czero \widetilde{G}/\widetilde{X})_{\rm num}$ induces the bijection on the isomorphism classes of objects (Theorem \ref{dbdlesmoothing}). In Section 6, we compare function complexes and homotopy function complexes for $\mcal / X\ (= \czero/ X , \dcal/ X),$ generalizing several results in Section 4. This leads to the proof of Theorem \ref{dsectionsmoothing}. In Section 7, we prove Theorem \ref{dDK} using Theorem \ref{dsectionsmoothing} along with Theorem \ref{dbdlesmoothing}.
\par\indent
Sections 8-10 are devoted to the proofs of Theorems \ref{hcofibrancy} and \ref{V}, which enable us to apply smooth homotopical results of diffeological spaces to $C^\infty$-manifolds. In Section 8, we introduce and study the two kinds of notions of diffeological polyhedra for the proofs of Theorems \ref{hcofibrancy} and \ref{V}. In Section 9, we first establish a diffeological version of \cite[Proposition 4.1]{Seg} (Proposition \ref{dSegal}) using the results in Section 8, and then introduce and study the notion of a Hurewicz cofibration in $\dcal$. Using the results of these studies, we prove Theorem \ref{hcofibrancy}. In Section 10, we introduce the notion of a locally contractible diffeological space and prove Theorem \ref{V}.
\par\indent
In Section 11, we first establish two theorems (Theorems \ref{mfdhcofibrancy} and \ref{mfdV}) from Theorems \ref{hcofibrancy} and \ref{V}. Then, Theorems \ref{mapsmoothing}-\ref{DK} follow immediately from Theorems \ref{dmapsmoothing}-\ref{dDK}. We also precisely investigate semiclassicality and hereditary $C^{\infty}$-paracompactness conditions, showing that many important $C^{\infty}$-manifolds satisfy these two conditions.

\hspace{2.5cm}\hrulefill\hspace{3.5cm} \par
The following chart indicates the interdependencies of the sections.
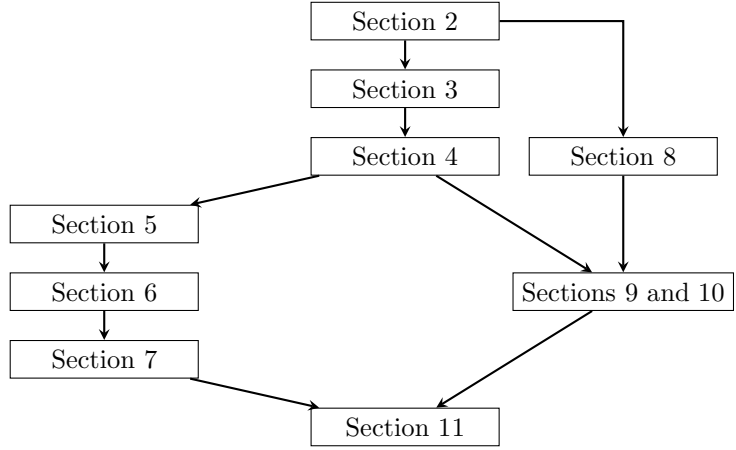
\begin{center}
	\begin{tikzpicture}[node distance=0.9cm]
	\node (sec2) [startstop] {Section 2};
	\node (sec3) [startstop, below of =sec2] {Section 3};
	\node (sec4) [startstop, below of =sec3] {Section 4};
	\node (sec5) [startstop, below of =sec4, xshift=-4cm] {Section 5};
	\node (sec6) [startstop, below of =sec5] {Section 6};
	\node (sec7) [startstop, below of =sec6] {Section 7};
	\node (sec11) [startstop, below of =sec7, xshift=4cm] {Section 11};
	\node (sec8) [startstop, right of =sec4, xshift=2cm] {Section 8};
	\node (sec9and10) [startstop, right of =sec6, xshift=6cm] {Sections 9 and 10};
	
	\draw [arrow] (sec2) -- (sec3);
	\draw [arrow] (sec3) -- (sec4);
	\draw [arrow] (sec5) -- (sec6);
	\draw [arrow] (sec6) -- (sec7);
	\draw [arrow] (sec8) -- (sec9and10);
	\draw [arrow] (sec2) -| (sec8);
	\draw [arrow] (sec4) -- (sec5);
	\draw [arrow] (sec7) -- (sec11);
	\draw [arrow] (sec9and10) -- (sec11);
	\draw [arrow] (sec4) -- (sec9and10);
	
	\end{tikzpicture}
\end{center}
\section{Diffeological spaces, arc-generated spaces, and $C^{\infty}$-manifolds}
We recall the convenient properties of diffeological spaces and arc-generated spaces, and then show that the category $C^{\infty}$ of (separated) $C^{\infty}$-manifolds can be fully faithfully embedded into the category $\dcal$ of diffeological spaces. We also make a review on the compactly generated model structure on $\dcal$ (\cite{origin}) and observe that the adjoint pairs $|\ |_{\dcal} : \scal \rightleftarrows \dcal : S^{\dcal}$ and $\widetilde{\cdot} : \dcal \rightleftarrows \czero : R$ are Quillen pairs.
\subsection{Categories $\dcal$ and $\czero$}
In this subsection, we recall the definition of a diffeological space and summarize the basic properties of the categories $\dcal$ and $\czero$, and the underlying topological space functor $\widetilde{\cdot} : \dcal \longrightarrow \czero$.
\par\indent
Let us begin with the definition of a diffeological space. A {\sl parametrization} of a set $X$ is a (set-theoretic) map $p: U \longrightarrow X$, where $U$ is an open subset of $\rbb^{n}$ for some $n$.
\begin{defn}\label{diffeological}
	\begin{itemize}
		\item[(1)] A {\sl diffeological space} is a set $X$ together with a specified set $D_X$ of parametrizations of $X$ satisfying the following conditions:
		\begin{itemize}
			\item[(i)](Covering)  Every constant parametrization $p:U\longrightarrow X$ is in $D_X$.
			\item[(ii)](Locality) Let $p :U\longrightarrow X$ be a parametrization such that there exists an open cover $\{U_i\}$ of $U$ satisfying $p|_{U_i}\in D_X$. Then, $p$ is in $D_X$.
			\item[(iii)](Smooth compatibility) Let $p:U\longrightarrow X$ be in $D_X$. Then, for every $n \geq 0$, every open set $V$ of $\rbb^{n}$ and every smooth map $F  :V\longrightarrow U$, $p\circ F$ is in $D_X$.
		\end{itemize}
		The set $D_X$ is called the {\sl diffeology} of $X$, and its elements are called {\sl plots}.
		\item[(2)] Let $X=(X,D_X)$ and $Y=(Y,D_Y)$ be diffeological spaces, and let $f  :X\longrightarrow Y$ be a (set-theoretic) map. We say that $f$ is {\sl smooth} if for any $p\in D_X$, \ $f\circ p\in D_Y$.
	\end{itemize}
\end{defn}
The convenient properties of $\dcal$ are summarized in the following proposition. Recall that a topological space $X$ is called {\sl arc-generated} if its topology is final for the continuous curves from $\rbb$ to $X$ and that $\czero$ denotes the category of arc-generated spaces and continuous maps. See \cite[pp. 230-233]{FK} for initial and final structures with respect to the underlying set functor.
\begin{prop}\label{conven}
	\begin{itemize}
		\item[$(1)$] The category ${\dcal}$ has initial and final structures with respect to the underlying set functor. In particular, ${\dcal}$ is complete and cocomplete.
		\item[$(2)$] The category $\mathcal{D}$ is cartesian closed.
		\item[$(3)$] The underlying set functor $\dcal \longrightarrow Set$
		is factored as the underlying topological space functor
		$\widetilde{\cdot}:\dcal \longrightarrow \czero$
		followed by the underlying set functor
		$\czero \longrightarrow Set$.
		Further, the functor
		$\widetilde{\cdot}:\dcal \longrightarrow \czero$
		has a right adjoint
		$R:\czero \longrightarrow \dcal$.
	\end{itemize}
	\begin{proof}
		See \cite[p. 90]{CSW}, \cite[pp. 35-36]{IZ}, and \cite[Propositions 2.1 and 2.10]{origin}.
	\end{proof}
\end{prop}
The following remark relates to Proposition \ref{conven}
\begin{rem}\label{convenrem}
	\begin{itemize}
	\item[$\rmone$] Let $\xfra$ be a concrete category (i.e., a category equipped with a faithful functor to Set); the faithful functor $\xfra \longrightarrow Set$ is called the underlying set functor. See \cite[Section 8.8]{FK} for the notions of an $\xfra$-embedding, an $\xfra$-subspace, an $\xfra$-quotient map, and an $\xfra$-quotient space. $\dcal$-subspaces and $\dcal$-quotient spaces are usually called diffeological subspaces and quotient diffeological spaces, respectively.
	\item[$\rmtwo$] For Proposition \ref{conven}(3), recall that the underlying topological space $\tilde{A}$ of a diffeological space $A = (A, D_A)$ is defined to be the set $A$ endowed with the final topology for $D_A$ and that $R$ assigns to an arc-generated space $X$ the set $X$ endowed with the diffeology
	$$
	D_{RX} = \text{\{continuous parametrizations of } X \text{\}.}
	$$
	Then, we can easily see that $\tilde{\cdot} \circ R = Id_{\ccal^0}$ and that the unit $A \longrightarrow R\widetilde{A}$ of the adjoint pair $(\widetilde{\cdot}, R)$ is set-theoretically the identity map.
	\item[$\rmthree$] Iglesias-Zemmour and many authors following him used the categorical convenience of $\dcal$ to deal with various differential-geometric and global-analytic problems (see, eg, \cite{IglesiasM, IZ, IKZ}).
	\end{itemize}
\end{rem}
The following remark explains that $\czero$ is the most suitable category as a target category of the underlying topological space functor for diffeological spaces.
\begin{rem}\label{suitable}
	The category $\czero$ has the following good properties:
	\begin{itemize}
		\item[{\rm (1)}] $\czero$ has initial and final structures with respect to the underlying set functor. In particular, $\czero$ is complete and cocomplete (\cite[Proposition 2.6]{origin}).
		\item[{\rm (2)}] $\czero$ is cartesian closed (\cite[Proposition 2.9]{origin}).
		\item[{\rm (3)}] A topological space $X$ is in $\czero$ if and only if $X$ is the underlying topological space of some diffeological space (Remark \ref{convenrem}(2)).
		\item[{\rm (4)}] The underlying topological space functor $\widetilde{\cdot} : \dcal \longrightarrow \czero$ preserves finite products (\cite[Proposition 2.13]{origin}).
	\end{itemize}

The category $\tcal$ of topological spaces does not have property $(2)$, $(3)$, or $(4)$. The most important property is property $\rmtwo$, which is essential for the proof of property $\rmfour$. However, other well-known categories of topological spaces having property $\rmtwo$ usually consist of topological spaces satisfying some separation axiom (\cite[Chapter 5]{MayAT}, \cite{GZ}), and hence cannot be the target category of the underlying topological space functor for $\dcal$.

The above good properties of $\czero$ are efficiently used in this paper. Properties $(2)$ and $(4)$ lead us to Lemmas \ref{enrich}, \ref{enrich/X}, and \ref{bdleforget}, which are the starting points of our smoothing theorems for maps, sections, and principal bundles respectively (see also Remark \ref{bdlerem}).

The categories $\dcal$ and $\czero$ share properties $(1)$ and $(2)$, which often enables us to deal with $\dcal$ and $\czero$ simultaneously (see Sections 5.3, 6, and 7).
\end{rem}
We end this subsection by introducing the separation condition for arc-generated spaces. An arc-generated space $X$ is called {\sl separated} if the diagonal is closed in the product $X \times_{\czero} X$. Separatedness is stronger than $T_{1}$-axiom and weaker than Hausdorff property. For a group $G$ in $\czero$, separatedness is equivalent to $T_{1}$-axiom. We can also observe the following closure properties of separatedness:
\begin{itemize}
	\item If arc-generated spaces $X$ and $Y$ are separated, then $X \times_\czero Y$ is also separated.
	\item If an arc-generated space $A$ admits a continuous injection into a separated arc-generated space, then $A$ is also separated.
\end{itemize}
\subsection{Fully faithful embedding of $C^{\infty}$ into $\dcal$}
In this subsection, we recall the notion of a $C^{\infty}$-manifold (in convenient calculus) from \cite{KM} and observe that the category $C^{\infty}$ of (separated) $C^{\infty}$-manifolds can be fully faithfully embedded into the category $\dcal$.
\par\indent
First, we make a review on the local calculus. Let $E$ be a locally convex vector space. A set-theoretic curve $c: \rbb \longrightarrow E$ is called {\sl smooth} if all derivatives exist and are continuous - this is a concept without problems. The final topology for the set of smooth curves into $E$ is called the {\sl $c^{\infty}$-topology} of $E$, which is finer than the original locally convex topology; for a Fr\'{e}chet space, the $c^{\infty}$-topology coincides with the original locally convex topology (\cite[Theorem 4.11(1)]{KM}). A set-theoretic map $f: U \longrightarrow V$ between $c^{\infty}$-open sets of locally convex vector spaces are defined to be {\sl smooth} or $C^{\infty}$ if $f$ maps smooth curves in $U$ to smooth curves in $V$. In the finite dimensional case, this gives the usual notion of smooth mappings (\cite[Corollary 3.14]{KM}). Smooth maps are obviously continuous for the $c^{\infty}$-topologies; however, a smooth map need not be continuous for the original locally convex topologies (\cite[Corollary 2.11]{KM}). 
\par\indent
Next, we recall the notion of a $C^\infty$-manifold. In convenient calculus, $C^\infty$-manifolds are defined by gluing $c^{\infty}$-open sets of convenient vector spaces, which are locally convex vector spaces satisfying a week completeness condition (\cite[Theorem 2.14]{KM}). The precise definition of a $C^\infty$-manifold is as follows:
\par\indent
A {\sl chart} $(U,u)$ on a set $M$ is a bijection $M\supset U \overset{u}{\longrightarrow} u(U)\subset E_U$ from a subset $U$ of $M$ onto a $c^\infty$-open subset of a convenient vector space $E_U$. A family $\{(U_\alpha,u_\alpha)\}_{\alpha\in A}$ of {\sl charts} on $M$ is called an {\sl atlas} for $M$ if $M=\cup\ U_\alpha$ and all chart changings
\[
	u_{\alpha\beta}:= u_\alpha \circ u_\beta^{-1}:u_\beta(U_\alpha \cap U_\beta)\longrightarrow u_\alpha(U_\alpha \cap U_\beta)
\]
are smooth maps between $c^\infty$-open subsets. Two atlas are called {\sl equivalent} if their union is again an atlas for $M$. An equivalence class of atlas is called a $C^\infty$-{\sl structure} on $M$. A $C^\infty$-{\sl manifold} $M$ is a set together with a $C^\infty$-structure on it. \par\indent
Smooth maps between $C^\infty$-manifolds $M$ and $N$ are defined using charts in the obvious manner. Thus, the category $C^\infty$-$mfd$ of $C^\infty$-manifolds and smooth maps are defined. We can easily see that a set-theoretic map $f: M \longrightarrow N$ is smooth if and only if $f$ preserves smooth curves.\par
The underlying topological space $\widetilde{M}$ of a $C^\infty$-manifold $M$ is defined to be the set $M$ endowed with the final topology for the smooth curves (\cite[27.4]{KM}).\par
Now we introduce a fully faithful embedding of $C^\infty$-${mfd}$ into $\dcal$. Define the functor $I: C^{\infty}$-${mfd} \longrightarrow \dcal$ by assigning to a $C^{\infty}$-manifold $M$ the set $M$ endowed with the diffeology
\[
D_{IM} = \{ \text{  $C^{\infty}$-parametrizations of $M$ } \}.
\]
\begin{lem}\label{predmfd}
	The functor $I:C^\infty$-${mfd} \longrightarrow \dcal$  is a fully faithful functor which preserves finite products. Further, the underlying topological space $\widetilde{M}$ of a $C^\infty$-manifold $M$ is just the underlying topological space $\widetilde{IM}$ of the diffeological space $IM$.
	\begin{proof}
		Recall that a (set-theoretic) map $f : M \longrightarrow N$ between $C^\infty$-manifolds is smooth if and only if $f$ preserves smooth curves and note that the smooth curves of a $C^\infty$-manifold $M$ are just the plots of $IM$ with source $\mathbb{R}$. Then, we see that $I$ is fully faithful.\par
		We have the equality $C^{\infty} (U, M \times N) = C^{\infty} (U, M) \times C^{\infty} (U, N)$ for an open set $U$ of $\rbb^{n}$ and $M, N \in C^{\infty}$-$mfd$, from which we see that $I$ preserves finite products.\par
		To prove that $\widetilde{M} = \widetilde{IM}$, it suffices to show that the ordinary topology of an open set $U$ of $\rbb^{n}$ is the final topology for the smooth curves, which is easily shown by Special Curve Lemma \cite[p. 18]{KM} (alternatively, see \cite[Theorem 4.11(1)]{KM}).
	\end{proof}
\end{lem}
From now on, {\bf we require that the underlying topological spaces of $C^\infty$-manifolds are separated as arc-generated spaces} (see the end of Section 2.1); the category of (separated) $C^\infty$-manifolds are denoted by $C^\infty$.

From Lemma \ref{predmfd}, we obtain the following result on the functor $I:C^\infty \longrightarrow \dcal$, which is the restriction of $I:C^\infty$-${mfd}\longrightarrow \dcal$ to $C^\infty$. 
\begin{prop}\label{dmfd}
	\begin{itemize}
	\item[{\rm (1)}] The category $C^\infty$ is closed under finite products and the functor $I: C^{\infty} \longrightarrow \dcal$ is a fully faithful functor which preserves finite products.
	\item[{\rm (2)}] The functor $I: C^{\infty} \longrightarrow \dcal$ restricts to the fully faithful functor from the category of Lie groups to that of diffeological groups.
	\item[{\rm (3)}] The functor $I: C^{\infty} \longrightarrow \dcal$ is compatible with the underlying topological space functors. In other words, the following diagram commutes:
	\[
	\begin{tikzcd}
	C^{\infty} \arrow[hook]{rr}{I} \arrow[swap]{rd}{\widetilde{\cdot}} & & \arrow{ld}{\widetilde{\cdot}} \dcal \\
	& \czero &
	\end{tikzcd}
	\]
	\end{itemize}
\begin{proof}
	{\rm (1)} We see from Lemma \ref{predmfd} and property (4) in Remark \ref{suitable} that $\tilde{\cdot}:C^\infty$-${mfd} \longrightarrow \czero$ preserves finite products, and hence that $C^\infty$ is closed under finite products (see the end of Section 2.1). The rest of the statement follows immediately from Lemma \ref{predmfd}.\par
	{\rm (2)} Since Lie groups and diffeological groups are defined to be groups in $C^{\infty}$ and $\dcal$ respectively, the result is immediate from Part 1.\par
	{\rm (3)} The result follows immediately from Lemma \ref{predmfd}.
\end{proof}
\end{prop}
The diffeological space $IM$ associated to a $C^{\infty}$-manifold $M$ is often denoted simply by $M$ if there is no confusion in context.
\begin{rem}\label{fullembedding}
	\begin{itemize}
	\item[{\rm (1)}] It is obvious that finite dimensional $C^\infty$-manifolds embed fully faithfully into the category $\dcal$ via the functor $I$. It was shown by Losik \cite{Losik} that Fr\'{e}chet manifolds embed fully faithfully into $\dcal$ via $I$.
	\item[{\rm (2)}] Though the definition of $I$ applies to $C^{\infty}$-manifolds in classical differential calculi, the functor $I$ need not be full. In fact, $I: C^{\infty}_{c} \longrightarrow \dcal$ is not full, where $C^{\infty}_{c}$ denotes the category of $C^{\infty}$-manifolds in Keller's $C^{\infty}_{c}$-theory (see Appendix B).
	\end{itemize}
\end{rem}
The following remark relates to separation conditions for $C^{\infty}$-manifolds.

\begin{rem}\label{KMsepar}
	In \cite[27.4]{KM}, Kriegl and Michor proposed the following three separation conditions for $C^{\infty}$-manifolds:
	\begin{itemize}
	\item[{\rm (i)}] $\widetilde{M}$ is separated as an arc-generated space.
	\item[{\rm (ii)}] $\widetilde{M}$ is Hausdorff as a topological space.
	\item[{\rm (iii)}] $M$ is smoothly Hausdorff (i.e., the smooth functions in $C^{\infty}(M, \rbb)$ separate points in $M$).
	\end{itemize}
	(The implications $\rmiii \Longrightarrow \rmii \Longrightarrow \rmi$ hold obviously. Use the fact that $\tilde{\cdot}:C^\infty$-${mfd}\longrightarrow \czero$ preserves finite products to see the equivalence of separation condition $(\rm i)$ and condition $(2)$ in \cite[27.4]{KM}.) They honestly wrote ``it is not so clear which separation property should be required for a manifold" and  temporarily required that $C^{\infty}$-manifolds are smoothly Hausdorff. Whichever separation condition is selected, all the essential results in this paper remain true. However, we choose separation condition $\rmi$, which seems most natural for our categorical setting; it also simplifies some results and definitions (see Remarks \ref{4properties}(3)-(4), the definition of the functor $J: C^{\infty}_{c\, {conv}} \longrightarrow C^{\infty}$ in Appendix B, and \cite{CGA}).
\end{rem}
We end this subsection with a remark concerning the semiclassicality condition on a $C^\infty$-manifold (see Section \ref{1.2}).
\begin{rem}\label{semiclassical}
	The $c^\infty$-topology of a locally convex vector space can be strictly finer than the original locally convex topology (\cite[Section 4]{KM}). Thus, we need the notion of a semiclassical $C^\infty$-manifold to use local convexity in local arguments.
\end{rem}
\subsection{Standard $p$-simplices and model structure on $\dcal$}
We recall the basics of the model structure on $\dcal$ introduced in \cite{origin}.\par
The principal part of our construction of a model structure on $\dcal$ is the construction of good diffeologies on the sets
$$
\Delta^p=\{(x_0,\ldots,x_p)\in\mathbb{R}^{p+1} \ |\ \underset{i}{\sum} x_i = 1,\ x_i\geq 0 \}\ \ \ (p\geq 0)
$$
which enable us to define weak equivalences, fibrations, and cofibrations and to verify the model axioms (see Definition \ref{WFC} and Theorem \ref{originmain}). The required properties of the diffeologies on $\Delta^{p} \ (p \geq 0)$ are expressed in the following four axioms:
\begin{axiom}
	The underlying topological space of $\Delta^p$ is the topological standard $p$-simplex for $p\geq 0$.
\end{axiom}
Recall that
$f:\Delta^p \longrightarrow \Delta^q$
is an {\sl affine map} if $f$ preserves convex combinations.
\begin{axiom}
	Any affine map $f:\Delta^p\longrightarrow \Delta^q$ is smooth.
\end{axiom}
For $K \in \scal$, the {\sl simplex category} $\Delta\downarrow K$ is defined to be the full subcategory of the overcategory $\scal \downarrow K$ consisting of maps $\sigma : \Delta[n] \rightarrow K$. By Axiom 2, we can consider the diagram $\Delta\downarrow K \longrightarrow \mathcal{D}$ sending $\sigma :\Delta[n] \longrightarrow K$ to $\Delta^n$. Thus, we define the {\sl realization functor}
$$
|\ |_{\dcal}: \mathcal{S}\longrightarrow \mathcal{D}
$$
by $|K|_{\mathcal{D}}= \underset{\mathrm{\Delta\downarrow} K}{\mathrm{colim}} \ \Delta^n$.
\par\indent
Consider the smooth map $|\dot{\Delta}[p]|_{\dcal} \longhookrightarrow |\Delta[p]|_{\dcal} = \Delta^{p}$ induced by the inclusion of the boundary $\dot{\Delta}[p]$ into $\Delta[p]$.
\begin{axiom}
	The canonical smooth injection
	$$\left| \dot{\Delta}[p] \right|_{\dcal} \longhookrightarrow \Delta^p$$
	is a $\dcal$-embedding.
\end{axiom}
The $\dcal$-homotopical notions, especially the notion of a $\dcal$-deformation retract, are defined in the same manner as in the category of topological spaces by using the unit interval $I=[0,1]$ endowed with a diffeology via the canonical bijection with $\Delta^{1}$ (\cite[Section 2.4]{origin}).
The {\sl $k^{th}$ horn} of $\Delta^p$ is a diffeological subspace of $\Delta^p$ defined by
\begin{equation*}
	\Lambda^{p}_{k} =  \{(x_0,\ldots,x_p)\in\Delta^p \ |\ x_i=0 \hbox{ for some }i\neq k\}.
\end{equation*}
\begin{axiom}
	The $k^{th}$ horn $\Lambda^p_k$ is a $\dcal$-deformation retract of $\Delta^p$ for $p \geq 1$ and $0 \leq k \leq p$.
\end{axiom}
For a subset $A$ of the affine $p$-space $\abb^{p} = \{(x_0, \ldots, x_p) \in \rbb^{p+1} \ | \ \sum x_i = 1 \}$, $A_{\mathrm{sub}}$ denotes the set $A$ endowed with the sub-diffeology of $\abb^{p} \ (\cong \rbb^{p})$.
The diffeological space $\Delta^{p}_{\mathrm{sub}}$, used in \cite{H} to study diffeological spaces by homotopical means, satisfies neither Axiom 3 nor 4 for $p \geq 2$ (\cite[Proposition A.2]{origin}). Thus, we must construct a new diffeology on $\Delta^p$, at least for $p \geq 2$.
\par\indent
Let $(i)$ denote the vertex $(0, \ldots, \underset{(i)}{1}, \ldots, 0)$ of $\Delta^p$, and let $d^i$ denote the affine map from $\Delta^{p-1}$ to $\Delta^p$, defined by
\begin{equation*}
	d^i((k))= \left \{
	\begin{array}{ll}
		(k) & \text{for} \ k<i,\\
		(k+1)& \text{for} \ k\geq i.
	\end{array}
	\right.
\end{equation*}

\begin{defn}\label{simplices}
	We define the {\sl standard $p$-simplices} $\Delta^p$ ($p\geq 0$) inductively. Set $\Delta^p=\Delta_{\mathrm{sub}}^p$ for $p\leq 1$. Suppose that the diffeologies on $\Delta^k$ ($k<p$) are defined.
	We define the map
	\begin{eqnarray*}
		\varphi_i: \Delta^{p-1}\times [0,1) & \longrightarrow &  \Delta^p
	\end{eqnarray*}
	by $\varphi_{i}(x, t) = (1-t)(i)+td^{i}(x)$, and endow $\Delta^p$ with the final structure for the maps $\varphi_{0}, \ldots, \varphi_{p}$.
\end{defn}
The following result is established in \cite[Propositions 3.2, 5.1, 7.1, and 8.1]{origin}.
\begin{prop}\label{axioms}
	The standard $p$-simplices $\Delta^p\ (p \geq 0)$ in Definition \ref{simplices} satisfy Axioms 1-4.
\end{prop}
Without explicit mention, the symbol $\Delta^p$ denotes the standard $p$-simplex defined in Definition \ref{simplices} and a subset of $\Delta^p$ is endowed with the sub-diffeology of $\Delta^p$. Since the diffeology of $\Delta^p$ is the sub-diffeology of $\abb^{p}$ for $p \leq 1$, the $\dcal$-homotopical notions, especially the notion of a $\dcal$-deformation retract, coincide with the ordinary smooth homotopical notions in the theory of diffeological spaces (see \cite[p. 108]{IZ} and \cite[Remark 2.14]{origin}).
\par\indent
By Axiom 2, we can define the singular complex $S^{\dcal}X$ of a diffeological space $X$ to have smooth maps $\sigma : \Delta^p \longrightarrow X$ as $p$-simplices, thereby defining the {\sl singular functor} $S^{\mathcal{D}} :\mathcal{D} \longrightarrow \mathcal{S}$.
We have the following result (\cite[Proposition 9.1 and Lemma 9.5]{origin}).
\begin{prop}\label{adjoint1}
	\begin{itemize}
	\item[{\rm (1)}] $|\ |_{\dcal}: \scal \rightleftarrows \dcal: S^{\dcal}$ is an adjoint pair.
	\item[{\rm (2)}] The composite of the two adjoint pairs
	$$
	|\ |_{\dcal}: \scal \rightleftarrows \dcal: S^{\dcal} \text{ and } \tilde{\cdot}: \dcal \rightleftarrows \ccal^{0}: R
	$$
	is just the adjoint pair
	$$
	|\ | : \scal \rightleftarrows \ccal^0 : S
	$$
	of the topological realization and singular functors.
	\end{itemize}
\end{prop}
Weak equivalences, fibrations, and cofibrations in $\mathcal{D}$ are defined as follows.
\begin{defn}\label{WFC}
	Define a map $f :X\longrightarrow Y$ in $\mathcal{D}$ to be
	\begin{itemize}
		\item[$(1)$]
		a {\sl weak equivalence} if $S^{\mathcal{D}} f:S^{\mathcal{D}} X\longrightarrow S^{\mathcal{D}} Y$ is a weak equivalence in the category of simplicial sets,
		\item[$(2)$]
		a {\sl fibration} if the map $f$ has the right lifting property with respect to the inclusions $\Lambda^p_k \longhookrightarrow\Delta^p$ for all $p>0$ and $0\leq k\leq p$, and
		\item[$(3)$]
		a {\sl cofibration} if the map $f$ has the left lifting property with respect to all maps that are both fibrations and weak equivalences.
	\end{itemize}
\end{defn}
Define the sets $\ical$ and $\jcal$ of morphisms of $\dcal$ by
\begin{eqnarray*}
	\ical & = & \{\dot{\Delta}^{p} \longhookrightarrow \Delta^{p} \ | \ p\geq 0 \},\\
\jcal & = & \{\Lambda^{p}_{k} \longhookrightarrow \Delta^{p} \ | \ p>0,\ 0 \leq k \leq p \}.
\end{eqnarray*}
Then, we have the following theorem (\cite[Theorem 1.3]{origin}). See \cite[Definition 15.2.1]{MP} for a compactly generated model category.
\begin{thm}\label{originmain}
	With Definition \ref{WFC}, $\mathcal{D}$ is a compactly generated model category whose object is always fibrant. $\ical$ and $\jcal$ are the sets of generating cofibrations and generating trivial cofibrations respectively.
\end{thm}
\begin{rem}\label{ingredients}
	The proofs of Proposition \ref{adjoint1} and Theorem \ref{originmain} are constructed using only the convenient properties of the category $\dcal$ (Proposition \ref{conven}) and Axioms 1-4 for the standard simplices.
\end{rem}
The following theorem, which is \cite[Theorem 1.4]{origin}, shows that weak equivalences in $\dcal$ are just smooth maps inducing isomorphisms on smooth homotopy groups. See \cite[Section 3.1]{CW} or \cite[Chapter 5]{IZ} for the smooth homotopy groups $\pi^{\dcal}_{p}(X, x)$ of a pointed diffeological space $(X, x)$. Note that $S^{\dcal}X$ is always a Kan complex (Theorem \ref{originmain}) and see \cite[p. 25]{GJ} for the homotopy groups $\pi_p(K, x)$ of a pointed Kan complex $(K, x)$.
\begin{thm}\label{homotopygp}
	Let $(X, x)$ be a pointed diffeological space. Then, there exists a natural bijection
	$$
	\varTheta_{X} : \pi^{\dcal}_{p}(X, x) \longrightarrow \pi_{p}(S^{\dcal}X, x) \ \  \text{for} \ \ p \geq 0,
	$$
	that is an isomorphism of groups for $p > 0$.
\end{thm}
The standard simplices $\Delta^{p}$ define not only the model structure on $\dcal$ but also a simplicial category structure on $\mathcal{D}$ via Proposition \ref{adjoint1}(1); the relation between these two structures on $\mathcal{D}$ is one of our greatest concerns (see Section 4). Furthermore, the standard simplices $\Delta^{p}_{\sub}$ are also used for constructing smooth maps and smooth homotopies (see Section 8). Thus, we end this subsection with the following results on the standard simplices.\par
A subset $A$ of $\Delta^p$ endowed with the sub-diffeology of $\Delta^p_{\sub}$ is denoted by $A_{\sub}$, which is compatible with the notation introduced above.
\begin{lem}\label{simplex}
	\begin{itemize}
	\item[{\rm (1)}] The horn $\Lambda^{p}_{k}$, and hence, the standard $p$-simplex $\Delta^{p}$ is $\dcal$-contractible.
	\item[{\rm (2)}] The set $\{ \Delta^{p}_{\sub} \}_{p \geq 0}$ satisfies Axioms 1 and 2 for the standard simplices but satisfies neither Axiom 3 nor 4.
	\item[{\rm (3)}] The map $id : \Delta^{p} \longrightarrow \Delta^{p}_{\sub}$ is smooth, which restricts to the diffeomorphism $id: \Delta^{p}-{\rm sk}_{p-2}\ \Delta^{p} \xrightarrow[\cong]{} (\Delta^{p} - {\rm sk}_{p-2}\ \Delta^{p})_{\sub}$, where ${\rm sk}_{p-2}$ $\Delta^{p}$ denotes the $(p-2)$-skeleton of $\Delta^{p}$.
	\end{itemize}
\end{lem}
\begin{proof}
See \cite[Remark 9.3, Lemma A.1, Proposition A.2, and Lemmas 3.1 and 4.2]{origin}.
\end{proof}
\subsection{Quillen pairs $(|\ |_{\dcal}, S^{\dcal})$ and $(\ \widetilde{\cdot}, R)$}
In this subsection, we see that the adjoint pairs $(|\ |_{\dcal}, S^{\dcal})$ and $(\widetilde{\cdot}, R)$ are Quillen pairs whose composite is a pair of Quillen equivalences.

We begin with the following remark which relates to the model structure on $\czero$.
\begin{rem}\label{arc}
	The category $\ccal^{0}$ is a compactly generated model category whose object is always fibrant (cf. \cite[Section 7.10]{Hi}); this is shown in the same manner as in the cases of the category $\tcal$ of topological spaces (\cite[Section 8]{DS}) and that of compactly generated Hausdorff spaces (\cite[Proposition 9.2 in Chapter I]{GJ}). We can easily see that the adjoint pair $I :\czero \rightleftarrows \tcal : \alpha$ (\cite[Lemma 2.7]{origin}) is a pair of Quillen equivalences.
\end{rem}
For the adjoint pairs $(|\ |_\dcal, S^\dcal)$ and $(\tilde{\cdot}, R)$, we have the following result.
\begin{lem}\label{Quillenpairs}
	$|\ |_{\dcal} : \scal \rightleftarrows \dcal : S^{\dcal}$ and $\widetilde{\cdot} : \dcal \rightleftarrows \czero : R$ are Quillen pairs.
\end{lem}
\begin{proof}
	Since $S^{\mathcal{D}}$ preserves both fibrations and trivial fibrations (\cite[Lemma 9.6]{origin}), $(|\ |_{\mathcal{D}},S^{\mathcal{D}})$ is a Quillen pair. Since $R$ also preserves both fibrations and trivial fibrations (Proposition \ref{conven}(3) and \cite[Lemma 9.8]{origin}), $(\widetilde{\cdot},R)$ is a Quillen pair.
\end{proof}
The following lemma states that the composite of the adjoint pairs $(|\ |_\dcal, S^\dcal)$ and $(\tilde{\cdot}, R)$ is a pair of Quillen equivalences (see Proposition \ref{adjoint1}(2)).
\begin{lem}\label{clequiv}
	$|\ |: \scal \rightleftarrows \ccal^{0}:S$ is a pair of Quillen equivalences.
\end{lem}
\begin{proof}
	This result is shown in the same manner as in the case of $|\ |:\scal \rightleftarrows \bcal:S$, where $\bcal$ is the category $\tcal$ of topological spaces or the category $\kcal$ of compactly generated Hausdorff spaces; see \cite[Theorem 16.1]{May} and \cite[Theorem 11.4 in Chapter I]{GJ} for the cases of $\tcal$ and $\kcal$ respectively.
\end{proof}
\begin{rem}\label{proof}
	As mentioned in the proof of Lemma \ref{clequiv}, two proofs of the fact that $(|\ |, S)$ is a pair of Quillen equivalences are known (\cite[Theorem 16.1]{May}, \cite[Theorem 11.4 in Chapter I]{GJ}). We prove Part 1 of Theorem \ref{Quillenequiv} in a way analogous to the proof of \cite[Theorem 16.1]{May} (see Section 3).
	\par\indent
	The other proof is based on Quillen's result that the topological realization of a Kan fibration is a Serre fibration (\cite[Theorem 10.10 in Chapter I]{GJ}). For the realization functor $|\ |_{\mathcal{D}}:\scal\longrightarrow\mathcal{D}$, the analogous result does not seem to hold; note that $|\ |_{\mathcal{D}}$ does not preserve finite products (in fact, $|\Delta[1] \times \Delta[1]|_{\dcal}$ is not the product $|\Delta[1]|_{\dcal} \times |\Delta[1]|_{\dcal}$ but a diffeological space obtained by patching together two copies of $\Delta^2$; see the argument in the proof of \cite[Proposition A.2(1)]{origin}). Thus, it does not seem that Theorem \ref{Quillenequiv}(1) can be proved according to the idea of the proof of \cite[Theorem 11.4 in Chapter I]{GJ}.
\end{rem}
\section{Quillen equivalences between $\scal$, $\dcal$, and $\czero$}
In this section, we develop the theory of singular homology for diffeological spaces and use it to establish the Quillen equivalences between $\scal$, $\dcal$, and $\czero$ (Theorem \ref{Quillenequiv}) and the basic properties of the subclasses $\wcal_\dcal$ and $\vcal_\dcal$ (Corollary \ref{W}).
\subsection{Singular homology of a diffeological space}
In this subsection, we introduce the singular homology of a diffeological pair (i.e., a pair of a diffeological space and its diffeological subspace) and establish the results used in the proof of Theorem \ref{Quillenequiv}.
\par\indent
To define the homology $H_{\ast}(X, A)$ of a diffeological pair $(X, A)$, we begin by recalling the homology groups of a simplicial pair (i.e., a pair of a simplicial set and its simplicial subset). For a category $\mathcal{A}$, $s\mathcal{A}$ denotes the category of simplicial objects in $\mathcal{A}$ (\cite[Section 2]{May}). Let $\zbb$-$mod$ denote the category of $\zbb$-modules, and let $\mathsf{Kom}_{\geq0}(\zbb$-$mod)$ denote the category of non-negatively graded chain complexes of $\zbb$-modules.
The free $\zbb$-module functor $\zbb\cdot: Set \longrightarrow \zbb\text{-}mod$ extends to the functor $\mathbb{Z}\cdot:\scal = sSet \longrightarrow s\zbb\text{-}mod$. The composite
\begin{center}
	$\scal \xrightarrow{\zbb\cdot} s\zbb\text{-}mod \longrightarrow \mathsf{Kom}_{\geq 0}(\zbb\text{-}mod)$
\end{center}
is also denoted by $\zbb\cdot$, where the functor $s\zbb\text{-}mod$ $\longrightarrow \mathsf{Kom}_{\geq 0}(\zbb\text{-}mod)$ is defined by assigning to a simplicial $\zbb$-module $M$ the chain complex
\begin{center}
	$\cdots \xrightarrow{\partial_1} M_1 \xrightarrow{\partial_0} M_0 \longrightarrow 0 $
\end{center}
with $\partial_n = \sum_{i=0}^{n} (-1)^i d_i$. Then, the {\sl homology} $H_{\ast}(K,L)$ of a simplicial pair $(K,L)$ is defined by
$$
H_{\ast}(K,L)=H_{\ast}(\zbb(K)/\zbb(L)).
$$
The homology of a simplicial pair $(K, \emptyset)$ is usually referred to as the homology of a simplicial set $K$, written $H_{\ast}(K)$.
\par\indent
Using the homology of a simplicial pair, we define the {\sl singular homology} $H_{\ast}(X,A)$ of a diffeological pair $(X,A)$ by
$$
H_{\ast}(X,A)=H_{\ast}(S^{\dcal} X,S^{\dcal} A).
$$
The singular homology of a diffeological pair $(X,\emptyset)$ is usually referred to as the {\sl singular homology} of a diffeological space $X$, written $H_{\ast}(X)$.
\par\indent
The singular homology for diffeological pairs has the following desirable properties. For a subset $S$ of a topological space $Z$, $\overline{S}$ and $S^{\circ}$ denote the closure and the interior of $S$ respectively.

\begin{prop}\label{homologytheory}
	\begin{itemize}
		\item[$(1)$]
		(Exactness) For a diffeological pair $(X,A)$, there exists a natural exact sequence
		\begin{center}
			\begin{tikzcd}
				&\cdots
				\arrow[d, phantom, ""{coordinate, name=Z}]
				&\phantom{H_p(X,A)} \arrow[dll,
				"\partial",
				rounded corners,
				to path={ -- ([xshift=2ex]\tikztostart.east)
					|- (Z) [near end]\tikztonodes
					-| ([xshift=-2ex]\tikztotarget.west)
					-- (\tikztotarget)}] \\
				H_p(A) \arrow[r]
				& H_p(X) \arrow[r]
				\arrow[d, phantom, ""{coordinate, name=Z}]
				& H_p(X,A) \arrow[dll,
				"\partial",
				rounded corners,
				to path={ -- ([xshift=2ex]\tikztostart.east)
					|- (Z) [near end]\tikztonodes
					-| ([xshift=-2ex]\tikztotarget.west)
					-- (\tikztotarget)}] \\
				H_{p-1}(A) \arrow[r]
				&  \cdots
				&
			\end{tikzcd}
		\end{center}
		\item[$(2)$]
		(Excision) If $(X,A)$ is a diffeological pair and $U$ is a diffeological subspace of $X$ such that $\overline{U}\subset A^{\circ}$, then the inclusion $(X-U,A-U)\longhookrightarrow (X,A)$ induces an isomorphism $H_{\ast}(X-U,A-U)\xrightarrow[\ \cong\ ]{} H_{\ast}(X,A)$.
		\item[$(3)$]
		(Homotopy) Let $f,g:(X,A)\longrightarrow (Y,B)$ be maps of diffeological pairs. If $f\simeq_{\dcal} g$, then the induced maps $f_{\ast},g_{\ast}:H_{\ast}(X,A)\longrightarrow H_{\ast}(Y,B)$ coincide.\vspace{2mm}
		\item[$(4)$]
		(Dimension)
		Let $\ast$ be a terminal object of $\mathcal{D}$. Then
		\[
		H_p({\ast})=
		\begin{cases}
		\zbb & p=0 \\
		0 & p\neq 0.
		\end{cases}
		\]
		\item[$(5)$]
		(Additivity) If $X$ is the coproduct of diffeological spaces $\{X_i \}$, then the inclusions $X_i \longrightarrow X$ induce an isomorphism
		$$
		\underset{i}{\bigoplus}\ H_\ast(X_i) \xrightarrow[\cong]{} H_\ast(X).
		$$
	\end{itemize}
\end{prop}
For the proof of Proposition \ref{homologytheory}, we need the following lemma.
\begin{lem}\label{prism}
	\begin{itemize}
	\item[{\rm (1)}] Let $X$ be a diffeological space. Then the singular complex $S^{\dcal}X$ is a subcomplex of the topological singular complex $S\widetilde{X}$.
	\item[{\rm (2)}] If $f  :\Delta^r \longrightarrow \Delta^p\times I$ is an affine map (i.e., a map preserving convex combinations), then $f$ is smooth.
	\end{itemize}
\end{lem}
\begin{proof} Parts 1 and 2 follow immediately from Proposition \ref{axioms} (Axioms 1 and 2, respectively).
\end{proof}
\begin{proof}[Proof of Proposition \ref{homologytheory}]
	Parts 1, 4, and	5 are easily verified. For the proof of Parts 2 and 3, we note that $\zbb S^{\dcal}X$ is a chain subcomplex of $\zbb S\widetilde{X}$ (Lemma \ref{prism}(1)), and verify that the relevant chain map and chain homotopies on $\zbb S\widetilde{X}$ restrict to ones on $\zbb S^{\dcal}X$.
	\par\indent
	(2) To verify the excision axiom for the topological singular homology, we use the chain map $Sd$ on $\zbb S\widetilde{X}$ and the chain homotopy $T$ on $\zbb S\widetilde{X}$ connecting $1_{\zbb S\widetilde{X}}$ and $Sd$, where $Sd$ and $T$ are defined using the barycentric subdivision of $\Delta^p$ and the triangulation of $\Delta^p \times I$ connecting the trivial triangulation of $\Delta^p$ and the barycentric subdivision of $\Delta^p$, respectively (\cite[Section 17 in Chapter IV]{Br}). Thus, $Sd$ and $T$ restrict to a chain map and a chain homotopy on $\zbb S^{\dcal}X$ respectively (see Proposition \ref{axioms}(Axiom 2) and Lemma \ref{prism}(2)).

	(3) To show that $\tilde{f}_\sharp \simeq \tilde{g}_\sharp:\zbb S\widetilde{X} \longrightarrow \zbb S\widetilde{Y}$, we use the chain homotopy $D$ on $\zbb S\widetilde{X}$ defined using the prism decomposition of $\Delta^p \times I$ (\cite[Section 11]{G}), which restricts to a chain homotopy on $\zbb S^\dcal X$ (see Lemma \ref{prism}(2)).
\end{proof}
The Mayer-Vietoris exact sequence and the homology exact sequence of a triple are also deduced from Proposition \ref{homologytheory} via the usual arguments (cf. \cite[p. 37]{Selick}); these exact sequences can be also derived from the consideration of the singular chain complexes.
\par\indent
Recall that $\mathring{\Delta}^p$ denotes the set $\Delta^p \backslash \dot{\Delta}^p$ endowed with the subdiffeology of $\Delta^p$ (\cite[Definition 4.1]{origin}). For a $\zbb$-module $M$ and $n \geq 0$, let $M[n]$ denote the graded module with $M[n]_n=M$ and $M[n]_i=0$ ($i\neq n$).

\begin{lem}\label{homologicallemma}
	Let $b$ denote the barycenter of $\Delta^p$. Then
	\[
	H_{\ast}(\mathring{\Delta}^p,\mathring{\Delta}^p-\{b\})\cong\zbb[p].
	\]
	\begin{proof}
		Define the reduced homology $\widetilde{H}_\ast(X)$ of a diffeological space $X$ in the same manner as the reduced homology of a topological space (see, eg, \cite[p. 38]{Selick}). Since $\mathring{\Delta}^p$ is $\dcal$-contractible (Lemma \ref{simplex}(3)), we have only to show that $\widetilde{H}_\ast(\mathring{\Delta}^p - \{b\}) \cong \zbb[p-1]$ (see Proposition \ref{homologytheory}(1)).
		\par\indent
		Since $\mathring{\Delta}^p - \{b\}$ contains a diffeological subspace which is diffeomorphic to $S^{p-1}$ as a $\dcal$-deformation retract, we can use Proposition \ref{homologytheory} and the Mayer-Vietoris exact sequence to see that $\widetilde{H}_\ast(\mathring{\Delta}^p - \{b\}) \cong \widetilde{H}_\ast (S^{p-1}) \cong \zbb [p-1]$.
	\end{proof}
\end{lem}

Let us introduce the notions of a $CW$-complex in $\dcal$ and its cellular homology.
\begin{defn}\label{CW}
	A {\sl $CW$-complex} $X$ in $\dcal$ is a diffeological space $X$ which is the colimit of a sequence in $\dcal$
	$$
	\emptyset = X^{-1}\overset{i_0}{\longhookrightarrow}X^0\overset{i_1}{\longhookrightarrow}X^1\overset{i_2}{\longhookrightarrow}X^2\overset{i_3}{\longhookrightarrow}\cdots
	$$
	such that each $i_n$ fits into a pushout diagram of the form
	\begin{center}
		\begin{tikzcd}
			\underset{\lambda\in\Lambda_n}{\coprod} \dot{\Delta}^n_\lambda \arrow{d} \arrow[hook]{r} & \underset{\lambda\in\Lambda_n}{\coprod} \Delta^n_\lambda \arrow{d}\\
			X^{n-1} \arrow[hook]{r}{i_n} & X^n.
		\end{tikzcd}
	\end{center}
\end{defn}
Recall that $\dcal$ is a cofibrantly generated model category having $\ical = \{\dot{\Delta}^p \longhookrightarrow \Delta^p\ |\ p \geq 0\}$ as the set of generating cofibrations. Thus, a $CW$-complex in $\dcal$ is a special type of $\ical$-cell complex, and hence a cofibrant object in $\dcal$. Let us see that the realizations of simplicial sets form an important class of $CW$-complexes in $\dcal$.

\begin{lem}\label{realization}
	Let $K$ be a simplicial set. Then, the realization $|K|_\dcal$ is a $CW$-complex in $\dcal$ having one $n$-cell for each non-degenerate $n$-simplex of $K$.
	\begin{proof}
		Let $K^n$ denote the $n$-skeleton of $K$. Then, $K$ is the colimit of the sequence
		$$
		\emptyset = K^{-1}\overset{i_0}{\longhookrightarrow} K^0 \xhookrightarrow{\ \ i_1\ \ } K^1 \xhookrightarrow{\ \ i_2\ \ } K^2 \xhookrightarrow{\ \ i_3\ \ } \cdots
		$$
		and each $i_n$ fits into the pushout diagram
		\[
		\begin{tikzcd}
		\underset{\lambda \in NK_n}{\coprod} \dot{\Delta}[n]_\lambda \arrow{d} \arrow[hook]{r} & \underset{\lambda \in NK_n}{\coprod} \Delta[n]_\lambda \arrow{d}\\
		K^{n-1} \arrow[hook]{r}{i_n} & K_n,
		\end{tikzcd}
		\]
		where $NK_n$ is the set of non-degenerate $n$-simplices of $K$ (see \cite[p. 8]{GJ}). Applying the realization functor $|\ |_\dcal$, we obtain the presentation of the $CW$-structure of $|K|_\dcal$ (see Propositions \ref{adjoint1}(1) and \ref{axioms}(Axiom 3)).
	\end{proof}
\end{lem}
For a $CW$-complex $X$ in $\dcal$, define the {\sl cellular chain complex} $(C(X), \partial_C)$ by
\[
	\begin{split}
	C_n(X) =& H_n (X^n, X^{n-1}), \\
	\partial_C =& \text{the connecting homomorphism coming from the } \\
	&\text{homology exact sequence of the triple } (X^n, X^{n-1}, X^{n-2}).
\end{split}
\]
(See \cite[p. 39]{Selick} for the verification of $\partial^2_C = 0$.) The homology $H_\ast C(X)$ is called the {\sl cellular homology} of $X$.

\begin{lem}\label{cellular}
	Let $X$ be a $CW$-complex in $\dcal$.
	\begin{itemize}
		\item[$(1)$] The $n^{th}$ component $C_n(X)$ is a free $\zbb$-module on the set of $n \text{-cells of } X$.
		\item[$(2)$] $H_\ast C(X)$ is isomorphic to $H_\ast(X)$.
	\end{itemize}
	\begin{proof}
		Express $X$ as in Definition \ref{CW}, and let $b_\lambda$ denote the barycenter of $\Delta^n_\lambda$. The point of $X^n$ corresponding to $b_\lambda$ is also denoted by $b_\lambda$. From \cite[Proposition 6.2, Lemmas 6.3(2) and 7.2]{origin}, we observe that
		\begin{center}
			\begin{tikzcd}
				X^n - \{b_\lambda | \lambda \in \Lambda_n\} &=& X^{n-1} \underset{\underset{\lambda \in \Lambda_n}{\coprod}\dot{\Delta}^n_\lambda}{\cup} \underset{\lambda\in\Lambda_n}{\coprod} \Delta^n_\lambda - \{b_\lambda\} \nonumber \\[-7mm]
				&\simeq& X^{n-1}. \ \ \ \ \ \ \ \ \ \ \ \ \ \ \ \ \ \ \ \ \ \ \ \ \ \ \ \ \
			\end{tikzcd}
		\end{center}
		By Proposition \ref{homologytheory} and Lemma \ref{homologicallemma}, we can prove the result in the same manner as in the topological case (see the proof of \cite[Theorem 5.4.2]{Selick}).
	\end{proof}
\end{lem}

\begin{prop}\label{homologyisone}
	For a simplicial set $K$, there exists a natural isomorphism of graded modules
	$$
	H_\ast (K) \underset{\cong}{\longrightarrow} H_\ast (|K|_\dcal).
	$$
	\begin{proof}
		Since $(|\ |_\dcal, S^\dcal)$ is an adjoint pair (Proposition \ref{adjoint1}(1)), we have the natural map $i_K : K \longrightarrow S^\dcal |K|_\dcal$, which induces the homomorphism
		$$
		H_\ast (K) \xrightarrow{\ \ i_{K_\ast}\ \ } H_\ast (S^\dcal |K|_\dcal) = H_\ast (|K|_\dcal).
		$$
		With the help of Lemma \ref{cellular}, we can show that $i_{K_\ast}$ is an isomorphism by an argument similar to that in the proof of \cite[Proposition 16.2(i)]{May}.
	\end{proof}
\end{prop}

We have proven all the results on the singular homology and the realization of a simplicial set which are used in the proof of Theorem \ref{Quillenequiv}(1). The following are additional remarks.

\begin{rem}\label{cohomology}
	\begin{itemize}
		\item[{\rm (1)}] Let $\pi$ be an abelian group. For a diffeological pair $(X, A)$, the {\sl singular homology} $H_\ast(X, A;\pi)$ {\sl with coefficients in} $\pi$ is defined by
		$$
		H_\ast(X, A;\pi) = H_\ast(\zbb S^\dcal X/\zbb S^\dcal A\otimes \pi ).
		$$
		Then, Proposition \ref{homologytheory} generalizes to the case of an arbitrary coefficient group $\pi$ (see \cite[p. 183]{Br} and the proof of Proposition \ref{homologytheory}).
		\par\indent
		For a diffeological pair $(X, A)$, the {\sl singular cohomology} $H^{\ast}(X, A;\pi)$ {\sl with coefficients in} $\pi$ is defined by
		$$
		H^{\ast}(X, A;\pi) = H^{\ast} \Hom(\zbb S^{\dcal}X/\zbb S^{\dcal}A, \pi).
		$$
		Then, it has desirable properties analogous to Proposition \ref{homologytheory} (see \cite[pp. 285-286]{Br} and the proof of Proposition \ref{homologytheory}).
		\item[{\rm (2)}] The universal coefficient theorems obviously hold for the singular homology and cohomology of diffeological spaces (cf. \cite[pp. 281-285]{Br}). Since the Eilenberg-Zilber theorem holds for simplicial sets (\cite[Corollary 29.6]{May}), the K\"{u}nneth theorems hold for singular homology and cohomology of diffeological spaces, and hence, the cross products on homology and cohomology, and the cup product on cohomology are defined and satisfy the same formulas as those in the topological case.
		\item[{\rm (3)}] Since the functor $S^{\dcal}$ preserves fibrations (Lemma \ref{Quillenpairs}), the Serre spectral sequence works for fibrations in $\dcal$ (see \cite[Section 32]{May}).
	\end{itemize}
\end{rem}

\begin{rem}\label{coincide}
	Hector \cite{H} defined the singular complex of a diffeological space $X$ using $\Delta^p_{\rm sub} \ (p \geq 0)$; his singular complex and its homology are denoted by $S^{\dcal}_{\mathrm{sub}}(X)$ and $H'_{\ast}(X)$ respectively. Iglesias-Zemmour defined the complex $\bvec{\mathsf{C}}_\star (X)$ of reduced groups of cubic chains for a diffeological space $X$, and called its homology $\bvec{\mathsf{H}}_{\ast}(X)$ the cubic homology of $X$ (\cite[pp. 182-183]{IZ}).
	\par\indent
	Similarly to the topological case, we can see that $H'_{\ast}(\Delta^{p}) \cong \bvec{\mathsf{H}}_{\ast} (\Delta^{p}) \cong \zbb [0]$, and hence that $\zbb S^\dcal (X)$, $\zbb S^\dcal_{\rm sub} (X)$, and $\bvec{\mathsf{C}}_\star (X)$ are naturally chain homotopy equivalent (see \cite{EM}). Thus, our singular homology, Hector's singular homology, and the cubic homology are naturally isomorphic, which justifies the comment in the footnote of \cite[p. 183]{IZ}.
	\par\indent
	It should be noted that the canonical natural chain homotopy equivalence $\zbb S^{\dcal}_{\mathrm{sub}}(X) \longhookrightarrow \zbb S^{\dcal}(X)$ exists (see \cite[Lemma 3.1 and Remark A.5]{origin}).
\end{rem}


\subsection{Proof of Theorem \ref{Quillenequiv}}
In this subsection, we give a proof of Theorem \ref{Quillenequiv}. Since the composite of the Quillen pairs
\[
\begin{tikzcd}
\scal \arrow[yshift = 0.1cm]{r}{\absno{\ }_{\dcal} } & \arrow[yshift = -0.1cm]{l}{S^{\dcal}}  \dcal  \arrow[yshift = 0.1cm]{r}{\widetilde{\cdot}} & \czero \arrow[yshift = -0.1cm]{l}{R}
\end{tikzcd}
\]
is just the pair of Quillen equivalences $\absno{\ }:\scal \rightleftarrows \czero:S$ (Lemmas \ref{Quillenpairs}-\ref{clequiv} and Proposition \ref{adjoint1}), we have only to prove Theorem \ref{Quillenequiv}(1) (see \cite[Corollary 1.3.15]{Hovey}).
\par\indent
We can reduce the proof of Theorem \ref{Quillenequiv}(1) by the following lemma on the Quillen pair $(|\ |_{\dcal}, S^{\dcal})$.
\begin{lem}\label{1streduction}
	The following are equivalent:
	\begin{itemize}
		\item[$\mathrm{(i)}$]
		$(|\ |_{D},S^{\mathcal{D}})$ is a pair of Quillen equivalences.
		\item[$\mathrm{(ii)}$]
		For $K\in\scal$ and $X\in\mathcal{D}$, the natural maps
		\[
		i_K:K\longrightarrow S^{\mathcal{D}}|K|_{\mathcal{D}}\hbox{  and  } p_X:|S^{\mathcal{D}} X|_{\mathcal{D}}\longrightarrow X
		\]
		are weak equivalences in $\scal$ and $\mathcal{D}$ respectively.

		\item[$\mathrm{(iii)}$]
		For $K\in\scal$, the natural map $i_K:K\longrightarrow S^{\mathcal{D}}|K|_{\mathcal{D}}$ is a weak equivalence in $\scal$.
	\end{itemize}
	\begin{proof}
		Note that every object of $\scal$ is cofibrant and that every object of $\dcal$ is fibrant. Then the implications (i)$\implies$(ii)$\implies$(iii) are obvious.\\
		(iii)$\implies$(i) Let $X$ be a diffeological space. For a map $f:K\longrightarrow S^{\mathcal{D}} X$, consider the commutative diagram
		\[
		\begin{tikzpicture}
		\node at (0,0)
		{\begin{tikzcd}[column sep=large]
			K \arrow{r}{f} \arrow[d,"i_K"'] & S^{\mathcal{D}} X\:\: \arrow[d,"i_{S^{\mathcal{D}} X}"] \\
			S^{\mathcal{D}} |K|_{\mathcal{D}} \arrow{r}{S^{\mathcal{D}}|f|_{\dcal}} \arrow[dr,"S^{\mathcal{D}} f\hat{\phantom{g}}"']&
			S^{\mathcal{D}}|S^{\mathcal{D}} X \arrow[d,"{S^{\mathcal{D}} p_X}"]|_{\dcal}\\
			& S^{\mathcal{D}} X,\:\:
			\end{tikzcd}};
		\draw[->] (1.75,1.20)--(2.45,1.20)--(2.45,-1.25)--(1.75,-1.25);
		\node [right] at (2.5,0) {$1_{S^{\mathcal{D}} X}$};
		\end{tikzpicture}
		\]
		where $f\hat{\phantom{g}}:|K|_{\mathcal{D}}\longrightarrow X$ is the left adjunct of $f:K\longrightarrow S^{\mathcal{D}} X$. Since $i_K$ is a weak equivalence, we have the equivalences
		\begin{center}
			\begin{tabular}{rcl}
				$f$ is a weak equivalence &$\Leftrightarrow$& $S^{\mathcal{D}} f\hat{\phantom{g}}$ is a weak equivalence\\
				&$\Leftrightarrow$&
				$f\hat{\phantom{g}}$ is a weak equivalence,
			\end{tabular}
		\end{center}
		which show that $(|\ |_{\mathcal{D}},S^{\mathcal{D}})$ is a pair of Quillen equivalences.
	\end{proof}
\end{lem}

Thus, we have only to show that for $K\in\scal$, the natural map $i_K:K\longrightarrow S^{\mathcal{D}}|K|_{\mathcal{D}}$ is a weak equivalence in $\scal$. Since $|\ |_\dcal$ and $S^\dcal$ preserve coproducts, we may assume that $K$ is connected.

Let us further reduce the proof of Theorem \ref{Quillenequiv}(1) to simpler cases.

\begin{lem}\label{2ndreduction}
	Consider the following conditions:
	\begin{itemize}
		\item[$\mathrm{(i)}$]
		$f:K\longrightarrow K'$ is a weak equivalence in $\scal$.
		\item[$\mathrm{(ii)}$]
		$|f|_{\dcal}:|K|_{\dcal}\longrightarrow|K'|_{\dcal}$ is a weak equivalence in $\mathcal{D}$.
		\item[$\mathrm{(iii)}$]
		$S^{\mathcal{D}}|f|_{\mathcal{D}}:S^{\mathcal{D}}|K|_{\mathcal{D}}\longrightarrow S^{\mathcal{D}}|K'|_{\mathcal{D}}$ is a weak equivalence in $\scal$.
	\end{itemize}
	Then, the implications $\rm (i)$$\implies$$\rm (ii)$$\implies$$\rm (iii)$ hold.
\end{lem}
\begin{proof}
	Since every object of $\scal$ is cofibrant, the implication $\rm (i) \Longrightarrow (ii)$ holds by Lemma \ref{Quillenpairs} and \cite[Corollary 7.7.2]{Hi}. The implication $\rm (ii) \Longrightarrow (iii)$ is immediate from the definition of a weak equivalence in $\dcal$.
\end{proof}
For a connected simplicial set $K$, consider weak equivalences in $\scal$
$$
K \ \hookrightarrow \ L \ \hookleftarrow \ M,
$$
where $L$ is a fibrant approximation of $K$ and $M$ is a minimal subcomplex of $L$ (\cite[$\mathsection$ 9]{May}). Then, by Lemma \ref{2ndreduction}, we have only to show that for a Kan complex $K$ with only one vertex,
$i_{K}: K \longrightarrow S^{\dcal}|K|_{\dcal}$ is a weak equivalence in $\scal$.

A simplicial map $\varpi:L\longrightarrow K$ is called a {\sl covering} if $\varpi$ is a fiber bundle with discrete fiber. For a Kan complex $K$ with only one vertex, a Kan fibration $\varpi:\tilde{K}\longrightarrow K$ with $\tilde{K}$ $1$-connected is constructed in a natural way (\cite[Definition 16.4]{May}). Since the fiber of $\varpi$ is a minimal complex, $\varpi$ is a covering in our sense (see \cite[Corollary 10.8 in Chapter I]{GZ}). Thus, we call the fibration $\varpi: \widetilde{K}\longrightarrow K$ the {\sl universal covering} of $K$. 

A smooth map $p:E\longrightarrow B$ is called a {\sl fiber bundle} if any $b\in B$ has an open neighborhood $U$ such that $p^{-1}(U)\cong U\times F$ over $U$ for some $F\in\mathcal{D}$. A smooth map $p:E\longrightarrow B$ is called a {\sl covering} if $p$ is a fiber bundle with discrete fiber.

\begin{lem}\label{covering}
	Let $K$ be a Kan complex with only one vertex, and let $\varpi:\tilde{K}\longrightarrow K$ be the universal covering of $K$.
	\begin{itemize}
		\item[$(1)$]
		$|\varpi|_{\mathcal{D}}:|\tilde{K}|_{\mathcal{D}}\longrightarrow |K|_{\mathcal{D}}$ is a covering.
		\item[$(2)$]
		$S^{\mathcal{D}}|\varpi|_{\mathcal{D}}:S^{\mathcal{D}}|\tilde{K}|_{\mathcal{D}}\longrightarrow S^{\mathcal{D}}|K|_{\mathcal{D}}$ is a covering.
	\end{itemize}
	\begin{proof}
		(1) We begin by showing the following three facts on $|\varpi|_{\dcal}: |\tilde{K}|_{\dcal} \longrightarrow |K|_{\dcal}$.
		\par\indent
		{\sl Fact 1: $|\varpi|_\dcal$ is a $\dcal$-quotient map.}~Recall from Section 2.3 that the realization $|L|_{\dcal}$ of a simplicial set $L$ is defined by $|L|_{\dcal} = \underset{\Delta\downarrow L}{\colim} \ \Delta^n$. Then, we have the commutative diagram in $\dcal$
		\begin{equation*}
			\begin{tikzcd}
				\underset{\mathrm{ob}(\Delta\downarrow\tilde{K})}{\coprod}{\Delta}^{n}\arrow{d}\arrow{r} & {|\tilde{K}|}_\dcal \arrow{d}{\mathrm{|\varpi|_\dcal}}\\
				\underset{\mathrm{ob}(\Delta\downarrow K)}{\coprod}{\Delta}^{n}\arrow[r] & {|K|}_\dcal
			\end{tikzcd}
		\end{equation*}
		consisting of the canonical surjective morphisms, where $\underset{\mathrm{ob}(\Delta\downarrow L)}{\coprod}\Delta^n$ denotes the coproduct of the standard simplices indexed by the objects of $\Delta\downarrow{L}$. Since the horizontal arrows and the left vertical arrow are $\dcal$-quotient maps, $|\varpi|_\dcal$ is also a $\dcal$-quotient map.
		\par\indent
		\noindent{\sl Fact 2: Each $g \in \pi_1 (K)$ acts on $|\tilde{K}|_\dcal$ as a diffeomorphism over $|K|_{\dcal}$.}~This fact is easily seen since the action of $\pi_1 (K)$ on $|\tilde{K}|_{\dcal}$ is induced from the obvious action of $\pi_1 (K)$ on $\tilde{K}$ via the realization functor.
		\par\indent
		\noindent{\sl Fact 3: $|\varpi|_{\dcal}$ is a topological covering with fiber $\pi_{1}(K)$.}~The fact follows from Proposition \ref{adjoint1}(2) and \cite[Theorem 4.2 in Chapter III]{GZ}.
		\vspace{0.1in}\par\indent
		Let us prove the result using Facts 1-3.
		For $x\in |K|_\dcal$, choose an open neighborhood $U$ which is topologically contractible (\cite[1.9 in Chapter III]{GZ}). Then
		$$
		|\varpi|_\dcal : |\varpi|^{-1}_{\dcal}(U) \longrightarrow U
		$$
		is a $\dcal$-quotient map (Fact 1 and \cite[Lemma 2.4(2)]{origin}) and $|\varpi|^{-1}_{\dcal}(U)$ is isomorphic to $\underset{g\in\pi_{1} (K)}{\coprod}U_g$, where each $U_g$ is an open diffeological subspace of $|\tilde{K}|_\dcal$ which is topologically isomorphic to $U$ (Fact 3). Since each $g\in\pi_1 (K)$ acts on $|\varpi|^{-1}_{\dcal}(U)$ as a diffeomorphism over $U$ (Fact 2), each $U_g$ is diffeomorphic to $U$ via $|\varpi|_\dcal$.
		\par\indent
		(2) Let $\sigma:\Delta^n \longrightarrow |K|_\dcal$ be a singular simplex, and consider the following lifting problem
		$$
		\begin{tikzcd}
		& {|\tilde{K}|_\dcal} \arrow[d,"\mathrm{|\varpi|_\dcal}"]\\
		\Delta^n \arrow[ur,dashed] \arrow[r,"\sigma"]& {|K|_\dcal}.
		\end{tikzcd}
		$$
		By Proposition \ref{axioms}(Axiom 1) and Fact 3 in the proof of Part 1, there is a continuous solution $\tau: \Delta^n \longrightarrow |\tilde{K}|_\dcal$ and $\{\tau \cdot g \}_{g\in\pi_1 (K)}$ is the set of all continuous solutions. Part 1 shows that the elements $\tau \cdot g$ are smooth, which completes the proof.
	\end{proof}
\end{lem}
Consider the morphism between homotopy exact sequences induced by the morphism of coverings
\begin{center}
	\begin{tikzcd}
		\pi_1(K) \arrow[d] \arrow[r,"="'] &\pi_1(K) \arrow[d]\\
		\tilde{K} \arrow[d,"\varpi"'] \arrow[r,"i_{\tilde{K}}"]& S^{\mathcal{D}}|\tilde{K}|_{\mathcal{D}} \arrow[d, "S^{\mathcal{D}}|\varpi|_{\mathcal{D}}"] \\
		K \arrow[r,"i_K"]& S^{\mathcal{D}} |K|_{\mathcal{D}}.
	\end{tikzcd}
\end{center}
For the proof of Theorem \ref{Quillenequiv}(1), we then have only to show that for a  1-connected Kan complex $K$, $i_{K}: K \longrightarrow S^{\dcal}|K|_{\dcal}$ is a weak equivalence in $\scal$.
Thus, the following lemma and Proposition \ref{homologyisone} complete the proof of Theorem \ref{Quillenequiv}(1) via the Whitehead theorem (\cite[Theorem 13.9]{May}).
\begin{lem}\label{1-connected}
	If $K$ is a $1$-connected Kan complex, then $S^{\mathcal{D}}|K|_{\mathcal{D}}$ is also a $1$-connected Kan complex.
\end{lem}
\begin{proof}
	By replacing $K$ with its minimal subcomplex, we may assume that $K$ has only one non-degenerate simplex of dimension $\leq 1$ (see Lemma \ref{2ndreduction}). Then $|K|_{\mathcal{D}}$ is a diffeological space obtained from the $0$-simplex $\ast$ by attaching simplices of dimension $\geq 2$ (see Lemma \ref{realization}). Since $\pi_{1}(S^{\dcal}|K|_{\dcal}) \cong \pi^{\dcal}_{1}(|K|_{\dcal})$ (Theorem \ref{homotopygp}), we can easily see that $\pi_{1}(S^{\dcal}|K|_{\dcal}) = 0$ by \cite[Lemma 9.9]{origin}, the transversality theorem (\cite[(14.7)]{BJ}), and the argument in the proof of Lemma \ref{cellular}.
\end{proof}

\subsection{Proof of Corollary \ref{W}}
In this subsection, we give a proof of Corollary \ref{W}.
\begin{proof}[Proof of Corollary \ref{W}] (1) Since $A$ is in $\wcal_{\dcal}$, there are a cofibrant diffeological space $A'$ and a $\dcal$-homotopy equivalence $f: A' \longrightarrow A$. Then, $\widetilde{A'}$ is a cofibrant arc-generated space (Lemma \ref{Quillenpairs}) and $\widetilde{f}: \widetilde{A'} \longrightarrow \widetilde{A}$ is a $\czero$-homotopy equivalence (see \cite[Proposition 2.13 and Lemma 2.12]{origin}).\par
	(2) We prove the result in two steps.\\
{\sl Step 1: The case where $A$ is cofibrant.} Since $\widetilde{\cdot}: \dcal \rightleftarrows \ccal^{0}: R$ is a pair of Quillen equivalences (Theorem \ref{Quillenequiv}(2)) and $\ccal^{0}_{f} = \ccal^{0}$ holds, $id: A \longrightarrow R\widetilde{A}$ is a weak equivalence in $\dcal$.

	\noindent {\sl Step 2: The case where $A$ is in $\wcal_{\dcal}$.} Let $A'$ and $f: A' \longrightarrow A$ be as in the proof of Part 1. Since $R$ preserves products and $id: I \longrightarrow R\tilde{I}$ is smooth (Proposition \ref{conven}(3)), the $\czero$-homotopy equivalence $\widetilde{f}: \widetilde{A'} \longrightarrow \widetilde{A}$ defines the $\dcal$-homotopy equivalence $R\widetilde{f}: R\widetilde{A'} \longrightarrow R\widetilde{A}.$ Thus, we have the commutative diagram in $\dcal$
	$$
	\begin{tikzcd}
	A' \arrow[d,swap,"id"]\arrow[r,"f"] \arrow[r,swap,"\simeq"] & A \arrow[d,"id"]\\
	R\widetilde{A}'\arrow[r,"R\widetilde{f}"]\arrow[r,swap,"\simeq"] & R\widetilde{A}.
	\end{tikzcd}
	$$
	Hence, $id : A \longrightarrow R\widetilde{A}$ is a weak equivalence by Step 1 and \cite[Lemma 9.4(2)]{origin}.

	(3) $\mathrm{(i)} \Longleftrightarrow \mathrm{(ii)}$ From the definition of a weak equivalence in $\dcal$ (Definition \ref{WFC}) and the equality $S^\dcal R = S$ (Proposition \ref{adjoint1}(2)), we see that $\mathrm{(i)}$ is equivalent to $\mathrm{(ii)}$.\\
	$\mathrm{(ii)}\Longleftrightarrow\mathrm{(iii)}$ Recall from Theorem \ref{originmain} and Remark \ref{arc} that $S^\dcal A$ and $S\widetilde{A}$ are Kan complexes. Since there are natural isomorphisms
	$$
	\pi^{\dcal}_{p}(A, a) \xrightarrow[\cong]{} \pi_{p}(S^{\dcal}A, a),
	$$
	$$
	\pi_{p}(\widetilde{A}, a) \xrightarrow[\cong]{} \pi_{p}(S\widetilde{A}, a)
	$$
	(Theorem \ref{homotopygp}), the equivalence $\mathrm{(ii)}\Longleftrightarrow\mathrm{(iii)}$ is obvious (see, eg, \cite[Section 3]{K}).
\end{proof}
We end this subsection with a remark on the subclass $\wcal_{\dcal}$.
\begin{rem}\label{WD}
In principle, the subclass $\ecal_{cf}$ of fibrant-cofibrant objects of a model category $\ecal$ is regarded as a subclass of objects behaving best from a model-categorical viewpoint. However, for the category $\dcal$, we adopt not $\dcal_{c}\ (= \dcal_{cf})$ but $\wcal_{\dcal}$ as a subclass of smooth homotopically good objects. The reasons are as follows:
	\begin{itemize}
		\item We are primarily concerned with $\dcal$-homotopical (or simplicial categorical) properties of diffeological spaces, not with model categorical properties of diffeological spaces (see Remark \ref{svsm}).
		\item We cannot expect that good diffeological spaces such as finite dimensional $C^{\infty}$-manifolds are cofibrant; in fact, cofibrant objects are far from seamless since they are sequential $\mathcal{I}$-cell complexes or their retracts (see Theorem \ref{originmain} for the set $\mathcal{I}$). However, we can show that a large class of $C^{\infty}$-manifolds is contained in $\wcal_{\dcal}$ (see Section 11).
	\end{itemize}

Note that the notion of the subclass $\wcal_{\dcal}$ is a hybrid of model categorical and simplicial categorical notions.
\end{rem}
\section{Smoothing of continuous maps}
We are concerned with function complexes $S^{\dcal}\dcal(A, B)$ and $S\czero(X, Y)$, and their relation (Theorem \ref{dmapsmoothing}). However, model category theory mainly deals with homotopy function complexes (\cite[Chapter 17]{Hi}). This section is devoted to the study of the relation between function complexes and homotopy function complexes for the category $\mcal$ $(= \czero, \dcal)$, which leads us to the proof of the smoothing theorem for continuous maps (Theorem \ref{dmapsmoothing}). More precisely, it enables us to reduce the proof of Theorem \ref{dmapsmoothing} to the proof of the weak equivalence between the homotopy function complexes of $\dcal$ and $\czero$.\par
We begin by clarifying significant differences between the simplicial categories $\czero$ and $\dcal$, which explains why subtle arguments are needed to obtain results for the category $\dcal$.

\subsection{Enrichment of cartesian closed categories}\label{4.1}
In this paper, we often deal with the categories $\czero$ and $\dcal$ as categories enriched over themselves and over $\scal$. In this subsection, we thus recall the basics of enrichment of a cartesian closed category over itself and establish a result on base change, which is used in this and later sections.\par
Let $\vcal$ be a cartesian closed category. Then, $\vcal$ itself is an $\vcal$-category (\cite[Proposition 6.2.6]{Borceux}). Further, $\vcal$ is both a tensored and cotensored $\vcal$-category whose tensor and cotensor are given by
\[
	V \otimes X = V \times X \:\:\:\:\:\text{and}\:\:\:\:\: Y^V = \vcal(V,Y)
\]
(\cite[Proposition 6.5.3]{Borceux}); in other words, the pairs
\[
	\cdot \times X:\vcal \rightleftarrows \vcal:\vcal(X,\cdot)\:\:{\rm and}\:\:\vcal(\cdot,Y):\vcal \rightleftarrows \vcal^\ast:\vcal(\cdot,Y) 
\]
are $\vcal$-adjoint pairs for $X$, $Y\in \vcal$, where $\vcal^\ast$ is the dual of $\vcal$ (see \cite[Propositions 6.2.2 and 6.7.4]{Borceux}).\par
The following result on base change is used in the next subsection to clarify differences between the simplicial categories $\czero$ and $\dcal$. See \cite[Section 6.4]{Borceux} for the basics of base change.
\begin{prop}\label{tensor}
	Let $L:\ucal \rightleftarrows \vcal:R$ be an adjoint pair between cartesian closed categories and regard $\vcal$ as a $\ucal$-category via $R$. Suppose that $L$ preserves terminal objects. Then, the following are equivalent:
	\begin{itemize}
		\item[{\rm (i)}] $L$ preserves finite products.
		\item[{\rm (ii)}] The adjoint pair $(L,R)$ can be enriched over $\ucal$.
		\item[{\rm (iii)}] $\vcal$ is a tensored $\ucal$-category.
		\item[{\rm (iv)}] $\vcal$ is a cotensored $\ucal$-category.
	\end{itemize}
\end{prop}
\begin{proof}
	We prove the implications $({\rm i})\Rightarrow ({\rm ii})\Rightarrow ({\rm iii})\Rightarrow({\rm i})$; we can prove the implications $({\rm ii})\Rightarrow ({\rm iv})\Rightarrow({\rm i})$ similarly. Note that $\vcal$ is regarded as a $\ucal$-category with hom-object $R\vcal(X,Y)$.\\
	$({\rm i})\Rightarrow({\rm ii})$ Noticing that $R$ is a right adjoint, we can see that the functor $R:\vcal \rightarrow \ucal$ is enriched over $\ucal$ in a canonical manner. Using condition $({\rm i})$, we also see that the functor $L:\ucal \longrightarrow \vcal$ and the natural bijection $\ucal(U,RX)\cong \vcal(LU,X)$ are enriched over $\ucal$ in a canonical manner.\\
	$({\rm ii})\Rightarrow({\rm iii})$ The pair $\cdot \times X:\vcal \rightleftarrows \vcal:\vcal(X,\cdot)$ is a $\vcal$-(and hence $\ucal$-)adjoint pair. By composing $(L,R)$ with $(\cdot \times X,\vcal(X,\cdot))$, we thus obtain the $\ucal$-adjoint pair
	\[
		L\cdot\times X:\,\ucal \rightleftarrows \vcal:R\vcal(X,\cdot),
	\]
	which shows that $\vcal$ is a tensored $\ucal$-category (\cite[Proposition 6.7.4]{Borceux}).\\
	$({\rm iii})\Rightarrow ({\rm i})$ Let $U\otimes X$ denote the tensor of $U\in \ucal$ and $X\in \vcal$. Then, we have the isomorphism in $\ucal$
	\[
		\ucal(U,R\vcal(X,Y))\cong R\vcal(U\otimes X,Y).
	\]
	Since $L$ preserves terminal objects, this isomorphism implies the bijection
	\[
		\ucal(U,R\vcal(X,Y))\cong \vcal(U\otimes X,Y),
	\]
	which shows that $U\otimes X\cong LU\times X$. Thus, we see from \cite[Proposition 6.5.4]{Borceux} that $LU\times LU' \cong L(U\times U')$.
\end{proof}
\subsection{Simplicial categories $\czero$ and $\mathcal{D}$}\label{4.2}
The categories $\scal$, $\mathcal{D}$, and $\czero$ are cartesian closed and the composite of the adjoint pairs
\[
\begin{tikzcd}
\scal \arrow[yshift = 0.1cm]{r}{\absno{\ }_{\dcal} } & \arrow[yshift = -0.1cm]{l}{S^{\dcal}}  \dcal  \arrow[yshift = 0.1cm]{r}{\widetilde{\cdot}} & \czero \arrow[yshift = -0.1cm]{l}{R}
\end{tikzcd}
\]
is just the adjoint pair $\absno{\ }:\scal \rightleftarrows \czero:S$ (Proposition \ref{conven}, Remark \ref{suitable}, and Proposition \ref{adjoint1}). Thus, $\mathcal{D}$ and $\czero$ are $\scal$-categories with function complexes $S^\dcal \dcal(X,Y)$ and $S\czero(X,Y)$, respectively. In this subsection, we give a precise definition of the map $\widetilde{\cdot}: S^{\dcal} \dcal (A, B) \longhookrightarrow S\czero(\widetilde{A}, \widetilde{B})$ dealt with in Theorem \ref{dmapsmoothing}, and discuss the differences between the simplicial categories $\czero$ and $\dcal$.
\par\indent
First, we prove the following result.
\begin{lem}\label{enrich}
The adjoint pair
\[
\widetilde{\cdot} : \dcal \rightleftarrows \czero : R
\]
can be enriched over $\scal$.
\end{lem}
\begin{proof}
	Since $\widetilde{\cdot}:\dcal\longrightarrow \czero$ preserves finite products (Remark \ref{suitable}), the adjoint pair $\widetilde{\cdot}:\dcal \rightleftarrows \czero:R$ can be enriched over $\dcal$ (see Proposition \ref{tensor}). Thus, we obtain the result using the equality $S^\dcal\circ R=S$.
\end{proof}
The simplicial functor $\widetilde{\cdot}: \dcal \longrightarrow \czero$ defines the natural inclusion between function complexes
\[
\widetilde{\cdot}: S^{\dcal}\dcal(A, B) \longrightarrow S\czero(\widetilde{A}, \widetilde{B}),
\]
which is the map that we are interested in (see Theorem \ref{dmapsmoothing}). Note that $\pi_{0}(\widetilde{\cdot})$ is just the natural map
\[
[A, B]_{\dcal} \longrightarrow [\widetilde{A}, \widetilde{B}]_{\czero}
\]
(\cite[Section 2.4]{origin}).

Next, we show the following basic result on the simplicial categories $\czero$ and $\dcal$.
\begin{prop}\label{enrich2}
	\begin{itemize}
	\item[{\rm (1)}] The $\scal$-category $\czero$ is both tensored and cotensored.
	\item[{\rm (2)}] The $\scal$-category $\dcal$ is neither tensored nor cotensored.
	\end{itemize}
	\begin{proof}
		{\rm (1)} Consider the adjoint pair $|\ |: \scal \rightleftarrows \czero:S$. Since $\czero$ is cartesian closed, the topological realization functor $|\ |: \scal \longrightarrow \czero$ preserves finite products (see the argument in \cite[3.5 in Chapter III]{GZ}). Thus, the result follows from Proposition \ref{tensor}.\par
		{\rm (2)} Consider the adjoint pair $|\ |_\dcal: \scal \rightleftarrows \dcal:S^\dcal$. Since $|\ |_{\dcal}: \scal \longrightarrow \dcal$ does not preserve finite products (Remark \ref{proof}), the result follows from Proposition \ref{tensor}.
	\end{proof}
\end{prop}
Let $\ecal$ be both a simplicial category and a model category. If $\ecal$ satisfies axioms M6 and M7 in \cite[Definition 9.1.6]{Hi}, then $\ecal$ is called a simplicial model category; axiom M6 in \cite[Definition 9.1.6]{Hi} requires that the simplicial category $\ecal$ is tensored and cotensored.\par
Recall the model structures on $\czero$ and $\dcal$ from Sections 2.3 and 2.4.
\begin{cor}\label{notenrich}
	\begin{itemize}
	\item[{\rm (1)}] $\czero$ is a simplicial model category.
	\item[{\rm (2)}] $\dcal$ is not a simplicial model category.
	\end{itemize}
	\begin{proof}
		{\rm (1)} By Proposition \ref{enrich2}(1), $\czero$ satisfies axiom M6. Axiom M7 is easily verified using \cite[Proposition 3.13 in Chapter II]{GJ}.\par
		{\rm (2)} By Proposition \ref{enrich2}(2), $\dcal$ does not satisfy axiom M6.
	\end{proof}
\end{cor}
\begin{rem}\label{inevitable}
	It is natural to develop smooth homotopy theory using the notions and results of model category theory. Then, it seems to be an inevitable feature of homotopical algebra of smooth spaces that one must deal with a category of smooth spaces which is not a simplicial model category. In fact, if a convenient category $\acal$ of smooth spaces is regarded as a simplicial category via a singular functor $S^\acal$ having a left adjoint, then $\acal$ does not seem to be a tensored or cotensored simplicial category (see Remark \ref{proof} and Proposition \ref{tensor}).
\end{rem}
\begin{rem}\label{homobj}
	This remark relates to the notation for hom-objects of enriched categories.
	
	Suppose that a category $\acal$ is enriched over a cartesian closed category $\vcal$. The hom-object of $\acal$ is denoted by $\acal (A, A')$. In the case of $\vcal = \mcal\: (= \dcal, \czero),$ the hom-object $\acal(A, A')$ is just the hom-set $\acal(A, A')$ endowed with the canonical $\mcal$-space structure, and hence this notation is no problem. However, in the case of $\vcal = \scal$, the hom-object (or function complex) $\acal(A, A')$ is a simplicial set whose $0^{\rm th}$ component is the hom-set $\acal(A, A')$, and hence we often write ${\rm Map}_{\acal}(A, A')$ for the function complex of $\acal$. Further if the $\scal$-category structure of $\acal$ is induced from an $\mcal$-category structure on $\acal$ via the singular functor $S^{\mcal} : \mcal \longrightarrow \scal$, then we usually write $S^{\mcal}\acal(A, A')$ for the function complex of $\acal$.
\end{rem}
\subsection{Function complexes and homotopy function complexes for $\czero$ and $\dcal$}
In this subsection, we show that under the cofibrancy condition on the source, the function complexes $S^{\dcal}\dcal(A, X)$ and $S\czero(A, X)$ are homotopy equivalent to the homotopy function complexes ${\rm map}_{\dcal}(A, X)$ and ${\rm map}_{\czero}(A, X)$ respectively. After briefly discussing the case of $\czero$ (Proposition \ref{cfctcpx}), we establish the result for $\dcal$ (Theorem \ref{dfctcpx}).
\par\indent
First, recall the following basic properties of the categories $\scal$, $\czero$, and $\dcal$:
\begin{itemize}
	\item The category $\scal$ is a simplicial model category whose objects are cofibrant (\cite[Example 9.1.13]{Hi}).
	\item The category $\czero$ is a simplicial model category whose objects are fibrant (Remark \ref{arc} and Corollary \ref{notenrich}(1)).
	\item The category $\dcal$ is both a simplicial category and a model category, but it is not a simplicial model category (Corollary \ref{notenrich}(2)). All objects of $\dcal$ are fibrant.
\end{itemize}

Next, let us review the basics of function complexes in a simplicial category and homotopy function complexes in a model category.

Let $\ecal$ be a simplicial category. Then, the function complex $\mathrm{Map}_{\ecal}(A,X)$ is assigned to $A,X\in\ecal$. The $\ecal$-homotopy set $[A,X]_\ecal$ is defined by
$$
[A,X]_\ecal=\pi_0\ \mathrm{Map}_{\ecal}(A,X).
$$
Let us call a $1$-simplex of ${\rm Map}_{\ecal} (A, X)$ an $\ecal$-{\sl homotopy}. Then, the $\ecal$-{\sl homotopy set} $[A, X]_{\ecal}$ is just the quotient set of $\ecal (A, X)$ by the equivalence relation $\simeq_{\ecal}$ generated by $\ecal$-homotopies.
We simply write $\simeq$ for $\simeq_{\ecal}$ if there is no confusion in context. The homotopy category $\pi_0\ecal$ of $\ecal$ is defined to be the category in which the objects are the same as those of $\ecal$ and the hom-sets are the $\ecal$-homotopy sets. These $\ecal$-homotopical notions specialize to the ordinary ones in the case of $\ecal = \mcal\: (= \czero,\dcal)$ (see \cite[Section 2.4]{origin}).

Next, let $\ecal$ be a model category. Then the homotopy function complex $\mathrm{map}_\ecal(A,X)$ is assigned to $A,X\in \ecal$. For simplicity, we assume that $X$ is fibrant. Then, $\mathrm{map}_\ecal(A,X)$ can be defined by
$$
\mathrm{map}_\ecal(A,X)=\ecal(\abf,X),
$$
where $\abf$ is a cosimplicial resolution of $A$, i.e., a cofibrant approximation to the constant cosimplicial object $cA$ at $A$ in the Reedy model category $\ecal^\Delta$ (\cite[p. 318]{Hi}). We further assume that $A$ is cofibrant. Then the set $\pi_0(\mathrm{map}_\ecal(A,X))$ is isomorphic to the classical homotopy set $\pi_\ecal(A,X)=\ecal(A,X)/\simeq_{\cel}$ (\cite[Proposition 17.7.1]{Hi}), and hence, to the hom-set $\ho\ \ecal(A,X)$ of the Quillen homotopy category $\ho\ \ecal$ of $\ecal $ (\cite[Theorem 8.3.9]{Hi}); see \cite[Definition 7.3.2(5) and Theorem 7.4.9]{Hi} for the relation $\simeq_{\cel}$ on $\ecal (A, X)$. (Throughout this paper, we refer to the homotopy relation in \cite[Definition 7.3.2(5)]{Hi} as the classical homotopy relation and write $\simeq_{\cel}$ for it in order to distinguish it from the $\ecal$-homotopy relation defined in the case where $\ecal$ is also a simplicial category.)  The classical homotopy category $\pi\ \ecal_{cf}$ of $\ecal$ is defined to be the category whose objects and hom-sets are fibrant-cofibrant objects of $\ecal$ and classical homotopy sets respectively. In the case of $\ecal = \czero$ or $\dcal$, we write $\pi\ \ecal_{c}$ for $\pi\ \ecal_{cf}$, since every object is fibrant.
\par\indent
Last, let $\ecal$ be a simplicial model category. For any object $A$, $A^\ast$ denotes the cosimplicial object in $\ecal$ whose $n^{\rm th}$ component is $A\otimes\Delta[n]$. If $A$ is cofibrant, the canonical map $A^\ast\longrightarrow cA$ in $\ecal^\Delta$ is a cosimplicial resolution for $A$ (\cite[Proposition 16.1.3]{Hi}), and hence the homotopy function complex $\mathrm{map}_\ecal(A,X)$ coincides with the function complex $\mathrm{Map}_\ecal(A,X)$ for fibrant $X$.
\par\indent
For the category $\czero$, we obtain the following result.
\begin{prop}\label{cfctcpx}
Let $A$ be a cofibrant arc-generated space and $X$ an arc-generated space. Then, the equality
$$
\begin{array}{rcl}
{\rm map}_{\czero} (A, X) & = & S\czero(A, X),
\end{array}
$$
and hence the equality
$$\pi_\czero(A,X) =  [A,X]_\czero$$
holds.
\end{prop}
\begin{proof}
Since $\czero$ is a simplicial model category (Corollary \ref{notenrich}(1)) and every object of $\czero$ is fibrant, the result follows from the comment above.
\end{proof}
The following theorem is a diffeological analogue of Proposition \ref{cfctcpx}.

\begin{thm}\label{dfctcpx}
Let $A$ be a cofibrant diffeological space and $X$ a diffeological space. Then there exists a homotopy equivalence
$$
\begin{array}{rcl}
{\rm map}_{\dcal}(A, X) & \simeq & S^{\dcal}\dcal(A, X)
\end{array}
$$
which is natural up to homotopy. Hence, the equality
$$
\pi_{\dcal}(A,X) = [A,X]_{\dcal}
$$
holds.
\end{thm}
Recall that $\dcal$ is not a simplicial model category (Corollary \ref{notenrich}(2)). For the proof of Theorem \ref{dfctcpx}, we thus need the following two results along with Theorem \ref{Quillenequiv}(1).

\begin{prop}\label{trivialcofibr}
Let $i:A\longhookrightarrow B $ be a trivial cofibration in $\dcal$. Then $A$ is a $\dcal$-deformation retract of $B$.
\begin{proof}
We prove the result in two steps.
\vspace{1mm}\par\noindent
{\sl Step 1: The case where $i$ is a sequential relative $\jcal$-cell complex.} By the definition of a sequential relative $\jcal$-cell complex \cite[Definition 15.1.2]{MP}, there is a sequence
$$A=B_0\xhookoverright{i_1}{5} B_1 \xhookoverright{i_2}{5} B_2 \xhookoverright{i_3}{5} \cdots \xhookoverright{i_n}{5} B_n \xhookoverright{i_{n+1}}{5} \cdots$$
in $\dcal$ such that $i$ coincides with the canonical map $A\longrightarrow \underset{\underset{n}{\longrightarrow}}{\lim}\ B_n$ and each $i_n$ fits into a pushout diagram of the form
$$
\begin{tikzcd}
\underset{\lambda}{\coprod}\ \Lambda^{p_\lambda}_{k_\lambda} \cdarrow{r,hook}\cdarrow{d} & \underset{\lambda}{\coprod}\ \Delta^{p_\lambda} \cdarrow{d}\\
B_{n-1}\cdarrow{r,"i_n",hook} & B_n.
\end{tikzcd}
$$
\indent
We construct a $\dcal$-deformation $R:B\times I\longrightarrow B$ onto $A$ as follows; see Section 1.4 for the (smooth) composite of $\dcal$-homotopies $h: X \times [a, b] \longrightarrow Y$, $k : X \times [b, c] \longrightarrow Y$ with $h(\cdot, b) = k(\cdot, b)$.
\begin{itemize}
\item
$R$ is defined to be the constant homotopy of $1_A$ on $A\times I$.
\item
$R$ is defined to be the constant homotopy of $1_{B_1}$ on $B_1\times [0,\frac{1}{2}]$ and a deformation onto $A$ on $B_1\times [\frac{1}{2},1]$ (\cite[Lemma 9.7]{origin}).
\item
Suppose that $R$ is defined on $B_{n-1}\times I$. Then $R$ is defined to be the constant homotopy of $1_{B_n}$ on $B_n\times[0,\frac{1}{2^n}]$, a deformation onto $B_{n-1}$ on $B_n\times [\frac{1}{2^{n}},\frac{1}{2^{n-1}}]$, and the composite
\[
\:\:\:\:\:\:\:\:\:\:\:\:\:\:\:\:\:\:B_n\times \left[\frac{1}{2^{n-1}},1\right]\xoverright{\ R\left(\ \cdot\ ,\frac{1}{2^{n-1}}\right)\times 1\ } B_{n-1}\times\left[\frac{1}{2^{n-1}},1\right] \xoverright{R} B_{n-1} \longhookrightarrow B_n
\]
on $ B_n\times[\frac{1}{2^{n-1}},1]$.
\end{itemize}
Since $B\times I=(\underset{\underset{n}{\longrightarrow}}{\lim}\ B_n)\times I = \underset{\underset{n}{\longrightarrow}}{\lim}\ (B_n\times I)$ (Proposition \ref{conven}(2)), $R:B\times I\longrightarrow  B$ is the desired smooth deformation.\\

\noindent
{\sl Step 2: The case where $i$ is a general trivial cofibration.} The trivial cofibration $i$ has the functional factorization
$$
\begin{tikzcd}
A \cdarrow{r,hook,"i'"} \cdarrow{rd,"i"'} & A'\cdarrow{d,"p'"}\\
 & B,
\end{tikzcd}
$$
where $i'$ is a sequential relative $\jcal$-cell complex and $p'$ is a (trivial) fibration (see \cite[the proof of Theorem 1.3]{origin}). Thus, for the commutative solid arrow diagram
$$
\begin{tikzcd}
A \cdarrow{r,"i'"} \cdarrow{d,"i"'} & A' \cdarrow{d,"p'"}\\
B \cdarrow{ur,"\ell",dashed} \cdarrow{r,"1"'}& B,
\end{tikzcd}
$$
the dotted arrows $\ell$ exists, making the diagram commute. Hence, we have the retract diagram
$$
\begin{tikzcd}
 A \cdarrow{r,"1"} \cdarrow{d,hook',"i"'} & A \cdarrow{r,"1"}  \cdarrow{d,hook',"i'"'} & A\cdarrow{d,hook',"i"} \\
 B\cdarrow{r,"\ell"'} & A' \cdarrow{r,"p'"'} & B.
\end{tikzcd}
$$
Since $i'$ is a sequential relative $\jcal$-cell complex, there are a retract $r':A'\longrightarrow A$ and a homotopy relative to $A$ $R':A'\times I \longrightarrow A'$ connecting $1_{A'}$ to $i'r'$ (Step 1). Define the map $r: B\longrightarrow A$ to be the composite $B \xoverright{\ell}A'\xoverright{r'}A$. Then, it is easily verified that $r$ is a retraction of $B$ onto $A$. Define the homotopy $R:B\times I\longrightarrow B$ to be the composite
$$
B\times I \xhookoverright{\ell\times 1}{5} A'\times I \xoverright{R'} A' \xoverright{p'}B.
$$
We can easily see that $R|_{A\times I}$ is the constant homotopy of $i$ and that $R$ is a homotopy connecting $1_B$ to $ir$, which complete the proof.
\end{proof}
\end{prop}
\begin{lem}\label{homotopyrel}
	Let $f: A \longrightarrow B$ be a morphism of $\dcal$. Then, the following are equivalent:
	\begin{itemize}
		\item[{\rm (i)}] $f: A \longrightarrow B$ is a $\dcal$-homotopy equivalence.
		\item[{\rm (ii)}] $f^{\sharp}: \dcal(B, X) \longrightarrow \dcal(A, X)$ is a $\dcal$-homotopy equivalence for any $X \in \dcal$, where $f^\sharp$ is defined by precomposition with $f$.
		\item[{\rm (iii)}] $S^{\dcal}f^{\sharp}: S^{\dcal} \dcal (B, X) \longrightarrow S^{\dcal}\dcal (A, X)$ is a homotopy equivalence of {\rm Kan} complexes for any $X \in \dcal$.
	\end{itemize}
\end{lem}
\begin{proof} ${\rm (i) \Longrightarrow (ii)}$ From Proposition \ref{conven}(2), we see that the induced map
	\[
	[C, f^{\sharp}]_{\dcal} : [C, \dcal(B, X)]_{\dcal} \longrightarrow [C, \dcal(A, X)]_{\dcal}
	\]
	is bijective for any $C \in \dcal$, and hence that $f^{\sharp}: \dcal(B, X) \longrightarrow \dcal(A, X)$ is a $\dcal$-homotopy equivalence.\\
	${\rm (ii) \Longrightarrow (iii)}$ See \cite[Lemma 9.4]{origin}.\\
	$\rmiii \Longrightarrow \rmi$ Applying the functor $\pi_{0}$ to the homotopy equivalence $S^{\dcal}f^{\sharp} : S^{\dcal}\dcal (B, X) \longrightarrow S^{\dcal}\dcal(A, X)$, we see that
	\[
	f^{\sharp} : [B, X]_{\dcal} \longrightarrow [A, X]_{\dcal}
	\]
	is bijective for $X\in\dcal$, which implies that $f: A \longrightarrow B$ is a $\dcal$-homotopy equivalence.
\end{proof}

\begin{proof}[Proof of Theorem \ref{dfctcpx}]
We prove the result in two steps.\vspace{0.1cm}\\
{\sl Step 1}: Suppose that $A = |K|_{\dcal}$ for some simplicial set $K$. We establish a homotopy equivalence
\[
{\rm map}_{\dcal} (|K|_{\dcal}, X) \simeq S^{\dcal}\dcal(|K|_{\dcal}, X).
\]
Since $\absno{\ \ \ }_{\dcal}:\scal\rightleftarrows \dcal:\sdcal$ is a Quillen pair (Lemma \ref{Quillenpairs}), the cosimplicial resolution $K^{\ast}\longrightarrow cK$ of $K$ induces the cosimplicial resolution $\absno{K^\ast}_{\dcal}\longrightarrow\absno{cK}_{\dcal}=c\absno{K}_{\dcal}$ of $\absno{K}_{\dcal}$ (\cite[Proposition 16.2.1(1)]{Hi}). Hence, we have the equality
\[
{\rm map}_{\dcal} (|K|_{\dcal}, X) = \dcal(|K^{\ast}|_{\dcal}, X).
\]
Thus, we have only to establish a homotopy equivalence
$$
\dcal(\absno{K^\ast}_{\dcal},X)  \simeq  \sdcal\dcal(\absno{K}_{\dcal},X).
$$

Consider the bisimplicial set
$$
\sdcal\dcal(\absno{K^\ast}_{\dcal},X) = \dcal(\Delta^\ast \times \absno{K^\ast}_{\dcal}  ,X)
$$
and note that
$$
\begin{array}{rcl}
\dcal(\Delta^0 \times \absno{K^\ast}_{\dcal}, X) & \cong & \dcal (\absno{K^\ast}_{\dcal},X),\\
\dcal(\Delta^\ast \times \absno{K^0}_{\dcal} , X) & \cong & \sdcal\dcal (\absno{K}_{\dcal},X)
\end{array}
$$
hold.

The canonical map $\iota : X = \dcal(\ast, X) \longrightarrow \dcal(\Delta^p,X)$ is a weak equivalence between fibrant objects (see \cite[Remark 9.3]{origin} and Lemma \ref{homotopyrel}). Thus, the induced map $\dcal(\absno{K^\ast}_{\dcal},X)\longrightarrow\dcal(\absno{K^\ast}_{\dcal},\dcal(\Delta^p,X))\cong \dcal(\Delta^p  \times\absno{K^\ast}_{\dcal},X)$ is a weak equivalence (\cite[Theorem 17.6.3(1)]{Hi}), and hence the horizontal face and degeneracy operators of $\dcal(\Delta^\ast \times\absno{K^\ast}_{\dcal},X)$ are weak equivalences.

Since $\absno{\ \ }_{\dcal}:\scal\rightleftarrows\dcal:\sdcal$ is a Quillen pair and $K\times(0)\hookrightarrow K\times \Delta[q]$ is a trivial cofibration in  $\scal$, $\absno{K}_{\dcal}=\absno{K\times(0)}_{\dcal} \longhookrightarrow\absno{K\times\Delta[q]}_{\dcal}$ is a trivial cofibration in $\dcal$, and hence a $\dcal$-homotopy equivalence (see Proposition \ref{trivialcofibr}). From this, we see that $|{proj}|_{\dcal} : |K \times \Delta[q]|_{\dcal} \longrightarrow |K|_{\dcal}$ is also a $\dcal$-homotopy equivalence. Thus, the induced map $\sdcal\dcal(\absno{K}_{\dcal},X)\longrightarrow\sdcal\dcal(\absno{K\times\Delta[q]}_{\dcal},X)=\dcal(\Delta^\ast\times\absno{K\times\Delta[q]}_{\dcal}, X)$ is a weak equivalence (Lemma \ref{homotopyrel}), and hence the vertical face and degeneracy operators of $\dcal(\Delta^\ast \times\absno{K^\ast}_{\dcal},X)$ are weak equivalences.

From these results, we have the diagram consisting of weak equivalences
$$
\begin{tikzcd}
 \dcal(\absno{K^\ast}_{\dcal},X) \cdarrow{r} & {\rm diag}\ \dcal(\Delta^\ast\times\absno{K^\ast}_{\dcal}, X) & \sdcal\dcal(\absno{K}_{\dcal},X) \cdarrow{l}
\end{tikzcd}
$$
(\cite[Corollary 15.11.12]{Hi}), which shows that $\dcal(|K^{\ast}|_{\dcal}, X) \simeq \sdcal\dcal(|K|_{\dcal}, X)$.\vspace{0.2cm}\\
{\sl Step 2}: We deal with the case where $A$ is a general cofibrant object of $\dcal$. We have the homotopy equivalence
\[
{\rm map}_{\dcal} (|S^{\dcal}A|_{\dcal}, X) \simeq S^{\dcal}\dcal(|S^{\dcal}A|_{\dcal}, X)
\]
(Step 1). Since the canonical map $p_{A}: |S^{\dcal}A|_{\dcal} \longrightarrow A$ is a weak equivalence (Theorem \ref{Quillenequiv}(1) and Lemma \ref{1streduction}), ${\rm map}_{\dcal}(|S^{\dcal}A|_{\dcal}, X) \simeq {\rm map}_{\dcal}(A, X)$ holds by \cite[Theorem 17.6.3(2)]{Hi}. Thus, we have only to show that the weak equivalence $p_A:\absno{\sdcal A}_{\dcal}\longrightarrow A$ is a $\dcal$-homotopy equivalence (see Lemma \ref{homotopyrel}).

By the Whitehead theorem (\cite[Theorem 7.5.10]{Hi}), there is a map $q_A: A\longrightarrow\absno{\sdcal A}_{\dcal}$ such that $p_Aq_A\simeq_\cel 1_A$ and $q_Ap_A\simeq_\cel 1_{\absno{\sdcal A}_{\dcal}}$. Thus there exist a cylinder object ${\rm Cyl}(A)$ of $A$ and a classical homotopy $h:\mathrm{Cyl}(A)\longrightarrow A$ such that the composite
$$
A\amalg A \xoverright{i_0+i_1}\mathrm{Cyl}(A)\xoverright{h}A
$$
is $p_Aq_A+1_A$. Factor the map $p_A$ as
\[
\begin{tikzcd}
\absno{\sdcal A}_\dcal \cdarrow{r,hook,"i'_A"} & (\absno{\sdcal A}_\dcal)' \cdarrow{r,"p'_A"} & A
\end{tikzcd}
\]
where $i'_A$ is a trivial cofibration and $p'_A$ is a fibration (see Theorem \ref{originmain}) and consider the solid arrow commutative diagram
\[
\begin{tikzcd}
A \cdarrow{d,hook',"i_0"'} \cdarrow{r,"q_A"}
& \absno{\sdcal A}_{\dcal} \cdarrow{r, hook, "i'_A"}
& (\absno{\sdcal A}_{\dcal})' \cdarrow{d,"p_A'"} \\
\mathrm{Cyl}(A) \cdarrow{urr,dashed,"H" } \cdarrow{rr,"h"'}& & A.
\end{tikzcd}
\]
Since $i_0$ is a trivial cofibration and $p_A'$ is a fibration, the dotted arrow $H$ exists, making the diagram commute. Then $q_A':=H\circ i_1$ satisfies that $p_A'q_A'=1_A$. On the other hand, we have the commutative diagram in the classical homotopy category $\pi\, \dcal_c$
\[
\begin{tikzcd}
(\absno{\sdcal A}_{\dcal})'\cdarrow{r,"p_A'"} & A\cdarrow{r,"q_A'"}\cdarrow{dr,"q_A"} & (\absno{\sdcal A}_{\dcal})'\\
\absno{\sdcal A}_{\dcal}\cdarrow{u,hook,"i_A'"} \cdarrow{ur,"p_A"} \cdarrow{rr,"1_{\absno{\sdcal A}_{\dcal}}"'}& & \absno{\sdcal A}_{\dcal}\cdarrow{u,hook,"i_A'"'}.
\end{tikzcd}
\]
Thus, we see that $i'_A\simeq_\cel q_A'p_A'i_A'$. Noticing that $\pi_{\dcal}(|K|_{\dcal}, Z) \cong [|K|_{\dcal}, Z]_{\dcal}$ for $K \in \scal$ and $Z\in \dcal$ (Step 1), we have $i'_{A} \simeq_{\dcal} q'_{A}p'_{A}i'_{A}$. Since $i_A'$ has a $\dcal$-homotopy inverse $r'_A$ (Proposition \ref{trivialcofibr}), we obtain $1_{(\absno{\sdcal A}_{\dcal})'} \simeq_{\dcal} q_A'p_A'$. Therefore, $p'_{A}: (|S^{\dcal}A|_{\dcal})' \longrightarrow A$, and hence $p_{A}: |S^{\dcal} A|_{\dcal} \longrightarrow A$ is a $\dcal$-homotopy equivalence.
\end{proof}
We derive a few corollaries of Theorem \ref{dfctcpx} using Theorem \ref{Quillenequiv}(2).\par
Let $\mcal$ denote one of the categories $\dcal$ and $\czero$. Let $\mcal_c$ (resp. $\wcal_\mcal$) denote the full subcategory of $\mcal$ consisting of cofibrant objects (resp. objects having the $\mcal$-homotopy type of a cofibrant object). Regarding $\mcal_c$ and $\wcal_\mcal$ as simplicial full subcategories of $\mcal$, we consider the homotopy categories $\pi_0 \mcal_c$ and $\pi_0\wcal_\mcal$ (see the beginning of this subsection). Then, we see that the canonical functor $\pi_0 \mcal_c \longrightarrow \pi_0 \wcal_\mcal$ is an equivalence of categories. Recall the equivalence $\pi\ \mcal_{c} \longrightarrow {\ho}\ \mcal$ between the classical homotopy category and the Quillen homotopy category (\cite[Theorem 8.3.9]{Hi}).
\begin{cor}\label{homotopycat}
	The following commutative diagram of equivalences of categories exists:
	\[
	\begin{tikzcd}
	\pi_0 \wcal_{\dcal} \arrow{d}{\pi_0 \,\widetilde{\cdot}} & \arrow{l} \pi_0 \dcal_{c} \arrow{r} \arrow{d}{\pi_0 \,\widetilde{\cdot}} & \pi\,\dcal_{c} \arrow{d}{\pi\,\widetilde{\cdot}} \arrow{r} & \ho\,\dcal \arrow{d}{\lbb\, \widetilde{\cdot}}\\
	\pi_0 \wcal_{\czero} & \arrow{l} \pi_0 \czero_{c} \arrow{r} & \pi\,\czero_{c} \arrow{r} & \ho\,\czero.
	\end{tikzcd}
	\]
	The horizontal functors are just inclusions on the object level. The vertical functors are induced by the underlying topological space functor $\widetilde{\cdot}: \dcal \longrightarrow \czero$.
\end{cor}
\begin{proof}
	By Proposition \ref{cfctcpx} and Theorem \ref{dfctcpx}, we have the isomorphisms of categories $\pi_0 \,\czero_{c} \longrightarrow \pi \ \czero_{c}$ and $\pi_0 \,\dcal_{c}\longrightarrow \pi\ \dcal_{c}$. Thus, we have the commutative diagram of functors as in the statement, whose horizontal functors are equivalences of categories. Since the right vertical arrow $\lbb\ \widetilde{\cdot}$, which is the total left derived functor of $\, \widetilde{\cdot}$, is an equivalence of categories by Theorem \ref{Quillenequiv}(2), the other vertical functors are also equivalences of categories.
\end{proof}
\begin{cor}\label{4equiv}
Let $A$ and $B$ be diffeological spaces in $\wcal_{\dcal}$. Then, the following are equivalent:
\begin{itemize}
	\item[{\rm (i)}] $f: A \longrightarrow B$ is a $\dcal$-homotopy equivalence.
	\item[{\rm (ii)}] $f: A \longrightarrow B$ is a weak equivalence in $\dcal$.
	\item[{\rm (iii)}] $\widetilde{f}: \widetilde{A} \longrightarrow \widetilde{B}$ is a $\czero$-homotopy equivalence.
	\item[{\rm (iv)}] $\widetilde{f}: \widetilde{A} \longrightarrow \widetilde{B}$ is a weak equivalence in $\czero$.
\end{itemize}
\end{cor}
\begin{proof}
	We prove the result in two steps. \vspace{0.2cm}\\
	{\sl Step 1: The case where $A$ and $B$ are cofibrant.} The result is immediate from Corollary \ref{homotopycat}.\\
	{\sl Step 2: The case where $A$ and $B$ are in $\wcal_{\dcal}$.} For a smooth map $g: X \longrightarrow Y$, the implications
	\[
	\begin{tikzcd}
	\text{$g$ is a $\dcal$-homotopy equivalence} \arrow[Rightarrow]{r} \arrow[Rightarrow]{d} &  \text{$g$ is a weak equivalence in $\dcal$}\\
	\text{$\widetilde{g}$ is a $\czero$-homotopy equivalence} \arrow[Rightarrow]{r} & \text{$\widetilde{g}$ is a weak equivalence in $\czero$}
	\end{tikzcd}
	\]
	hold (see \cite[Lemma 9.4 and Section 2.4]{origin}). We thus choose $\dcal$-homotopy equivalences $A'\longrightarrow A$ and $B \longrightarrow B'$ with $A'$ and $B'$ cofibrant and consider the composite
	\[
	A' \xrightarrow[\ \scalebox{1}{$\simeq_{\dcal}$}\ ]{} A \xrightarrow{\ \ f\ \ } B \xrightarrow[\ \scalebox{1}{$\simeq_{\dcal}$}\ ]{} B'.
	\]
	Then, the result follows from Step 1.
\end{proof}
The following result is a diffeological version of \cite[Proposition 3]{Mi}.
\begin{cor}\label{product}
Let $A$ and $B$ be diffeological spaces in $\wcal_{\dcal}$. Then the product $A\times B$ is also in $\wcal_{\dcal}$.
\end{cor}
\begin{proof}
We prove the result in two steps.
\vspace{2mm}\newline\noindent
{\sl Step 1}: We show that $\absno{K\times L}_\dcal$ is $\dcal$-homotopy equivalent to $\absno{K}_\dcal\times \absno{L}_\dcal$ for $K, L\in \scal$.

We have
\[
S^{\dcal}\dcal (|M|_{\dcal}, X) \simeq {\rm Map}_{\scal} (M, S^{\dcal}X)
\]
for $M \in \scal$ and $X \in \dcal$ (see Theorem \ref{dfctcpx}, Lemma \ref{Quillenpairs}, \cite[Proposition 17.4.16]{Hi}, and the comment before Proposition \ref{cfctcpx}). Thus, by the cartesian closedness of $\scal$ and $\dcal$, we have
$$
\begin{array}{rcl}
\sdcal\dcal(\absno{K\times L}_{\dcal},X) &\simeq & \mathrm{Map}_{\scal}(K\times L,\sdcal X)\\
& \cong & \mathrm{Map}_{\scal}(L, \mathrm{Map}_{\scal}(K,\sdcal X))\\
& \simeq & \mathrm{Map}_{\scal}(L,\sdcal\dcal(\absno{K}_{\dcal},X))\\
& \simeq & \sdcal\dcal(\absno{L}_{\dcal},\dcal(\absno{K}_{\dcal},X))\\
& \cong & \sdcal\dcal(\absno{K}_{\dcal}\times\absno{L}_{\dcal},X),
\end{array}
$$
which implies that $\absno{K\times L}_{\dcal} \simeq_\dcal \absno{K}_{\dcal}\times\absno{L}_{\dcal}$.

\vspace{2mm}\par\noindent
{\sl Step 2}: From Step 1 and Corollary \ref{4equiv}, we have the composite of $\dcal$-homotopy equivalences
$$
\absno{\sdcal A\times \sdcal B}_{\dcal}\longrightarrow \absno{\sdcal A}_{\dcal}\times \absno{\sdcal B}_{\dcal} \xoverright{p_A\times p_B} A\times B,
$$
which completes the proof.
\end{proof}
\subsection{Proof of Theorem \ref{dmapsmoothing}}
We end this section by proving Theorem \ref{dmapsmoothing}.
\begin{proof}[Proof of Theorem \ref{dmapsmoothing}]
\par\indent
Since $A$ is in $\wcal_{\dcal}$, we may assume that $A$ is a cofibrant object (see Lemma \ref{homotopyrel} and Corollary \ref{4equiv}). Note that the left Quillen functor $\widetilde{\cdot}:\dcal \longrightarrow \czero$ defines the map
\[
\widetilde{\cdot} : {\rm map}_{\dcal} (A, X) \longrightarrow {\rm map}_{\czero} (\widetilde{A}, \widetilde{X})
\]
(\cite[Proposition 16.2.1(1)]{Hi}). By Proposition \ref{cfctcpx} and Theorem \ref{dfctcpx}, we have only to show that this map is a weak equivalence. Since $\widetilde{\cdot} : \dcal \rightleftarrows \czero : R$ is a Quillen pair (Lemma \ref{Quillenpairs}), we have the isomorphism
\[
{\rm map}_{\dcal} (A, R\widetilde{X}) \cong {\rm map}_{\czero} (\widetilde{A}, \widetilde{X})
\]
(see \cite[Proposition 17.4.16]{Hi}). Since $X$ is in $\vcal_{\dcal}$, $id : X \longrightarrow R\widetilde{X}$ is a weak equivalence between fibrant objects, and hence
\[
{\rm map}_{\dcal} (A, X) \longrightarrow {\rm map}_{\dcal} (A, R\widetilde{X})
\]
is a weak equivalence (\cite[Theorem 17.6.3(1)]{Hi}), which completes the proof.
\end{proof}

\section{Smoothing of continuous principal bundles}
In this section, we give precise definitions of smooth and continuous principal bundles, and then establish a smoothing theorem for principal bundles (Theorem \ref{dbdlesmoothing}) using Theorem \ref{dmapsmoothing}. Throughout this section, $\mathcal{C}$ denotes one of the categories $C^\infty$, $\mathcal{D}$, $\mathcal{C}^0$, and $\mathcal{T}$ (see Remarks \ref{suitable} and \ref{arc} for the category $\tcal$).
\subsection{$\ccal$-partitions of unity}
In this subsection, we introduce the notion of a $\ccal$-partition of unity and the related notions.
\par
Consider the faithful functors
\[
\begin{tikzcd}
	C^\infty \overset{I}{\longhookrightarrow} \mathcal{D} \overset{\widetilde\cdot}{\longrightarrow} \mathcal{C}^0 \overset{I}{\longhookrightarrow} \mathcal{T} \tag{5.1}
\end{tikzcd}
\]
and recall from Proposition \ref{dmfd} and the definition of $\czero$ that the two functors denoted by $I$ are fully faithful. Equip the category $\mathcal{C}$ with the faithful functor $U:\mathcal{C}\rightarrow \mathcal{T}$, which is the composite of functors in (5.1). The underlying topological space $UX$ of $X\in\mathcal{C}$ is often denoted by $X$ if there is no confusion in context.\par
For an object $X$ of $\ccal$, a $\ccal$-subspace of $X$ is defined to be a subset $A$ of $X$ endowed with the initial structure for the inclusion $A\longhookrightarrow X$ with respect to the obvious underlying set functor $\ccal \longrightarrow Set$ (see Remark \ref{convenrem}(1)). If $\ccal= \dcal$, $\czero$, or $\tcal$, or if $A$ is an open set of $UX$, then $A$ admits a unique $\ccal$-subspace structure.\par
Note that $\mathbb{R}$ is canonically viewed as an object of $\mathcal{C}$ in a way compatible with the faithful functors in (5.1). A morphism of $\ccal$ with target $\mathbb{R}$ is often called a $\ccal$-function.
\begin{defn}\label{partition}
	Let $X$ be an object of $\mathcal{C}$.
	\begin{itemize}
		\item[(1)] A set $\{\varphi_i:X\longrightarrow \mathbb{R}\}$ of $\ccal$-functions is called a $\ccal$-partition of unity if $\{U\varphi_i:UX\longrightarrow U\mathbb{R}\}$ is a partition of unity in the usual sense (i.e., $\{\mathrm{carr}\ \varphi_i\}$, and hence $\{\mathrm{supp}\ \varphi_i\}$ is a locally finite covering of $X$, $\varphi_i(x) \geq 0$ for any $i$ and any $x\in X$, and $\sum_i \varphi_i = 1$).
		\item[(2)] A covering $\{U_i\}$ of $X$ is called $\ccal$-numerable if there exists a $\ccal$-partition of unity $\{ \varphi_{i}: X \longrightarrow \rbb \}$ subordinate to $\{U_i\}$ (i.e., $\mathrm{supp}\, \varphi_{i} \subset U_{i}$).
		\item[(3)] $X$ is called $\ccal$-paracompact if any open covering of $X$ is $\ccal$-numerable. $X$ is called hereditarily $\ccal$-paracompact if any open $\ccal$-subspace is $\ccal$-paracompact.
	\end{itemize}
\end{defn}
\begin{rem}\label{paracpt} 
	\begin{itemize}
	\item[{\rm (1)}] In the definition of a $\mathcal{C}$-numerable covering, $\{ U_{i} \}$ is often assumed to be an open covering. However, the definition and many important results apply to a covering by subsets (see, eg, \cite{Seg} and \cite{tom}).
	\item[{\rm (2)}] For a $C^{\infty}$-manifold $M$, $C^{\infty}$-paracompactness and $\dcal$-paracompactness are equivalent (see Proposition \ref{dmfd}). Thus, both $C^{\infty}$-paracompactness and $\dcal$-paracompactness are often referred to as smooth paracompactness.\par
	Both $\czero$-paracompactness and $\tcal$-paracompactness imply (ordinary) paracompactness, and the converse also holds under Hausdorff condition.
	\end{itemize}
\end{rem}
\begin{rem}\label{regular}
	An object $X$ of $\ccal$ is called $\ccal$-regular if for any $x\in X$ and any open neighborhood $U$ of $x$, there exists a $\ccal$-function $f: X \longrightarrow \rbb$ such that $f(x) = 1$ and ${\rm carr} (f) \subset U$ (cf. \cite[Convention 14.1]{KM}). If $X$ is $\ccal$-paracompact and $UX$ is a $T_{1}$-space, then $X$ is $\ccal$-regular. 
	\par\indent
	For a $C^{\infty}$-manifold $M$, $C^{\infty}$-regularity and $\dcal$-regularity are equivalent, and hence these are often referred to as smooth regularity. The notion of smooth regularity is needed in Section 11. If $\ccal$ is $\czero$ or $\tcal$, then $\ccal$-regularity is just complete regularity.\par\indent
	We can easily see that $\ccal$-regularity is inherited by every $\ccal$-subspace whose underlying topology is the subspace topology (especially by every open $\ccal$-subspace).
\end{rem}
For the basics of smooth paracompactness and smooth regularity, refer to \cite[Sections 14 and 16]{KM}, in which these notions are defined for a Hausdorff topological space $X$ with a sheaf $\acal_X$ of subalgebras of the sheaf $F^0_X$ of algebras of continuous functions satisfying the following condition: The algebra $\acal_X(U)$ of sections is closed under composites with elements of $C^\infty(\mathbb{R},\mathbb{R})$ for any open set $U$. (This is a corrected version of the notion introduced in \cite[Convention 14.1]{KM}. We call such a space a smoothly ringed space.) We can assign to an object $X$ of $\ccal$ with $UX$ Hausdorff the smoothly ringed space $(UX,\ccal(\cdot,\mathbb{R}))$. Note that separation axiom $T_{1}$ is equivalent to Hausdorff condition under $\ccal$-regularity (and hence under $\ccal$-paracompactness).
\subsection{Principal bundles in $\ccal$}
Let us introduce the notion of a principal bundle in $\ccal$.
\begin{defn}\label{bdle}
	Let $B$ be an object of $\mathcal{C}$ and $G$ a group in $\mathcal{C}$ (\cite[p. 75]{Mac}).
	\begin{itemize}
		\item[(1)] Let $\mathcal{C} G$ denote the category of right $G$-objects of $\mathcal{C}$ (i.e., objects of $\ccal$ endowed with a right $G$-action). Regard $B$ as an object of $\mathcal{C} G$ with trivial $G$-action.\par
		An object $\pi:E \longrightarrow B$ of the overcategory $\mathcal{C} G/B$ is called a principal $G$-bundle if there exists an open cover $\{U_i\}$ of $B$ such that $E|_{U_i} : = \pi^{-1} (U_{i}) \cong  U_i\times G$ in $\mathcal{C} G/U_i$.\par
		Let $\mathsf{P} \ccal G/B$ denote the full subcategory of $\ccal G/B$ consisting of principal $G$-bundles.
		\item[(2)] A principal $G$-bundle $\pi:E\longrightarrow B$ is called $\ccal$-numerable if it has a trivialization open cover $\{U_i\}$ which is $\mathcal{C}$-numerable.\par
		Let $(\mathsf{P} \ccal G /B)_{\rm num}$ denote the full subcategory of $\mathsf{P} \ccal G/B$ consisting of $\ccal$-numerable principal $G$-bundles.
	\end{itemize}
\end{defn}
We see that a morphism $f: A \longrightarrow B$ of $\ccal$ induces the pullback functors $f^{\ast} : \mathsf{P} \ccal G / B \longrightarrow \mathsf{P} \ccal G/ A$ and $f^{\ast} : (\mathsf{P} \ccal G / B)_{\rm num} \longrightarrow (\mathsf{P} \ccal G / A)_{\rm num}$.
\begin{rem}\label{4properties}
	\begin{itemize}
    \item[{\rm (1)}] Unlike $\dcal$, $\czero$, and $\tcal$, the category $C^{\infty}$ is not complete or cocomplete. However, we can easily see that all the finite limits needed in Definition \ref{bdle} and the comment after it exist in $C^{\infty}$.
	\item[{\rm (2)}] The well-known description of a principal bundle using an open cover and transition functions applies to $\ccal$ $(= \cinf, \dcal, \czero, \tcal)$. Thus, the categories $\mathsf{P} \ccal G / B$ and $(\mathsf{P} \ccal G/ B)_{\rm num}$ are essentially small.
	\item[{\rm (3)}] Let $B$ be a $C^{\infty}$-manifold and $G$ a Lie group. Then, the fully faithful embedding $C^{\infty} \longhookrightarrow \dcal$ induces an isomorphism of categories $\mathsf{P} C^{\infty} G / B \longrightarrow \mathsf{P} \dcal G / B$, which restricts to an isomorphism $(\mathsf{P} C^{\infty} G / B)_{\rm num} \longrightarrow (\mathsf{P} \dcal G / B)_{\rm num}$. To prove this, we have only to show that the total space $E$ of a principal $G$-bundle over $B$ in $\dcal$ is necessarily separated (see Proposition \ref{dmfd}). First, consider the commutative solid arrow diagram in $\dcal$
	\[
	\begin{tikzcd}
	E \arrow[hook]{rrd}{diag} \arrow[dashed]{rd} \arrow{rdd} \\
	& E\underset{B}{\times} E \arrow{d} \arrow[hook]{r} & E\times E \arrow{d} \\
	& B \arrow[hook]{r}{diag} & B\times B.
	\end{tikzcd}
	\]
	Then, there exists a (unique) dotted arrow from $E$ to the fiber product $E\times_B E$, making the diagram commute. Since $B$ is separated, $E \times_B E$ is a closed subset of $E\times E$. Note that a local trivialization $E|_U \cong U\times G$ defines the local trivialization $E \times_B E|_U \cong U\times G\times G$. Then, we can easily construct a retract diagram in $\dcal$
	\[
		E \times_B E|_U \longhookrightarrow E|_U \times E|_U \longrightarrow E \times_B E|_U,
	\]
	which shows that the topology of $E\times_B E$ is the subspace topology of $E\times E$. Thus, we have only to show that $E$ is closed in $E\times_B E$, which is easily seen from the separatedness of $G$ using the local trivializations $E \times_B E|_U \cong U\times G\times G$.
	\item[{\rm (4)}] Even if we select separation condition $\rmii$ in Remark \ref{KMsepar}, the result in Part 3 remains true. However, if we select separation condition $\rmiii$ in Remark \ref{KMsepar}, the situation is different. We can show that $\mathsf{P} C^{\infty} G / B \longhookrightarrow \mathsf{P} \dcal G / B$ is a fully faithful functor which restricts to an isomorphism $(\mathsf{P} C^{\infty} G / B)_{\rm num} \longrightarrow (\mathsf{P} \dcal G / B)_{\rm num}$ and that $\mathsf{P} C^{\infty} G / B \longrightarrow \mathsf{P} \dcal G / B$ is an isomorphism under the assumption that $B$ is $C^{\infty}$-regular. But, we cannot prove that $\mathsf{P} C^{\infty} G / B \longhookrightarrow \mathsf{P} \dcal G / B$ is an isomorphism without any additional assumption.
	\end{itemize}
\end{rem}
The following lemma is simple but important.
\begin{lem}\label{gpd}
	The category $\mathsf{P} \ccal G / B$ is a groupoid.
\end{lem}
\begin{proof}
	We observe that the natural isomorphism of monoids
	\[
	\ccal G / V(V \times G, V \times G) \cong \ccal(V, G)
	\]
	exists, and hence that $\mathsf{P} \ccal G/ V(V \times G, V \times G)$ $(= \ccal G/ V (V \times G, V \times G))$ is a group. From this, we can easily see that $\mathsf{P} \ccal G / B$ is a groupoid.
\end{proof}
The following lemma and remark explain that the underlying topological object of a principal $G$-bundle in $\dcal$ should be considered as a principal $\widetilde{G}$-bundle in $\czero$.
\begin{lem}\label{bdleforget}
	Let $B$ be a diffeological space and $G$ a diffeological group. Then, the functor $\widetilde{\cdot}:\mathcal{D}\longrightarrow \mathcal{C}^0$ induces a functor
	\[
		\mathsf{P} \dcal G/B \longrightarrow \mathsf{P}\czero {\widetilde{G}}/\widetilde{B},
	\]
	which restricts to a functor
	\[
	(\mathsf{P}\dcal G/B)_{\mathrm{num}} \longrightarrow (\mathsf{P}\czero {\widetilde{G}}/\widetilde{B})_{\mathrm{num}}.
	\]
	\begin{proof}
		For an object $\pi : E \longrightarrow B$ of $\ccal G / B$, a local trivialization $E|_{U} \cong U \times G$ is identified with a pullback diagram in $\ccal$
	\[
	\begin{tikzcd}
	U \times G \arrow[hook]{r}{j} \arrow[swap]{d}{proj} & E \arrow{d}{\pi} \\
	U \arrow[hook]{r}{i} & B
	\end{tikzcd}
	\]
	such that $i$ is an open $\ccal$-embedding and $j$ is $G$-equivariant.
	
	Recall from Remark \ref{suitable} that the functor $\widetilde{\cdot}: \dcal \longrightarrow \czero$ preserves finite products. Then, we see that $\widetilde{G}$ is a group in $\czero$ and that $\widetilde{\cdot}$ induces the functor $\widetilde{\cdot} : \dcal G / B \longrightarrow \czero {\widetilde{G}} / \widetilde{B}$. Further, recall that the underlying set functors $\dcal \longrightarrow Set$ and $\czero \longrightarrow Set$ create limits (Proposition \ref{conven}(1) and \cite[Proposition 2.6]{origin}) and note that Lemma 2.3 in \cite{origin} remains true for the concrete category $\czero$ and that $\widetilde{\cdot}$ transforms open $\dcal$-embeddings to open $\czero$-embeddings (cf. \cite[Lemma 3.18]{CSW}). Then, we see that $\widetilde{\cdot}$ transforms local trivializations of an object $\pi : E \longrightarrow B$ of $\dcal G / B$ to ones of the object $\widetilde{\pi} : \widetilde{E} \longrightarrow \widetilde{B}$ of $\czero {\widetilde{G}} / \widetilde{B}$, which completes the proof.
	\end{proof}
\end{lem}
We call a group in $\czero$ an {\sl arc-generated group}. In the proof of Lemma \ref{bdleforget}, it is also shown that if $G$ is a diffeological group, then $\widetilde{G}$ is an arc-generated group.
\begin{rem}\label{bdlerem}
	Lemma \ref{bdleforget} does not hold if we replace $\czero$ with $\tcal$. In fact, since $U:\mathcal{D}\longrightarrow \mathcal{T}$, which is the composite $\mathcal{D}\overset{\widetilde{\cdot}}{\longrightarrow} \mathcal{C}^0 \overset{I}{\longhookrightarrow} \mathcal{T}$, does not preserve finite products, the underlying topological space $UG$ of a diffeological group $G$ is not necessarily a topological group (see \cite[4.16-4.26]{KM} and Proposition \ref{dmfd}). However, suppose that one of the following conditions is satisfied:
	\begin{itemize}
		\item[$\rmi$]  $UG$ is locally compact.
		\item[$\rmii$] $UG$ and $UX$ satisfy the first axiom of countability.
	\end{itemize}
	Then, $UG$ is a topological group and the isomorphism $\mathsf{P} \czero \widetilde{G} / \widetilde{X} \cong \mathsf{P}\tcal UG / UX$, and hence the functor $U: \mathsf{P}\dcal G / X \longrightarrow \mathsf{P} \tcal UG / UX$ exists (see \cite[Lemma 2.8 and Proposition 2.9]{origin} and observe that every arc-generated space is locally arcwise connected).
\end{rem}
\subsection{Fiber bundles in $\ccal$}
We discuss the smoothing of continuous sections of fiber bundles in the next section. Thus, we give a precise definition of a fiber bundle in $\ccal$ and study fiber bundles, especially in $\mcal$ $(= \dcal,\:\czero)$ (see Section 1.4).
\begin{defn}\label{fiberbdle}
	An object $p: E \longrightarrow B$ of the overcategory $\ccal/B$ is called a fiber bundle if there exist an object $F$ of $\ccal$ and an open cover $\{ U_{i} \}$ of $B$ such that $E|_{U_{i}}:= p^{-1}(U_i) \cong U_{i} \times F$ in $\ccal / U_{i}$; the object $F$ is called the fiber of $p: E \longrightarrow B$ and $p$ is often called an $F$-bundle.\par
	An $F$-bundle $p:E\longrightarrow B$ is called $\ccal$-numerable if it has a trivialization open cover $\{U_i\}$ which is $\ccal$-numerable.
\end{defn}
\begin{lem}\label{bdlequotient}
	Let $p:E\longrightarrow B$ be a fiber bundle in $\ccal$ (with nonempty fiber). Then, $p$ is a $\ccal$-quotient map whose underlying continuous map $Up$ is a topological quotient map.
	\begin{proof}
		For any $x\in B$, there exist an open neighborhood $U$ and a $\ccal$-section $s$ on $U$. Thus, we can easily see that $p$ is a $\ccal$-quotient map and that $Up$ is a $\tcal$-quotient map.
	\end{proof}
\end{lem}
In the rest of this subsection, we restrict ourselves to the case of $\ccal=\mcal$ $(=\dcal,\:\czero)$.
\begin{lem}\label{bdleforget2}
	The underlying topological space functor $\widetilde{\cdot}:\dcal\longrightarrow \czero$ sends an $F$-bundle in $\dcal$ to an $\widetilde{F}$-bundle in $\czero$.
	\begin{proof}
		For an object $p:E\longrightarrow B$ of $\ccal/B$, a local trivialization $E|_U \cong U\times F$ is identified with a pullback diagram in $\ccal$
		\[
		\begin{tikzcd}
		U\times F \arrow[hook]{r}{j} \arrow[swap]{d}{proj} & E \arrow{d}{p} \\ 
		U \arrow[hook]{r}{i} & B
		\end{tikzcd}
		\]
		such that $i$ is an open $\ccal$-embedding. Thus, we can prove the result by an argument similar to that in the proof of Lemma \ref{bdleforget}.
	\end{proof}
\end{lem}
Since Lemma \ref{bdleforget2} is an analogue of Lemma \ref{bdleforget}, statements analogous to those of Remark \ref{bdlerem} hold; the precise statements are left to the reader.\par
We endow the automorphism group ${\rm Aut}_{\mcal}(F)$ of an object $F$ of $\mcal$ with the initial structure for the two maps
\[
\begin{tikzcd}
{\rm Aut}_{\mcal}(F) \arrow[yshift = 2]{r}{{incl}} \arrow[swap, yshift = -2]{r}{{inv}} &  \mcal(F, F),
\end{tikzcd}
\]
where ${incl}(\varphi) = \varphi$ and ${inv}(\varphi) = \varphi^{-1}$.
\begin{lem}\label{auto}
	Let $\mcal$ denote one of the categories $\dcal$ and $\czero$ and let $F$ be an object of $\mcal$.
	\begin{itemize}
	\item[{\rm (1)}] There exists a bijection, natural in $A$,
	\[
		\mcal(A, {\rm Aut}_{\mcal}(F)) \cong {\rm Aut}_{\mcal / A} (A \times F).
	\]
	Hence, ${\rm Aut_{\mcal}}(F)$ is a group in $\mcal$.
	\item[{\rm (2)}] The isomorphism classes of principal ${\rm Aut}_{\mcal}(F)$-bundles over $X$ in $\mcal$ bijectively correspond to those of $F$-bundles over $X$ in $\mcal$.
	\end{itemize}
\end{lem}
\begin{proof} 
	(1) From the definition of ${\rm Aut}_\mcal(F)$, we can easily observe that a map $\varphi: A \longrightarrow {\rm Aut}_{\mcal}(F)$ is a morphism of $\mcal$ if and only if $\hat{\varphi}: = \underset{a \in A}{\coprod} \varphi(a): A \times F \longrightarrow A \times F$ is an isomorphism of $\mcal / A$. This implies the first statement, from which the second statement follows.\par
	(2) Recall the notion of a fiber bundle with structure group $G$ and the description of such a bundle using transition functions (see, eg, \cite[Chapter 2]{AP}), and note that they also apply to the category $\mcal$. Since an $F$-bundle in $\mcal$ is regarded as an $F$-bundle with structure group ${\rm Aut}_\mcal(F)$ (Part 1), the result holds.
\end{proof}
We can construct an equivalence of categories which induces the bijection between the sets of isomorphism classes in Lemma \ref{auto}(2); see Lemma \ref{equivar}(1) and see also \cite[Theorem 6.1]{CW17} for the case of $\mcal = \dcal$.\par
We end this subsection with the following remark.
\begin{rem}\label{notloctriv}
	Though we are mainly concerned with numerable locally trivial bundles (Definitions \ref{bdle} and \ref{fiberbdle}), we also mention other important objects of $\dcal/B$.\par
	Iglesias-Zemmour introduced weaker notions of a fiber bundle and a principal bundle in $\dcal$, which are defined by local triviality of the pullback along any plot (\cite[Chapter 8]{IZ}). However, no analogue exists in $C^\infty$, $\ccal^0$, or $\mathcal{T}$ and in addition, such bundles are difficult to deal with homotopically, as shown in \cite[Section 3]{CW17} using results of \cite{Wu,CW19}. Thus, we deal with locally trivial bundles under numerability conditions (see Section 5.4).\par
	Christensen-Wu \cite{CW16a, CW16b} introduced the tangent bundle $TB$ of a diffeological space $B$ as a vector space in $\dcal/B$. But, $TB$ need not be even a fiber bundle in the weaker sense.
\end{rem}
\subsection{Smoothing of principal bundles}
In this subsection, we establish a smoothing theorem for principal bundles (Theorem \ref{dbdlesmoothing}). We begin by making a brief review on the classification of principal bundles in $\mcal$ ($=\dcal$, $\czero$).\par
Milnor \cite{Milnor56} constructed the universal principal bundle $\pi:EG\longrightarrow BG$ for a topological group $G$ (see \cite[Chapter 4]{Hwse}). Since the construction also works in the category $\mathcal{C}^0$, we have the universal principal bundle $\pi:EG\longrightarrow BG$ for an arc-generated group $G$.\par
Magnot-Watts \cite{MJ} and Christensen-Wu \cite{CW17} constructed the universal principal bundles for a diffeological group $G$, which have the same underlying set and sleightly different diffeologies. Since Christensen-Wu's one $\pi:EG\longrightarrow BG$ classifies $\dcal$-numerable principal $G$-bundles without any assumption on the base, we adopt theirs as the universal principal bundle for $G$.
\par\indent
For an essentially small category $\acal$, $K\acal$ denotes the set of isomorphism classes of objects of $\acal$.
\begin{prop}\label{prop6.6}
	\begin{itemize}
		\item[{\rm(1)}] Let $X$ be an arc-generated space and $G$ an arc-generated group. Then, there exists a natural bijection
		\[
		K(\mathsf{P}\czero G/X)_{\rm num} \underset{\cong}{\longrightarrow} [X,\:BG]_{\czero}.
		\]
		\item[{\rm (2)}] Let $X$ be a diffeological space and $G$ a diffeological group. Then, there exists a natural bijection
		\[
		K(\mathsf{P}\dcal G/X)_{\rm num} \underset{\cong}{\longrightarrow} [X,\: BG]_{\mathcal{D}}.
		\]
		\item[{\rm (3)}] Let $G$ be a diffeological group. Then, $\widetilde{\pi}: \widetilde{EG} \longrightarrow \widetilde{BG}$ is a universal principal $\widetilde{G}$-bundle.
	\end{itemize}
\end{prop}

\begin{proof}
	\begin{itemize}
		\item[(1)] The relevant arguments in \cite[Chapter 4]{Hwse} also apply to the arc-generated case.
		\item[(2)] See \cite[Theorem 5.10]{CW17}.
		\item[(3)] By Lemma \ref{bdleforget}, $\widetilde{\pi}: \widetilde{EG} \longrightarrow \widetilde{BG}$ is a $\czero$-numerable principal $\widetilde{G}$-bundle in $\czero$. Since $EG$ is $\dcal$-contractible (\cite[Corollary 5.5]{CW17}), $\widetilde{EG}$ is contractible, which implies the universality of $\widetilde{\pi}$ by the arc-generated version of \cite[Theorem 7.5]{Dold}.\qedhere
	\end{itemize}
\end{proof}
Recall the model structures of $\dcal$ and $\ccal^0$ from Sections 2.3-2.4 and the subclass $\vcal_\dcal$ from Section 1.3.
\begin{cor}\label{bdlecpx}
	\begin{itemize}
	\item[{\rm (1)}] Let $\mcal$ be one of the categories $\dcal$ and $\czero$. Then, every fiber bundle in $\mcal$ is a fibration with respect to the standard model structure of $\mcal$.
	\item[{\rm (2)}] Let $\pi:E\longrightarrow B$ be an $F$-bundle in $\mathcal{D}$. If two of $F$, $E$, and $B$ are in $\mathcal{V}_\mathcal{D}$, then so is the third.
	\end{itemize}
\end{cor}
\begin{proof}$\rmone$ {\sl The case of $\mcal = \dcal$.} We have only to show that for an $F$-bundle $\pi: E \longrightarrow B$ in $\dcal$, $S^{\dcal}\pi : S^{\dcal}E \longrightarrow S^{\dcal}B$ is an $S^{\dcal}F$-bundle in $\scal$ (see \cite[Lemma 9.6(1)]{origin}). The affine space $\abb^p=\{(x_0,\ldots,x_p)\in\mathbb{R}^{p+1}|\sum x_i=1\}$ is smoothly paracompact (\cite[Theorem 7.3]{BJ}). Since the inclusion $\Delta^p \longhookrightarrow \abb^p$ is a smooth map and a closed topological embedding (Lemma \ref{simplex}(3) and Proposition \ref{axioms}), $\Delta^{p}$ is also smoothly paracompact. Thus, every $F$-bundle over $\Delta^p$ is necessarily $\dcal$-numerable, and hence trivial (see Lemmas \ref{simplex}(1) and \ref{auto}(2), Proposition \ref{prop6.6}(2)). From this, we see that the pullback of $S^{\dcal}p$ along any map $\Delta[p] \longrightarrow S^{\dcal} B$ is isomorphic to $\Delta[p] \times S^{\dcal}F$ over $\Delta[p]$.\vspace{2mm}
\par\noindent
{\sl The case of $\mcal = \czero$.} We have only to show that for an $F$-bundle $\pi: E \longrightarrow B$ in $\czero$, $S\pi : SE \longrightarrow SB$ is an $S F$-bundle in $\scal$. Since the topological standard $p$-simplex $\Delta^p_{\rm top}$ is paracompact and contractible, an argument similar to that in the case of $\mcal = \dcal$ applies.\par
$\rmtwo$ Recall that $\widetilde{\pi}: \widetilde{E} \longrightarrow \widetilde{B}$ is an $\widetilde{F}$-bundle (Lemma \ref{bdleforget2}). Then we have the morphism of simplicial fiber bundles
\[
\begin{tikzcd}
S^{\dcal}E \arrow[hook]{r} \arrow[swap]{d}{S^{\dcal}\pi} & S\widetilde{E} \arrow{d}{S\widetilde{\pi}} \\ 
S^{\dcal} B \arrow[hook]{r} & S\widetilde{B}
\end{tikzcd}
\]
(Part 1), from which we can easily see the result.
\end{proof}
\begin{cor}\label{BV}
	Let $G$ be a diffeological group. Then, $G$ is in $\vcal_{\dcal}$ if and only if $BG$ is in $\vcal_{\dcal}$.
\end{cor}
\begin{proof}
	Consider the universal principal $G$-bundle $\pi : EG \longrightarrow BG$ and recall that $\widetilde{\pi}:\widetilde{EG}\longrightarrow \widetilde{BG}$ is the universal principal $\widetilde{G}$-bundle (Proposition \ref{prop6.6}(3)). Since $S^{\dcal}EG \simeq S\widetilde{EG} \simeq \ast$ (see the proof of Proposition \ref{prop6.6}(3)), the result easily follows from Corollary \ref{bdlecpx}.
\end{proof}
Now, we can prove the following theorem.
\begin{thm}\label{dbdlesmoothing}
	Let $X$ be a diffeological space and $G$ a diffeological group. If $X$ is in $\wcal_\mathcal{D}$ and $G$ is in $\vcal_\mathcal{D}$, then the functor $\widetilde{\cdot}:\mathcal{D}\longrightarrow\mathcal{C}^0$ induces the bijection
	\[
	K(\mathsf{P}\mathcal{D} G/X)_{\rm num} \underset{\cong}{\longrightarrow} K(\mathsf{P}\czero {\widetilde{G}}/\widetilde{X})_{\rm num}.
	\]
\end{thm}
\begin{proof} By Proposition \ref{prop6.6}, the natural map $K(\mathsf{P}\dcal G/ X)_{\rm num} \longrightarrow K(\mathsf{P}\czero {\widetilde{G}}/ \widetilde{X})_{\rm num}$ is identified with
	\[
	[X, BG]_{\dcal} \longrightarrow [\widetilde{X}, \widetilde{BG}]_{\czero}.
	\]
Thus, we see that this map is bijective from Theorem \ref{dmapsmoothing} and Corollary \ref{BV}.
\end{proof}
\begin{rem}\label{bdleparacpt}
	In addition to the assumptions in Theorem \ref{dbdlesmoothing}, assume that $X$ is $\dcal$-paracompact. Then, we have the natural bijection
	\[
	K(\mathsf{P}\dcal G/X) \xrightarrow[\ \ \ \cong\ \ \ ]{} K(\mathsf{P}\czero {\widetilde{G}}/\widetilde{X}).
	\]
\end{rem}

\section{Smoothing of continuous sections}
In this section, we prove the smoothing theorem for continuous sections (Theorem \ref{dsectionsmoothing}). For this, we first establish Quillen equivalences between the overcategories of $\scal$, $\dcal$, and $\czero$. We then study the relation between function complexes and homotopy function complexes for the overcategory $\mcal / X\ (= \czero/ X, \dcal / X)$, generalizing several results in Section 4.

\subsection{Quillen equivalences between the overcategories of $\scal, \dcal,$ and $\czero$}
In this subsection, we establish Quillen equivalences between the overcategories of $\scal, \dcal,$ and $\czero$, generalizing Theorem \ref{Quillenequiv}.

We begin with an observation on adjoint pairs. Let $F: \acal \rightleftarrows \bcal: G$ be an adjoint pair. Suppose that $\acal$ has pullbacks. Then, the functor
\[
F : \acal / A \longrightarrow \bcal / FA
\]
has a right adjoint for any $A \in \acal$. In fact, the right adjoint $G': \bcal / FA \longrightarrow \acal/A$ is defined by assigning to an object $q : Z \longrightarrow FA$ the object $G'q: G'Z \longrightarrow A$ which is defined by the pullback diagram in $\acal$
\[
\begin{tikzcd}
G'Z \arrow{r} \arrow[swap]{d}{G'q} & \arrow{d}{Gq} GZ\\
A \arrow{r}{i_{A}} & GFA,
\end{tikzcd}
\]
where $i_{A}$ is the unit of the adjoint pair $(F, G)$.

From this, we see that the functors
\[
|\ |_{\dcal} : \scal/K \longrightarrow \dcal / |K|_{\dcal} \ \text{\ and\ }\ \widetilde{\cdot}: \dcal / X \longrightarrow \czero / \widetilde{X}
\]
have the right adjoints $(S^{\dcal} \cdot)'$ and $R'$ respectively.

Recall that the overcategory of a model category has a canonical model structure (\cite[Theorem 7.6.5(2)]{Hi}) and note that if $F: \acal \rightleftarrows \bcal: G$ is a Quillen pair between model categories $\acal$ and $\bcal$, then $F : \acal / A \rightleftarrows \bcal / FA : G'$ is also a Quillen pair. 

We have the following result, which generalizes Theorem \ref{Quillenequiv}.
\begin{prop}\label{Quillenequiv/X}
	\begin{itemize}
	\item[{\rm (1)}] $|\ |_{\dcal} : \scal / K \rightleftarrows \dcal / |K|_{\dcal} : (S^{\dcal} \cdot)'$ is a pair of Quillen equivalences for any simplicial set $K$.
	\item[{\rm (2)}] The Quillen pair
	\[
	\widetilde{\cdot} : \dcal / X \rightleftarrows \czero / \widetilde{X} : R'
	\]
	is a pair of Quillen equivalences if and only if $X$ is in $\vcal_{\dcal}$.
	\end{itemize}
\end{prop}
\begin{proof}
    $\rmone$ Note that every object of $\mathcal{S}/K$ is cofibrant and that $E\longrightarrow |K|_\mathcal{D}$ is fibrant in $\mathcal{D}/|K|_\mathcal{D}$ if and only if $E\longrightarrow |K|_\mathcal{D}$ is a fibration in $\mathcal{D}$. Then, we can easily see from Theorem \ref{Quillenequiv}(1) and \cite[Theorem 13.1.13]{Hi} that ($|\ |_\mathcal{D}$, $(S^\mathcal{D} \cdot)'$) is a pair of Quillen equivalences.\par
    $\rmtwo$ Note that $f: A \longrightarrow X$ is cofibrant in $\dcal / X$ if and only if $A$ is cofibrant in $\dcal$ and that $E \longrightarrow \widetilde{X}$ is fibrant in $\czero / \widetilde{X}$ if and only if $E \longrightarrow \widetilde{X}$ is a fibration in $\czero$. By the definition, $X$ is in $\vcal_{\dcal}$ if and only if the unit $id: X \longrightarrow R\widetilde{X}$ of the adjoint pair $(\widetilde{\cdot}, R)$ is a weak equivalence. Thus, the result follows from Theorem \ref{Quillenequiv}(2) and the fact that $\dcal$ is right proper (see Theorem \ref{originmain} and \cite[Corollary 13.1.3]{Hi}).  
\end{proof}
The composite of the two pairs of Quillen equivalences
\[
\begin{tikzcd}
\scal / K \arrow[yshift = 2]{r}{\absno{\ }_{\dcal}}  & \arrow[yshift = -2]{l}{(S^{\dcal} \cdot)'} \dcal / \absno{K}_{\dcal}  \arrow[yshift = 2]{r}{\widetilde{\cdot}} & \arrow[yshift = -2]{l}{R'} \czero / \absno{K}
\end{tikzcd}
\]
is just the pair of Quillen equivalences
\[
|\ | : \scal / K \rightleftarrows \czero / |K| : (S \cdot)'
\]
(see Proposition \ref{adjoint1}).

\subsection{Enrichment of overcategories}
To further investigate the categories $\czero/X$ and $\dcal/X$, we have to establish basic results on the enrichment of overcategories of cartesian closed categories.\par
Let $\vcal$ be a cartesian closed category with pullbacks. For objects $p: E \longrightarrow X$, $p': E' \longrightarrow X $ of $\vcal/ X$, define the object $\vcal / X(E, E')$ of $\vcal$ by the pullback diagram in $\vcal$
\[
\begin{tikzcd}
\vcal / X (E, E') \arrow{r} \arrow{d} & \vcal (E, E') \arrow{d}{p'_{\sharp}}\\
\ast \arrow{r} & \vcal (E, X), \tag{6.1}
\end{tikzcd}
\]
where the lower horizontal arrow is the map from the terminal object $\ast$ to $\vcal (E, X)$ corresponding to $p$. The object $\vcal / X (X, E)$ is often denoted by $\Gamma(X, E)$. The underlying set of $\vcal/X(E,E')$, defined to be the set $\vcal(\ast, \vcal/X(E,E'))$, is just the hom-set $\vcal/X(E,E')$ of the category $\vcal/X$.

\begin{prop}\label{AX}
	Let $\vcal$ be a cartesian closed category with pullbacks and $X$ an object of $\vcal$. Then, $\vcal/X$ is a tensored and cotensored $\vcal$-category with hom-object $\vcal/ X (E, E')$.
\end{prop}
\begin{proof}
	We prove the results in three steps.\vspace{0.2cm}\\
	{\sl Step 1. $\vcal/X$ is a $\vcal$-category.} From the definition of $\vcal/ X (E, E')$, we can observe that the natural bijections
	\begin{align}
		\vcal (V, \vcal / X (E, E')) & \cong \vcal / X (V \times E, E') \tag{6.2}\\
		& \cong  \vcal / V \times X (V \times E, V \times E') \nonumber
	\end{align}
	exist. Using these, we can define the composition morphism $\vcal/X(E,E')\times \vcal/X(E',E'')\longrightarrow \vcal/X(E,E'')$ and the unit morphism $\ast \longrightarrow \vcal/X(E,E)$ satisfying associativity axiom and unit axiom (\cite[Definition 6.2.1]{Borceux}).\vspace{0.2cm}\\
	{\sl Step 2. $\vcal/X$ is tensored.} Let $V$ and $E$ be objects of $\vcal$ and $\vcal / X$, respectively. To see that $V \times E$ is the tensor of $V$ and $E$, we have to see that the natural bijection
	\[
	\vcal (V, \vcal / X (E, E')) \cong \vcal / X (V \times E, E')
	\]
	in (6.2) is an isomorphism in $\vcal$. This is easily seen by replacing $V$ with $U \times V$.\vspace{0.2cm}\\
	{\sl Step 3. $\vcal/X$ is cotensored.} For an object $V$ of $\vcal$ and an object $p:E\longrightarrow X$ of $\vcal/X$, define the object $p_V:E^V\longrightarrow X$ of $\vcal/X$ by the pullback diagram in $\vcal$
	\[
	\begin{tikzcd}
	E^{V} \arrow{r} \arrow[swap]{d}{p_{V}}  & \vcal(V, E) \arrow{d}{p_{\sharp}} \\
	X \arrow{r} & \vcal(V, X),
	\end{tikzcd}
	\]
	where the lower horizontal arrow is the map induced by the canonical map $V \longrightarrow \ast$. To see that $E^{V}$ is the cotensor of $V$ and $E$, we have to prove that
	\[
	\vcal / X (D, E^{V}) \cong \vcal(V, \vcal/X(D,E))
	\]
	holds in $\vcal$. By the definition of $E^{V}$, the natural bijection
	\[
	\vcal / X (D, E^{V}) \cong \vcal / X (V \times D, E)
	\]
	exists. By replacing $D$ with $U \times D$, we see that this bijection is an isomorphism in $\vcal$, which along with Step 2, implies the desired natural isomorphism.
\end{proof}
\begin{cor}\label{adjoint/X}
	Let $\vcal$ be a cartesian closed category with pullbacks and $X$ an object of $\vcal$.
	\begin{itemize}
	\item[{\rm (1)}] Let $E$ be an object of $\vcal / X$. Then, 
	\[
	\cdot \times E : \vcal \rightleftarrows \vcal / X : \vcal / X (E, \cdot) \text{   and    } E^\text{\textbullet}:\vcal \rightleftarrows (\vcal/X)^\ast:\vcal/X(\cdot,E)
	\]
	are $\vcal$-adjoint pairs.
	\item[{\rm (2)}] $X \times \cdot : \vcal \rightleftarrows \vcal / X : \Gamma (X, \cdot)$ is a $\vcal$-adjoint pair.
	\item[{\rm (3)}] Let $f : Z \longrightarrow X$ be a morphism of $\vcal$. Then, 
	\[
	f_{\sharp} : \vcal / Z \rightleftarrows \vcal/ X : f^{\ast}
	\]
	is a $\vcal$-adjoint pair, where $f_{\sharp}$ is defined by assigning to $p : D \longrightarrow Z$ the object $D \xrightarrow{\ \ p \ \ } Z \xrightarrow{\ \ f \ \ } X$ of $\vcal / X$ and $f^{\ast}$ is the pullback functor along $f$.
	\end{itemize}
\end{cor}
\begin{proof}
	(1) The result is immediate from Proposition \ref{AX} (see \cite[Proposition 6.7.4]{Borceux}).\par
	(2) The first $\vcal$-adjoint pair in Part 1 specializes to the $\vcal$-adjoint pair $(X\times\cdot,\Gamma(X,\cdot))$ by setting $E=X$.\par
	(3) It is obvious that $(f_\sharp,f^\ast)$ is an adjoint pair. Using (6.2), we can easily see that $f_{\sharp}$ and $f^{\ast}$ are $\vcal$-functors and that the natural isomorphism in $\vcal$
	\[
	\vcal / Z (D, f^{\ast} E) \cong \vcal / X (f_{\sharp} D, E)
	\]
	exists.
\end{proof}
From Corollary \ref{adjoint/X}(1), we have the natural isomorphisms in $\vcal$
\[
\vcal (V, \vcal / X (D, E) ) \cong \vcal / X (V \times D, E) \cong \vcal / X (D, E^{V}). \tag{6.3}
\]
We also have the following natural isomorphism.
\begin{cor}\label{isom/X}
	The natural isomorphism in $\vcal$ 
	\[
	\vcal(V, \vcal/X (E,E')) \cong \vcal /V \times X (V \times E, V \times E')
	\] 
	exists.
	\begin{proof}
		Recall from (6.2) the natural bijection $\vcal(V, \vcal/X (E,E')) \cong \vcal /V \times X (V \times E, V \times E')$. By replacing $V$ with $U \times V$, we see that this bijection is an isomorphism in $\vcal$.
	\end{proof}
\end{cor}
\begin{rem}\label{lccc}
	Suppose that $\vcal$ is a locally cartesian closed category. Then, $\vcal/X$ is a cartesian closed category with internal hom $\vcal_X(E,E')$, and hence a tensored and cotensored $\vcal/X$-category with hom-object $\vcal_X(E,E')$ (\cite[Proposition 6.5.3]{Borceux}). Thus, we see from Corollary \ref{adjoint/X} and \cite[Proposition 6.7.4]{Borceux} that
	\[
		\vcal/X (E,E') = \Gamma (X, \vcal_X(E,E')) \ \ \text{and} \ \ E^V = \vcal_X(X\times V, E)
	\]
	hold. Recall the local cartesian closedness of $\dcal$ and $\scal$ from \cite{BH, Dubuc} and \cite[Theorem 1.42]{John}.
\end{rem}
Given an adjoint pair $L:\ucal \rightleftarrows \vcal:R$ between cartesian closed categories, we can regard $\vcal/X$ as a $\ucal$-category via $R$. We thus consider when $L:\ucal/U\rightleftarrows \vcal/LU:R'$ can be enriched over $\ucal$ (see Section 6.1 for $R'$) and when $\vcal/X$ is tensored and cotensored as a $\ucal$-category. The following two results answer these questions, generalizing Proposition \ref{tensor}.
\begin{lem}\label{LR'}
	Let $L:\ucal \rightleftarrows \vcal:R$ be an adjoint pair between cartesian closed categories with pullbacks. Suppose that $L$ preserves terminal objects. Then, the following are equivalent:
	\begin{itemize}
		\item[{\rm (i)}] $L$ preserves finite products.
		\item[{\rm (ii)}] The adjoint pair $L:\ucal/U \rightleftarrows \vcal/LU:R'$ can be enriched over $\ucal$ for any $U\in \ucal$.
		\item[{\rm (iii)}] The adjoint pair $L:\ucal/U \rightleftarrows \vcal/LU:R'$ can be enriched over $\ucal$ for some $U\in \ucal$ with $\ucal(\ast,U)\neq \emptyset$.
	\end{itemize}
\begin{proof}
	By the assumption, $\ucal/U$ and $\vcal/LU$ are enriched over $\ucal$; $\ucal/U(D,D')$ and $R\vcal/LU(E,E')$ are hom-objects of the $\ucal$-categories $\ucal/U$ and $\vcal/LU$ respectively.\\
	${\rm (i)}\Rightarrow{\rm (ii)}$ Note that $L$ preserves finite products and that $R$ is a right adjoint. Then, we can observe from (6.2) that the adjoint pair $L:\ucal/U\rightleftarrows \vcal/LU:R'$ is enriched over $\ucal$ in a canonical manner for any $U\in \ucal$.\\
	${\rm (ii)}\Rightarrow{\rm (iii)}$ Obvious.\\
	${\rm (iii)}\Rightarrow{\rm (i)}$  Choose an object $U$ of $\ucal$ for which $L:\ucal/U \rightleftarrows \vcal/LU:R'$ is enriched over $\ucal$ and which admits a morphism $u:\ast \longrightarrow U$, and consider the composite of $\ucal$-adjoint pairs
	\[
	\begin{tikzcd}
	\ucal = \ucal/\ast \arrow[yshift = 2]{r}{u_\sharp}  & \arrow[yshift = -2]{l}{u^\ast} \ucal/U \arrow[yshift = 2]{r}{L} & \arrow[yshift = -2]{l}{R'} \vcal/LU \arrow[yshift = 2]{r}{c_\sharp} & \arrow[yshift = -2]{l}{c^\ast} \vcal/\ast = \vcal,
	\end{tikzcd}
	\]
	where $c:LU\longrightarrow \ast$ is the canonical map to the terminal object $\ast$ (see Corollary \ref{adjoint/X}(3)). Since this composite is just $(L,R)$, Proposition \ref{tensor} shows that ${\rm (iii)}\Rightarrow {\rm (i)}$.
\end{proof}
\end{lem}
\begin{prop}\label{tensor/X}
	Let $L:\ucal \rightleftarrows \vcal:R$ be an adjoint pair between cartesian closed categories. Suppose that $\vcal$ has pullbacks and that $L$ preserves terminal objects. Then, the following are equivalent:
	\begin{itemize}
		\item[{\rm (i)}] $L$ preserves finite products.
		\item[{\rm (ii)}] $\vcal/X$ is both a tensored and cotensored $\ucal$-category for any $X\in \vcal$.
		\item[{\rm (iii)}] $\vcal/X$ is a tensored or cotensored $\ucal$-category for some $X\in \vcal$ with $\vcal(\ast,X)\neq \emptyset$.
	\end{itemize}
\begin{proof}
	Note that $\vcal/X$ is a $\ucal$-category with hom-object $R\vcal/X(E,E')$.\\
	${\rm (i)\Rightarrow (ii)}$ For $E\in \vcal/X$, consider the composites of $\ucal$-adjoint pairs
	\[
		\begin{tikzcd}
		\ucal \arrow[yshift = 2]{r}{L}  & \arrow[yshift = -2]{l}{R} \vcal \arrow[yshift = 2]{r}{\cdot\times E} & \arrow[yshift = -2]{l}{\vcal/X(E,\cdot)} \vcal/X,\\
		\ucal \arrow[yshift = 2]{r}{L}  & \arrow[yshift = -2]{l}{R} \vcal \arrow[yshift = 2]{r}{E^{\text{\textbullet}}} & \arrow[yshift = -2]{l}{\vcal/X(\cdot,E)} (\vcal/X)^\ast
		\end{tikzcd}
	\]
	(Proposition \ref{tensor} and Corollary \ref{adjoint/X}(1)), which show that $\vcal/X$ is both a tensored and cotensored $\ucal$-category (see \cite[Proposition 6.7.4]{Borceux}).\\
	${\rm (ii)\Rightarrow (iii)}$ Obvious.\\
	${\rm (iii)\Rightarrow (i)}$ We prove the implication in the tensored case; the proof in the cotensored case is similar.\par
	Let $X$ be an object of $\vcal$ which admits a morphism $x:\ast \longrightarrow X$. Suppose that $\vcal/X$ is a tensored $\ucal$-category and write $U\otimes E$ for the tensor of $U\in \ucal$ and $E\in \vcal/X$. For $E\in \vcal/X$, we then consider the composite of $\ucal$-adjoint pairs
	\[
	\begin{tikzcd}
	\ucal \arrow[yshift = 2]{r}{\cdot\otimes E}  & \arrow[yshift = -2]{l}{R\vcal/X(E,\cdot)} \vcal/X \arrow[yshift = 2]{r}{c_\sharp} & \arrow[yshift = -2]{l}{c^\ast} \vcal/\ast = \vcal,
	\end{tikzcd}
	\]
	where $c:X\longrightarrow \ast$ is the canonical map to the terminal object $\ast$ (see \cite[Proposition 6.7.4]{Borceux} and Corollary \ref{adjoint/X}(3)). Note that $R\vcal/X(E,\cdot)\circ c^\ast = R\vcal(E,\cdot)$ and that every object $V$ of $\vcal$ can be regarded as an object of $\vcal/X$ by endowing it with the structure morphism $V\longrightarrow \ast \overset{x}{\longrightarrow} X$. Then, we see that $\vcal$ is a tensored $\ucal$-category (\cite[Proposition 6.7.4]{Borceux}), and hence that $L$ preserves finite products (see Proposition \ref{tensor}).
\end{proof}
\end{prop}

\subsection{Simplicial categories $\czero/X$ and $\mathcal{D}/X$}\label{6.3}
In this subsection, we establish the basic properties of the $\scal$-categories $\czero/X$ and $\dcal/X$ and discuss the differences between these simplicial categories.\par
For $\mcal=\czero$, $\dcal$, the category $\mcal/X$ is an $\mcal$-category (Proposition \ref{AX}), and hence an $\scal$-category with function complex $S^\mcal\mcal/X(E,E')$ (see Section 1.4).\par
The following lemma is a generalization of Lemma \ref{enrich}.
\begin{lem}\label{enrich/X}
	Let $X$ be a diffeological space. Then, the adjoint pair
	\[
	\widetilde{\cdot} : \dcal/ X \rightleftarrows \czero / \widetilde{X} : R'
	\]
	can be enriched over $\scal$.
	\begin{proof}
		Since $\widetilde{\cdot}:\dcal \rightleftarrows \czero:R$ is an adjoint pair whose left adjoint preserves finite products (Remark \ref{suitable}), $(\widetilde{\cdot},R')$ can be enriched over $\dcal$ (see Lemma \ref{LR'}). Thus, we obtain the result by applying the functor $S^\dcal$ (see Proposition \ref{adjoint1}).
	\end{proof}
\end{lem}
Since $\mcal/X$ is a simplicial category, $\mcal/X$-homotopies are defined as 1-simplices of function complexes $S^\mcal \mcal/X(E,E')$ (see Section 4.3). We can easily see from (6.3) that an $\mcal/X$-homotopy is just an $\mcal$-homotopy over $X$ (or vertical $\mcal$-homotopy), and hence that the $\mcal/X$-homotopy set $[E,E']_{\mcal/X}$ is just the set of vertical $\mcal$-homotopy classes of morphisms from $E$ to $E'$ over $X$.\par
The simplicial functor $\widetilde{\cdot}:\dcal/X \longrightarrow \czero/\widetilde{X}$ introduced in Lemma \ref{enrich/X} defines the natural inclusion
\[
	\widetilde{\cdot}:S^\dcal \dcal/X(D,E) \longhookrightarrow S\czero/\widetilde{X}(\widetilde{D},\widetilde{E}),
\]
which specializes to the natural inclusion
\[
	\widetilde{\cdot}:S^\dcal \Gamma (X,E) \longhookrightarrow S\Gamma (\widetilde{X},\widetilde{E}).
\]
This is the map which are dealt with in Theorem \ref{dsectionsmoothing}. Note that $\pi_0(\widetilde{\cdot})$ is just the natural map from the vertical smooth homotopy classes of smooth sections to the vertical continuous homotopy classes of continuous sections.\par
Next, we give the following basic results for $\czero/X$ and $\dcal/X$, generalizing Proposition \ref{enrich2} and Corollary \ref{notenrich}.
\begin{prop}\label{enrich2/X}
	\begin{itemize}
	\item[{\rm (1)}] The $\scal$-category $\czero/X$ is both tensored and cotensored for any $X\in \czero$.
	\item[{\rm (2)}] The $\scal$-category $\dcal/X$ is neither tensored nor cotensored for any nonempty $X\in \dcal$.
	\end{itemize}
	\begin{proof}
		The results follows from Propositions \ref{enrich2}, \ref{tensor}, and \ref{tensor/X}.
	\end{proof}
\end{prop}
Note that the tensor $K\otimes E$ and the cotensor $E^K$ of the $\scal$-category $\czero/X$ are given by
\[
K \otimes E = |K| \times E\ \text{and} \ E^{K} = E^{|K|} \tag{6.4}
\]
(see the proofs of Propositions \ref{AX} and \ref{tensor/X}).
\begin{cor}\label{notenrich/X}
	\begin{itemize}
	\item[{\rm (1)}] $\czero / X$ is a simplicial model category for any $X \in \czero$.
	\item[{\rm (2)}] $\dcal / X$ is not a simplicial model category for any nonempty $X \in \dcal$.
	\end{itemize}
	\begin{proof}
		{\rm (1)} By Proposition \ref{enrich2/X}(1), we have only to verify that $\czero/X$ satisfies axiom M7 in \cite[Definition 9.1.6]{Hi}. Since the tensor $K\otimes E$ is given by $K \otimes E = |K| \times E$ (see (6.4)), axiom M7 follows from Corollary \ref{notenrich}(1) and \cite[Proposition 9.3.7]{Hi}.\par
		{\rm (2)} The result follows from Proposition \ref{enrich2/X}(2).
	\end{proof}
\end{cor}
\begin{rem}\label{S/K}
	Note that $\scal/K$ is a tensored and cotensored $\scal$-category whose tensor $L\otimes X$ is given by $L\otimes X = L\times X$ (see Proposition \ref{AX}). Then, we can show that $\scal / K$ is a simplicial model category for any $K \in \scal$ by an argument similar to that in the proof of Corollary \ref{notenrich/X}(1).\\
\end{rem}
\subsection{Function complexes and homotopy function complexes for $\czero/X$ and $\mathcal{D}/X$}
In this subsection, we study the relationship between function complexes $S^{\mcal}\mcal/ X (D, E)$ and homotopy function complexes ${\rm map}_{\mcal / X} (D, E)$ for $\mathcal{M}/X$ $(= \czero / X \ \text{and} \ \dcal/ X)$, which leads us to the proof of Theorem \ref{dsectionsmoothing}. 
\par
First, we deal with the case of $\mathcal{M}=\czero$.
\begin{prop}\label{cfctcpx/X}
  Let $f:D\longrightarrow X$ and $p:E\longrightarrow X$ be objects of $\czero/X$. Suppose that $D$ is cofibrant in $\czero$ and that $p$ is a fibration in $\czero$. Then, the equality
  \[
  {\rm map}_{\czero/X}(D,E) = S\czero/X(D,E) 
  \]
  holds. In particular, for a cofibrant object $X$ and a fibration $p:E\longrightarrow X$ in $\czero$, the equality
  \[
 {\rm map}_{\czero/X}(X,E) = S\Gamma(X,E)
  \]
  holds.
\end{prop}
  \begin{proof}
    Note that $h:Z\longrightarrow X$ is cofibrant in $\czero/X$ if and only if $Z$ is cofibrant in $\czero$ and that $h:Z\longrightarrow X$ is fibrant in $\czero/X$ if and only if $h$ is a fibration in $\czero$. Then, the result follows from Corollary \ref{notenrich/X}(1) and the comment before Proposition \ref{cfctcpx}.
  \end{proof}
Next, we establish a diffeological version of Proposition \ref{cfctcpx/X}. Since $\mathcal{D}/X$ is not a simplicial model category unlike $\czero/X$ (Corollary \ref{notenrich/X}), we must impose stronger assumptions on $X$ and $E$.
\begin{thm}\label{dfctcpx/X}
  Let $K$ be a simplicial set and $p:E\longrightarrow |K|_\mathcal{D}$ be a $\dcal$-numerable $F$-bundle in $\dcal$. Then, there exists a homotopy equivalence
  \[
 {\rm map}_{\mathcal{D}/|K|_\mathcal{D}}(|K|_\mathcal{D},E) \simeq S^\mathcal{D}\Gamma(|K|_\mathcal{D}, E)
  \]
  which is natural up to homotopy.
\end{thm}
\begin{rem}\label{unnecessary}
  We can prove that every cofibrant object of $\mathcal{D}$ is a $\mathcal{D}$-paracompact. Thus, the $\dcal$-numerability condition on $p:E\longrightarrow |K|_\mathcal{D}$ in Theorem \ref{dfctcpx/X} is actually unnecessary.
\end{rem}
For the proof of Theorem \ref{dfctcpx/X}, we need a few lemmas. First, recall the pullback functor $\vcal/Z \overset{f^\ast}{\longleftarrow} \vcal/X$ from Corollary \ref{adjoint/X}(3) and note that the morphism $\vcal/Z(f^\ast D, f^\ast E) \overset{f^\ast}{\longleftarrow} \vcal/X(D,E)$ between the hom-objects specializes to the morphism $\Gamma(Z, f^\ast E) \overset{f^\ast}{\longleftrightarrow} \Gamma(X,E)$.\\
\begin{lem}\label{dgammahomotopy}
  Let $p:E\longrightarrow X$ be a $\dcal$-numerable $F$-bundle in $\mathcal{D}$ and $f:Z\longrightarrow X$ be a $\mathcal{D}$-homotopy equivalence. Then, the induced map
  \[
  \Gamma(Z,f^\ast E)\overset{f^\ast}{\underset{}{\longleftarrow}} \Gamma(X,E)
  \]
  is a $\dcal$-homotopy equivalence.
  \begin{proof}
    Choose a $\mathcal{D}$-homotopy inverse $g$ of $f$ and consider the pullback diagrams
    \begin{center}
      \begin{tikzcd}
        g^\ast f^\ast E \arrow[r]\arrow[d] & f^\ast E \arrow[r]\arrow[d] & E \arrow[d]\\
        X \arrow[r,"g"] & Z \arrow[r,"f"] & X
      \end{tikzcd}
    \end{center}
    and the induced morphisms of $\dcal$
    \[
    \Gamma(X,g^\ast f^\ast E)\overset{g^\ast}{\longleftarrow} \Gamma(Z,f^\ast E) \overset{f^\ast}{\longleftarrow} \Gamma(X,E).
    \]\par
    Next, choose a $\mathcal{D}$-homotopy $H:X\times I \longrightarrow X$ connecting $1_X$ to $f\circ g$ and an isomorphism $E\times I \overset{\Phi}{\underset{\cong}{\longleftarrow}}H^\ast E$ over $X\times I$ such that $\Phi_{0} : = \Phi|_{X \times (0)}=1_E$ (see \cite[Corollary 4.15]{CW17}). Since $\Gamma(X \times I, E \times I) \cong \dcal(I, \Gamma(X, E))$ (Corollary \ref{isom/X}), we have the commutative diagram
    \[
    \begin{tikzcd}
      \mathcal{D}(I,\Gamma(X,E))\arrow[d, "({i_0}^\sharp{,}{i_1}^\sharp)",swap] & \hspace{-1.5cm} \cong \Gamma(X\times I,E\times I)   & \Gamma(X\times I,H^\ast E)\arrow[l,"\Phi_\sharp"',"\cong"] & \Gamma(X,E) \arrow[l, "H^\ast",swap] \arrow[dlll, "(1_{\Gamma(X,E)}{,}\Phi_{1\sharp}\circ g^\ast \circ f^\ast)"]\\
      \Gamma(X,E)\times\Gamma(X,E),
    \end{tikzcd}
    \]
    which shows that $\Phi_{1 \sharp} \circ g^{\ast} \circ f^{\ast} \simeq_{\dcal} 1_{\Gamma(X, E)}$, and hence that $f^{\ast}$ has a left $\dcal$-homotopy inverse. Since $\Phi_{1 \sharp}$ is an isomorphism, $g^{\ast}$ has a right $\dcal$-homotopy inverse. Since $g^{\ast}$ also has a left $\dcal$-homotopy inverse, $g^{\ast}$, and hence $f^{\ast}$ is a $\dcal$-homotopy equivalence.
  \end{proof}
\end{lem}
Recall from the proof of Proposition \ref{AX} the definition of the cotensor in the $\vcal$-category $\vcal/X$.
\begin{lem}\label{EA}
  If $p:E\longrightarrow X$ be an $F$-bundle in $\mathcal{D}$, then the cotensor $p_A:E^A\longrightarrow X$ is a $\mathcal{D}(A,F)$-bundle in $\mathcal{D}$ for any $A\in \dcal$.
  \begin{proof}
  	We can identify an local trivialization $E |_{U} \cong U \times F$ with a pullback diagram in $\dcal$
  	\[
  	\begin{tikzcd}
  	U \times F \arrow[hook]{r} \arrow[swap]{d}{proj}  & E \arrow{d}{p} \\
  	U \arrow[hook]{r} & X
  	\end{tikzcd}\tag{6.5}
  	\]
  	(see the proof of Lemma \ref{bdleforget2}). Since $\dcal(A, \cdot)$ preserves limits (Proposition \ref{conven}), we obtain from pullback diagram (6.5) the pullback diagrams in $\dcal$
  	\[
  	\begin{tikzcd}
  	U \times \dcal(A, F) \arrow{r} \arrow{d}{proj} & \dcal(A, U) \times \dcal (A, F) \arrow{d}{proj} \arrow[hook]{r}   & \dcal(A, E) \arrow{d}{p_{\sharp}}\\
  	U \arrow{r} & \dcal(A, U) \arrow[hook]{r} & \dcal(A, X),  
  	\end{tikzcd}
  	\]
  	which give a local trivialization of $p_{A}$ (see \cite[Exercise 8 on page 72]{Mac}).
  \end{proof}
\end{lem}
\begin{rem}\label{cgammahomotopy}
	Lemmas \ref{dgammahomotopy} and \ref{EA} remain true even if the category $\dcal$ is replaced by $\czero$, since the arguments in their proofs also apply to the arc-generated context. The arc-generated version of Lemma \ref{dgammahomotopy} is also needed in the proof of Theorem \ref{dsectionsmoothing}.
\end{rem}
\begin{proof}[Proof of Theorem \ref{dfctcpx/X}] Recall the Quillen pair
  \[
  |\ |_\mathcal{D}:\scal/K \rightleftarrows \mathcal{D}/|K|_\mathcal{D}:(S^\mathcal{D}\cdot)'
  \]
 (Proposition \ref{Quillenequiv/X}(1)). Since $\mathcal{S}/K$ is a simplicial model category (Remark \ref{S/K}), the canonical morphism $K^\ast (:= K \times \Delta[\ast]) \longrightarrow cK$ is a cosimplicial resolution of the terminal object $K$ of $\scal / K$, and hence $|K^\ast|_\mathcal{D}\longrightarrow c|K|_\mathcal{D}$ is a cosimplicial resolution of the terminal object $|K|_{\dcal}$ of $\dcal / |K|_{\dcal}$ (see \cite[Proposition 16.2.1]{Hi}). Hence, we have
  \[
  {\rm map}_{\mathcal{D}/|K|_\mathcal{D}}(|K|_\mathcal{D},E)=\mathcal{D}/|K|_\mathcal{D}(|K^\ast|_\mathcal{D},E)
  \]
(see Corollary \ref{bdlecpx}(1)). Thus, we have only to establish a homotopy equivalence
\[
	\dcal/|K|_\dcal (|K^\ast|_\dcal, E) \simeq S^\dcal \Gamma (|K|_\dcal, E).
\]

Consider the bisimplicial set
  \[
  S^\mathcal{D}\mathcal{D}/|K|_\mathcal{D}(|K^\ast|_\mathcal{D},E)=\mathcal{D}/|K|_\mathcal{D}(\Delta^\ast \times |K^\ast|_\mathcal{D},E)
  \]
  and note that
  \[
  \mathcal{D}/|K|_\mathcal{D}(\Delta^0 \times |K^\ast|_\mathcal{D}, E)\cong \mathcal{D}/|K|_\mathcal{D}(|K^\ast|_\mathcal{D},E),
  \]
  \[
  \mathcal{D}/|K|_\mathcal{D}(\Delta^\ast \times |K^0|_\mathcal{D}, E)\cong S^\mathcal{D}\Gamma(|K|_\mathcal{D},E)
  \]
  hold (see (6.3)).
  
  We can observe from (6.3) that the canonical map $E \longrightarrow E^{\Delta^{p}}$ is a vertical $\dcal$-homotopy equivalence, and hence, a weak equivalence between fibrant objects of $\mathcal{D}/|K|_\mathcal{D}$ (see Lemma \ref{EA} and Corollary \ref{bdlecpx}(1)). Thus, the map
  \[
  \mathcal{D}/|K|_\mathcal{D} (|K^\ast|_\mathcal{D},E)\longrightarrow \mathcal{D}/|K|_\mathcal{D}(|K^\ast|_\mathcal{D},E^{\Delta^p})=\mathcal{D}/|K|_\mathcal{D}(\Delta^p\times |K^\ast|_\mathcal{D},E)
  \]
  is a weak equivalence (see \cite[Theorem 17.6.3(1)]{Hi}), and hence the horizontal face and degeneracy operators of $\mathcal{D}/|K|_\mathcal{D}(\Delta^\ast \times |K^\ast|_\mathcal{D},E)$ are weak equivalences.

  Since the map $|proj|_{\dcal}:|K\times \Delta[q]|_\mathcal{D}\longrightarrow |K|_\mathcal{D}$ is a $\mathcal{D}$-homotopy equivalence (Lemma \ref{Quillenpairs}, \cite[Corollary 7.7.2]{Hi} and Corollary \ref{4equiv}), we have the homotopy equivalences
    \begin{eqnarray*}
  	\dcal / |K|_{\dcal} (\Delta^{\ast} \times |K \times \Delta[q]|_{\dcal}, E) & \cong & S^{\dcal} \dcal / |K|_{\dcal} (|K \times \Delta[q]|_{\dcal}, E)\\
  	& \cong & S^{\dcal} \Gamma (|K \times \Delta[q]|_{\dcal}, {|proj|_{\dcal}}^{\ast}E)\\
  	& \simeq & S^{\dcal} \Gamma (|K|_{\dcal}, E)
  \end{eqnarray*}
  (Corollary \ref{adjoint/X}(3) and Lemma \ref{dgammahomotopy}), from which we see that the vertical face and degeneracy operators of $\mathcal{D}/|K|_\mathcal{D}(\Delta^\ast \times |K^\ast|_\mathcal{D},E)$ are weak equivalences.
  
  From these results, we have the diagram consisting of weak equivalences
  \[
  S^\mathcal{D}\Gamma(|K|_\mathcal{D},E)\longrightarrow\:{\rm diag}\: \mathcal{D}/|K|_\mathcal{D}(\Delta^\ast \times |K^\ast|_\mathcal{D},E) \longleftarrow \mathcal{D}/|K|_\mathcal{D}(|K^\ast|_\mathcal{D},E)
  \]
  (\cite[Corollary 15.11.12]{Hi}), which completes the proof.
\end{proof}
\subsection{Proof of Theorem \ref{dsectionsmoothing}}
In this subsection, we give a proof of Theorem \ref{dsectionsmoothing}.
\begin{proof}[Proof of Theorem \ref{dsectionsmoothing}]
  Note that $p_X:|S^\mathcal{D}X|_\mathcal{D}\longrightarrow X$ is a $\mathcal{D}$-homotopy equivalence and hence $\widetilde{p_X}:|S^\dcal X| \longrightarrow \widetilde{X}$ is a $\czero$-homotopy equivalence (Theorem \ref{Quillenequiv}(1) and Corollary \ref{4equiv}). Use the local triviality of $E$ to see that the canonical map $\widetilde{p_X^\ast E} \longrightarrow \widetilde{p_X}^\ast \widetilde{E}$ is an isomorphism in $\czero/|S^\dcal X|$ (cf. the proof of Lemma \ref{bdleforget2}). Then, by Lemma \ref{dgammahomotopy} and Remark \ref{cgammahomotopy}, we may assume that $X=|K|_\mathcal{D}$ for some $K\in \scal$. Thus, we consider the cosimplicial resolution $|K^\ast|_\mathcal{D}\longrightarrow c|K|_\mathcal{D}$ of the terminal object $|K|_{\dcal}$ of $\dcal/|K|_{\dcal}$ (see the proof of Theorem \ref{dfctcpx/X}).
  Using the Quillen pair
  \[
  \widetilde{\cdot} : \dcal / |K|_{\dcal} \rightleftarrows \czero / |K| : R'
  \]
  (Proposition \ref{Quillenequiv/X}(2)), we have the isomorphism between the homotopy function complexes
  \[
  \dcal / |K|_{\dcal} (|K^{\ast}|_{\dcal}, R'\widetilde{E}) \xrightarrow[\cong]{} \czero/|K| (|K^{\ast}|, \widetilde{E})
  \]
  (see Lemma \ref{bdleforget2}, Corollary \ref{bdlecpx}(1), and \cite[Proposition 17.4.16]{Hi}). Since $R'\widetilde{E}\longrightarrow |K|_\dcal$ is a fibration in $\dcal$ with fiber $R\widetilde{F}$, $id : E \longrightarrow R'\widetilde{E}$ is a weak equivalence between fibrant objects in $\dcal / |K|_{\dcal}$ (see \cite[Lemma 9.6(1)]{origin}). Thus, we see that the natural map
  \[
  	\dcal/ |K|_\dcal (|K^\ast|_\dcal, E) \longrightarrow \dcal/|K|_\dcal(|K^\ast|_\dcal, R'\widetilde{E}),
  \]
  and hence the natural map
  \[
  \dcal / |K|_{\dcal}(|K^{\ast}|_{\dcal}, E) \longrightarrow \czero / |K| (|K^{\ast}|, \widetilde{E})
  \]
  is a weak equivalence, which shows that
  \[
  S^{\dcal} \Gamma (|K|_{\dcal}, E) \longrightarrow S\Gamma (|K|, \widetilde{E})
  \]
  is a weak equivalence (see Proposition \ref{cfctcpx/X} and Theorem \ref{dfctcpx/X}).
\end{proof}
\section{Dwyer-Kan equivalence between $(\mathsf{P}\dcal G / X)_{\rm num}$ and $(\mathsf{P}\czero \widetilde{G} / \widetilde{X})_{\rm num}$}
In this section, we establish the Dwyer-Kan equivalence between $(\mathsf{P}\dcal G / X)_{\rm num}$ and $(\mathsf{P}\czero \widetilde{G} / \widetilde{X})_{\rm num}$ under the conditions that $X\in \wcal_\dcal$ and that $G\in \vcal_\dcal$ (Theorem \ref{dDK}). After introducing the notion of a groupoid enriched over a cartesian closed category, we use the smoothing theorem for sections (Theorem \ref{dsectionsmoothing}) to prove a smoothing result for gauge transformations, which along with Theorem \ref{dbdlesmoothing} implies Theorem \ref{dDK}.
\par\indent
Recall that $\mcal$ denotes one of the categories $\czero$ and $\dcal$ (see Section 1.4). 
\subsection{Enrichment of categories embedded into $\mcal$}\label{7.1}
Let $\acal$ and $\bcal$ be $\mcal$-categories and $I : \acal \longhookrightarrow \bcal$ an embedding (i.e., a faithful functor). We say that $I: \acal \longhookrightarrow \bcal$ is an $\mcal$-embedding if the inclusion $\acal (A, A') \longhookrightarrow \bcal(IA, IA')$ is an $\mcal$-embedding for any $A, A' \in \acal$ (see Remark \ref{convenrem}(1)).

If a category $\acal$ has an embedding $\acal \longhookrightarrow \bcal$ into an $\mcal$-category $\bcal$, then $\acal$ has the canonical $\mcal$-category structure for which $\acal \longhookrightarrow \bcal$ is an $\mcal$-embedding (see Proposition \ref{conven} and Remark \ref{suitable}). Thus, we have the commutative diagram of $\mcal$-embeddings
\[
\begin{tikzcd}
 & & \mcal {G} \arrow[hook]{dr} & \\
 \mathsf{P} \mcal {G} /X \arrow[hook]{r} & \mcal {G} / X \arrow[hook]{ur} \arrow[hook]{dr} & & \mcal\\
  & & \mcal / X \arrow[hook]{ur} &   \tag{7.1}
\end{tikzcd}
\]
for every $X \in \mcal$ and every group $G$ in $\mcal$ (see Definition \ref{bdle} for the categories in (7.1)). Note that the hom-object $\mcal/X(E,E')$ of $\mcal/X$ introduced here coincides with that introduced by pullback diagram (6.1). Since every $\mcal$-category is also an $\scal$-category via the singular functor $S^{\mcal} : \mcal \longrightarrow \scal$ (Proposition \ref{adjoint1}), we can regard (7.1) as a commutative diagram of $\scal$-categories. By Definition \ref{bdle}, $\mathsf{P}\mcal G/X$ is an $\mcal$-full subcategory and hence an $\scal$-full subcategory of $\mcal G/X$.\par
Assume given a commutative diagram of categories 
\[
\begin{tikzcd}
	\acal \arrow{r}{F} \arrow[hook']{d} & \acal' \arrow[hook']{d}\\
	\bcal \arrow{r}{\Phi} & \bcal'
\end{tikzcd}
\]
in which the vertical arrows are embeddings and the lower horizontal arrow is an $\mcal$-functor between $\mcal$-categories. Then, $F$ is an $\mcal$-functor with respect to the canonical $\mcal$-subcategory structures of $\acal$ and $\acal'$. Recall that if $\czero$ is regarded as a $\dcal$-category via $R:\czero\longrightarrow \dcal$, then $\widetilde{\cdot}:\dcal\longrightarrow \czero$ is a $\dcal$-functor (see the proof of Lemma \ref{enrich}). Then, we see that the $\dcal$-functor $\widetilde{\cdot}:\dcal\longrightarrow \czero$ restricts to the $\dcal$-(and hence $\scal$-)functors $\widetilde{\cdot}:\dcal/X\longrightarrow \czero/\widetilde{X}, \widetilde{\cdot}:\dcal G\longrightarrow \czero {\widetilde{G}}$, $\widetilde{\cdot}:\dcal G/X \longrightarrow \czero {\widetilde{G}}/\widetilde{X}$, and $\widetilde{\cdot}:\mathsf{P}\dcal G/X\longrightarrow \mathsf{P}\czero {\widetilde{G}}/\widetilde{X}$.
\begin{rem}\label{mfdenrich}
	Since $C^{\infty}$ is a full subcategory of $\dcal$ (Proposition \ref{dmfd}), the categories in (7.1) are also $\dcal$-categories and hence $\scal$-categories even if $\mcal$ and $X$ are replaced with $C^{\infty}$ and $M$ respectively. Especially, $C^{\infty}$, $C^{\infty} / M$, and $\mathsf{P} C^{\infty} {G}/ M$ are $\dcal$-(and hence $\scal$-)full subcategories of $\dcal$, $\dcal / M$, and $\mathsf{P}\dcal {G} / M$, respectively, and their hom-objects are denoted by $C^{\infty}(N, N')$, $C^{\infty}/M (E, E')$, and $\mathsf{P} C^{\infty} {G}/ M (P, P')$, respectively (see Remark \ref{homobj}).
\end{rem}
\subsection{Enriched groupoid $\mathsf{P}\mcal {G} / X$}
In this subsection, we introduce the notion of a groupoid enriched over a cartesian closed category and show that $\mathsf{P}\mcal {G} / X$ is a groupoid enriched over $\mcal$. For a principal $G$-bundle $\pi: P \longrightarrow X$ in $\mcal$, the gauge group ${\rm Gau}_{\mcal}(P)$ is introduced as the automorphism group of $P$ in the $\mcal$-groupoid $\mathsf{P}\mcal {G} / X$.
\begin{defn}\label{V-gpd}
	Let $\vcal$ be a cartesian closed category. A $\vcal$-category $\acal$ is called a $\vcal$-groupoid if for any pair $A, A' \in \acal$, there exists a morphism of $\vcal$
	\[
	\cdot^{-1} : \acal(A, A') \longrightarrow \acal(A', A)
	\]
	such that the following diagram commutes:
	\[
	\begin{tikzcd}
	\acal(A', A) \times \acal(A, A') \arrow[swap]{d}{\text{composition}}& & \arrow{d} \acal(A, A') \arrow[swap]{ll}{(\cdot^{-1}, 1_{\acal(A, A')})} \arrow{rr}{(1_{\acal(A, A')},\ \cdot^{-1})} & & \acal(A, A') \times \acal(A', A) \arrow{d}{\text{composition}} \\ 
	\acal(A', A') & & \arrow{ll}{\text{unit}} \ast \arrow[swap]{rr}{\text{unit}} & & \acal(A, A).
	\end{tikzcd}
	\]
\end{defn}
For an object $A$ of a $\vcal$-groupoid $\acal$, the {\sl automorphism group} ${\rm Aut}_{\acal}(A)$ of $A$ is defined by
\[
{\rm Aut}_{\acal} (A) = \acal(A, A),
\]
which is a group in $\vcal$. $\scal$-groupoids are just simplicial groupoids in the sense of \cite{DK}.
\begin{lem}\label{gaugegp}
	The $\mcal$-category $\mathsf{P}\mcal {G} / X$ is an $\mcal$-groupoid, and hence an $\scal$-groupoid via the singular functor $S^{\mcal} : \mcal \longrightarrow \scal$.
\end{lem}
\begin{proof}
	Since $\mathsf{P}\mcal {G} / X$ is a groupoid (Lemma \ref{gpd}), we have the inverse-assigning function
	\[
	\cdot^{-1} : \mathsf{P}\mcal {G} / X (E, E') \longrightarrow \mathsf{P}\mcal {G} / X (E', E)
	\]
	for $E, E' \in \mathsf{P}\mcal {G} / X$. Thus, we have only to show that the map $\cdot^{-1}$ is a morphism of $\mcal$.
	
	Observe that the natural bijections
	\begin{align}
		\mcal (A, \mathsf{P}\mcal {G} / X (E, E')) & =  \mcal (A, \mcal {G} / X (E, E')) & \tag{7.2} \\
		& \cong \mcal G/X (A\times E, E') &\nonumber\\
	 	& \cong  \mcal {G} / A \times X (A \times E, A \times E') & \nonumber\\
	    & = \mathsf{P}\mcal G/ A\times X(A\times E, A\times E') &\nonumber
	\end{align}
exist (see Section 7.1). Then, we can easily see that $\cdot^{-1}$ is a morphism of $\mcal$.
\end{proof}
Let $\pi : P \longrightarrow X$ be a principal $G$-bundle in $\mcal$. Then, the {\sl group} ${\rm Gau}_{\mcal} (P)$ {\sl of gauge transformations} is defined by
\[
 {\rm Gau}_{\mcal} (P) = {\rm Aut}_{\mathsf{P}\mcal {G} / X} (P).
\]
By the definition, we have
\begin{align*}
	{\rm Gau}_\mcal (P) &= \mathsf{P}\mcal G/X(P,P)\\
	& = \mcal G/X(P,P)
\end{align*}
(see Section 7.1). When $\mathsf{P} \mcal {G}/X$ is regarded as an $\scal$-groupoid, the automorphism group of $P$ in $\mathsf{P}\mcal {G} / X$ is just $S^{\mcal} {\rm Gau}_{\mcal}(P)$.
\subsection{Smoothing of gauge transformations}
In this subsection, we formulate the gauge group ${\rm Gau}_{\mcal}(P)$ as a space of sections, and apply Theorem \ref{dsectionsmoothing} to establish a smoothing result of gauge transformations. For this, we begin by refining well-known constructions (e.g., \cite[Section 37]{KM}) in the convenient category $\mathcal{M}$.\par
Let $\pi:P\longrightarrow X$ be a principal $G$-bundle in $\mathcal{M}$ and $F$ an object of $\mathcal{M}$ endowed with a left $G$-action. Then, the object $P\times_GF\longrightarrow X$ of $\mcal / X$ is defined by
\[
P\times_GF = P\times F/\sim,
\]
where $(ug,f)\sim(u,gf)$ for every $u\in P$, $g\in G$, and $f\in F$. Note that a left $G$-object $F$ of $\mathcal{M}$ can be regarded as a right $G$-object of $\mathcal{M}$ by $f\cdot g:=g^{-1}\cdot f\:(f\in F,g\in G)$.
\begin{lem}\label{equivar}
  Let $\pi:P\longrightarrow X$ be a principal $G$-bundle in $\mathcal{M}$ and $F$ a left $G$-object of $\mathcal{M}$.
  \begin{itemize}
  \item[{\rm (1)}] The canonical map $P \times_{G} F \longrightarrow X$ is an $F$-bundle in $\mcal$.
  \item[{\rm (2)}] There exists a natural isomorphism in $\mcal$
  \[
  \mathcal{M} G(P,F)\cong\Gamma(X,P\times_GF).
  \]
  \end{itemize}
\end{lem}
\begin{proof}
	$\rmone$ Consider a commutative diagram in $\mcal$
	\[
	\begin{tikzcd}
	U \times G \times F \arrow[hook]{r}{j \times 1} \arrow{d} & P \times F \arrow{d}\\
	U \times F \arrow[hook]{r}{\overline{j \times 1}} \arrow{d} & P \times_{G} F \arrow{d} \\
	U \arrow[hook]{r}{i} & X 
	\end{tikzcd}
	\]
	such that the outer rectangle is the local trivialization of the $G \times F$-bundle $P \times F$ induced by a local trivialization of the principal $G$-bundle $P$ (see the proofs of Lemmas \ref{bdleforget} and \ref{bdleforget2}) and the middle horizontal arrow $\overline{j \times 1}$ is induced from the upper horizontal arrow $j \times 1$ via the quotient space functor $\cdot / G : \mcal {G} \longrightarrow \mcal$. Since $j \times 1$ is an open $\mcal$-embedding, $\overline{j \times 1}$ is also an open $\mcal$-embedding. Further, since the lower square is a pullback diagram in $Set$, it is a pullback diagram in $\mcal$, and hence a local trivialization of $P \times_{G} F$.\par
	$\rmtwo$ Define the map
  \[
  \Phi : \mathcal{M} G(P,F)\longrightarrow \Gamma(X,P\times_GF)
  \]
  by assigning to an equivariant map $\varphi$ the section $[1_{P}, \varphi]$ obtained from the equivariant map $(1_{P}, \varphi) : P \longrightarrow P \times F$ via the quotient space functor $\cdot / G : \mcal {G} \longrightarrow \mcal$. The map $\Phi$ is then a bijection. In fact, the inverse of $\Phi$ is defined by assigning to a section $s$ the equivariant map $\varphi_{s}$ making the following diagram commute:
  \[
  \begin{tikzcd}
  P \arrow{r}{(1_{P}, \varphi_{s})} \arrow[swap]{d}{\pi} & P \times F \arrow{d}\\
  X \arrow{r}{s} & P \times_{G} F.
  \end{tikzcd}
  \]
 
  Next, we show that the natural bijection $\Phi$ is an isomorphism in $\mcal$. Replacing $\pi : P \longrightarrow X$ with $1 \times \pi : A \times P \longrightarrow A \times X$, we have the natural bijection
  \[
  \mcal {G} (A \times P, F) \longrightarrow \Gamma(A \times X, (A \times P) \times_{G} F).
  \]
  Observe that the natural bijections
  \begin{align*}
  	  \mcal {G} (A \times P, F) & \cong  \mcal(A, \mcal {G} (P, F)),\tag{7.3}\\
  \Gamma(A \times X, (A \times P) \times_{G} F) & \cong  \mcal(A, \Gamma (X, P \times_{G} F) )  
  \end{align*}
  exist (see \cite[Lemma 2.5 and its proof]{origin} and (6.2) for the second bijection). Then, we have the natural bijection
  \[
  \mcal (A, \mcal {G}(P, F)) \cong \mcal(A, \Gamma (X, P \times_{G} F)),
  \] 
  which shows that $\Phi: \mcal {G} (P, F) \longrightarrow \Gamma(X, P \times_{G} F)$ is an natural isomorphism in $\mcal$.
\end{proof}
  Let $\pi:P\longrightarrow X$ be a principal $G$-bundle in $\mathcal{M}$. Then, the (non-principal) $G$-bundle $P[G,{\rm conj}]\longrightarrow X$ is defined by
  \begin{equation*}
    \begin{split}
      P[G,{\rm conj}] &= P\times_G(G,{\rm conj})\\
      &= P \times G/(ug,h)\sim(u,ghg^{-1}).
    \end{split}
  \end{equation*}
\begin{lem}\label{conjbdle}  
  Let $\pi:P\longrightarrow X$ be a principal $G$-bundle in $\mathcal{M}$.
  \begin{itemize}
    \item[{\rm (1)}] $P[G,{\rm conj}] \longrightarrow X$ is a group in $\mathcal{M}/X$. Thus, $\Gamma(X,P[G,{\rm conj}])$ is a group in $\mcal$.
    \item[{\rm (2)}] ${\rm Gau}_\mathcal{M}(P)$  is isomorphic to $\Gamma(X,P[G,{\rm conj}])$ as a group in $\mathcal{M}$.
  \end{itemize}
  \begin{proof}
    (1) Define the multiplication
    \[
    P[G,{\rm conj}]\times_X P[G,{\rm conj}]\overset{\mu}{\longrightarrow} P[G,{\rm conj}]
    \]
    by
    \[
    ([u,g_1],[u,g_2])\longmapsto [u,g_1g_2]
    \]
    and verify that $\mu$ is well-defined and associative. Then, we can easily find the unit map $e:X\longrightarrow P[G,{\rm conj}]$ and the inverse map $\nu: P[G,{\rm conj}]\longrightarrow P[G,{\rm conj}]$ which makes $(P[G,{\rm conj}], \mu)$ a group in $\mathcal{M}/X$. Since $\Gamma(X, \cdot)$ is a right adjoint (Corollary \ref{adjoint/X}(2)), $\Gamma(X,P[G,{\rm conj}])$ is a group in $\mcal$.\par
    (2) Note that $\mathcal{M} G(P,(G,{\rm conj}))$ is a group in $\mathcal{M}$ with respect to the obvious multiplicative structure (see (7.3)). Next, observe that the natural isomorphism in Lemma \ref{equivar}(2) specializes to the natural isomorphism of groups in $\mcal$
    \[
    \mathcal{M} G(P,(G,{\rm conj}))\cong\Gamma(X,P[G,{\rm conj}]).
    \]
    Then, we have only to show that there exists a natural isomorphism of groups in $\mathcal{M}$
    \[
    \Psi : \mathcal{M} G(P,(G,{\rm conj}))\longrightarrow {\rm Gau}_\mathcal{M}(P).
    \]
    Define the map $\Psi$ by assigning to an equivariant map $\varphi$ the automorphism $\hat{\varphi}:P\longrightarrow P$ such that $\hat{\varphi}(u)=u\varphi(u)$. We can easily see that $\Psi$ is a group isomorphism. Next, let us see that $\Psi$ is an isomorphism of groups in $\mcal$. By replacing $\pi : P \longrightarrow X$ with $1 \times \pi: A \times P \longrightarrow A \times X$, we have the group isomorphism
    \[
    \Psi : \mcal {G} (A \times P, (G, {\rm conj})) \longrightarrow {\rm Gau}_{\mcal}(A \times P).
    \]
    Using (7.3) and (7.2), we have the natural isomorphism of groups
    \[
    \mcal(A, \mcal {G}(P, (G, {\rm conj}))) \cong \mcal(A, {\rm Gau}_{\mcal} (P)),
    \]
    which shows that $\Psi$ is an isomorphism of groups in $\mcal$ and completes the proof.
  \end{proof}
\end{lem}
\begin{prop}\label{gaugesmoothing}
  Let $\pi:P\longrightarrow X$ be a $\dcal$-numerable principal $G$-bundle in $\mathcal{D}$. If $X$ is in $\mathcal{W}_\mathcal{D}$ and $G$ is in $\mathcal{V}_\mathcal{D}$, then the natural inclusion
  \[
  S^\mathcal{D}{\rm Gau}_\mathcal{D}(P)\longhookrightarrow S{\rm Gau}_\czero (\widetilde{P})
  \]
  is a weak equivalence.
  \begin{proof}
    Since $P[G,{\rm conj}]=P\times_G(G,{\rm conj})=P\times G/\sim$, we have
    \[
    (P[G,{\rm conj}])^{\widetilde{}} = (P\times G/\sim)^{\widetilde{}} \: =\widetilde{P}\times\widetilde{G}/\sim \: = \widetilde{P}[\widetilde{G},{\rm conj}]
    \]
    (see \cite[Lemma 2.11 and Proposition 2.13]{origin}). Thus, the result follows from Lemmas \ref{equivar} and \ref{conjbdle} and Theorem \ref{dsectionsmoothing}.
  \end{proof}
\end{prop}
\subsection{Proof of Theorem \ref{dDK}}
In this subsection, we give a proof of Theorem \ref{dDK}.
\par\indent
$\mathsf{P} \mcal {G}/X$ is an $\scal$-category with function complex $S^{\mcal}\mathsf{P}\mcal {G}/ X(P, P')$ and the underlying topological space functor $\widetilde{\cdot}$ defines the $\mathcal{S}$-functor
\[
  \widetilde{\cdot}:\mathsf{P}\mathcal{D} G/X \longrightarrow \mathsf{P}\czero {\widetilde{G}}/\widetilde{X}
\]
(see Section 7.1).\par
Let us recall the notion of a Dwyer-Kan equivalence of simplicial categories (\cite[p. 2044]{Bergner}). First recall the homotopy category $\pi_0 \mathcal{A}$ of a simplicial category $\mathcal{A}$ from Section 4.3. A simplicial functor $F:\mathcal{A}\longrightarrow \mathcal{B}$ is called a {\sl Dwyer-Kan equivalence} if the following two conditions are satisfied:
\begin{itemize}
	\item[(i)] $\pi_0 F:\pi_0 \mathcal{A}\longrightarrow \pi_0\mathcal{B}$ is an equivalence of categories.
  \item[(ii)] $F:{\rm Map}_\mathcal{A}(A,A')\longrightarrow {\rm Map}_\mathcal{B}(FA,FA')$ is a weak equivalence in $\mathcal{S}$ for any $A$, $A'\in \mathcal{A}$.
\end{itemize}

If $\acal$ is a simplicial groupoid, then the isomorphism classes of objects of $\pi_0 \acal$ are just those of objects of $\acal$ and the following holds:
\[
{\rm Map}_\mcal (A,A') \cong
\begin{cases}
{\rm Map}_\mcal(A,A)= {\rm Aut}_\mcal(A) & \text{if $A\cong A'$} \\
\emptyset & \text{if $A \ncong A'$}.
\end{cases}
\]
Thus, a simplicial functor $F:\mathcal{A}\longrightarrow \mathcal{B}$ between simplicial groupoids is a Dwyer-Kan equivalence if and only if it satisfies  the following conditions:
\begin{itemize}
  \item[${\rm (i)}'$] $F$ induces the bijection between the isomorphism classes of objects of $\mathcal{A}$ and those of objects of $\mathcal{B}$.
  \item[$\rmii'$] $F: {\rm Aut}_{\acal} (A) \longrightarrow {\rm Aut}_{\bcal}(FA)$ is a weak equivalence in $\scal$ for any $A \in \acal$.
\end{itemize}
\begin{proof}[Proof of Theorem \ref{dDK}]
  Since $(\mathsf{P}\dcal G/X)_{\rm num}$ and $(\mathsf{P}\czero {\widetilde{G}}/\widetilde{X})_{\rm num}$ are simplicial groupoids (Lemma \ref{gaugegp}), we have only to show that the simplicial functor $\widetilde{\cdot} : (\mathsf{P}\dcal {G}/ X)_{\rm num} \longrightarrow (\mathsf{P}\czero {\widetilde{G}} / \widetilde{X})_{\rm num}$ satisfies conditions $\rmi'$ and $\rmii'$, which are verified in Theorem \ref{dbdlesmoothing} and Proposition \ref{gaugesmoothing}.
\end{proof}
\section{Diffeological polyhedra}

Polyhedra, which are a special type of $CW$-complexes, not only behave well as sources but also play an important role as targets (recall the notion of a barycentric map).\par
In the diffeological context, we must introduce two kinds of notions of a diffeological polyhedron, which correspond to the two kinds of standard simplices $\{ \Delta^{p} \}_{p \geq 0}$ and $\{ \Delta^{p}_{\mathrm{sub}} \}_{p \geq 0}$. Diffeological polyhedra corresponding to $\{ \Delta^{p} \}_{p \geq 0}$, which are called $\dcal$-polyhedra, form an important class of cofibrant objects, and hence behave well as sources. Diffeological polyhedra corresponding to $\{ \Delta^{p}_{\mathrm{sub}} \}_{p \geq 0}$ are more directly related to constructions of smooth maps and smooth homotopies and especially provide targets of barycentric maps (Lemmas \ref{barycentricmap} and \ref{linhomotopy}). These two kinds of diffeological polyhedra are not diffeomorphic but naturally $\dcal$-homotopy equivalent (Theorem \ref{homotopyequiv}).
\subsection{Basic properties of two kinds of diffeological polyhedra}
In this subsection, we introduce and study two kinds of diffeological polyhedra.

Let us begin by recalling the definition of a simplicial complex. A {\sl simplicial complex} $K$ consists of a set $\overline{K}$ of ${\sl vertices}$ and a set $\Sigma_{K}$ of finite nonempty subsets of $\overline{K}$ called {\sl simplices} such that
\begin{enumerate}[(i)]
	\item
	Any set consisting of exactly one vertex is a simplex.
	\item
	Any nonempty subset of a simplex is a simplex.
\end{enumerate}
\begin{flushleft}
	A simplicial map between simplicial complexes is defined in an obvious manner. The category of simplicial complexes and simplicial maps is denoted by $\scal\ccal$. The notion of a subcomplex of a simplicial complex $K$, especially that of the $p$-skeleton $K^p$ of $K$, is also defined in an obvious manner (\cite[Section 1 in Chapter 3]{Spa}).
\end{flushleft}
\par\indent For a simplicial complex $K$, the set $\absno{K}$ is defined by
$$
\absno{K}=
\left\{
x:\overline{K}\longrightarrow [0,1] \,\Big{|}\, \{v\in \overline{K}\ |\ x(v)>0\}\: \in \Sigma_K \hbox{  and  } \underset{v\in \overline{K}}{\sum} x(v)=1
\right\}.
$$
For a vertex $v$ of $K$, $x_{v} \in |K|$ is defined by
\[
x_{v}(v') =
\begin{cases}
 1 & \text{if $v' = v$}, \\
 0 & \text{if $v' \neq v$}.
\end{cases}
\]
We usually identify $v \in \overline{K}$ with $x_{v} \in |K|$. For a $p$-simplex $\sigma = \{v_{0}, \ldots, v_{p} \}$ of $K$, the closed $p$-simplex $|\sigma|$ of $|K|$ is defined by
\[
|\sigma| = \{ x \in |K|\ |\ x(v) > 0 \Rightarrow v \in \sigma \},
\]
which is just the convex full of $\sigma = \{ v_{0}, \ldots, v_{p} \}.$
\par\indent
The $\dcal$-{\sl polyhedron} $|K|_{\dcal}$ of a simplicial complex $K$ is defined to be the set $|K|$ endowed with the final diffeology for the inclusions $\{ |\sigma| \longhookrightarrow |K|\ |\ \sigma \in \Sigma_{K} \},$ where the closed $p$-simplex $|\sigma|$ is regarded as a diffeological space via an affine identification with $\Delta^{p}$. Thus, we have the functor $|\ |_{\dcal}: \scal\ccal \longrightarrow \dcal$ (see Proposition \ref{axioms}). The underlying topological space of $|K|_{\dcal}$ is just the topological polyhedron of $K$ (Proposition \ref{axioms} and \cite[Lemma 2.11]{origin}), which is also denoted by $|K|$ (\cite[p. 111]{Spa}). Since the $\dcal$-polyhedron $|K|_{\dcal}$ can be identified with the realization of $K$ viewed as a simplicial set by choosing a well-ordering on the vertices (\cite[Example 1.4]{May}), $|K|_{\dcal}$ is a cofibrant object.
\par\indent
The $\dcal$-polyhedron $|L|_\dcal$ of a subcomplex $L$ of $K$ is called a $\dcal$-subpolyhedron of $|K|_\dcal$. The following is a basic result on $\dcal$-subpolyhedra.
\begin{lem}\label{Dsub}
	Let $K$ be a simplicial complex.
	\begin{itemize}
		\item[{\rm (1)}] The inclusion $|L|_\dcal \longhookrightarrow |K|_\dcal$ is a cofibration and a $\dcal$-embedding for any subcomplex $L$ of $K$.

		\item[{\rm (2)}] Let $\{L_i\}$ be a set of subcomplexes of $K$. Then, $L:= \cup \:L_i = (\cup\: \overline{L}_i, \cup\: \Sigma_{L_i})$ is also a subcomplex of $K$ and the canonical map
		\[
			\coprod |L_i|_\dcal \longrightarrow |L|_\dcal
		\]
		is a $\dcal$-quotient map.
	\end{itemize}
	\begin{proof}
		\begin{itemize}
 			\item[{\rm (1)}] By choosing a well-ordering on the vertices of $K$, we can regard the inclusion $L \longhookrightarrow K$ as a monomorphism (i.e., a cofibration) in $\scal$. Thus, $|L|_\dcal \longhookrightarrow |K|_\dcal$ is a cofibration in $\dcal$ (Lemma \ref{Quillenpairs}) and a $\dcal$-embedding (\cite[Proposition 9.2]{origin}).
			\item[{\rm (2)}] The result is immediate from the definition of a $\dcal$-polyhedron.\qedhere
		\end{itemize}
	\end{proof}
\end{lem}

For a simplicial complex $K$, we can introduce another diffeology on $|K|$ by using $\{\Delta^{p}_{\mathrm{sub}}\}_{p \geq 0}$ instead of $\{ \Delta^{p}\}_{p \geq 0}$ (see Section 2.3); the resulting diffeological space is denoted by $|K|'_{\dcal}$, whose underlying topological space is also the topological polyhedron of $K$ (see Lemma \ref{simplex}(2) and \cite[Lemma 2.11]{origin}).
\par\indent
Let $\Delta(p)$ denote the simplicial complex defined by $\overline{\Delta(p)} = \{ 0, \cdots, p \}$ and $\Sigma_{\Delta(p)} = \{ \text{ nonempty subsets of $\overline{\Delta(p)}$} \}.$ Then, $|\Delta(p)|_{\dcal} = \Delta^{p}$ and $|\Delta(p)|'_{\dcal} = \Delta^{p}_{\mathrm{sub}}$ hold obviously. The simplicial complex $\Delta(1)$ is often denoted by $I$.
\par\indent
Regard a $p$-simplex $\sigma$ of a simplicial complex $K$ as a subcomplex of $K$ which is isomorphic to $\Delta(p)$. Then, $|\sigma|_{\dcal}$ (resp. $|\sigma|'_{\dcal}$) is just the closed $p$-simplex $|\sigma|$ viewed as a diffeological space via an affine identification with $\Delta^{p}$ (resp. $\Delta^{p}_{\mathrm{sub}}$).
\begin{rem}\label{embedding}
	Let $K$ be a simplicial complex and $\iota: L \longhookrightarrow K$ the inclusion of a subcomplex $L$ into $K$. Then, $|\iota|_{\dcal}: |L|_{\dcal} \longhookrightarrow |K|_{\dcal}$ is a $\dcal$-embedding (see Lemma \ref{Dsub}(1)). On the other hand, $|\iota|'_{\dcal}: |L|'_{\dcal} \longhookrightarrow |K|'_{\dcal}$ need not be a $\dcal$-embedding; consider the subcomplex $L$ of $K = \Delta(2)$ defined by
	\[
	\Sigma_{L} = \{ \{(0)\}, \{(1)\}, \{(2)\}, \{(0), (1)\}, \{(0), (2)\} \}
	\]
	and see \cite[the proof of Proposition A.2(1)]{origin}. However, if $L$ is a $p$-simplex $\sigma$ viewed as a subcomplex of $K$, then the map $\Delta^{p}_{\mathrm{sub}} \cong |\sigma|'_{\dcal} \xhookrightarrow[]{|\iota|'_{\dcal}} |K|'_{\dcal}$ is a $\dcal$-embedding. This result is shown as follows. Fix an isomorphism $\Delta(p) \cong \sigma$. Since a retraction $\rho: K \longrightarrow \Delta(p)$ to $\Delta(p) \cong \sigma \xhookrightarrow[]{\ \  \iota \ \ } K$ is defined by $\rho(v) = (0)$ for $v \not\in \sigma$, we obtain the retract diagram
				\begin{center}
		\begin{tikzpicture} 
		\node [below] at (0,0.25) {\begin{tikzcd}
			\absno{\sigma}'_{\dcal} \arrow{r}{\absno{\iota}'_{\dcal}}
			& \absno{K}'_{\dcal} \arrow{r}{\absno{\rho}'_{\dcal}}
			& \absno{\sigma}'_{\dcal},
			\end{tikzcd}};

		\draw[->] (-2.0,-.6) -- (-2.0,-.75)--(1.7,-.75)-- (1.7,-0.5);
		\node at (0,-1.) {$1$};
		\end{tikzpicture}
	\end{center}
which implies that $|\iota|'_\dcal$ is a $\dcal$-embedding.
\end{rem}
For a covering $\ucal = \{ U_{\alpha} \}_{\alpha \in A}$ of a diffeological space $X$ by subsets, we define the {\sl nerve} $N\ucal$ to be the simplicial complex such that $\overline{N\ucal} = \{ \alpha \in A\ |\ U_{\alpha} \neq \emptyset \}$ and $\{ \alpha_{0}, \ldots, \alpha_{p} \} \in \Sigma_{N \ucal}$ if and only if $U_{\alpha_{0}} \cap \cdots \cap U_{\alpha_{p}} \neq \emptyset$.
\begin{lem}\label{barycentricmap}
	Let $\ucal = \{U_{\alpha} \}_{\alpha \in A}$ be a covering of a diffeological space $X$ by subsets and $\{\varphi_{\alpha}\}_{\alpha \in A}$ be a smooth partition of unity subordinate to $\ucal$. Then the map $\varphi:X \longrightarrow |N\ucal|'_{\dcal}$ defined by $\varphi(x) = \underset{\alpha \in A}{\sum} \varphi_{\alpha}(x)\alpha$ is smooth.
	\begin{proof}
		For a given $x\in X$, $\sigma := \{\alpha \ |\  x\in \mathrm{supp}\ \varphi_\alpha\}$ is a finite subset of $A$ such that $x\in U_\sigma := \underset{\alpha \in \sigma}{\cap} U_\alpha$. Noticing that $\{\mathrm{supp}\ \varphi_\alpha \ | \ \alpha \in A\}$ is locally finite, we define the open neighborhood $V_\sigma$ of $x$ by
		\[
			V_\sigma = U_\sigma - \underset{\alpha \not\in\sigma}{\cup} \mathrm{supp}\ \varphi_\alpha.
		\]
		Then, there exists a smooth map $\phi_{\sigma}: V_{\sigma} \longrightarrow |\sigma|'_{\dcal}$ making the diagram
\[
\begin{tikzcd}
 & & \absno{\sigma}'_{\dcal} \arrow[hook']{d}\\
 \arrow{rru}{\phi_{\sigma}} V_{\sigma} \arrow[hook]{r} & X \arrow{r}{\varphi} & \absno{N\ucal}'_{\dcal}
\end{tikzcd}
\]
commute. Hence, $\varphi$ is smooth.
	\end{proof}
\end{lem}
\begin{lem}\label{linhomotopy}
	Let $X$ be a diffeological space, and $K$ a simplicial complex. Let $f,g:X\longrightarrow \absno{K}_{\dcal}'$ be smooth maps. Suppose that there exists a set $\{ \phi_{i} : X_{i} \longrightarrow X \}$ of smooth maps such that $\sum \phi_{i}: \coprod X_{i} \longrightarrow X$ is a $\dcal$-quotient map and $f(\phi_{i} (X_{i}) )$ and $g(\phi_{i} (X_{i}))$ are contained in the same closed simplex of $|K|'_{\dcal}$ for any $i$. Then, $f$ is $\dcal$-homotopic to $g$.
\end{lem}
\begin{proof}
	By the assumption, we can define the set-theoretic map $h: X \times I \longrightarrow |K|'_{\dcal}$ by $h(x, t) = (1 - t)f(x) + t g(x)$. Since $f \circ \phi_{i}$ and $g \circ \phi_{i}$ factor through a $\dcal$-embedding $\Delta^{p}_{\mathrm{sub}} \cong |\sigma|'_{\dcal} \longhookrightarrow |K|'_{\dcal}$ corresponding to some simplex $\sigma$ of $K$ (Remark \ref{embedding}) and $\sum \phi_{i} \times 1: \coprod (X_{i} \times I) \longrightarrow X \times I$ is a $\dcal$-quotient map (\cite[Lemma 2.5]{origin} and Proposition \ref{conven}(2)), $h$ is a smooth homotopy between $f$ and $g$.
\end{proof}
Lemmas \ref{barycentricmap} and \ref{linhomotopy} along with the following remark explain why we need not only $\{\Delta^p\}_{p\ge 0}$ but also $\{\Delta^p_{\rm sub}\}_{p\ge 0}$.
\begin{rem}\label{nonsmooth}
	We show that a barycentric map to $|N\ucal|_\dcal$ and a linear homotopy between smooth maps to $|K|_\dcal$ need not be smooth. For this, we need the following fact:
	\begin{quote}
		{\rm F{\scriptsize ACT}}. Let $c:(-\varepsilon, \varepsilon)\longrightarrow \Delta^2_{\rm sub}$ be a smooth curve. If the images $c((-\varepsilon, 0))$ and $c((0, \varepsilon))$ are contained in the open 1-simplices $(0,1)$ and $(0,2)$ respectively, then $c$ is not smooth as a curve into $\Delta^2$.
	\end{quote}
	(See \cite[the proof of Proposition A.2(1)]{origin}).
	\begin{itemize}
		\item[{\rm (1)}] (A non-smooth barycentric map to $|N\ucal|_\dcal$) Lemma \ref{barycentricmap} does not hold if $|\ |'_\dcal$ is replaced by $|\ |_\dcal$. In fact, choose positive numbers $r'$ and $r$ such that $\sqrt{\frac{2}{3}} < r' < r < \sqrt{\frac{3}{2}}$ and consider the open covering $\ucal=\{U_0, U_1, U_2 \}$ of $X=\Delta^2$ defined by $U_i = \{x\in \Delta^2 \,|\, d(x,(i)) < r \}$, where $d(\cdot, \cdot)$ is the standard distance function of $\rbb^3$. We construct smooth functions $f_i:\Delta^2 \longrightarrow [0,1]$ with ${\rm supp}\, f_i = \{x\in \Delta^2 \,|\, d(x,(i))\leq r' \}$ (see \cite[p. 64]{BJ}), and define the partition of unity $\{\varphi_0, \varphi_1, \varphi_2 \}$ subordinate to $\ucal$ by
		\[
			\varphi_i = f_i/(f_0 + f_1 + f_2).
		\]
		Let $P$ be the point of $\Delta^2$ such that $d(P,(1))=d(P,(2))= r'$ and define the smooth curve $\ell:(-\varepsilon, \varepsilon)\longrightarrow \Delta^2$ by $\ell(t) = P+ t(0,-1,1)$. Since the composite $(-\varepsilon, \varepsilon) \xrightarrow{\ \ \ \ell\ \ } \Delta^2 \xrightarrow{\ \ \varphi\ \ } |N\ucal|_\dcal = \Delta^2$ is not smooth, the barycentric map $\varphi:\Delta^2 \longrightarrow |N\ucal|_\dcal$ is not smooth. (See Fact and Fig. 8.1.)
		\item[{\rm (2)}] (A non-smooth linear homotopy between smooth maps to $|K|_\dcal$) The essence of the proof of Lemma \ref{linhomotopy} is the following:
		\begin{quote}\vspace{2mm}
			Under the assumptions of Lemma \ref{linhomotopy}, the linear homotopy $h:X\times I \longrightarrow |K|'_\dcal$ between $f$ and $g$ is smooth.
		\end{quote}\vspace{2mm}
		This does not hold if $|\ |'_\dcal$ is replaced by $|\ |_\dcal$. In fact, consider the linear embeddings $f,g:[0,1]\longrightarrow \Delta^2$ depicted in Fig. 8.2 (see Proposition \ref{axioms} for the smoothness of $f$ and $g$). Then, the linear homotopy
		\[
			h:[0,1] \times I \longrightarrow \Delta^2
		\]
		is not smooth (see Fact).\par
		We, however, show that Lemma \ref{linhomotopy} remains true even if $|\ |'_\dcal$ is replaced by $|\ |_\dcal$ (see Corollary \ref{linhomotopy2}).
	\end{itemize}\par
	\begin{center}
		\begin{tikzpicture}[thick, baseline=0pt]
		\coordinate[label=left:$(1)$]  (1) at (0,0);
		\coordinate[label=right:$(2)$] (2) at (4,0);
		\coordinate[label=above:$(0)$] (0) at (2,3.464);
		
		\coordinate[](a) at ($ (1)!.33!(2) $);
		\coordinate[](b) at ($ (1)!.66!(2) $);
		\coordinate[](c) at ($ (2)!.33!(0) $);
		\coordinate[](d) at ($ (2)!.66!(0) $);
		\coordinate[](e) at ($ (0)!.33!(1) $);
		\coordinate[](f) at ($ (0)!.66!(1) $);
			
		\path[clip, draw] (1) -- (2) -- (0) -- cycle;
		

		\filldraw[fill=gray, opacity=0.3] (0,0) circle (2.66cm);
		\filldraw[fill=gray, opacity=0.3] (2,3.464) circle (2.66cm);
		\filldraw[fill=gray, opacity=0.3] (4,0) circle (2.66cm);		
	
		\draw (b) arc (0:60:2.66cm);
		\draw (a) arc (180:120:2.66cm);
		\draw (c) arc (-60:-120:2.66cm);
		\end{tikzpicture}
		\begin{tikzpicture}[thick, baseline=0pt]
		\coordinate[label=left:$(1)$]  (1) at (0,0);
		\coordinate[label=right:$(2)$] (2) at (4,0);
		\coordinate[label=above:$(0)$] (0) at (2,3.464);
		
		\coordinate[](a) at ($ (1)!.5!(0) $);
		\coordinate[](b) at ($ (0)!.5!(2) $);
		
		\coordinate[label=above:$f$](f) at ($ (1)!.5!(2) $);
		\coordinate[label=above:$g$](g) at ($ (a)!.5!(b) $);
		
		\draw (1) -- (2) -- (0) -- cycle;
		\draw[middlearrow={>}] (a) -- (b);
		\draw[middlearrow={>}] (1) -- (2);
		\end{tikzpicture}\\
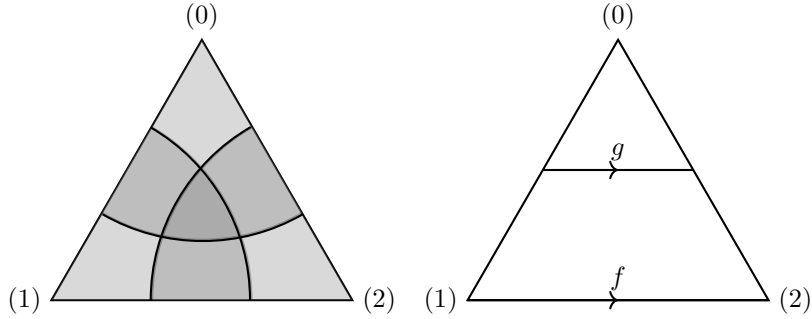
		{\rm Fig 8.1 Supports of $\varphi_0$, $\varphi_1$, and $\varphi_2$.$\quad$ Fig 8.2 The linear embeddings $f$ and $g$.}
	\end{center}
\end{rem}

\subsection{$\dcal$-homotopy equivalence between two kinds of diffeological polyhedra}
Since $id: \Delta^{p} \longrightarrow \Delta^{p}_{\mathrm{sub}}$ is smooth, $id: |K|_{\dcal} \longrightarrow |K|'_{\dcal}$ is a smooth map which is a homeomorphism with respect to the underlying topologies (see Lemma \ref{simplex}). The rest of this section is devoted to proving the following theorem.
\begin{thm}\label{homotopyequiv}
	Let $K$ be a simplicial complex. Then the (set-theoretic) identity
	$$
	id:\absno{K}_{\dcal}\longrightarrow\absno{K}_{\dcal}'
	$$
	is a $\dcal$-homotopy equivalence.
\end{thm}
For the proof of Theorem \ref{homotopyequiv}, we recall the local diffeological structures of $\Delta^{p}$ and $\Delta^{p}_{\mathrm{sub}}$ from \cite[Section 4]{origin}. Let $I = \{i_{0}, \ldots, i_{k} \}$ be a subset of the ordered set $\{0, \ldots, p \};$ we always assume that the elements $i_{0}, \ldots, i_{k}$ are arranged such that $i_{0} < \cdots < i_{k}$. Then, $(i_{0}, \ldots, i_{k})$ denotes the open $k$-simplex defined by
\[
(i_{0}, \ldots, i_{k}) = \{ (x_{0}, \ldots, x_{p} ) \in \Delta^{p} \ | \ x_{i} = 0\ \text{if and only if\ $i \not\in \{i_{0}, \ldots, i_{k} \} $} \},
\]
and $U_{I}$ denotes the diffeological subspace of $\Delta^{p}$ defined by
\[
U_{I} = \{ (x_{0}, \ldots, x_{p}) \in \Delta^{p} \ | \ x_{i} > 0 \ \text{for} \ i \in I \}.
\]
$U_{I}$ is an open neighborhood of the open simplex $(i_{0}, \ldots, i_{k})$ of $\Delta^{p}$. For
$I = \{ i \}$, $U_{I}$ is just the $i^{\mathrm{th}}$ half-open simplex $\Delta^{p}_{\hat{i}}$ (\cite[Definition 4.1]{origin}).
\begin{prop}\label{goodnbd}
	Let $I = \{i_{0}, \ldots, i_{k} \}$ be a subset of the ordered set $\{0, \ldots, p \}$ and define the map
\[
\varPhi_{I}: U_I \longrightarrow \mathring{\Delta}^{k} \times \Delta^{p-k}_{\hat{0}}
\]
by
\[
\varPhi_{I} (x_0, \ldots, x_p) = ( \frac{x_{i_0}(0)+ \cdots + x_{i_k}(k)}{x_{i_0}+ \cdots + x_{i_k}}, (x_{i_0}+ \cdots + x_{i_k})(0)+x_{j_1}(1)+ \cdots +x_{j_{p-k}}(p-k)),
\]
where $\{j_1, \ldots, j_{p-k} \} = \{0, \ldots, p \} \backslash \{i_0, \ldots, i_k \}$.
\begin{itemize}
\item[{\rm (1)}] $\varPhi_{I}: U_{I} \longrightarrow \mathring{\Delta}^{k} \times \Delta^{p-k}_{\hat{0}}$ is a diffeomorphism.
\item[{\rm (2)}] $\varPhi_{I}: U_{I\ \sub} \longrightarrow \mathring{\Delta}^{k} \times \Delta^{p-k}_{\hat{0}\ \sub}$ is a diffeomorphism.
\end{itemize}
\end{prop}
\begin{proof}
	See \cite[Proposition 4.5, Remark 4.7(2), and Lemma 4.2]{origin}.
\end{proof}
In the proof of Theorem \ref{homotopyequiv}, $I$ denotes a subset $\{i_{0}, \ldots, i_{k} \}$ of the ordered set $\{0, \ldots, p \}$; the unit interval is denoted by $[0, 1]$ to avoid confusion. For a subset $B$ of a topological space $A$, $\overline{B}$ and $B^{\circ}$ denote the closure and the interior of $B$ respectively. We write $h:f\simeq_\dcal g$ for a $\dcal$-homotopy $h:X \times [0,1]\longrightarrow Y$ connecting $f$ to $g$.
\begin{proof}[Proof of Theorem \ref{homotopyequiv}.]
	Suppose that there exists a smooth map $\psi_{K}:\absno{K}_{\dcal}\longrightarrow\absno{K}_{\dcal}$ satisfying the following conditions:
	\begin{enumerate}
		\item[$\bullet$]
		$\psi_{K}$ preserves each closed simplex.
		\item[$\bullet$]
		$\psi_{K}$ is $\dcal$-homotopic to $1_{\absno{K}_{\dcal}}$.
		\item[$\bullet$]
		$\psi_{K}$ is smooth as a map from $\absno{K}_{\dcal}'$ to $\absno{K}_{\dcal}$.
	\end{enumerate}
	Then $\psi_{K}:\absno{K}_{\dcal}'\longrightarrow \absno{K}_{\dcal}$ is a $\dcal$-homotopy inverse of $id:\absno{K}_{\dcal}\longrightarrow\absno{K}_{\dcal}'$ (see Lemma \ref{linhomotopy}). We thus construct such a smooth map $\psi_{K}$ in five steps.\vspace{0.2cm}\\
	Step 1:
		{\sl Constructions of $\lambda^p:\Delta^p_{\hat{0}}\longrightarrow\Delta^p_{\hat{0}}$ and $\Lambda^{p}: \lambda^{p} \simeq_{\dcal} 1_{\Delta^{p}_{\hat{0}}}$}. For a positive number $\epsilon < 1$, we choose a monotone increasing smooth function $\lambda:[0,1)\longrightarrow[0,1)$ such that
		\[
			\lambda(t)=
		\begin{cases}
		0 & (0\leq t\leq\frac{\epsilon}{2})\\
		t & (t\geq \epsilon)
		\end{cases}
		\ \text{\ and \ \ $0 < \lambda(t) < t$\ \ on\ \ $(\frac{\epsilon}{2}, \epsilon)$},
		\]
		and define $\Lambda: [0, 1) \times [0, 1] \longrightarrow [0, 1)$ to be the linear homotopy connecting $\lambda$ to $1_{[0, 1)}$; explicitly,
		\[
		\Lambda (t, s) = (1 - s) \lambda (t) + st.
		\]
		Then, the smooth maps $\lambda^p : \Delta^p_{\hat{0}}\longrightarrow\Delta^p_{\hat{0}}$ and $\Lambda^{p}: \Delta^{p}_{\hat{0}} \times [0, 1] \longrightarrow \Delta^{p}_{\hat{0}}$ are defined by the commutative diagrams
		\begin{center}
		$
		\begin{tikzcd}
		\Delta^{p-1}\times [0,1) \arrow[r,"1 \times \lambda"] \arrow[d,"\varphi_0",swap] & \Delta^{p-1}\times [0,1) \arrow[d,"\varphi_0"]\\
		\Delta^p_{\hat{0}}\arrow[r,"\lambda^{p}"]&\Delta^{p}_{\hat{0}},
		\end{tikzcd}
		$
		$
		\begin{tikzcd}
		\Delta^{p-1} \times [0, 1) \times [0, 1] \arrow{r}{1 \times \Lambda} \arrow[swap]{d}{\varphi_{0} \times 1} & \Delta^{p-1} \times [0 , 1) \arrow{d}{\varphi_{0}} \\
		\Delta^{p}_{\hat{0}} \times [0, 1] \arrow{r}{\Lambda^{p}} & \Delta^{p}_{\hat{0}}
		\end{tikzcd}
		$
		\end{center}
		(see Definition \ref{simplices} and \cite[Lemmas 4.3 and 2.5]{origin}). The map $\Lambda^{p}$ is a $\dcal$-homotopy connecting $\lambda^{p}$ to $1_{\Delta^{p}_{\hat{0}}}$ and is given by
		\[
		\Lambda^{p} (x, s) = (1 - s) \lambda^{p} (x) + sx.
		\]\vspace{0.2cm}
Step 2:
		{\sl Notations}. We introduce some notations, which are used only in this proof.

		Let $f: A \longrightarrow A$ be a continuous self-map of a topological space $A$. Then, we set
		\[
		\mathrm{carr}\ f = \{ x \in A \ | \ f(x) \neq x \}\ \  \ \text{and}\ \ \  \mathrm{supp}\ f = \overline{\mathrm{carr}\ f}.
		\]
		In the case where $A$ is an open set of $\Delta^{p}$, we also set
	\[
	\mathrm{carr}_{k}\ f  =  \{ x \in A \ | \ f(x) \neq x,\ f(x) \in \mathrm{sk}_{k}\ \Delta^{p} \}	\hspace{0.3cm} \text{and} \hspace{0.3cm}  \mathrm{supp}_{k}\ f  =  \overline{\mathrm{carr}_{k}\ f}
	\]
	(see \cite[Definition 4.1]{origin} for the $k$-skeleton $\mathrm{sk}_{k}\ \Delta^{p}$).

		Recalling that $I=\{i_0,\ldots,i_k\}$, we next define the open neighborhood $V_{I}$ of the open $k$-simplex $(i_{0}, \ldots, i_{k})$ in $\Delta^{p}$ by
		\[
		V_{I} = \{ (x_{0}, \ldots, x_{p}) \in \Delta^{p} \ | \ x_{i} > x_{j}   \ \text{for} \ i \in I \ \text{and} \ j \not\in I \},
		\]
		which is obviously contained in $U_{I}$.\vspace{0.2cm}\\
		Step 3:
		{\sl Constructions of $\psi^p_k:\Delta^p\longrightarrow\Delta^p$ and $h^{p}_{k}:\psi^{p}_{k} \simeq_{\dcal} 1_{\Delta^{p}}.$} First, we construct sequences of smooth maps
		$$
		\cdots\longrightarrow\Delta^p\xoverright{\psi^p_k}\Delta^p\xoverright{\psi^p_{k-1}}\cdots \xoverright{\psi^p_1}\Delta^p\xoverright{\psi^p_0}\Delta^p\ \ (p\geq 0)
		$$
		satisfying the following conditions:
		\begin{itemize}
			\item[$(\rm i)$]
			$\psi^p_k$ is equivariant with respect to the obvious action of the symmetric group $S_{p+1}$ on $\Delta^{p}$ and satisfies the inclusion relation
			\[
			{\rm supp}\ \psi^{p}_{k}\ \subset \underset{i_{0} < \cdots < i_{k}}{\coprod} V_{ \{ i_{0} \ldots i_{k} \}}.
			\]
			\item[$(\rm ii)$]
			$\psi^p_k$ preserves each closed simplex of $\Delta^{p}$ and coincides with $\psi^{p-1}_k$ on each closed $(p-1)$-simplex under an affine identification with $\Delta^{p-1}$.
			\item[$(\rm iii)$]
			$\psi^{p}_{k} = 1_{\Delta^{p}}$ on the $k$-skeleton ${\rm sk}_{k}\ \Delta^{p}$, and hence $\psi^p_k=1_{\Delta^p}$ for $k\geq p$.
			\item [$(\rm iv)$]
			$(\mathrm{supp}_{0}\ \psi^{p}_{0})^{\circ} \cup \cdots \cup (\mathrm{supp}_{k}\ \psi^{p}_{k})^{\circ}$ is a neighborhood of $\mathrm{sk}_{k}\ \Delta^{p}$.
		\end{itemize}

		By condition $\rmiii$, we may assume that $k < p$. The maps $\psi^{p}_{k}$ are constructed by induction on $k$. To construct $\psi^{p}_{0} : \Delta^{p} \longrightarrow \Delta^{p}$, choose a positive number $\epsilon_{0} < \frac{1}{2}$ and consider the map $\lambda^{p} : \Delta^{p}_{\hat{0}} \longrightarrow \Delta^{p}_{\hat{0}}$ for $\epsilon = \epsilon_{0}$ (Step 1). Since $\mathrm{supp} \ \lambda^{p} \subset V_{\{ 0 \}},$ $\psi^{p}_{0}$ is defined by $\psi^{p}_{0} \ |_{V_{\{0\}}} = \lambda^{p}$ and condition $\rmi$. The maps $\psi^p_0\:(p\geq0)$ obviously satisfy conditions (i)-(iv).

		Suppose that the maps $\psi^{p}_{d}$ are constructed for $d<k$. Let us construct $\psi^{p}_{k} : \Delta^{p} \longrightarrow \Delta^{p}$. First, choose a smooth function $\phi^{k}: \Delta^{k} \longrightarrow [0, 1]$ satisfying the following conditions:
		\begin{itemize}
			\item $\phi^{k} \equiv 0$ near $\dot{\Delta}^{k}$.
			\item $(\mathrm{supp}_{0}\ \psi^k_{0})^{\circ} \cup \cdots \cup (\mathrm{supp}_{k-1}\ \psi^{k}_{k-1})^{\circ} \cup (\phi^{k})^{-1}(1)^{\circ} = \Delta^{k}$.
			\item $\phi^{k}$ is $S_{k+1}$-invariant.
		\end{itemize}
	(Use \cite[Lemma 3.1]{origin} and \cite[Remark 7.4]{BJ}.) Set $m_{k} = \underset{0 \leq i \leq k}{\min} \min\ \{ y_{i} \ | \ (y_{0}, \ldots, y_{k}) \in \mathrm{supp}\ \phi^{k} \}$ and $\epsilon_{k} = \frac{1}{2}m_{k}$, and let $\lambda^{p-k} : \Delta^{p-k}_{\hat{0}} \longrightarrow \Delta^{p-k}_{\hat{0}}$ be the map defined in Step 1 for $\epsilon = \epsilon_{k}$. Recalling that $U_{\{ 0, \ldots, k \}}$ is diffeomorphic to $\mathring{\Delta}^{k} \times \Delta^{p-k}_{\hat{0}}$ via $\varPhi_{\{ 0, \ldots, k \}}$ (Proposition \ref{goodnbd}(1)), we consider the self-map $\lambda^{p}_{k}$ on $U_{\{ 0, \ldots, k \}} ( \cong \mathring{\Delta}^{k} \times \Delta^{p-k}_{\hat{0}})$ defined by
	\[
	(y, z) \longmapsto (y, \phi^{k}(y)\lambda^{p-k}(z) + ( 1 - \phi^{k}(y))z ).
	\]
	Since $\lambda^{p}_{k}(y, z) = (y, \Lambda^{p-k} (z, 1- \phi^{k}(y) ) ),$ $\lambda^{p}_{k}$ is smooth. Note that
	\[
	\mathrm{supp} \ \lambda^{p}_{k} \cong \{ (y, z) \ | \ y \in \ \mathrm{supp} \  \phi^{k}, z_{1} + \cdots + z_{p-k} \leq \epsilon_{k} \},
\]
	\[
	(\mathrm{supp}_{k}\ \lambda^{p}_{k})^{\circ} \cong \{ (y, z) \ | \ y \in (\phi^{k})^{-1} (1)^{\circ}, z_{1} + \cdots + z_{p-k} < \frac{\epsilon_{k}}{2} \}.
	\]
	Recalling the definition of $\varPhi_{\{ 0, \ldots, k \}}$ (Proposition \ref{goodnbd}) and the choice of $\epsilon_{k}$, we see that if $x \in {\rm supp}\ \lambda^{p}_{k}$, then $(y, z) = \varPhi_{\{ 0, \ldots, k \} }(x)$ satisfies the inequalities
	\[
	x_{j} = z_{j-k} \leq \epsilon_{k} = \frac{1}{2} m_{k} \leq \frac{1}{2} y_{i} < x_{i}
	\]
	for $i \leq k$ and $k < j$, and hence that $\mathrm{supp}$ $\lambda^{p}_{k} \ \subset  \ V_{\{ 0, \ldots, k \}}$.~Thus, $\psi^{p}_{k}$ is defined by $\psi^{p}_{k} |_{V_{\{0, \ldots, k\}}} = \lambda^{p}_{k} |_{V_{\{0, \ldots, k\}}}$ and condition $\rmi$. We can easily see that the maps $\psi^{p}_{k}$ satisfy conditions $\rmi$-$\rmiv$.
	\par\indent
	Next, we construct the $\dcal$-homotopy $h^p_k$ connecting $\psi^p_k$ to $1_{\Delta^p}$. Define the map $\Lambda^p_k:U_{ \{ 0, \ldots, k \} } \times [0, 1] \longrightarrow U_{ \{ 0, \ldots, k \} }$ by
	\[
	(y, z, s) \mapsto (y, \phi^{k}(y) \Lambda^{p-k} (z, s) + (1-\phi^{k}(y))z )
	\]
 under the identification $U_{\{ 0, \cdots, k \}} \cong \mathring{\Delta}^{k} \times \Delta^{p-k}_{\hat{0}}$ (Proposition \ref{goodnbd}(1)). Since the $\Delta^{p-k}_{\hat{0}}$-component of $\Lambda^p_k$ is just $\Lambda^{p-k}(z,\phi^k(y)s+(1-\phi^k(y)))$, it is a smooth map, which restricts to a $\dcal$-homotopy connecting $\psi^{p}_{k} |_{V_{ \{ 0, \ldots, k \} }}$ to $1_{V_{ \{ 0, \ldots, k \} }}$. Thus, this $\dcal$-homotopy extends to the desired $\dcal$-homotopy $h^{p}_{k}: \Delta^{p} \times [0, 1] \longrightarrow \Delta^{p}$ in the obvious manner. Then, $h^p_k$ is an $S_{p+1}$-equivariant $\dcal$-homotopy that preserves each closed simplex of $\Delta^p$ and coincides with $h^{p-1}_{k}$ on each closed $(p-1)$-simplex under an affine identification with $\Delta^{p-1}$. Further, $h^{p}_{k}$ coincides with the constant homotopy at $1_{\Delta^{p}}$ on the $k$-skeleton ${\rm sk}_{k}\ \Delta^{p}$. \vspace{0.2cm}\\
	Step 4:
		{\sl Constructions of $\psi_{K}:\absno{K}_{\dcal}\longrightarrow\absno{K}_{\dcal}$ and $h_{K}:\psi_{K} \simeq_{\dcal} 1_{|K|_{\dcal}}$ }. By condition $\rmii$ on the maps $\psi^{p}_{k}$ (Step 3), $\{ \psi^{p}_{k} \}_{p \geq 0}$ defines the smooth map $\psi_k:\absno{K}_{\dcal}\longrightarrow\absno{K}_{\dcal}$.
		Thus, the smooth map $\psi_{K}:\absno{K}_{\dcal}\longrightarrow\absno{K}_{\dcal}$ is defined by $\psi_{K}= \psi_{0}\circ\psi_1\circ\psi_2\circ\cdots$; note that $ \psi_{K}=\psi_{0}\circ\psi_{1}\circ\cdots\circ\psi_{p-1}$ on a closed $p$-simplex of $\absno{K}_{\dcal}$.

		Let us see that $\psi_{K} \simeq_{\dcal} 1_{\absno{K}_{\dcal}}:\absno{K}_{\dcal}\longrightarrow\absno{K}_{\dcal}$. Since $h^p_k$ coincides with $h^{p-1}_k$ on each $(p-1)$-simplex, $\{ h^{p}_{k} \}_{p \geq 0}$ defines the $\dcal$-homotopy $h_k:\absno{K}_{\dcal}\times [0, 1]\longrightarrow\absno{K}_{\dcal}$ connecting $\psi_k$ to $1_{\absno{K}_{\dcal}}$ (see \cite[Lemma 2.5]{origin}). 
		Set $\widehat{h}_k=(h_k,p_{[0,1]}):|K|_\dcal \times [0,1] \longrightarrow |K|_\dcal \times [0,1]$, where $p_{[0,1]}$ is the projection of $|K|_\dcal \times [0,1]$ onto $[0,1]$, and define the map $\widehat{h}_K:|K|_\dcal \times [0,1] \longrightarrow |K|_\dcal \times [0,1]$ by $\widehat{h}_K=\widehat{h}_0 \circ \widehat{h}_1 \circ \widehat{h}_2 \circ \cdots$. Since $\widehat{h}_K = \widehat{h}_0 \circ \widehat{h}_1 \circ \cdots \circ \widehat{h}_{p-1}$ on $|K^p|_\dcal\times [0,1]$ (see the end of Step 3) and $|K|_\dcal \times [0,1]= \underset{\longrightarrow}{\lim}\, (|K^p|_\dcal \times [0,1])$ (see Propositions \ref{adjoint1}(1) and \ref{conven}(2)), we see that $\widehat{h}_K$ is a well-defined smooth map. Thus, the smooth map $h_K:|K|_\dcal \times [0,1] \longrightarrow |K|_\dcal$, defined as the composite
		\[
		|K|_\dcal \times [0,1] \overset{\widehat{h}_K}{\longrightarrow} |K|_\dcal \times [0,1] \overset{proj}{\longrightarrow} |K|_\dcal,
		\]
		is a $\dcal$-homotopy connecting $\psi_K$ to $1_{|K|_\dcal}$.\vspace{0.2cm}\\
Step 5:
		{\sl Smoothness of $\psi_{K}: |K|'_{\dcal} \longrightarrow |K|_{\dcal}$.} We show that $\psi_{K}$ is smooth as a map from $|K|'_{\dcal}$ to $|K|_{\dcal}$. By the construction, we have only to show that $\psi^{p}\ (:= \psi^{p}_{0} \circ \cdots \circ \psi^{p}_{p-1}): \Delta^{p}_{\sub} \longrightarrow \Delta^{p}$ is smooth. Since $\mathring{\Delta}^{p}_{\sub} = \mathring{\Delta}^{p}$ (Lemma \ref{simplex}(3)), it suffices to see that $\psi^{p}$ is smooth on the neighborhood $(\mathrm{supp}_{0}\ \psi^{p}_{0})^{\circ}\ \cup \cdots \cup \ (\mathrm{supp}_{p-1}\ \psi^{p}_{p-1})^{\circ}$ of the boundary.

		First, we show that for $k < l$, $\psi^{p}_{l}$ preserves $(\mathrm{supp}_{k}$  $\psi^{p}_{k})^{\circ}$ and makes the following diagram commute:
		\[
		\begin{tikzcd}
		(\mathrm{supp}_{k} \ \psi^{p}_{k})^{\circ} \arrow{rr}{\psi^{p}_{l}} \arrow[swap]{rd}{\psi^{p}_{k}} & & \arrow{ld}{\psi^{p}_{k}} (\mathrm{supp}_{k}\ \psi^{p}_{k})^{\circ}\\
		       &  \Delta^{p} &
		\end{tikzcd}\tag{8.1}
		\]
		By the construction in Step 3, we have
	\begin{eqnarray*}
	({\rm supp}_{k} \ \psi^{p}_{k})^{\circ} & = & \underset{i_{0} < \cdots < i_{k} }{\coprod} ({\rm supp}_{k}\ \psi^{p}_{k})^{\circ}_{\{i_{0}, \ldots, i_{k} \}},\\
	{\rm supp} \ \psi^{p}_{l} & \subset & \underset{i'_{0} < \cdots < i'_{l} }{\coprod} V_{ \{ i'_{0}, \ldots, i'_{l} \} },
	\end{eqnarray*}
	where $(\mathrm{supp}_{k} \ \psi^{p}_{k} )^{\circ}_{\{ i_{0}, \ldots, i_{k} \}}  := (\mathrm{supp}_{k} \ \psi^{p}_{k})^{\circ} \ \cap \ V_{\{ i_{0}, \ldots, i_{k} \}}.$ Thus, we see how $\psi^{p}_{l} |_{V_{ \{ i'_{0}, \ldots, i'_{l} \} }}$ maps the points of $(\mathrm{supp}_{k} \ \psi^{p}_{k})^{\circ}_{\{ i_{0}, \ldots, i_{k} \} }$. Since $V_{\{ i_{0}, \ldots, i_{k} \}} \cap V_{\{ i'_{0}, \ldots, i'_{l} \}} = \emptyset$ if $\{i_{0}, \ldots, i_{k} \} \not\subset \{ i'_{0}, \ldots, i'_{l} \}$, we may assume that $\{ i_{0}, \ldots, i_{k} \} \subset \{ i'_{0}, \ldots, i'_{l} \}.$ For simplicity, we assume that $\{ i_{0}, \ldots, i_{k} \} = \{ 0, \ldots, k \}$ and $\{ i'_{0}, \ldots, i'_{l} \} = \{ 0, \ldots, l \}$. By the definition, $\psi^{p}_{l} |_{V_{\{ 0, \ldots, l \}}}$ sends $(x_{0}, \ldots, x_{p})$ to $(x'_{0}, \ldots, x'_{p})$ such that $(x'_{0}, \ldots, x'_{l}) = r(x_{0}, \ldots, x_{l})$ for some $r \geq 1$ and $(x'_{l+1}, \ldots, x'_{p}) = s(x_{l+1}, \ldots, x_{p})$ for some $s \leq 1$. Thus, under the identification $U_{\{ 0, \ldots, k \}} \cong \mathring{\Delta}^{k} \times \Delta^{p-k}_{\hat{0}}$, $\psi^{p}_{l} |_{V_{\{0, \ldots, l \}}}$ sends $(y, z)$ to $(y, z')$ with $z_{1} + \cdots + z_{p-k} \geq z'_{1} + \cdots + z'_{p-k}$. Hence, $\psi^{p}_{l}$ preserves $(\mathrm{supp}_{k}\ \psi^{p}_{k})^{\circ}$ and makes diagram (8.1) commute; note that under the identification $U_{ \{ i_{0}, \ldots, i_{k} \} } \cong \mathring{\Delta}^{k} \times \Delta^{p-k}_{\hat{0}}$, $\psi^{p}_{k}|_{(\rm supp_{k} \ \psi^{p}_{k})^{\circ}}$ sends $(y, z)$ to $(y, (0))$.
	\par\indent
	From this, we have the commutative diagram in $Set$
	\[
	\begin{tikzcd}
	{\phantom A} \arrow[xshift = -1.78ex ,dash]{rrrr} & \arrow[xshift = - 6ex ,dash]{rr}{\psi^{p}} &  & \arrow[xshift = 1.9ex, dash]{r}  &  {\phantom A} \arrow[yshift = 1.6ex, dash]{dd} \arrow{dd} \\
	& & (\mathrm{supp}_{k} \ \psi^{p}_{k})^{\circ}_{\{ i_{0}, \ldots, i_{k} \}} \arrow{dr}{\psi^{p}_{k}} & &  \\
	\arrow[yshift = 1.6ex,dash]{uu} \arrow{rrr}{\psi^{p}_{k}} \arrow[hook']{dd} (\mathrm{supp}_{k}\ \psi^{p}_{k})^{\circ}_{\{ i_{0}, \ldots, i_{k} \}} \arrow{urr}{\ \ \ \ \psi^{p}_{k+1} \circ \cdots \circ \psi^{p}_{p-1}} & &  & \Delta^{p} \arrow{r}{\psi^{p}_{0}\circ \cdots \circ \psi^{p}_{k-1} } &\  \Delta^{p} \\
	& & & (i_{0}, \ldots, i_{k}) \arrow[hook]{u}&  \\
	\mathring{\Delta}^{k} \times \Delta^{p-k}_{\hat{0}} \arrow{rrr}{proj} & &  & \mathring{\Delta}^{k}, \arrow{u}{\rotatebox[]{90}{$\cong$}} &
	\end{tikzcd}
	\]
	which is a diagram in $\dcal$ if the subsets of $\Delta^{p}, \Delta^{k}$, and $\Delta^{p-k}$ are endowed with the subdiffeologies of the standard simplices (Definition \ref{simplices}). Thus, using Proposition \ref{goodnbd}, we can see that $\psi^{p}_{k}: ({\rm supp}_{k}\ \psi^{p}_{k})^{\circ}_{ \{i_{0}, \ldots, i_{k} \}\ {\rm sub} } \longrightarrow \Delta^{p}$, and hence $\psi^{p}: \Delta^{p}_{\rm sub} \longrightarrow \Delta^{p}$ is smooth.
\end{proof}
\begin{cor}\label{linhomotopy2}
	Lemma \ref{linhomotopy} remains true even if $|\ |'_{\dcal}$ is replaced by $|\ |_{\dcal}$.
\end{cor}
\begin{proof}
	The result follows from Lemma \ref{linhomotopy} and Theorem \ref{homotopyequiv}.
\end{proof}

\section{Homotopy cofibrancy theorem}
In this section, we prove the homotopy cofibrancy theorem (Theorem \ref{hcofibrancy}), which is a diffeological version of a theorem of tom Dieck \cite[Theorem 4]{tom}. For the proof, we establish a diffeological version of \cite[Proposition 4.1]{Seg} (Section 9.1) and then introduce and study the notion of a Hurewicz cofibration in $\dcal$ (Section 9.2).
\subsection{Diffeological spaces associated to a covering}
In this subsection, we introduce the diffeological spaces $BR_{U}$ and $BX_{U}$ associated to a covering $U = \{ U_{\alpha} \}_{\alpha \in A}$ of a diffeological space $X$ and prove that $BX_{U}$ is $\dcal$-homotopy equivalent to $X$ under the $\dcal$-numerability condition on $U$ (Proposition \ref{dSegal}).\par
Recall the nerve functor $N:Cat\longrightarrow \scal$, where $Cat$ denotes the category of small categories (\cite[p. 271]{Mac}, \cite[p. 5]{GJ}). Similarly to the topological case, we can define the classifying space functor $B:Cat \longrightarrow \dcal$ by the commutative diagram
\begin{center}
	\begin{tikzpicture} 
	\node [below] at (0,0.25) {\begin{tikzcd}
		Cat \arrow{r}{N}
		& \scal \arrow{r}{|\:|_\dcal}
		& \dcal.
		\end{tikzcd}};
	
	\draw[->] (-1.4,-.5) -- (-1.4,-.75)--(1.4,-.75)-- (1.4,-0.5);
	\node at (0,-1.) {$B$};
	\end{tikzpicture}
\end{center}\par
We first extend these three functors. For a finitely complete category $\mathcal{A}$, $\mathsf{Cat}(\mathcal{A})$ denotes the category of category objects in $\mathcal{A}$ (see \cite[p. 267]{Mac}). Thus, $\mathsf{Cat}(\dcal)$ denotes the category of diffeological categories, which are defined to be category objects in $\dcal$. Let $s\dcal$ denote the category of simplicial diffeological spaces (see Section 3.1 for the notation).
\begin{defn}\label{BC}
	\begin{itemize}
		\item[(1)] The nerve $N\ccal$ of a diffeological category $\ccal$ can be regarded as a simplicial diffeological space in an obvious manner. Hence, the nerve functor $N:\mathsf{Cat}(\dcal) \longrightarrow s\dcal$ is defined.
		\item[(2)] We define the {\sl realization} $|X_{\bullet}|_{\dcal}$ of a simplicial diffeological space $X_{\bullet}$ to be the coend $\int_{}^{n} X_{n} \times \Delta^{n}$ (\cite[Section 6 in Chapter IX]{Mac}), thereby defining the realization functor $|\:|_\dcal:s\dcal\longrightarrow \dcal$. Then, $|X_{\bullet}|_{\dcal}$ is just the quotient diffeological space $\underset{n}{\coprod}\ X_{n} \times \Delta^{n} /\sim$, where the equivalence relation $\sim$ is defined by $(x, \theta_{\ast}(u)) \sim (\theta^{\ast}x, u)$ for $x \in X_{n}$, $u \in \Delta^{m}$, and $\theta: [m] \longrightarrow [n]$.
		\item[(3)] We define the classifying space functor $B:\mathsf{Cat}(\dcal)\longrightarrow \dcal$ by the commutative diagram
		\begin{center}
			\begin{tikzpicture} 
			\node [below] at (0,0.25) {\begin{tikzcd}
				\mathsf{Cat}(\dcal) \arrow{r}{N}
				& s\dcal \arrow{r}{|\:|_\dcal}
				& \dcal.
				\end{tikzcd}};
			
			\draw[->] (-1.4,-.5) -- (-1.4,-.75)--(1.8,-.75)-- (1.8,-0.5);
			\node at (0,-1.) {$B$};
			\end{tikzpicture}
		\end{center}
	\end{itemize}
\end{defn}
Define the functor $\delta: Set\longrightarrow \dcal$ to assign to a set $A$ the set $A$ endowed with the discrete diffeology. Observing that $\delta$ has left and right adjoints, we see that $\delta$ is a fully faithful functor preserving limits and colimits. Hence, the induced functors $\delta:\scal = sSet \longrightarrow s\dcal$ and $\delta:Cat=\mathsf{Cat}(Set) \longrightarrow \mathsf{Cat}(\dcal)$ are also fully faithful and the following diagram is commutative:
\[
\begin{tikzcd}
	Cat \arrow[hook']{d}{\delta} \arrow{r}{N} & \scal \arrow[hook']{d}{\delta} \arrow{r}{|\:|_\dcal} & \dcal \arrow{d}[sloped,above]{=}\\
	\mathsf{Cat}(\dcal) \arrow{r}{N} & s\dcal \arrow{r}{|\:|_\dcal} & \dcal.
\end{tikzcd}
\]
Therefore, $Cat$ and $\scal$ can be regarded as full subcategories of $\mathsf{Cat}(\dcal)$ and $s\dcal$, respectively. Conversely, a diffeological category $\ccal$ with ob $\ccal$ and mor $\ccal$ discrete (resp. a simplicial diffeological space $X_{\bullet}$ with each $X_n$ discrete) is often regarded simply as a small category (resp. a simplicial set).\par
The following is a basic result on the realization of a simplicial diffeological space.
\begin{lem}\label{realizationskeleta}
	For a simplicial diffeological space $X_{\bullet}$, define the diffeological spaces $|X_{\bullet}|^{(n)}_{\dcal}$ by the following conditions:
	\begin{itemize}
		\item[$\rm(i)$] $|X_{\bullet}|_{\dcal}^{(0)} = X_{0}$.
		\item[$\rm(ii)$] Suppose that $|X_{\bullet}|_{\dcal}^{(n-1)}$ is defined. Then, $|X_{\bullet}|_{\dcal}^{(n)}$ is defined by the pushout diagram
		\[
		\begin{tikzcd}
		X_{n} \times \dot{\Delta}^{n} \underset{sX_{n-1} \times \dot{\Delta}^{n}}{\cup} s X_{n-1} \times \Delta^{n} \arrow{r} \arrow[hook']{d} &  \absno{X_{\bullet}}_{\dcal}^{(n-1)} \arrow[hook']{d} \\
		X_{n} \times \Delta^{n} \arrow{r} & \absno{X_{\bullet}}_{\dcal}^{(n)},
		\end{tikzcd}
		\]
		where $sX_{n-1}$ is the subset $\underset{i}{\cup}\ s_{i}X_{n-1}$ of $X_n$ endowed with the quotient diffeology for the surjection $\underset{i}{\sum} s_{i}: \underset{i}{\coprod}\ X_{n-1} \longrightarrow sX_{n-1}$.
	\end{itemize}
	Then, $\lim\limits_{\longrightarrow} \ |X_{\bullet}|^{(n)}_{\dcal} = |X_{\bullet}|_{\dcal}$ holds.
	\begin{proof}
		For a simplicial set $K$, define the diffeological space $|K|^{(n)}_\dcal$ by $|K|^{(n)}_\dcal = |K^n|_\dcal$, where $K^n$ is the $n$-skeleton of $K$ (see the proof of Lemma \ref{realization}). Since the underlying set functor $U:\dcal\longrightarrow Set$ preserves limits and colimits (\cite[Proposition 2.1(1)]{origin}), we have the equalities in $Set$
		\[
		|X_{\bullet}|_\dcal = |UX_\bullet|_\dcal \:\:\: {\rm and} \:\:\: |X_\bullet|^{(n)}_\dcal = |UX_\bullet|^{(n)}_\dcal,
		\]
		which imply the equalities in $Set$
		\[
		|X_{\bullet}|_{\dcal} = |UX_{\bullet}|_{\dcal} = \lim\limits_{\longrightarrow} |UX_{\bullet}|^{(n)}_{\dcal} = \lim\limits_{\longrightarrow}|X_{\bullet}|^{(n)}_{\dcal}.
		\]
		Noticing that the diffeology of $|X_{\bullet}|_{\dcal}$ is final for the canonical maps $X_{n} \times \Delta^{n} \longrightarrow |X_{\bullet}|_{\dcal}\ (n \geq 0)$, we see that $\lim\limits_{\longrightarrow} |X_{\bullet}|^{(n)}_{\dcal} = |X_{\bullet}|_{\dcal}$ holds in $\dcal$.
	\end{proof}
\end{lem}
We now introduce the diffeological categories $R_U$ and $X_U$ associated to a covering $U=\{U_\alpha\}_{\alpha\in A}$ of a diffeological space $X$ by subsets.\par\indent
First, we define $R_{U}$ to be the category whose objects are the nonempty finite subsets $\sigma$ of $A$ with $U_{\sigma} := \underset{\alpha \in \sigma}{\cap} U_{\alpha} \neq \emptyset$ and whose morphisms are their reverse inclusions. Regard $R_U$ as a diffeological category by endowing ${\rm ob}(R_U)$ and ${\rm mor}(R_U)$ with discrete diffeologies. Then, $NR_{U}$ is just the barycentric subdivision of the nerve $NU$ of the covering $U$ (see \cite[p. 123]{Spa} and \cite[p. 108]{Seg}).
\par\indent
Next, we define the category $X_{U}$ by
\begin{eqnarray*}
{\rm ob}(X_{U}) & = & \{ (x, \sigma) \ | \ \mathrm{\sigma \  is\ a\ nonempty\ finite\ subset\ of\ } A\ {\rm and}\ x \in U_{\sigma} \},\\
{\rm mor} (X_{U}) & = & \{ (x, \sigma) \rightarrow (x, \tau) \ | \ \sigma \supset \tau \}
\end{eqnarray*}
and endow $\mathrm{ob}(X_{U})$ and $\mathrm{mor}(X_{U})$ with diffeologies via the obvious bijections $\mathrm{ob}(X_{U}) \cong \underset{\sigma}{\coprod}\ U_{\sigma}$ and $\mathrm{mor}(X_{U}) \cong \underset{\sigma\supset \tau}{\coprod}\, U_{\sigma}$.\par
The morphism of diffeological categories
\[
	\pi:X_U \longrightarrow R_U
\]
is defined by $\pi(x,\sigma) = \sigma$.
\begin{rem}\label{originaldef}
	See \cite[Section 4]{Seg} for simple but somewhat imprecise definitions of $R_{U}$ and $X_{U}$. Compare them with the definitions above; the key point is that distinct elements $\sigma$ and $\tau$ of $\Sigma_{NU}$ should be distinguished as objects of $R_{U}$ even if $U_{\sigma}$ and $U_{\tau}$ coincide.
\end{rem}
To investigate the relation between $BX_U$ and $X$, we construct a smooth injection $BX_U \longhookrightarrow X\times BR_U$. Since $BX_U=|NX_U|_\dcal$ and $BR_U=|NR_U|_\dcal$, we set $BX_U^{(n)}=|NX_U|_\dcal^{(n)}$ and $BR_U^{(n)}=|NR_U|_\dcal^{(n)}$. Then, we have the pushout diagrams
\[
\begin{tikzcd}
	\underset{\sigma_n \supsetneq \cdots \supsetneq \sigma_0}{\coprod} U_{\sigma_n} \times \dot{\Delta}^{n} \arrow{r} \arrow[hook']{d} & BX_U^{(n-1)} \arrow[hook']{d} & \underset{\sigma_n \supsetneq \cdots \supsetneq \sigma_0}{\coprod} X \times \dot{\Delta}^{n} \arrow[hook']{d} \arrow{r} & X\times BR_U^{(n-1)} \arrow[hook']{d}\\
	\underset{\sigma_n \supsetneq \cdots \supsetneq \sigma_0}{\coprod} U_{\sigma_n}\times \Delta^n \arrow{r} & BX_U^{(n)}, & \underset{\sigma_n \supsetneq \cdots \supsetneq \sigma_0}{\coprod} X\times \Delta^n \arrow{r} & X \times BR_U^{(n)}
\end{tikzcd}
\]
and the identities
\[
	BX_U = \underset{\longrightarrow}{\lim}\: BX_U^{(n)}, X\times BR_U = \underset{\longrightarrow}{\lim}\: X\times BR_U^{(n)}
\]
(see Lemma \ref{realizationskeleta} and Proposition \ref{conven}(2)). From these, we observe that the set-theoretic identities
\[
	BX_U = \underset{\sigma_n \supsetneq \cdots \supsetneq \sigma_0}{\coprod} U_{\sigma_n}\times \mathring{\Delta}^n \;\;\;{\rm and}\;\;\; X\times BR_U = \underset{\sigma_n \supsetneq \cdots \supsetneq \sigma_0}{\coprod} X \times \mathring{\Delta}^n
\]
hold. We thus see that the canonical inclusions $U_{\sigma_n}\times \Delta^n \longhookrightarrow X \times \Delta^n$ define a smooth injection
\[
BX_{U} \longhookrightarrow X \times BR_{U}.
\]
Then, the map $pr:BX_U\longrightarrow X$ is defined to be the $X$-component of this injection; the $BR_U$-component of this injection is just the map $B\pi:BX_U\longrightarrow BR_U$.\par
Let $[\sigma_p \supsetneq \cdots \supsetneq \sigma_0]$ denote the copy of $\Delta^p$ indexed by a nondegenerate $p$-simplex $\sigma_p \supsetneq \cdots \supsetneq \sigma_0$ of $NR_U$. Then, the canonical maps
\[
	\underset{\sigma_p \supsetneq \cdots \supsetneq \sigma_0}{\coprod} U_{\sigma_p} \times [\sigma_p \supsetneq \cdots \supsetneq \sigma_0] \longrightarrow BX_U \;\;\; {\rm and } \underset{\sigma_p \supsetneq \cdots \supsetneq \sigma_0}{\coprod} X \times [\sigma_p \supsetneq \cdots \supsetneq \sigma_0] \longrightarrow X \times BR_U
\]
are $\dcal$-quotient maps.

\begin{rem}\label{prBpi}
	We see that the smooth injection $(pr,B\pi):BX_U \longhookrightarrow X\times BR_U$ need not be a $\dcal$-embedding or a topological embedding.\par
	Let $X=\mathbb{R}$ and consider the open cover $U=\{(-\infty,1),(-1,\infty)\}$. Then, we have the isomorphism
	\[
		X\times BR_U \cong \mathbb{R}\times ([0,\frac{1}{2}]\cup_{\{\frac{1}{2}\}} [\frac{1}{2}, 1]).
	\]
	Under this isomorphism, $BX_U$ is identified with the subset
	\[
		(-\infty,1)\times(0) \cup (-1,1)\times [0,1] \cup (-1,\infty)\times(1)
	\]
	set-theoretically. Consider a smooth curve
	\[
		c:\mathbb{R} \longrightarrow X\times BR_U \cong \mathbb{R}\times ([0,\frac{1}{2}]\cup_{\{\frac{1}{2}\}} [\frac{1}{2}, 1])
	\]
	such that
	\begin{itemize}
		\item $c(0)=(1,1)$.
		\item $-1<c_1(t)<1$, $\frac{1}{2}<c_2(t)<1$ for $\: t\neq 0$.
	\end{itemize}
	(Here, $c_1$ and $c_2$ are the $\mathbb{R}$- and $([0,\frac{1}{2}]\cup_{\{\frac{1}{2}\}} [\frac{1}{2}, 1])$-components of $c$ respectively.) Noticing that the map
	\[
		(-\infty,1)\times (0) \sqcup (-1,1)\times ([0,\frac{1}{2}]\cup_{\{\frac{1}{2}\}} [\frac{1}{2},1]) \sqcup (-1, \infty)\times (1)\longrightarrow BX_U
	\]
	is a $\dcal$-quotient map, we see that the corestriction $c:\mathbb{R}\longrightarrow BX_U$ is not smooth, and hence that $BX_U\longhookrightarrow X\times BR_U$ is not a $\dcal$-embedding.\par
	Next, consider the topological quotient map
	\[
		(-\infty,1) \times (0) \sqcup (-1,1)\times [0,1] \sqcup (-1,\infty)\times (1)\longrightarrow \widetilde{BX_U}
	\]
	(see \cite[Lemma 2.11]{origin}). Then, the set
	\[
		V=\{(x,y)\in (-1,1)\times [0,1]\:|\: (x-\frac{1}{2})^2 + (y-1)^2 <\frac{1}{4} \} \cup (0,2)\times (1)
	\]
	is an open neighborhood of $(1,1)\in BX_U$. But, $V$ is not an open neighborhood of $(1,1)\in BX_U$ with respect to the subspace topology of $X\times BR_U$. Hence, $BX_U \longhookrightarrow X\times BR_U$ is not a topological embedding.
\end{rem}
The rest of this subsection is devoted to the proof of the following result, which is a diffeological version of \cite[Proposition 4.1]{Seg}.
\begin{prop}\label{dSegal}
	Let $U=\{U_\alpha\}_{\alpha \in A}$ be a covering of a diffeological space $X$ by subsets. If $U$ is $\dcal$-numerable, then $pr:BX_{U} \longrightarrow X$ is a $\dcal$-homotopy equivalence.
\end{prop}
For the proof, we need a lemma on the barycentric subdivision of a $\dcal$-polyhedron. For a simplicial complex $K$, we have the natural bijection
\[
\eta_{K}: |\sd\ K|_{\dcal} \longrightarrow |K|_{\dcal}
\]
(see \cite[p. 123]{Spa} for the barycentric subdivision $\sd\ K$ of $K$); we can easily see that $\eta_{K}$ is smooth and that $\widetilde{\eta}_{K}$ is a homeomorphism (Proposition \ref{axioms} and \cite[Lemma 2.11]{origin}). Thus, $|\sd\ K|_\dcal$ is topologically identified with $|K|_\dcal$ via $\widetilde{\eta_K}$.
\begin{lem}\label{homotopyequiv2}
	Let $K$ be a simplicial complex. Then, the smooth bijection 
	\[
	\eta_{K}:|\sd\ K|_{\dcal} \longrightarrow |K|_{\dcal}
	\]
	has a $\dcal$-homotopy inverse $\rho_{K}:|K|_{\dcal} \longrightarrow |\sd\ K|_{\dcal}$ such that under the topological identification of $|\sd\ K|_{\dcal}$ and $|K|_{\dcal}$, $\rho_{K}$ preserves each closed simplex of $|\sd\ K|$.
\end{lem}
\begin{proof}
	We have only to construct a smooth map $\rho'_{K}: |K|_{\dcal} \longrightarrow |\sd\ K|'_{\dcal}$ which preserves each closed simplex of $|\sd\ K|$. In fact, Corollary \ref{linhomotopy2} shows that the composite $|K|_{\dcal} \xrightarrow{\ \rho'_{K}\ } |\sd\ K|'_{\dcal} \xrightarrow{\ \  \psi_{\mathrm{sd}K}  \ \ } |\sd\ K|_{\dcal}$ is the desired $\dcal$-homotopy inverse of $\eta_{K}$, where $\psi_L:|L|'_\dcal \longrightarrow |L|_\dcal$ is the smooth map constructed in the proof of Theorem \ref{homotopyequiv}.
	\par\indent
	We can reduce the construction of $\rho'_{K}$ to the case of $K = \Delta(p)$. More precisely, to construct such a natural map $\rho'_{K}: |K|_{\dcal} \longrightarrow |\sd\ K|'_{\dcal}$, it suffices to construct smooth maps
	$$
	\rho'^p: \Delta^{p} = |\Delta(p)|_{\dcal} \longrightarrow |\sd\ \Delta(p)|_{\dcal}' \ \ \ (p=0,1,2,\ldots)
	$$
	satisfying the following conditions:
	\begin{itemize}
		\item[$(\rm i)$] $\rho'^p$ is equivariant with respect to the obvious  actions of the symmetric group $S_{p+1}$.
		\item[$(\rm ii)$] $\rho'^{p}$ preserves each $(p-1)$-face $\Delta^{p-1}_{(i)}$ and coincides with $\rho'^{p-1}$ on $\Delta^{p-1}_{(i)}$ under an affine identification of $\Delta^{p-1}_{(i)}$ with $\Delta^{p-1}$.
	\end{itemize}
	(See \cite[Definition 4.1]{origin} for the notation.)

	Consider the open cover $\vcal_{p} = \{ V_{(i_{0}, \cdots, i_{k} )} \ | \ 0 \leq i_{0} < \cdots < i_{k} \leq p  \}$ of $|\sd\ \Delta(p)|$, where $V_{(i_{0}, \ldots, i_{k})}$ is the star of the barycenter $b_{(i_{0}, \ldots, i_{k})}$ of an open simplex $(i_{0}, \ldots, i_{k})$ of $\Delta^{p}$. Since $|\sd\ \Delta(p)|$ is topologically identified with $\Delta^p$ and the nerve of $\vcal_{p}$ is just $\sd\ \Delta(p)$ (\cite[p. 114]{Spa}), we have only to construct a smooth partition of unity $\{ \rho'_{ (i_{0}, \ldots, i_{k}) } \}$ on $\Delta^p$ subordinate to the open cover $\vcal_{p}$ satisfying the two conditions corresponding to conditions $\rmi$ and $\rmii$, which are denoted by conditions $\rmi'$ and $\rmii'$ respectively (see Lemma \ref{barycentricmap}); the precise formulation of conditions $\rmi'$ and $\rmii'$ is left to the reader.
	\par\indent
	For $p = 0$, $\{ \rho'_{ (i_{0}, \ldots, i_{k}) } \} (= \{ \rho'_{(0)} \})$ is defined in the obvious manner. Suppose that we have constructed a partition of unity $\{\rho'_{(i_{0}, \ldots, i_{k})} \}$ on $\Delta^{p-1}$ subordinate to $\vcal_{p-1}$ satisfying conditions $\rmi'$ and $\rmii'$. Then, the functions $\rho'_{(i_{0}, \ldots, i_{k})}$ with $(i_{0}, \ldots, i_{k}) \neq (0, \ldots, p)$ are defined on $\dot{\Delta}^{p}$ by conditions $\rmi'$ and $\rmii'$. Extend the functions $\rho'_{(i_{0}, \ldots, i_{k})}$ to smooth functions on $\Delta^{p}$ with ${\rm supp}\ \rho'_{(i_{0}, \ldots, i_{k})}\subset V_{(i_{0}, \ldots, i_{k})}$ (see \cite[Lemma 9.11]{origin}) and average them over $S_{p+1}$. We then obtain the desired smooth partition of unity $\{\rho'_{(i_{0}, \ldots, i_{k})}\}$ on $\Delta^p$ by setting $\rho'_{(0, \ldots, p)}=1-\underset{k < p}{\sum}  \ \rho'_{(i_{0}, \ldots, i_{k})}$.
\end{proof}
\begin{proof}[Proof of Proposition \ref{dSegal}]
	We prove the result in two steps.\vspace*{2mm}\\
	Step 1: {\sl Construction of a splitting $\psi$ to $pr:BX_U \longrightarrow X$.} We choose a smooth partition of unity $\{\varphi_\alpha\}$ subordinate to $U$, which defines the smooth function $\varphi:X \longrightarrow |NU|'_\dcal$ (see Lemma \ref{barycentricmap}). Noticing that $BR_{U} = |\mathrm{sd}\ NU|_{\dcal}$, we then define the smooth map $\phi$ to be the composite
	\[
	X \xrightarrow{\ \, \varphi\ \, } |NU|'_{\dcal} \xrightarrow[\simeq]{\psi_{NU} } |NU|_{\dcal} \xrightarrow[\simeq]{\rho_{NU}} BR_{U}
	\]
	(see Theorem \ref{homotopyequiv} and Lemma \ref{homotopyequiv2} for $\psi_{NU}$ and $\rho_{NU}$, respectively). 
	
	Consider the solid arrow diagram in $\dcal$
	\[
	\begin{tikzcd}
	& & X \arrow{d}{(1_{X}, \phi)} \arrow[dashed,swap]{lld}{\psi} \\
	 BX_{U} \arrow[hook,swap]{rr}{(pr, B\pi)} & & X \times BR_{U}.
	\end{tikzcd}
	\]
	Then, the dotted arrow $\psi$ exists, making the diagram commute. In fact, for $x \in X$, consider the open neighborhood $V_{\sigma}$ of $x$ defined in the proof of Lemma \ref{barycentricmap}. Then, $(1_{X}, \phi)|_{V_\sigma}$ factorizes the $\dcal$-embedding $U_{\sigma} \times |\sd\ \Delta(p)|_{\dcal} \longhookrightarrow X \times BR_{U}$, which is the product of the inclusion $U_{\sigma} \longhookrightarrow X$ and the inclusion $|\sd\ \Delta(p)|_{\dcal} \longhookrightarrow BR_{U}$ of the barycentric subdivision of the closed simplex corresponding to $\sigma$ (see Lemma \ref{Dsub}(1) and \cite[the proof of Lemma 4.2]{origin}). Since the closed $p$-simplices of $|\sd\ \Delta(p)|_\dcal$ are of the form $[\sigma \supsetneq \cdots]$, we have the $\dcal$-quotient map $\coprod U_\sigma \times [\sigma \supsetneq \cdots] \longrightarrow U_\sigma \times |\sd\ \Delta(p)|_{\dcal}$ (see \cite[Lemma 2.5]{origin} and Proposition \ref{conven}(2)). Since the composite $O \overset{f}{\longrightarrow} V_\sigma \xrightarrow[]{(1_X,\phi)} U_\sigma \times |\sd \ \Delta(p)|_\dcal$ lifts locally along $\coprod U_\sigma \times [\sigma \supsetneq \cdots]\longrightarrow U_\sigma \times |\sd\ \Delta(p)|_\dcal$ for any plot $f$ of $V_\sigma$, we see that $(1_{X}, \phi)$ factorizes through a smooth map $\psi:X \longrightarrow BX_{U}$.\par
	Note that $pr\circ \psi = 1_X$ holds.
	\vspace*{2mm}\\
Step 2: {\sl Construction of a $\dcal$-homotopy between $\psi \circ pr$ and $1_{BX_U}$.} We must show that $\psi \circ pr \simeq_{\dcal} 1_{BX_{U}}$. For this, we construct a $\dcal$-homotopy $h$ between the two composites
\[
\begin{tikzcd}
BX_{U} \arrow[yshift = 0.2cm]{r}{\psi \circ pr} \arrow[yshift = -0.2,swap]{r}{1_{BX_{U}}} & BX_{U} \arrow[hook]{r}{(pr, B\pi)} &X \times BR_{U}
\end{tikzcd}
\]
so that it factorizes through the smooth injection $BX_U \longhookrightarrow X\times BR_U$. Since $(pr, B\pi) \circ \psi \circ pr = (pr, \phi \circ pr),$ we define the $X$-component $h_{X}$ of $h$ to be the constant homotopy at $pr$.\par
To construct the $BR_{U}$-component $h_{BR_{U}}$, we consider the composites
\[
\begin{tikzcd}
BX_{U} \arrow[yshift = 0.2cm]{r}{\phi \circ pr} \arrow[yshift = -0.2,swap]{r}{B\pi} & BR_{U} \xrightarrow[\simeq]{\eta_{NU}}\absno{NU}_{\dcal}
\end{tikzcd}
\]
(see Lemma \ref{homotopyequiv2} for $\eta_{NU}$). Recall that the canonical map $\coprod U_{\sigma_p} \times [\sigma_{p} \supsetneq \cdots \supsetneq \sigma_{0}] \longrightarrow BX_{U}$ is a $\dcal$-quotient map and consider the composites
\[
\begin{tikzcd}
U_{\sigma_p} \times [\sigma_{p} \supsetneq \cdots \supsetneq \sigma_{0}] \arrow{r} & BX_{U} \arrow[yshift = 0.2cm]{r}{\phi \circ  pr} \arrow[yshift = -0.2,swap]{r}{B\pi} & BR_{U} \xrightarrow[\simeq]{\eta_{NU}}\absno{NU}_{\dcal}.
\end{tikzcd}
\]
For $x \in U_{\sigma_p}$, consider the open neighborhood $V_{\sigma}$ of $x$ defined in the proof of Lemma \ref{barycentricmap} and set $\sigma'_{p} = \sigma \cup \sigma_{p}$ and $W_{\sigma'_{p}} = V_{\sigma} \cap U_{\sigma_{p}}$. Then, $W_{\sigma'_p}$ is an open neighborhood of $x$ contained in $U_{\sigma_p}$ and the images of $W_{\sigma'_p} \times [\sigma_{p} \supsetneq \cdots \supsetneq \sigma_{0}]$ by the two composites are contained in the closed simplex of $|NU|_{\dcal}$ corresponding to $\sigma'_{p}$. Thus, $\eta_{NU} \circ \phi \circ pr \simeq_{\dcal} \eta_{NU} \circ B\pi$ by Corollary \ref{linhomotopy2}. Hence, by composing the map $\rho_{NU}: |NU|_{\dcal} \longrightarrow BR_{U}$, we have
\[
\phi \circ pr \simeq_{\dcal} \ \rho_{NU} \circ \eta_{NU} \circ \phi \circ pr \simeq_{\dcal}\ \rho_{NU} \circ \eta_{NU} \circ B\pi \simeq_{\dcal} B\pi
\]
(see Lemma \ref{homotopyequiv2}). Construct three $\dcal$-homotopies corresponding to these three $\dcal$-homotopy relations using the linear homotopies and the $\dcal$-homotopies $h_{K}$ constructed in the proof Theorem \ref{homotopyequiv} (see the proofs of Corollary \ref{linhomotopy2} and Lemma \ref{homotopyequiv2}).\vspace{0.2cm} We then define the $\dcal$-homotopy $h_{BR_{U}}$ to be the composite of the three $\dcal$-homotopies.\par
Last, we must verify that the $\dcal$-homotopy $h:BX_U\times I \longrightarrow X\times BR_U$ factorizes through $BX_U \longhookrightarrow X\times BR_U$. Let $\sigma'_p$ and $W_{\sigma'_p}$ be as above and note that the composite
\[
W_{\sigma'_p} \times [\sigma_{p} \supsetneq \cdots \supsetneq \sigma_{0}] \times I \longrightarrow BX_{U} \times I \xrightarrow{\ \ h\ \ } X \times BR_{U}
\]
factorizes through the $\dcal$-embedding $U_{\sigma'_{p}} \times |{\rm sd}\ \Delta(p')|_{\dcal} \longhookrightarrow X \times BR_{U}$, which is the product of the inclusion $U_{\sigma'_{p}} \longhookrightarrow X$ and the inclusion $|{\rm sd}\ \Delta(p')|_{\dcal} \longhookrightarrow BR_{U}$ of the barycentric subdivision of the closed simplex corresponding to $\sigma'_{p}$. Since $|\sd\ \Delta(p')|_\dcal$ is the union of the simplices of the form $[\sigma_{p}'\supsetneq \cdots]$, the $\dcal$-homotopy $h:BX_{U} \times I \longrightarrow X \times BR_{U}$ factorizes through $(pr,B\pi):BX_U \longhookrightarrow X\times BR_U$ (see the argument showing that $(1_X, \phi)$ factorizes through $(pr,B\pi):BX_U\longhookrightarrow X\times BR_U$).
\end{proof}
\begin{rem}\label{proofs}
The proof of Proposition \ref{dSegal} is much more subtle than that of \cite[Proposition 4.1]{Seg} since we need two kinds of notions of a diffeological polyhedron and the natural map $|{\rm sd}\ K|_{\dcal} \longrightarrow |K|_{\dcal}$ is not an isomorphism unlike the topological case (see the definition of $\phi : X \longrightarrow BR_{U}$ in the proof of Proposition \ref{dSegal}).
\end{rem}
\subsection{Hurewicz cofibrations in $\dcal$}
In this subsection, we introduce the notion of a Hurewicz cofibration in $\dcal$ using a weak version of the homotopy extension property, and then establish its basic properties.
\begin{defn}\label{hcofibr}
	A smooth map $i: A \longrightarrow X$ is called a Hurewicz cofibration (or an $h$-cofibration) if $i:A \longrightarrow X$ satisfies the following weak homotopy extension property (WHEP):
\par\indent
Assume given a solid arrow diagram in $\dcal$
\[
\begin{tikzcd}
X \times (0) \underset{A \times (0)}{\cup} A \times I \arrow{rr}{f + h} \arrow{d} & & Z\\
X \times I \arrow[dashed,yshift=-0.7ex]{rru},  & &
\end{tikzcd}
\]
where the vertical arrow is the canonical map from the pushout of the diagram
\[
X \times (0) \xleftarrow{i \times (0)} A \times (0) \longhookrightarrow A \times I
\]
to $X \times I$. If $h_{t} = h_{0}$ for sufficiently small $t$, the dotted arrow exists, making the diagram commute.
\end{defn}
The reason why we define an $h$-cofibration in $\dcal$ not by HEP but by WHEP is explained in Remark \ref{difference}(2).
\begin{rem}\label{remhcofibr}
	\begin{itemize}
	\item[{\rm (1)}] Let $i: A \longrightarrow X$ be a smooth map and assume given a solid arrow diagram as in Definition \ref{hcofibr}. If $i: A \longrightarrow X$ is an $h$-cofibration, then we have an extension $\hat{h}:X\times I \longrightarrow Z$ after deforming the homotopy $h: A \times I \longrightarrow Z$ using a nondecreasing smooth function $\lambda: I \longrightarrow I$ such that $\lambda \equiv 0$ on $[0, \epsilon]$ and $\lambda \equiv id$ on $[2\epsilon , 1]$ for some $\epsilon > 0$. Further, if we deform the homotopy $h: A \times I \longrightarrow Z$ so that
\begin{eqnarray*}
	h_{t} = \left \{
	\begin{array}{ll}
		h_{0} & \text{if $t \in [0, \epsilon]$}\\
		h_{1} & \text{if $t \in [1-\epsilon, 1]$}
	\end{array}
	\right.
\end{eqnarray*}
for some $\epsilon$ with $0 < \epsilon < \frac{1}{2}$, we have an extension $\hat{h}:X\times I \longrightarrow Z$ such that
\begin{eqnarray*}
	\hat{h}_{t} = \left \{
	\begin{array}{ll}
		\hat{h}_{0} & \text{if $t \in [0, \epsilon']$}\\
		\hat{h}_{1} & \text{if $t \in [1-\epsilon', 1]$}
	\end{array}
	\right.
\end{eqnarray*}
for any fixed $\epsilon'$ with $0<\epsilon'<\epsilon$. Thus, we always choose such homotopies and extensions if necessary.
\item[{\rm (2)}] Let $i:A \longrightarrow X$ be an $h$-cofibration. Since we have a solution of the extension problem
	\[
	\begin{tikzcd}
	X \times (0) \underset{A \times (0)}{\cup} A \times I \arrow{r}{1} \arrow{d} & X \times (0) \underset{A \times (0)}{\cup} A \times I \\
	X \times I. \arrow[dashed,swap,yshift=0.7ex]{ru} &
	\end{tikzcd}
	\]
after deforming the homotopy $A \times I \longrightarrow X \times (0) \underset{A \times (0)}{\cup} A \times I$ (see Part 1), we see that there exist an open diffeological subspace $U$ of $X$ and a retract diagram
		\begin{center}
			\begin{tikzpicture} 
			\node [below] at (0,0.25) {\begin{tikzcd}
				A \arrow{r}{i}
				& U \arrow{r}{r}
				& A,
				\end{tikzcd}};

			\draw[->] (-1.4,-.5) -- (-1.4,-.75)--(1.4,-.75)-- (1.4,-0.5);
			\node at (0,-1.) {$1$};
			\end{tikzpicture}
		\end{center}
and hence that $i: A \longrightarrow X$ is an $\dcal$-embedding. Therefore, given an $h$-cofibration $i:A\longrightarrow X$, we can regard the source $A$ as a diffeological subspace of $X$.
\end{itemize}

\end{rem}
We first show the following lemma.

\begin{lem}\label{basicshcofibr}
\begin{itemize}	
\item[{\rm (1)}] If $i: A \longhookrightarrow X$ is an $h$-cofibration, then $i \times 1: A \times U \longhookrightarrow X \times U$ is also an $h$-cofibration for any $U \in \dcal$.
\item[{\rm (2)}] Assume given a pushout diagram in $\dcal$
\[
\begin{tikzcd}
A' \arrow[hook]{r}{i'} \arrow{d} & X' \arrow{d}\\
A \arrow[hook,swap]{r}{i} & X.
\end{tikzcd}
\]
If $i'$ is an $h$-cofibration, then $i$ is also an $h$-cofibration.
\item[{\rm (3)}] The class of $h$-cofibrations is closed under composites, coproducts, and retracts.
\item[{\rm (4)}] Assume given a sequence in $\dcal$
\[
X_{0} \xhookrightarrow[]{\ \  i_{0}\ \ } X_{1} \xhookrightarrow[]{\ \ i_{1}\ \ } X_{2} \xhookrightarrow[]{\ \ i_{2}\ \ } \cdots
\]
If the $i_{k}$ are $h$-cofibrations, then $i: X_{0} \longhookrightarrow X := \lim\limits_{\rightarrow} X_{k}$ is also an $h$-cofibration.
\end{itemize}
\end{lem}
\begin{proof}(1) The result follows from Proposition \ref{conven}(2).\par
	(2) We must solve the extension problem in Definition \ref{hcofibr} under the assumption that $h_{t} = h_{0}$ for sufficiently small $t$.
	Since $i' : A' \longhookrightarrow X'$ is an $h$-cofibration, there exists a morphism $\hat{k} : X' \times I \longrightarrow Z$ making the diagram
		\[
	\begin{tikzcd}
	\arrow{d} X' \times (0) \underset{A' \times (0)}{\cup} A' \times I \arrow{rr}{f\absno{_{X'} + h}_{A' \times I}} & & Z\\
	X' \times I \arrow[swap, yshift=-0.7ex]{rru}{\hat{k}} &
	\end{tikzcd}
	\]
	commute. Noticing that $X \times I = X' \times I \cup_{A' \times I} A \times I$ (Proposition \ref{conven}(2)), we see that
	\[
	\hat{k} + h : X \times I = X' \times I \cup_{A' \times I} A \times I \longrightarrow Z
	\]
	is the desired solution.\par
(3) Obvious (see Remark \ref{remhcofibr}(1) for the closure property under composites).\par
(4) Since $X \times I = \lim\limits_{\rightarrow} X_{n} \times I$ (Proposition \ref{conven}(2)), the extension problem in question can be solved (see Remark \ref{remhcofibr}(1)).
\end{proof}
A diffeological pair $(X, A)$ is an $NDR$-$pair$ (= neighborhood deformation retract pair) if there is a smooth map $u: X \longrightarrow I$ such that $u^{-1}(0) = A$ and a $\dcal$-homotopy $\varphi: X \times I \longrightarrow X$ relative to $A$ such that $\varphi_{0} = 1_{X}$ and $\varphi_{1}(x) \in A$ for $u(x) < 1$.\par

By the following lemma, the notion of an ${\rm NDR}$-pair gives a useful criterion for a map to be an $h$-cofibration.
\begin{lem}\label{NDR} Let $(X, A)$ be an ${NDR}$-pair. Then the inclusion $i: A \longrightarrow X$ is an $h$-cofibration.
\end{lem}
\begin{proof}
Choose a smooth map $u$ and a $\dcal$-homotopy $\varphi$ as in the definition of an ${\rm NDR}$-pair. We have to solve the extension problem in Definition \ref{hcofibr} under the assumption that $h_{t} = h_{0}$ for sufficiently small $t$.

Choose $\epsilon > 0$ such that $h_{t} = h_{0}$ for $t \in [0, \epsilon]$ and a non-decreasing function $\rho : I \longrightarrow I$ such that $\rho^{-1}(0) = [0, \frac{1}{3}\epsilon]$ and $\rho^{-1}(1) = [\frac{2}{3}\epsilon, 1].$ Set $U_{1} = \{ (x, t) \in X \times I \ | \ (1 - u(x)) t < \epsilon \}$ and $U_{2} = \{ (x, t) \in X \times I \ | \ (1 - u(x)) t > \frac{2}{3} \epsilon \}$, which form an open cover of $X \times I$. Define the smooth maps $\hat{h}_{i}: U_{i} \longrightarrow Z \ (i = 1, 2)$ by
\begin{eqnarray*}
\hat{h}_{1}(x, t) & = & f( \varphi (x, \rho ( (1- u(x))t  )   )),\\
\hat{h}_{2}(x, t) & = & h(\varphi (x, 1), (1 - u(x)) t).
\end{eqnarray*}
Since $\hat{h}_{1} = \hat{h}_{2}$ on $U_{1} \cap U_{2}$, we have the smooth map
\[
\widehat{h} = \widehat{h}_{1} + \widehat{h}_{2} : X \times I = U_{1} \cup U_{2} \longrightarrow Z,
\]
which is the desired extension of $f + h$.
\end{proof}
We call cofibrations in Definition \ref{WFC} {\sl Quillen cofibrations} or {\sl $q$-cofibrations}. Then, the following result holds.
\begin{prop}\label{gimpliesh}
	Every $q$-cofibration in $\dcal$ is an $h$-cofibration in $\dcal$.
\end{prop}
\begin{proof}
	Let $i: A \longrightarrow X$ be a $q$-cofibration. We show that $i$ is an $h$-cofibration in two steps.\vspace*{2mm}\\
	{\sl Step 1: The case where $i$ is the inclusion $\dot{\Delta}^p\longhookrightarrow \Delta^p$ $(p\ge 0)$}. Since the result is obvious for $p = 0$, we may assume that $p > 0$. We have only to show that $(\Delta^{p}, \dot{\Delta}^{p})$ is an NDR-pair (see Lemma \ref{NDR}).
	
	For the pair $(\Delta^{p}, \dot{\Delta}^p)$, we construct a smooth map $u$ and a smooth homotopy $\varphi$ as in the definition of an $\rm NDR$-pair. Let $\mu : I \longrightarrow I$ be a nondecreasing smooth function such that $\mu = id$ on $[0, 1-2\epsilon]$ and $\mu \equiv 1$ on $[1-\epsilon, 1]$ for some $\epsilon \in (0, \frac{1}{2})$. Define the smooth map $u: \Delta^{p} \longrightarrow I$ by $u(t_{0}, \ldots, t_{p}) = \mu((p+1)^{p+1} t_{0} \cdots t_{p})$ (see Lemma \ref{simplex}(3)). Then, we can construct the desired smooth homotopy $\varphi: \Delta^{p} \times I \longrightarrow \Delta^{p}$ using \cite[Lemma 9.11]{origin}.\vspace*{2mm}\\
	{\sl Step 2: The case where $i$ is a general $q$-cofibration}. Consider the cofibration-trivial fibration factorization of $i$
	\[
	A \xrightarrow{\ \ \ i_{\infty}\ \ \ } G^{\infty}(\ical, i) \xrightarrow{\ \ \ p_{\infty}\ \ \ } X
	\]
	constructed in \cite[the proof of Theorem 1.3]{origin} and note that $i_{\infty}$ is a sequential relative $\ical$-cell complex. By solving the lifting problem
	\par\indent
	\[
	\begin{tikzcd}
	A \arrow{r}{i_{\infty}} \arrow[swap]{d}{i} & G^{\infty}(\ical, i) \arrow{d}{p_{\infty}}\\
	X \arrow[swap]{r}{1_{X}} \arrow[dashed]{ur} & X,
	\end{tikzcd}
	\]
	we can show that $i$ is a retract of $i_{\infty}$, and hence that $i$ is a $h$-cofibration (see Step 1 and Lemma \ref{basicshcofibr}).
\end{proof}
From Proposition \ref{gimpliesh} and Remark \ref{remhcofibr}(2), we have the implications
\[
\text{$q$-cofibration $\Rightarrow$ $h$-cofibration $\Rightarrow$ $\dcal$-embedding.}
\]
\begin{rem}\label{difference}
\begin{itemize}
\item[{\rm (1)}] For an $h$-cofibration $i: A \longhookrightarrow X$, the canonical map $X \times (0) \cup_{A \times (0)} A \times I \longhookrightarrow X \times I$ need not be a $\dcal$-embedding; consider the $h$-cofibration $\dot{\Delta}^{1} \longhookrightarrow \Delta^{1}$ and see the proof of \cite[Proposition A.2(1)]{origin}.
\item[{\rm (2)}] Topological $h$-cofibrations are defined by the ${HEP}$ (\cite[p. 41]{MayAT}). However, diffeological $h$-cofibrations, introduced in this subsection, are defined by the ${WHEP}$ so that Proposition \ref{gimpliesh} holds true. In fact, since the implications
	\begin{equation*}
	\begin{split}
	&\text{\;\;\;\;\;\;\;\;\;\;\;}i:A\longrightarrow X \text{ satisfies the } HEP \\
	&\iff X\times(0)\underset{A \times (0)}{\cup} A\times I \longhookrightarrow X\times I  \text{ has a retraction}\\
	&\text{\;\;}\Longrightarrow X\times (0) \underset{A\times (0)}{\cup} A\times I \longhookrightarrow X\times I \text{ is a $\dcal$-embedding}
	\end{split}
	\end{equation*}
	 hold, even $\dot{\Delta}^{1} \longhookrightarrow \Delta^{1}$ does not satisfy the ${HEP}$ (see Part 1).
\item[{\rm (3)}] I do not know under what conditions the notions of an $h$-cofibration and an ${NDR}$-pair are equivalent or whether there exists a model structure on $\dcal$ whose cofibrations are just $h$-cofibrations (cf. \cite[p. 43 and p. 46]{MayAT}, \cite[p. 340]{MP}).
\end{itemize}
\end{rem}
We need several more results on $h$-cofibrations to prove Theorem \ref{hcofibrancy}.
\begin{lem}\label{factorization}
Every smooth map $f: X \longrightarrow Y$ has a functorial factorization
\[
X \xhookrightarrow[]{\ i\ } Y' \xrightarrow[]{\ p\ } Y
\]
such that $i$ is an $h$-cofibration and $p$ is a $\dcal$-homotopy equivalence.
\end{lem}
\begin{proof}
	We define the mapping cylinder $Mf$ by the pushout diagram in $\dcal$
	\[
	\begin{tikzcd}
	X \arrow{r}{f} \arrow[swap, hook']{d}{i_{1}} & Y \arrow[hook']{d}\\
	X \times I \arrow{r} & Mf,
	\end{tikzcd}
	\]
	where the map $i_{1}$ is defined by $i_{1} (x) = (x, 1)$.
 Consider the factorization
 \[
 \begin{tikzcd}
 X \arrow[hook]{r}{i} \arrow[swap]{rd}{f} & Mf \arrow{d}{p}\\
              &  Y,
 \end{tikzcd}
 \]
 where the maps $i$ and $p$ are defined to be the composite $ X \xhookrightarrow[]{\ i_{0} \ } X \times I \longrightarrow Mf$ and the sum $Mf = X \times I \cup_{X} Y \xrightarrow{f \circ proj + 1_{Y}} Y,$ respectively. Then, we can observe that $(Mf, X)$ is an ${\rm NDR}$-pair, and hence that this factorization is the desired one (see Lemma \ref{NDR}).
\end{proof}
For a diffeological space $A$, let $A/\dcal$ denote the category of diffeological spaces under $A$. Then, the notion of a $\dcal$-homotopy under $A$ and that of a $\dcal$-homotopy equivalence under $A$ are defined in the obvious manner (cf. \cite[p. 44]{MayAT}). Since $A/\dcal$ is a simplicial category (see Section 7.1), these homotopical notions can be also defined as $A/\dcal$-homotopical notions (see Section 4.3).
\begin{lem}\label{homotopyunder}
Let $i: A \longrightarrow X$ and $j: A \longrightarrow Y$ be $h$-cofibrations and let $f:X \longrightarrow Y$ be a map under $A$. If $f$ is a $\dcal$-homotopy equivalence, then $f$ is a $\dcal$-homotopy equivalence under $A$.
\end{lem}
\begin{proof}
The result is a diffeological version of Proposition in \cite[p. 44]{MayAT} and can be proved by a similar argument; see Remark \ref{remhcofibr}(1).
\end{proof}
\begin{prop}\label{homotopypair}
	Assume given a commutative diagram in $\dcal$
	\[
	\begin{tikzcd}
	A \arrow[hook]{r}{i} \arrow[swap]{d}{d} & X \arrow{d}{f} \\
	B \arrow[hook,swap]{r}{j} & Y
	\end{tikzcd}
	\]
	in which $i$ and $j$ are $h$-cofibrations and $d$ and $f$ are $\dcal$-homotopy equivalences. Then, $(f, d): (X, A) \longrightarrow (Y, B)$ is a $\dcal$-homotopy equivalence of pairs.
\end{prop}
\begin{proof}
This result is a diffeological version of Proposition in \cite[p. 45]{MayAT}. Recall the fact that if a morphism $\phi$ of a category $\ccal$ has a left inverse $\alpha$ and a right inverse $\beta$, then $\alpha = \beta$ and hence $\phi$ is an isomorphism. Then, we have only to show that a given $\dcal$-homotopy inverse $e$ of $d$ and a given $\dcal$-homotopy $h$ connecting $e \circ d$ to $1_{A}$ extend to a $\dcal$-homotopy inverse $g$ of $f$ and a $\dcal$-homotopy $H$ connecting $g \circ f$ to $1_{X}$ (recall Remark \ref{remhcofibr}(1)).
\par\indent
We can choose a $\dcal$-homotopy inverse $g$ of $f$ as in \cite[p. 46]{MayAT}. Choose $\dcal$-homotopies $\gamma: X \times I \times (0) \longrightarrow X$ and $\Gamma: A \times I \times I \longrightarrow X$ as in \cite[p. 46]{MayAT} and consider the extension problem
\[
\begin{tikzcd}
 X \times I \times (0) \underset{A\times I \times (0)}{\cup} A \times I \times I  \arrow[hook']{d} \arrow{rr}{\gamma + \Gamma} && X\\
 X \times I \times I. \arrow[dashed, yshift=-0.7ex]{rru} &&
\end{tikzcd}
\]
(Note that unlike in \cite[p. 46]{MayAT}, the domain of $\gamma$ is $X\times I \times (0)$.) Since $i \times 1 : A \times I \longhookrightarrow X \times I$ is an $h$-cofibration (Lemma \ref{basicshcofibr}(1)), we have a solution $\Lambda$ of the above extension problem. Then, the sum $\Sigma$ of the three $\dcal$-homotopies $\Lambda|_{X\times (0) \times I}, \Lambda|_{X \times I \times (1)}$ and $\Lambda|_{X \times (1) \times I}$ is a $\dcal$-homotopy connecting $g \circ f$ to $1_{X}$. We thus obtain the desired $\dcal$-homotopy $H$ by reparametrizing $\Sigma$ (see Remark \ref{remhcofibr}(1)).
\end{proof}
\begin{lem}\label{homotopypushout}
	Assume given a commutative diagram in $\dcal$
	\[
	\begin{tikzcd}
	A_{1}\arrow[swap]{d}{h_{1}} & A_{0} \arrow[swap,hook']{l}{f_{1}} \arrow{r}{f_{2}} \arrow[swap]{d}{h_{0}} & A_{2} \arrow[swap]{d}{h_{2}}\\
	 B_{1} & B_{0} \arrow[hook']{l}{g_{1}} \arrow[swap]{r}{g_{2}}  & B_{2}
	\end{tikzcd}
	\]
	in which $f_{1}$ and $g_{1}$ are $h$-cofibrations. If $h_{0}$, $h_{1}$, and $h_{2}$ are $\dcal$-homotopy equivalences, then so is their pushout
	\[
	A_{1} {\cup}_{A_{0}} A_{2} \longrightarrow B_{1} {\cup}_{B_{0}} B_{2}.
	\]
\end{lem}
\begin{proof}
	From Proposition \ref{homotopypair} (and its proof), we see that if $f_{2}$ and $g_{2}$ are also $h$-cofibrations, then the result holds. Thus, we have only to show that the map $1_{A_1} \cup_{A_0} p:A_1 \cup_{A_0} A'_2 \longrightarrow A_1 \cup_{A_0} A_2$ induced by the commutative diagram
		\[
		\begin{tikzcd}
		A_{1}\arrow[swap]{d}{1_{A_{1}}} & A_{0} \arrow[swap,hook']{l}{f_{1}} \arrow[hook]{r}{i} \arrow[swap]{d}{1_{A_{0}}} & A'_{2} \arrow{d}{p}\\
		A_{1} & A_{0} \arrow[hook']{l}{f_{1}} \arrow[swap]{r}{f_{2}}  & A_{2}
		\end{tikzcd}
		\]
		is a $\dcal$-homotopy equivalence, where $A'_{2} = Mf_{2}$ and the factorization $A_{0} \xhookrightarrow[]{\ i_{1}\ } A'_{2} \xrightarrow{\ p \ } A_{2}$ of $f_{2}$ is the one constructed in the proof of Lemma \ref{factorization}.
\par\indent
Note that $A'_{2} \longhookrightarrow A_{1} {\cup}_{A_{0}} A'_{2}$ is an $h$-cofibration (Lemma \ref{basicshcofibr}(2)) and consider the extension problem in $\dcal$
\[
\begin{tikzcd}
(A_{1} {\cup}_{A_{0}} A'_{2}) \times (0) \cup_{A'_{2} \times (0)} A'_{2} \times I \arrow{rr}{1_{A_{1} {\cup}_{A_{0}} A'_{2}} + h} \arrow[hook']{d} & & A_{1} {\cup}_{A_{0}} A'_{2}, \\
(A_{1} {\cup}_{A_{0}} A'_{2}) \times I \arrow[swap,dashed]{rru}{k} & &
\end{tikzcd}
\]
where $h$ is a reparametrization of the canonical deformation of $A'_{2}$ onto $A_{2}$ such that $h_t=h_0$ for sufficiently small $t$ (see Remark \ref{remhcofibr}(1)). Then, we have a solution $k$ of the extension problem.
\par\indent
Note that since $p:A_{2}'\longrightarrow A_2$ has a canonical splitting, $p:A'_2 \longrightarrow A_2$, and hence $1_{A_1}\cup_{A_0}p: A_1 \cup_{A_0} A_2'\longrightarrow A_1\cup_{A_0} A_2$ is a $\dcal$-quotient map. Thus, we see that $k_{1} = k(\cdot, 1)$ is factored as
\[
\begin{tikzcd}
A_{1} {\cup}_{A_{0}} A'_{2} \arrow{r}{k_{1}} \arrow[swap]{d}{1_{A_{1}} {\cup}_{A_{0}} p} & A_{1} {\cup}_{A_{0}} A'_{2}\\
A_{1} {\cup}_{A_{0}} A_{2} \arrow[swap]{ru}{j} &
\end{tikzcd}
\]
in $\dcal$, and hence that $j \circ (1_{A_{1}} {\cup}_{A_{0}} p) \simeq_\dcal 1_{A_{1} {\cup}_{A_{0}} A'_{2}}$. Using \cite[Lemma 2.5]{origin}, we also see that the $\dcal$-homotopy $(1_{A_{1}} {\cup}_{A_{0}} p) \circ k$ is factored as
\[
\begin{tikzcd}
(A_{1} {\cup}_{A_{0}} A'_{2}) \times I \arrow{r}{k} \arrow[swap]{d}{(1_{A_{1}} {\cup}_{A_{0}} p) \times 1} & A_{1} {\cup}_{A_{0}} A'_{2} \arrow{d}{1_{A_{1}} {\cup}_{A_{0}} p}\\
(A_{1} {\cup}_{A_{0}} A_{2}) \times I \arrow{r}{\bar{k}} & A_{1} {\cup}_{A_{0}} A_{2},
\end{tikzcd}
\]
which shows that $(1_{A_{1}} {\cup}_{A_{0}} p) \circ j \simeq_\dcal 1_{A_{1} {\cup}_{A_{0}} A_{2}}$.
\end{proof}
\begin{lem}\label{homotopyseqcolim} Assume given a commutative diagram
	\[
	\begin{tikzcd}
	X_{0} \arrow[hook]{r}{i_{1}} \arrow[swap]{d}{f_{0}} & X_{1} \arrow[hook]{r}{i_{2}} \arrow[swap]{d}{f_{1}} & X_{2} \arrow[hook]{r}{i_{3}} \arrow[swap]{d}{f_{2}} & \cdots\\
	Y_{0} \arrow[hook,swap]{r}{j_{1}} & Y_{1} \arrow[hook,swap]{r}{j_{2}} & Y_{2} \arrow[hook,swap]{r}{j_{3}} & \cdots
	\end{tikzcd}
	\]
		in which the $i_{k}$ and $j_{k}$ are $h$-cofibrations. If the maps $f_{k}$ are $\dcal$-homotopy equivalences, then so is their colimit
	\[
	\lim\limits_{\rightarrow} \ f_{k} : \lim\limits_{\rightarrow} X_{k} \longrightarrow \lim\limits_{\rightarrow} Y_{k}.
	\]
\end{lem}
\begin{proof}
	The result can be proved using Proposition \ref{homotopypair} and its proof.
\end{proof}
\subsection{Proof of Theorem \ref{hcofibrancy}}
In this section, we prove Theorem \ref{hcofibrancy} using the results in Sections 9.1 and 9.2.
\par\indent
Let us begin by proving a few lemmas. Regard the subclass $\wcal_{\dcal}$ of $\dcal$ as a full subcategory of $\dcal$ and recall that
\[
\ical = \{\dot{\Delta}^{p} \longhookrightarrow \Delta^{p} \ | \ p \geq 0  \}.
\]
See \cite[Definitions 15.1.1 and 15.1.2]{MP} for a sequential relative $\ical$-cell complex.
\begin{lem}\label{cofibrfactorization}
	Let $f:X_{0} \longrightarrow X_{1}$ be a morphism of $\wcal_{\dcal}$ and $h_{0}: A_{0} \longrightarrow X_{0}$ a $\dcal$-homotopy equivalence with $A_{0}$ cofibrant. Then, there exist a sequential relative $\ical$-cell complex $i: A_{0} \longhookrightarrow A_{1}$ and a $\dcal$-homotopy equivalence $h_{1}: A_{1} \longrightarrow X_{1}$, making the diagram
	\[
	\begin{tikzcd}
	A_{0} \arrow[hook]{r}{i} \arrow[swap]{d}{h_{0}} & \arrow{d}{h_{1}} A_{1}\\
	X_{0} \arrow{r}{f} & X_{1}
	\end{tikzcd}
	\]
	commute.
\end{lem}
\begin{proof}
	Define the maps
	\[
	A_{0} \xhookrightarrow[]{\ \ \ i\ \ \ } A_{1} \xrightarrow{\ \ \ h_{1} \ \ \ } X_{1}
	\]
	by applying the cofibration-trivial fibration factorization in \cite[the proof of Theorem 1.3]{origin} to $f \circ h_{0}$. We can see from Corollary \ref{4equiv} that $i$ and $h_{1}$ are the desired maps.
\end{proof}
\begin{lem}\label{posc}
	\begin{itemize}
	\item[{\rm (1)}] Assume given a diagram in $\dcal$
	\[
	\begin{tikzcd}
	X_{1} & X_{0} \arrow[hook',swap]{l}{f_{1}} \arrow{r}{f_{2}} & X_{2}
	\end{tikzcd}
	\]
	with $f_{1}$ $h$-cofibration. If $X_{0}, X_{1},$ and $X_{2}$ are in $\wcal_{\dcal}$, then so is $X_{1} {\cup}_{X_{0}} X_{2}$.
	\item[{\rm (2)}] Assume given a sequence in $\dcal$
	\[
	X_{0} \xhookrightarrow[]{\ \ i_{0}\ \ } X_{1} \xhookrightarrow[]{\ \ i_{1}\ \ } X_{2} \xhookrightarrow[]{\ \ i_{2}\ \ } \cdots
	\]
	such that the $i_{n}$ are $h$-cofibrations. If the $X_{n}$ are in $\wcal_{\dcal}$, then so is $X = \lim\limits_{\rightarrow} X_{n}$.
	\end{itemize}
\end{lem}
\begin{proof}
	Using Lemma \ref{cofibrfactorization}, we can derive Parts 1 and 2 from Lemmas \ref{homotopypushout} and \ref{homotopyseqcolim} respectively; recall Proposition \ref{gimpliesh}.
\end{proof}
\begin{proof}[Proof of Theorem \ref{hcofibrancy}] Since $pr: BX_{U} \longrightarrow X$ is a $\dcal$-homotopy equivalence (Proposition \ref{dSegal}), we prove that $BX_{U}$ is in $\wcal_{\dcal}$.

Recall from Section 9.1 the pushout diagram
\begin{equation}
\begin{tikzcd}
\underset{\sigma_{n} \supsetneq \cdots \supsetneq \sigma_0}{\coprod} U_{\sigma_{n}} \times \dot{\Delta}^{n} \arrow{r} \arrow[hook']{d} & BX^{(n-1)}_{U} \arrow[hook']{d}\\
\underset{\sigma_{n} \supsetneq \cdots \supsetneq \sigma_0}{\coprod} U_{\sigma_{n}} \times \Delta^{n} \arrow{r} & BX^{(n)}_{U}
\end{tikzcd}
\tag{$9.1$}
\end{equation}
and the identity
\begin{equation}
BX_{U} = \lim\limits_{\rightarrow} BX^{(n)}_{U}.
\tag{$9.2$}
\end{equation}

First, we show that $BX^{(n)}_{U}$ is in $\wcal_{\dcal}$ by induction on $n$. Since $BX^{(0)}_{U} = \underset{\sigma}{\coprod}\ U_{\sigma}$, the result holds for $n = 0$. Suppose that the result holds up to $n-1$. Then, we see from pushout diagram $(9.1)$ that $BX^{(n)}_{U}$ is in $\wcal_{\dcal}$, using Proposition \ref{gimpliesh}, Lemma \ref{basicshcofibr}, Corollary \ref{product} and Lemma \ref{posc}(1).

From pushout diagram $(9.1)$, we also see that $BX^{(n-1)}_{U} \longhookrightarrow BX^{(n)}_{U}$ is an $h$-cofibration (Lemma \ref{basicshcofibr}). Thus, we obtain the result from that for $BX^{(n)}_{U}$, (9.2), and Lemma \ref{posc}(2).
\end{proof}
\begin{rem}\label{refinement}
	For an ordinal number $\gamma$, let $\omega_{\gamma}$ denote the $\gamma^{\rm th}$ infinite cardinal number (\cite[p. 66]{Cie}) and define the subclass $\wcal_{\dcal\, \gamma}$ of $\wcal_\dcal$ by
	\[
	\wcal_{\dcal\, \gamma} = \{ A \in \dcal \ | \ \text{$A \simeq_{\dcal} |K|_{\dcal}$ for some simplicial set $K$ with ${\rm card}\ K \leq \omega_\gamma$  } \},
	\]
	where we set ${\rm card}\ K = |\underset{n \geq 0}{\coprod}\ K_{n}|$. The class $\wcal_{\dcal\, 0}$ is just the subclass of diffeological spaces having the $\dcal$-homotopy type of the realization of a countable simplicial set (cf. \cite[Section 1]{Mi}).
	
	We give a basic result on $\wcal_{\dcal\, \gamma}$ (Part 1), a diffeological version of \cite[Proposition 2]{Mi}(Part 2), and a refined version of Theorem \ref{hcofibrancy} (Part 3).
	\begin{itemize}
	\item[{\rm (1)}] Define the fibrant approximation functor $\hat{\cdot} : \scal \longrightarrow \scal$ by assigning to a simplicial set $K$ the Kan complex $\hat{K}$ such that $K \longhookrightarrow \hat{K} \longrightarrow \ast$ is the usual trivial cofibration-fibration factorization of $K \longrightarrow \ast$ (see, eg, \cite{K}). Then, we have the implications
	\begin{align*}
	\:\:\:\:\:\:{\rm card}\ K \leq \omega_{\gamma} & \Leftrightarrow {\rm card}\: \hat{K} \leq \omega_{\gamma}\\
	& \Rightarrow |\pi_{i} (\hat{K}, k_{0})| \leq \omega_{\gamma} \ \text{for any $i$ and any $k_{0}$}\\
	& \Rightarrow \hat{K}\: \text{is homotopy equivalent to some $Kan$}\\ &\:\:\:\:\:\:\:\text{complex $L$ with {\rm card} $L\leq \omega_\gamma$}.
	\end{align*}
	From these, we obtain the following equivalence:
	\begin{equation*}
		\:\:\:\:\:\:\:\:\:\:\:\:\:\:\:\:\:\:\:A\in \wcal_{\dcal\,\gamma} \Leftrightarrow A \in \wcal_\dcal \: \text{and } |\pi^\dcal_i(A,a)|\leq \omega_\gamma \:\text{for any $i$ and any $a$}.
	\end{equation*}
	(see Theorem \ref{Quillenequiv}, and Corollaries \ref{W} and \ref{4equiv}). A similar result holds for the similarly defined subclass $\wcal_{\czero \gamma}$ of $\wcal_\czero$.
	\item[{\rm (2)}] Let us consider the case of $\gamma=0$. We see from \cite[Theorem 1]{Mi} that for $A\in \czero$, the following equivalences hold:
	\begin{align*}
		\:\:\:\:\:\: A\in \wcal_{\czero\, 0} & \Leftrightarrow A \ \text{is homotopy equivalent to a countable }\\
		& \:\:\:\:\:\:\:\:CW \text{-complex}.\\
		& \Leftrightarrow A \ \text{is homotopy equivalent to a conuntable}\\
		& \:\:\:\:\:\:\:\: polyhedron.
	\end{align*}
	Thus, the class $\wcal_{\czero\, 0}$ is just the intersection of Milnor's class $\wcal_0$ and $\czero$.\par
	Further, using Corollary \ref{homotopycat} and \cite[Proposition 2]{Mi}, we have the following result:
	\vspace{0.2cm}
	\begin{quote}
		If $X$ is a diffeological space in $\wcal_\dcal$ with $\widetilde{X}$ Lindel\"of, then $X$ is in $\wcal_{\dcal\,0}$.
	\end{quote}
	\vspace{0.2cm}
	Theorem \ref{hcofibrancy} along with this result gives a useful criterion for a diffeological space to be in $\wcal_{\dcal\,0}$, which is used in the proof of Proposition \ref{countable}.
	\item[{\rm (3)}] Theorem \ref{hcofibrancy} can be refined as follows:
	\vspace{0.2cm}
	\begin{quote}
		Let $X$ be a diffeological space and $U = \{ U_{\alpha} \}_{\alpha \in A}$ a $\dcal$-numerable covering of $X$ by subsets. If $|A| \leq \omega_\gamma$ and $U_{\sigma} \in \wcal_{\dcal\, \gamma}$ for any nonempty finite $\sigma \subset A$, then $X$ is in $\wcal_{\dcal\, \gamma}$.
	\end{quote}
	\vspace{0.2cm}
	For the proof, we have only to show that Lemma \ref{cofibrfactorization} can be refined as follows: If $X_{0}, X_{1} \in \wcal_{\dcal\, \gamma}$ and $A_{0} = |K_{0}|_{\dcal}$ with {\rm card} $K_{0} \leq \omega_{\gamma}$, we can take $A_{1}$ such that $A_{1} = |K_{1}|_{\dcal}$ for some $K_{1}$ with {\rm card} $K_{1} \leq \omega_{\gamma}$. This can be done by using Part 1 and Theorem \ref{Quillenequiv}(1).
	
	The topological analogue, which is a refined version of \cite[Theorem 4]{tom}, can be also shown by a similar argument.
	\end{itemize}
\end{rem}

\section{Locally contractible diffeological spaces}
In this section, we introduce the notion of a locally contractible diffeological space and prove that such a diffeological space is in $\vcal_{\dcal}$ (Theorem \ref{V}).\par
Recall that a diffeological space $X$ is called $\dcal$-contractible if $X$ is $\dcal$-homotopy equivalent to the terminal object $\ast$.
\begin{defn}\label{loccontr}
	A diffeological space $X$ is called locally $\dcal$-contractible if every $x\in X$ has a neighborhood basis consisting of $\dcal$-contractible open diffeological subspaces.
\end{defn}
$\dcal$-contractible (resp. locally $\dcal$-contractible) diffeological spaces are simply called contractible (resp. locally contractible) diffeological spaces if there is no confusion in context.\par
Note that if a diffeological space $X$ is contractible (resp. locally contractible), $\widetilde{X}$ is contractible (resp. locally contractible) as a topological space (see \cite[Section 2.4]{origin}).
\par\indent
We define the iterated barycentric subdivision $\sd^{k}\Delta^{p}$ by $\sd^{k}\Delta^{p} = |\sd^{k}\Delta(p)|_{\dcal}$ and identify it with $|\Delta(p)|_{\dcal} = \Delta^{p}$ topologically (see \cite[p. 124]{Spa} for the iterated barycentric subdivision $\sd^{k}K$ of a simplicial complex $K$). For the proof of Theorem \ref{V}, we introduce the $\dcal$-polyhedron $(\Delta^{p} \times I)^{(k)}$ connecting ${\rm sd}^{k}\Delta^{p}$ to $\Delta^{p}$.

First, we introduce the $\dcal$-polyhedra $(\Delta^p \times I)'$ such that
\begin{itemize}
	\item $\widetilde{(\Delta^p\times I)'}$ = $\Delta^{p}_{\mathrm{top}} \times I_{\mathrm{top}}$, where $\Delta^p_{\mathrm{top}}$ and $I_{\mathrm{top}}$ denotes the topological standard $p$-simplex and the topological unit interval, respectively.
	\item The subsets $\Delta^p \times (0)$ and $\Delta^p \times (1)$ are $\dcal$-subpolyhedra which are isomorphic to $\sd\,\Delta^p$ and $\Delta^p$ respectively.
\end{itemize}
For this, we inductively define the simplicial complexes $(\Delta(p)\times I)'$ whose $\dcal$-polyhedra are the desired $(\Delta^p \times I)'$. For $p = 0$, we define $(\Delta(p)\times I)'$ by $(\Delta(p)\times I)'=\Delta(0)\times I\ (\cong I = \Delta(1))$. For $p = 1$, we define $(\Delta(p)\times I)'$ to be the cone of the simplicial complex $(0)\times I \cup \Delta(1)\times(1)\cup(1)\times I$, where the cone $CK$ of a simplicial complex $K$ is defined by
\[
\begin{tikzcd}
\overline{CK} = \ast \coprod \overline{K} \ \text{and} \ \Sigma_{CK} = \ast \coprod \Sigma_{K} \coprod \{\ast \coprod\ \sigma\ |\ \sigma\ \in \Sigma_{K} \}
\end{tikzcd}
\]
(see the left side of Fig 10.1). Suppose that $(\Delta(p-1)\times I)'$ is constructed. Then the simplicial complex $(\Delta(p)\times I)'$ is defined to be the cone of the simplicial complex $L_p$, where $L_p$ is the union of $\Delta(p)\times (1)$ and $p+1$ copies of $(\Delta(p-1)\times I)'$ whose polyhedron is topologically identified with $\Delta^p \times (1) \cup_{\dot{\Delta}^p\times (1)} \dot{\Delta}^p \times I$.\par
Using the construction of the $\dcal$-polyhedron $(\Delta^p\times I)'$ repeatedly, we obtain the $\dcal$-polyhedron $(\Delta^p\times I)^{(k)}$ such that
\begin{itemize}
	\item $\widetilde{(\Delta^p\times I)^{(k)}} = \Delta^{p}_{\mathrm{top}}\times I_{\mathrm{top}}$.
	\item The subsets $\Delta^p\times (0)$ and $\Delta^p\times (1)$ are $\dcal$-subpolyhedra which are isomorphic to $\sd^k\Delta^p$ and $\Delta^p$ respectively.
\end{itemize}
The simplicial complex corresponding to $(\Delta^p\times I)^{(k)}$ is denoted by $(\Delta(p)\times I)^{(k)}$.\\
\begin{center}
	$I$
	\begin{tikzpicture}[thick, baseline=0pt]
	\draw (-2,-2) -- (2,-2) -- (2,2) -- (-2,2) -- (-2,-2);
	\draw (-2,2) -- (0,-2) -- (2,2);
	\node[] at (0, 2.2)  (c)     {$\Delta^{1}$};
	\node[] at (0, -2.4) (c) {$\dcal$-polyhedron $(\Delta^1\times I)'$};
	\end{tikzpicture}
	\hspace{4pc}
	$I$
	\begin{tikzpicture}[thick, baseline=0pt]
	\draw (-2,-2) -- (2,-2) -- (2,2) -- (-2,2) -- (-2,-2);
	\draw (-2,2) -- (0,0) -- (2,2);
	\draw (-2,0) -- (2,0);
	\draw (-2,0) -- (-1,-2) -- (0,0);
	\draw (0,0) -- (1,-2) -- (2,0);
	\draw (0,0) -- (0,-2);
	\node[] at (0, 2.2)  (c)     {$\Delta^{1}$};
	\node[] at (0, -2.4) (c) {$\dcal$-polyhedron $(\Delta^1\times I)^{(2)}$};
	\end{tikzpicture}\\
	Fig 10.1
\end{center}

\begin{lem}\label{sdsmoothing}
	Let $X$ be a diffeological space and $\sigma$: $\Delta^p \longrightarrow X$ a continuous map. Suppose that $\sigma$ satisfies the following conditions:

	\begin{itemize}
		\item $\dot{\sigma} = \sigma\ |_{\dot{\Delta}^{p}}$ is smooth.
		\item $\sigma$ is smooth as a map from $\sd^k\Delta^p$ to $X$.
	\end{itemize}
	Then there exists a smooth map $\sigma'$: $\Delta^p \longrightarrow X$ such that $\sigma'$ is $\czero$-homotopic to $\sigma$ relative to $\dot{\Delta}^p$.

	\begin{proof}
		Consider the $\dcal$-subpolyhedron
		\[
			{\rm sd}^k\Delta^p \times (0) \cup \underset{i}{\cup}\ (\Delta^{p-1}_{(i)} \times I)^{(k)}
		\]
		of $(\Delta^p \times I)^{(k)},$ where $(\Delta^{p-1}_{(i)}\times I)^{(k)}$ is a copy of $(\Delta^{p-1}\times I)^{(k)}$ for $0\leq i \leq p$ (see \cite[Definition 4.1]{origin}). Noticing that the composite
		\[
		\underset{i}{\cup}\ (\Delta^{p-1}_{(i)} \times I)^{(k)} \xrightarrow{proj} \dot{\Delta}^{p} \xrightarrow{\ \ \dot{\sigma}\ } X
		\]
		is smooth (Proposition \ref{axioms}), we then have the commutative solid arrow diagram in $\dcal$
	\[
	\begin{tikzcd}
	\sd^k\Delta^p\times(0)\:{\scriptstyle\cup}\underset{i}{\cup}(\Delta^{p-1}_{(i)}\times I)^{(k)}    \arrow{rrr}{\sigma+\dot{\sigma}\circ proj} \arrow[hook']{d} & & & X \arrow{d} \\
	(\Delta^p\times I)^{(k)} \arrow{rrr} \arrow[dashed, yshift=-0.7ex]{urrr}{h} & & & \ast
	\end{tikzcd}
	\]
	(see Lemma \ref{Dsub}). Since the left and right vertical arrows are a trivial cofibration and a fibration respectively (Lemmas \ref{Dsub} and \ref{clequiv} and Theorem \ref{originmain}), there is a dotted arrow $h$ making the diagram commute. Then, $\sigma':= h(\cdot,1)$ is the desired smooth map from $\Delta^p$ to $X$.
	\end{proof}
\end{lem}

\begin{proof}[Proof of Theorem \ref{V}.]
	Let $\iota_{X}$ denote the natural inclusion $\sdcal X \longhookrightarrow S \widetilde{X}$. Note that $\sdcal X$ and $S\tilde{X}$ are Kan complexes and consider the following statements:
	\vspace{0.1cm}

	(${\rm A}_p$) For any locally contractible diffeological space $X$ and any fixed base point $x$, $\pi_{i}(\iota_X):\pi_{i}(\sdcal X, x) \longrightarrow \pi_{i}(S\tilde{X}, x)$ is injective for $i = p - 1$ and surjective for $i=p$.
	\vspace{0.1cm}

	(${\rm B}_{p}$) For any locally contractible diffeological space $X$ and any continuous map $\tau:\Delta^{p+1} \longrightarrow X$ with $\tau\ |_{\dot{\Delta}^{p+1}}$ smooth, there exists a smooth map $\tau': \Delta^{p+1} \longrightarrow X$ such that $\tau'$ is $\czero$-homotopic to $\tau$ relative to $\dot{\Delta}^{p+1}$.
	\vspace{0.1cm}

	We arrange the statements in the following order:
	\[
	({\rm A}_{0}), ({\rm B}_{0}), ({\rm A}_{1}), \ldots, ({\rm A}_{p}), ({\rm B}_{p}), ({\rm A}_{p+1}), \ldots
	\]
	Noticing that $({\rm A}_{p})$ obviously holds for $p = 0$, we prove all statements in two steps, which implies the result (see Corollary \ref{W}(3)).\vspace{0.2cm}\\
	{\sl Step 1}. Suppose that the implications hold up to $({\rm A}_{p})$. Then, we show that $({\rm B}_p)$ holds.\par
	Choose a covering $\ucal$ of $X$ consisting of contractible open sets. Take a sufficiently large integer $k$ and choose for each $(p+1)$-simplex $\Delta^{p+1}_{\alpha}$ of $\sd^{k}\Delta^{p+1}$, an element $U_{\alpha}$ of $\ucal$ such that $\tau(\Delta^{p+1}_{\alpha}) \subset U_{\alpha} $ (see, eg, \cite[p. 178]{Spa}). By Lemma \ref{sdsmoothing}, we have only to construct a $\czero$-homotopy
	\begin{center}
		$ h:\Delta^{p+1} \times I \longrightarrow X $
	\end{center}
	such that
	\begin{itemize}
		\item
		$h|_{\Delta^{p+1} \times (0)}=\tau$.
		\item
		$h$ is the constant homotopy of $\tau|_{\dot{\Delta}^{p+1}}$ on $\dot{\Delta}^{p+1} \times I$.
		\item
		$h|_{\Delta^{p+1} \times (1)}$ is smooth as a map from ${\rm sd}^{k}\Delta^{p+1}$ to $X$.
	\end{itemize}

	First, let us construct $h$ on $\Delta^{p+1} \times (0) \cup \dot{\Delta}^{p+1} \times I \cup {\rm sk}_{p}(\sd^{k}\Delta^{p+1}) \times I$, where  ${\rm sk}_{l}(\sd^{k}\Delta^{p+1}$) denotes the $l$-skeleton of $\sd^{k}\Delta^{p+1}$. Define $h$ on $\Delta^{p+1} \times (0) \cup \dot{\Delta}^{p+1} \times I \cup {\rm sk}_{0}(\sd^{k}\Delta^{p+1}) \times I$ to be the sum of $\tau$ and the constant homotopy of $\tau \ |_{\dot{\Delta}^{p+1} \cup\, {\rm sk_{0}} ({\rm sd}^{k} \Delta^{p+1}) }$. Suppose that $h$ is defined on $\Delta^{p+1} \times (0) \cup \dot{\Delta}^{p+1} \times I\ \cup\ {\rm sk}_{l}\ (\sd^{k}\Delta^{p+1}) \times I$ for $l < p$ so that the image $h(\Delta^{l}_{\alpha^{(l)}} \times I)$ is contained in $\bigcap \limits_{ \Delta^{l}_{\alpha^{(l)}} \subset \Delta^{p+1}_{\alpha} } U_\alpha$ and $h|_{\Delta^{l}_{\alpha^{(l)}} \times (1)}$ is smooth for any $l$-simplex $\Delta^{l}_{\alpha^{(l)}}$ of $\sd^{k}\Delta^{p+1}$. We have only to construct for each $(l+1)$-simplex $\Delta^{l+1}_{\alpha^{(l+1)}}$ of $\sd^{k}\Delta^{p+1}$ a solution of the lifting problem in $\czero$

	\begin{center}
		\begin{equation}
		\label{LABEL}
		\begin{tikzcd}
		\Delta^{l+1}_{\alpha^{(l+1)}}\times(0) \cup             \dot{\Delta}^{l+1}_{\alpha^{(l+1)}} \times I \arrow[r,"h"] \arrow[d, hook'] &         \bigcap \limits_{ \Delta^{l+1}_{\alpha^{(l+1)}} \subset     \Delta^{p+1}_{\alpha} }U_\alpha \arrow[d]\\
		\Delta^{l+1}_{\alpha^{(l+1)}} \times I \arrow[r] \cdarrow{ur,dashed} \arrow{r}& \ast
		\end{tikzcd}
		\end{equation}
	\end{center}
	such that its restriction to $\Delta^{l+1}_{\alpha^{(l+1)}} \times (1)$ is smooth.
	\par\indent
	We give an outline of the construction; it can be made precise by solving lifting problems in the model categories $\czero$ and $\dcal$. Since the left and right vertical arrows in (10.1) is a trivial cofibration and a fibration respectively, there exists a continuous solution $\Gamma$ to the lifting problem. Set $\gamma = \Gamma (\cdot, 1)$ and $U_{\alpha^{(l+1)}}$ =  $\bigcap \limits_{ \Delta^{l+1}_{\alpha^{(l+1)}} \subset \Delta^{p+1}_{\alpha} }U_\alpha$. Then, there exists a smooth map $\gamma' : \Delta^{l+1}_{\alpha^{(l+1)}} \longrightarrow U_{\alpha^{(l+1)}}$ such that $\gamma'\ |_{\dot{\Delta}^{l+1}_{\alpha^{(l+1)}}} = \gamma\ |_{\dot{\Delta}^{l+1}_{\alpha^{(l+1)}}}$ by the injectivity of $\pi_{l} (\iota_{U_{\alpha^{(l+1)}}})$. Further, we can use the surjectivity of $\pi_{l+1}(\iota_{U_{\alpha^{(l+1)}}})$ to choose $\gamma'$ so that $h+\gamma': \Delta^{l+1}_{\alpha^{(l + 1)}} \times (0) \,\cup\, \dot{\Delta}^{l+1}_{\alpha^{(l + 1)}} \times I \,\cup\, \Delta^{l+1}_{\alpha^{(l + 1)}} \times (1) \longrightarrow U_{\alpha^{(l+1)}}$ extends to $\Delta^{l+1}_{\alpha^{(l + 1)}} \times I$; this extension is the desired solution of lifting problem (10.1).
	\par\indent
	Last, recall that $U_{\alpha}$ is contractible and extend $h$ to $\Delta^{p+1}_{\alpha} \times (1)$ smoothly and further to $\Delta^{p+1}_{\alpha} \times I$ continuously for each $(p+1)$-simplex $\Delta^{p+1}_{\alpha}$. It completes the construction of the desired $\czero$-homotopy $h:\Delta^{p+1} \times I \longrightarrow X$. \vspace{0.2cm}

	\noindent
	{\sl Step 2}. Suppose that the implications hold up to $({\rm B}_{p})$. Then, we show that $({\rm A}_{p+1})$ holds.\par
	Let $\sigma$ be a smooth map from $\Delta^{p}$ to $X$ with $\sigma|_{\dot{\Delta}^{p}} \equiv 0$, and suppose that $\pi_{p}(\iota_X)([\sigma])=0$. Then there exists a continuous map $\theta:\Delta^{p+1} \longrightarrow X$ such that $\theta|_{\Delta^{p}_{(0)}} = \sigma$ and $\theta|_{\Lambda^{p+1}_{0}}=0$ (see \cite[Section 4.1]{origin} for the notation). Thus, from $({\rm B}_{p})$, we see that $[\sigma] = 0$ in $\pi_{p}(\sdcal X)$.

	Let $\tau$ be a continuous map from $\Delta^{p+1}$ to $X$ with $\tau|_{\dot{\Delta}^{p+1}}=0$. Then, we see from $({\rm B}_{p})$ that there exists a smooth map $\tau': \Delta^{p+1} \longrightarrow X$ such that $\tau'\ |_{\dot{\Delta}^{p+1}} = 0$ and $\pi_{p}(\iota_{X})([\tau']) = [\tau]$.
\end{proof}

\section{Applications to $C^{\infty}$-manifolds}
In Section 11.1, we prove the main theorems on $C^\infty$-manifolds (Theorems \ref{mapsmoothing}-\ref{DK}) using Theorems \ref{dmapsmoothing}-\ref{dDK}. In Sections 11.2 and 11.3, we study (semi)classicality condition and hereditary $C^\infty$-paracompactness, respectively. In Section 11.4, we apply the results obtained in Sections 11.2-11.3 to many important $C^\infty$-manifolds, showing that most of the $C^\infty$-manifolds introduced by Kriegl, Michor, and their coauthors are hereditarily $C^{\infty}$-paracompact and semiclassical.\par
In Sections 11.2-11.4, we use  some notions introduced in Appendix B and the results obtained in Appendix C.

\subsection{Proofs of Theorems \ref{mapsmoothing}-\ref{DK}}
In this subsection, we give proofs of Theorems \ref{mapsmoothing}-\ref{DK}. 

From Theorems \ref{hcofibrancy} and \ref{V}, we first derive the following theorems, which form the foundation of applications of smooth homotopy theory of diffeological spaces to $C^{\infty}$-manifolds.
\begin{thm}\label{mfdhcofibrancy}
 If a $C^{\infty}$-manifold $M$ is hereditarily $C^{\infty}$-paracompact and semiclassical, then $M$ is in $\wcal_{\dcal}$.
\end{thm}
\begin{proof}
	We prove the result in two steps.\\
	{\sl Step 1}. Let $V$ be a subset of a convenient vector space of $E$ which is open with respect to the locally convex topology. Then, $V$ has a convex open cover. Thus, we see from Theorem \ref{hcofibrancy} that if $V$ is $C^{\infty}$-paracompact, then $V$ is in $\wcal_{\dcal}$.\\
	{\sl Step 2}. We show that $M$ is in $\mathcal{W}_\mathcal{D}$. Let $\{ (U_{\alpha}, u_{\alpha}) \}_{\alpha \in A}$ be a semiclassical atlas (i.e., an atlas as in the definition of a semiclassical $C^{\infty}$-manifold) of $M$ and $\sigma$ a nonempty finite subset of $A$. Since $u_{\alpha}: U_{\sigma} \longrightarrow u_{\alpha} (U_{\sigma})$ is a diffeomorphism for any $\alpha \in \sigma$ and $u_{\alpha}(U_{\sigma})$ is open with respect to the locally convex topology of $E_{\alpha}$, $U_{\sigma}$ is in $\wcal_{\dcal}$ (Step 1) and hence $M$ is in $\mathcal{W}_\mathcal{D}$ (see Theorem \ref{hcofibrancy}).
\end{proof}
\begin{thm}\label{mfdV}
	Every $C^{\infty}$-manifold is in $\vcal_{\dcal}$.
\end{thm}
\begin{proof} Since every $C^{\infty}$-manifold is locally contractible as a diffeological space (\cite[Lemma 4.17]{KM}), the result follows immediately from Theorem \ref{V}.
\end{proof}
We have the following corollary; compare Corollary \ref{mfdst} with Examples \ref{Hawaiian} and \ref{quotient}.
\begin{cor}\label{mfdst}
	Let $M$ be a $C^{\infty}$-manifold and $x$ a point of $M$. Then, the smooth homotopy groups $\pi^{\dcal}_{\ast}(M,x)$ and the smooth homology groups $H_{\ast}(M)$ are naturally isomorphic to the homotopy groups $\pi_{\ast}(\widetilde{M},x)$ and the homology groups $H_{\ast}(\widetilde{M})$ of the underlying topological space $\widetilde{M}$ respectively.
\end{cor}
\begin{proof}
	By Theorem \ref{mfdV} and Corollary \ref{W}, the natural inclusion
	\[
		S^\dcal M \longhookrightarrow S \widetilde{M}
	\]
	is a weak equivalence in $\scal$, which induces isomorphisms
 	\begin{align*}
		H_\ast(M) \underset{\cong}{\longrightarrow} H_\ast(\widetilde{M}),\\
		\pi^\dcal_\ast (M,x) \underset{\cong}{\longrightarrow} \pi_\ast(\widetilde{M},x). &\qedhere
 	\end{align*}
\end{proof}
Theorems \ref{mapsmoothing}-\ref{DK} are now immediate consequences of Theorems \ref{dmapsmoothing}-\ref{dDK}. Recall from Remark \ref{mfdenrich} that $C^\infty$ and $C^\infty/M$ are $\scal$-full subcategories of $\dcal$ and $\dcal / M$ respectively and note that $\mathsf{P}C^\infty G/M$ is isomorphic to $\mathsf{P}\dcal G/ M$ as $\scal$-categories (see Remark \ref{4properties}(3)).
\begin{proof}[Proof of Theorem \ref{mapsmoothing}] The result follows from Theorem \ref{dmapsmoothing} by using Theorems \ref{mfdhcofibrancy}-\ref{mfdV}.
\end{proof}
\begin{proof}[Proof of Theorem \ref{sectionsmoothing}]
	The result follows from Theorem \ref{dsectionsmoothing}, by using Theorems \ref{mfdhcofibrancy}-\ref{mfdV}.
\end{proof}
\begin{proof}[Proof of Theorem \ref{DK}] The result follows from Theorem \ref{dDK} and Remark \ref{bdleparacpt} by using Theorems \ref{mfdhcofibrancy}-\ref{mfdV}.
\end{proof}

We end this subsection by establishing a variant of Theorem \ref{mfdhcofibrancy}. Recall that a topological space $X$ is called {\sl hereditarily Lindel\"{o}f} if any open set of $X$ is Lindel\"{o}f. Recall also the subclass $\wcal_{\dcal\, 0}$ of $\wcal_{\dcal}$ from Remark \ref{refinement}.
\begin{prop}\label{countable}
	For a semiclassical $C^\infty$-manifold $M$, consider the following conditions:
	\begin{itemize}
		\item[{\rm (i)}] $M$ is hereditarily Lindel\"of and $C^\infty$-regular.
		\item[{\rm (ii)}] $M$ is Lindel\"of and hereditarily $C^\infty$-paracompact.
		\item[{\rm (iii)}] $M$ is in $\wcal_{\dcal\,0}$.
	\end{itemize}
	Then, the implications $(\rm i) \Longrightarrow (\rm ii) \Longrightarrow (\rm iii)$ hold.
\begin{proof}
	The result follows from \cite[Theorem 16.10]{KM} (or Proposition \ref{Lindelof}), Theorem \ref{mfdhcofibrancy}, and Remark \ref{refinement}(2).
\end{proof}
\end{prop}

\subsection{Classical atlases}
In this subsection, we introduce two forms semiclassical $C^{\infty}$-manifolds appear.
\par\indent
A $C^{\infty}$-manifold $M$ is called {\sl classical} if $M$ admits a semiclassical atlas $\{(U_{\alpha}, u_{\alpha}) \}_{\alpha \in A}$ such that the chart changing $u_{\alpha\beta}: u_{\beta} (U_{\alpha} \cap U_{\beta}) \longrightarrow u_{\alpha}(U_{\alpha} \cap U_{\beta})$ is a homeomorphism with respect to the locally convex topologies of $E_{\alpha}$ and $E_{\beta}$ for any $\alpha$, $\beta \in A$. If $M$ has a classical atlas $\{ (U_{\alpha}, u_{\alpha}) \}_{\alpha \in A}$, then the classical topology on $M$ is defined to be the final topology for the inclusions $\{ u_{\alpha} (U_{\alpha})\cong U_\alpha \longhookrightarrow M\}_{\alpha \in A}$, where $u_{\alpha}(U_{\alpha})$ is endowed with the locally convex topology; the set $M$ endowed with the classical topology is denoted by $M_{cl}$. The inclusion $u_{\alpha}(U_{\alpha})\cong U_\alpha \longhookrightarrow M_{cl}$ is an open topological embedding and $id: \widetilde{M} \longrightarrow M_{cl}$ is continuous.
\begin{lem}\label{FrechetSilva}
	If a $C^\infty$-manifold $M$ is modeled on Fr\'{e}chet spaces or Silva spaces, then $M$ is classical and $\widetilde{M}=M_{cl}$ holds.
\end{lem}
\begin{proof}
	The result follows from \cite[Theorem 4.11(1)-(2)]{KM}; note that the strong dual of a Fr\'{e}chet-Schwartz space is just a Silva space (\cite[Corollary 12.5.9 and the comment on it in p. 270]{Jar}).
\end{proof}
From the viewpoint of application of Lemma \ref{FrechetSilva}, the following lemma is important. Recall the definition of a submanifold from \cite[27.11]{KM}.
\begin{lem}\label{FSsubmfd} If a $C^{\infty}$-manifold $M$ is modeled on Fr\'{e}chet spaces (resp. Silva spaces), then every submanifold of $M$ is also modeled on Fr\'{e}chet spaces (resp. Silva spaces).
\end{lem}
\begin{proof} The result follows from the fact that a closed linear subspace $F$ of a Fr\'{e}chet space (resp. Silva space) $E$ is a Fr\'{e}chet space (resp. Silva space) (\cite[p. 49]{Sch}, \cite[Proposition 1]{Yo}).
\end{proof}
Next, we see that a $C^{\infty}$-manifold in Keller's $C^{\infty}_{c}$-theory (or a $C^\infty_c$-manifold in short) defines a classical $C^{\infty}$-manifold.
See Appendix B for the category $C^{\infty}_{c\ conv}$ and the faithful functor $J : C^{\infty}_{c\ conv} \longrightarrow C^{\infty}$.
\begin{lem}\label{Kclassical}
	If a $C^{\infty}$-manifold $M$ comes from a $C^{\infty}_{c}$-manifold (i.e., $M \cong JM'$ for some $M' \in C^{\infty}_{c\ conv}$), then $M$ is classical.
\end{lem}
\begin{proof}
	Obvious.
\end{proof}
Lemmas \ref{FrechetSilva} and \ref{Kclassical} apply to many important $C^{\infty}$-manifolds (see Section 11.4).
\begin{rem}\label{other}
	\begin{itemize}
	\item[{\rm (1)}] Lemma \ref{FrechetSilva} can be regarded as a specialization of Lemma \ref{Kclassical} (see Remark \ref{full}).
	\item[{\rm (2)}] Lemma \ref{Kclassical} remains true even if we replace Keller's $C^{\infty}_{c}$-theory with another classical infinite dimensional calculus satisfying the three conditions in Remark \ref{classical}.
	\end{itemize}
\end{rem}
\subsection{Hereditary $C^{\infty}$-paracompactness}
The condition of being hereditarily $C^\infty$-paracompact and the condition of being hereditarily Lindel\"{o}f and $C^\infty$-regular are important from a smooth homotopical viewpoint (see Theorem \ref{mfdhcofibrancy} and Proposition \ref{countable}); recall also that the implication
\[
\text{hereditarily Lindel\"{o}f and }C^\infty\text{-regular }\Rightarrow\text{ herditarily }C^\infty\text{-paracompact}
\]
holds (see Proposition \ref{countable}).\par
In this subsection, we study these two conditions, recalling the basics of $C^{\infty}$-paracompactness from \cite[Section 16]{KM} and Appendix C. More concretely, for each of these two conditions, we give a necessary and sufficient condition of the form
\vspace*{2mm}\\
(11.1) topological condition on $M$ + condition on the model vector spaces of $M$
\vspace*{2mm}\\
(Propositions \ref{hLindelof} and \ref{model}) along with its application (Corollaries \ref{nuclear} and \ref{metricmfd}).
\begin{rem}\label{characterization}
	Given a smooth topological condition $\pcal$ on a $C^\infty$-manifold $M$, it is of practical importance to give a necessary and sufficient condition of the form as in (11.1); such a characterization enables us to directly apply the results of topology and functional analysis. Proposition \ref{smregular} also gives such a characterization.
\end{rem}
Some results are discussed in a rather general context using the notation of Section 5; $\ccal$ denotes one of the categories $C^\infty$, $\dcal$, $\czero$ and $\tcal$, and $U:\ccal \longrightarrow \tcal$ denotes the underlying topological space functor for $\ccal$.\par

We begin by studying the heredity of the relevant properties on objects of $\ccal$; see Remark \ref{regular} for the heredity of $\ccal$-regularity.
\begin{lem}\label{hered}
	Hereditary $\ccal$-paracompactness and hereditary Lindel\"{o}f property pass from an object of $\ccal$ to every $\ccal$-subspace whose underlying topology is the subspace topology.
	\begin{proof}
		We prove the result for hereditary $\ccal$-paracompactness; the proof for hereditary Lindel\"{o}f property is similar.\par
		Let $X$ be a hereditarily $\ccal$-paracomapct object of $\ccal$ and $A$ a $\ccal$-subspace of $X$ such that $UA$ is a topological subspace of $UX$. Suppose that an open set $V$ of $A$ and an open cover $\{V_{i} \}$ of $V$ are given. Then, we can choose a family $\{U_{i} \}$ of open sets of $X$ such that $V_{i} = U_{i} \cap A$ and use the $\ccal$-paracompactness of $U: = \cup\ U_{i}$ to construct a $\ccal$-partition of unity on $V$ subordinate to $\{ V_{i} \}$.
	\end{proof}
\end{lem}
From Lemma \ref{hered}, we obtain the following result.
\begin{cor}\label{submfd}
	If a $C^\infty$-manifold $M$ is hereditarily $C^\infty$-paracompact (resp. hereditarily Lindel\"{o}f and $C^\infty$-regular), then every submanifold of $M$ is also hereditarily $C^\infty$-paracompact (resp. hereditarily Lindel\"{o}f and $C^\infty$-regular).
\begin{proof}
	Since the underlying topology of a submanifold is the subspace topology of the ambient $C^{\infty}$-manifold (\cite[27.11]{KM}), the result follows from Lemma \ref{hered} and Remark \ref{regular}.
\end{proof}
\end{cor}
\begin{rem}\label{heredity}
	A hereditary topological property is defined to be a topological property that carries over from a topological space to its subspaces. However, we say that a property of objects of $\ccal$ is {\sl hereditary} if it carries over from an object of $\ccal$ to every open $\ccal$-subspace, since the underlying topology of a $\ccal$-subspace need not be the subspace topology in the case of $\ccal \neq \tcal$ (consider the $\ccal$-subspace $\{0\}\cup \{\frac{1}{n}\,|\, n\in \mathbb{N}\}$ of $\mathbb{R}$ for $\ccal=\dcal,\czero$). Hereditary $\ccal$-paracompactness, hereditary Lindel\"{o}f property, and $\ccal$-regularity are obviously hereditary.\par
	Many important hereditary properties in $\ccal$, such as these three properties, carry over from an object of $\ccal$ to its $\ccal$-subspace whose underlying topology is the subspace topology (see Lemma \ref{hered} and its proof).
\end{rem}
Next, we give a necessary and sufficient condition for a $C^\infty$-manifold to be hereditarily Lindel\"{o}f and $C^\infty$-regular.
\begin{prop}\label{hLindelof}
	For a $C^\infty$-manifold $M$, the following are equivalent:
	\begin{itemize}
		\item[$(\rm i)$] $M$ is hereditarily Lindel\"{o}f and $C^{\infty}$-regular.
		\item[$(\rm ii)$] $M$ is Lindel\"of and regular as a topological space and is modeled on hereditarily Lindel\"of, $C^\infty$-regular convenient vector spaces.
	\end{itemize}
	\begin{proof}
		The result follows from Proposition \ref{smregular}.
	\end{proof}
\end{prop}
\begin{cor}\label{nuclear}
	Suppose that a $C^\infty$-manifold $M$ is Lindel\"of and regular as a topological space and that every model vector space of $M$ is a separable Hilbert space, a nuclear Fr\'{e}chet space, a nuclear Silva space, or a strict inductive limit of a sequence of nuclear Fr\'{e}chet spaces. Then, every submanifold of $M$ is hereditarily Lindel\"of and $C^\infty$-regular, and hence hereditarily $C^\infty$-paracompact.
	\begin{proof}
		Note that for a metric space, separability, second countablity, and Lindel\"{o}f property are equivalent (\cite[Theorem 16.11]{Willard}). Then, we see that any $c^{\infty}$-open set $V$ of a model vector space $E$ of $M$ is Lindel\"{o}f (see \cite[the proof of Theorem 16.10]{KM}), and hence that $E$ is hereditarily Lindel\"of. Thus, the result follows from \cite[Corollary 16.16 and Theorem 16.10]{KM}, Proposition \ref{hLindelof}, and Corollary \ref{submfd}.
	\end{proof}
\end{cor}
Next, we give a necessary and sufficient condition for a $C^\infty$-manifold to be hereditarily $C^\infty$-paracompact.
\begin{prop}\label{model}
	For a $C^\infty$-manifold $M$, the following conditions are equivalent:
	\begin{itemize}
		\item[{\rm (i)}] $M$ is hereditarily $C^\infty$-paracompact.
		\item[{\rm (ii)}] $M$ is paracompact Hausdorff as a topological space and is modeled on hereditarily $C^\infty$-paracompact convenient vector spaces.
	\end{itemize}
\end{prop}
	For the proof, we need the following lemma. An object $X$ of $\ccal$ is called {\sl locally hereditarily $\ccal$-paracompact} if any $x\in X$ has a hereditarily $\ccal$-paracompact open neighborhood.
\begin{lem}\label{heredCparacpt}
	For $X\in \ccal$ with $UX$ $T_1$-space, the following are equivalent:
	\begin{itemize}
		\item[{\rm (i)}] $X$ is hereditarily $\ccal$-paracompact.
		\item[{\rm (ii)}] $UX$ is paracompact Hausdorff and $X$ is locally hereditarily $\ccal$-paracompact.
	\end{itemize}
	\begin{proof}
		${\rm (i)} \Longrightarrow {\rm (ii)}$ Obvious (see the end of Section 5.1).\\
		${\rm (ii)} \Longrightarrow {\rm (i)}$ Assume given an open set $V$ of $X$ and an open cover $\{V_i\}_{i\in I}$ of $V$. Then, we must construct a $\ccal$-partition of unity on $V$ subordinate to $\{V_i\}_{i\in I}$.\par
		In this proof, for a subset $S$ of $X$, $\overline{S}$ and $S\text{\textasciicircum}$ denote the closure of $S$ in $UX$ and the intersection of $S$ with $V$, respectively.\par
		First, we choose a locally finite open cover $\{U_\alpha\}_{\alpha\in A}$ of $X$ such that each $U_\alpha$ is hereditarily $\ccal$-paracompact. Noticing that $UX$ is paracompact Hausdorff (and hence normal), we choose for each $\alpha \in A$, open sets $U'_\alpha$ and $U''_\alpha$ such that
		\begin{itemize}
			\item $U'_\alpha \subset \overline{U'_\alpha} \subset U''_\alpha \subset \overline{U''_\alpha} \subset U_\alpha$,
			\item $\cup_\alpha U'_\alpha = X$.
		\end{itemize}\par
		Next, consider the open cover $\{U_\alpha\text{\textasciicircum}\}_{\alpha\in A}$ of $V$. Since $UX$ is normal, we can choose for each $x\in U_\alpha\text{\textasciicircum}$, an open neighborhood $W_\alpha^x$ of $x$ satisfying the following conditions:
		\begin{itemize}
			\item $\overline{W_\alpha^x} \subset U_\alpha \cap V_i$ for some $i$.
			\item \begin{tabular}{cc}
				$
				\left \{
				\begin{array}{ll}
				x\in \overline{U'_\alpha} \Longrightarrow W_\alpha^x \subset U''_\alpha, \\
				x \notin \overline{U'_\alpha} \Longrightarrow W_\alpha^x \subset U_\alpha \backslash \overline{U'_\alpha}
				\end{array}
				\right.
				$
				\end{tabular}
		\end{itemize}
		Since $U_\alpha\text{\textasciicircum}$ is $\ccal$-paracompact, we choose a $\ccal$-partition of unity $\{\varphi_\alpha^x \}$ on $U_\alpha\text{\textasciicircum}$ subordinate to the open cover $\{W_\alpha^x\}$; note that $\varphi^x_\alpha$ extends to a $\ccal$-function on $V$ by setting $\varphi^x_\alpha=0$ on $V-\overline{W_\alpha^x}$. Further, we define the (set-theoretic) map
		\[
			r:U_\alpha\text{\textasciicircum} \longrightarrow I
		\]
		by choosing for each $x\in U_\alpha\text{\textasciicircum}$, an element $i$ of $I$ with $\overline{W_\alpha^x} \subset U_\alpha \cap V_i$.\par
		Let us see that the family
		\[
			\Phi = \{{\rm supp}\,\varphi^x_\alpha \,|\, \alpha\in A, \,x\in \overline{U'_\alpha}\text{\textasciicircum}\}
		\]
		of closed sets in $V$ is locally finite. For any $z \in V$, we can choose an open neighborhood $U$ such that
		\[
			\sharp \{\alpha\in A \,|\, U\cap U_\alpha\text{\textasciicircum} \neq \emptyset\} < \infty.
		\]
		For $\alpha\in A$ with $U\cap U_\alpha\text{\textasciicircum} \neq \emptyset$, we choose a smaller open neighborhood $U(\alpha)$ of $z$ satisfying the following conditions:
		\begin{itemize}
			\item $z\notin U_\alpha \Longrightarrow U(\alpha)=U\backslash \overline{U''_\alpha}$.
			\item $z\in U_\alpha \Longrightarrow U(\alpha)$ meets the supports of only finitely many $\varphi_\alpha^x$.
		\end{itemize}
		Then, the intersection $\underset{U\cap U_\alpha\text{\textasciicircum}\neq \emptyset}{\cap} U(\alpha)$ is an open neighborhood of $z$ which meets only finitely many members of $\Phi$.\par
		Thus, we see that
		\[
			\varphi_i := \underset{\alpha\in A,\, x\in \overline{U'_\alpha}\text{\textasciicircum}_i}{\textstyle\sum} \varphi_\alpha^x
		\]
		is a $\ccal$-function on $V$ with ${\rm supp}\,\varphi_i \subset V_i$, where $\overline{U'_\alpha}\text{\textasciicircum}_i$ is the inverse image of $i\in I$ under the (set-theoretic) map
		\[
			\overline{U'_\alpha}\text{\textasciicircum} \longhookrightarrow U_\alpha\text{\textasciicircum} \overset{r}{\longrightarrow} I.
		\]
		Further, we see from the definition that
		\[
			\underset{i\in I}{\textstyle\sum} \varphi_i = \underset{\alpha \in A,\, x\in \overline{U'_\alpha}\text{\textasciicircum}}{\textstyle\sum} \varphi_\alpha^x > 0\ \ \  \text{  on }V.
		\]
		Therefore, we obtain the desired $\ccal$-partition of unity $\{\phi_i\}_{i\in I}$ by setting
		\begin{align*}
			\phi_i = \varphi_i / \underset{i\in I}{\textstyle\sum} \varphi_i. &\qedhere
		\end{align*}	
	\end{proof}
\end{lem}

\begin{proof}[Proof of Proposition \ref{model}.]
		${\rm (i)}\Longrightarrow {\rm(ii)}$ The implication follows from Remark \ref{regular} and Proposition \ref{smregular}.\\
		${\rm (ii)}\Longrightarrow {\rm(i)}$ The implication follows from Lemma \ref{heredCparacpt}.
\end{proof}

\begin{cor}\label{metricmfd}
	Suppose that a $C^\infty$-manifold $M$ is paracompact Hausdorff as a topological space and that every model vector space of $M$ is a $C^\infty$-paracompact Fr\'{e}chet space (e.g., a Hilbert space or a nuclear Fr\'{e}chet space), a nuclear Silva space, or a strict inductive limit of a sequence of nuclear Fr\'{e}chet spaces. Then, every submanifold of $M$ is hereditarily $C^\infty$-paracompact.
\end{cor}
For the proof, we need the following lemma.
\begin{lem}\label{metricdparacpt}
	Let $X$ be an object of $\ccal$ with $UX$ metrizable. Then, $X$ is hereditarily $\ccal$-paracompact if and only if $X$ is $\ccal$-paracompact.
\begin{proof}
	The result follows from \cite[Theorem 16.15]{KM} (see the comment after Remark \ref{regular}).
\end{proof}
\end{lem}
Note that Lemma \ref{metricdparacpt} has little meaning in the case of $\ccal=\czero$ and $\tcal$; in fact, every object $X$ with $UX$ metrizable is hereditarily $\ccal$-paracompact for $\ccal=\ccal^0,\tcal$.\par
\begin{proof}[Proof of Corolary \ref{metricmfd}]
	Note that a Fr\'{e}chet space is hereditarily $C^\infty$-paracompact if and only if it is $C^\infty$-paracompact (see Lemma \ref{metricdparacpt} and \cite[Theorem 4.11(1)]{KM}). Then, we see that a Hilbert space and a nuclear Fr\'echet space are hereditarily $C^\infty$-paracompact (see \cite[Corollary 16.16 and Theorem 16.10]{KM}). Also recall that model vector spaces as in Corollary \ref{nuclear} are hereditarily $C^\infty$-paracompact. Then, the result follows from Proposition \ref{model} and Corollary \ref{submfd}.
\end{proof}


We end this subsection with the following remark, in which we precisely state the relation between Theorem \ref{mfdhcofibrancy} and the results of Palais and Heisey and the relation between Proposition \ref{countable} and the result of Milnor (see Remark \ref{MPH}(2)).
\begin{rem}\label{topmfd}
In this remark, a topological manifold is defined to be a Hausdorff topological space locally homeomorphic to an open set of a locally convex vector space; this definition is the usual one but not that in \cite[27.1]{KM}.
\begin{itemize}
	\item[{\rm (1)}] We explain that Theorem \ref{mfdhcofibrancy} can be regarded as a smooth refinement of a result generalizing those of Palais \cite[Theorem 14]{Palais} and Heisey \cite[Theorem II.10]{Heisey}, which give sufficient conditions for an infinite dimensional topological manifold to have the homotopy type of a $CW$-complex.\par
	We can prove that for a topological manifold $M$, the following implications hold:\vspace{2mm}
	\begin{quote}
		\begin{itemize}
			\item[] $M$ is paracompact and is modeled on hereditarily paracompact locally convex vector spaces
			\item[$\Rightarrow$] $M$ is hereditarily paracompact
			\item[$\Rightarrow$] $M$ has the homoropy type of a $CW$-complex.
		\end{itemize}
	\end{quote}\vspace{2mm}
	In fact, the argument in the proof of Theorem \ref{mfdhcofibrancy} is applicable to the topological context by using \cite[Theorem 4]{tom} instead of Theorem \ref{hcofibrancy} (see Lemma \ref{heredCparacpt} for the first implication). This result is a topological version of Theorem \ref{mfdhcofibrancy}, which generalizes the results of Palais and Heisey (see \cite[Proposition III.1]{Heisey}).
	\item[{\rm (2)}] We explain that Proposition \ref{countable}, a variant of Theorem \ref{mfdhcofibrancy}, can be regarded as a smooth refinement of a result generalizing that of Milnor \cite[Corollary 1]{Mi}, which states that every separable metrizable topological manifold has the homotopy type of a countable $CW$-complex.\par
	Recall that a Lindel\"{o}f, regular topological space is paracompact (\cite[Corollary 20.8]{Willard}). Then, we see that for a topological manifold $M$, the following implications hold:\vspace{2mm}
	\begin{quote}
		$\text{\;\;\;\;\;}M$ is hereditarily Lindel\"{o}f and regular\\
		$\Rightarrow M$ is Lindel\"{o}f and hereditarily paracompact\\
		$\Rightarrow M$ has the homotopy type of a countable $CW$-complex
	\end{quote}\vspace{2mm}
	(see Part 1 and \cite[Proposition 2]{Mi}). This is a topological analogue of Proposition \ref{countable}, which generalizes the result of Milnor. 
\end{itemize}
\end{rem}
\subsection{Hereditarily $C^{\infty}$-paracompact, semiclassical $C^{\infty}$-manifolds}
In this subsection, we show that many important $C^{\infty}$-manifolds are hereditarily $C^{\infty}$-paracompact and semiclassical, and hence are in $\wcal_{\dcal}$. In fact, we show that many of the $C^{\infty}$-manifolds studied in \cite[Chapter IX and Section 47]{KM} are hereditarily Lindel\"{o}f, $C^{\infty}$-regular, and classical, and hence are in $\wcal_{\dcal\, 0}$; recall that for a $C^{\infty}$-manifold $M$, the implication
\[
\text{$M$ is classical $\Longrightarrow$ $M$ is semiclassical}
\]
holds (Section 11.2) and that for a semiclassical $C^{\infty}$-manifold $M$, the implications
\[
\begin{tikzcd}
\text{$M$ is hereditarily Lindel\"{o}f and $C^{\infty}$-regular} \arrow[r, Rightarrow] \arrow[d,Rightarrow]   & \text{$M$ is in $\wcal_{\dcal\, 0}$} \arrow[d, Rightarrow]\\
\text{$M$ is hereditarily $C^{\infty}$-paracompact} \arrow[r, Rightarrow] & {\text{$M$ is in $\wcal_{\dcal}$}}
\end{tikzcd}
\]
hold (Theorem \ref{mfdhcofibrancy} and Proposition \ref{countable}).
\begin{thm}\label{goodmfds} Let $M$ be a $C^{\infty}$-manifold.
	\begin{itemize}
		\item[$(1)$] Suppose that $M$ is paracompact and modeled on Hilbert spaces, nuclear Fr\'{e}chet spaces, or nuclear Silva spaces. Then, every submanifold of $M$ is hereditarily $C^{\infty}$-paracompact and classical, and hence is in $\wcal_{\dcal}$.
		\item[$(2)$] Suppose that $M$ is Lindel\"{o}f and regular as a topological space and is modeled on separable Hilbert spaces, nuclear Fr\'{e}chet spaces, or nuclear Silva spaces. Then, every submanifold of $M$ is hereditarily Lindel\"{o}f, $C^{\infty}$-regular, classical, and hence is in $\wcal_{\dcal\, 0}$.
		\item [$(3)$] Suppose that $M$ is Lindel\"{o}f and regular as a topological space and is semiclassical and modeled on strict inductive limits of sequences of nuclear Fr\'{e}chet spaces. Then, $M$ is hereditarily Lindel\"{o}f and $C^{\infty}$-regular, and hence is in $\wcal_{\dcal\, 0}$.
	\end{itemize}
\end{thm}
\begin{proof} (1) Note that a (separated) $C^\infty$-manifold modeled on Fr\'{e}chet spaces or Silva spaces is Hausdorff (see Remark \ref{full}). Then, the result follows from Corollary \ref{metricmfd}, Lemmas \ref{FrechetSilva}-\ref{FSsubmfd}, and Theorem \ref{mfdhcofibrancy}.\par
(2) The result follows from Corollary \ref{nuclear}, Lemmas \ref{FrechetSilva}-\ref{FSsubmfd}, and Proposition \ref{countable}.\par
(3) The result follows from Corollary \ref{nuclear} and Proposition \ref{countable}.
\end{proof}
\begin{exa}\label{Hilbert}
	Theorem \ref{goodmfds}(1) applies to the Grassmannian $Gr(H)$ of a polarized Hilbert space $H$ (\cite[Section 7.1]{PS}) and the Hilbert manifold models $L^{2}_{r}(S, M)$ for mapping spaces (\cite[p. 781]{Eells}); note that in the literature, infinite dimensional manifolds are often assumed to be paracompact (\cite[p. 765]{Eells}, \cite[p. 24]{PS}).
\end{exa}
Let us apply Theorem \ref{goodmfds} to various $C^{\infty}$-manifolds of mappings (see \cite[Chapter IX]{KM}).
\par\indent
Let $M$ and $N$ be second countable finite dimensional $C^{\infty}$-manifolds; we assume that $M$ and $N$ are connected and that $\dim \ M \leq \dim\ N$ whenever we consider the manifolds of embeddings. Then, the $C^{\infty}$-manifold $\cfra^{\infty}(M, N)$ of $C^{\infty}$-maps from $M$ to $N$ is defined (\cite[Theorem 42.1]{KM}) and the Lie group ${\rm Diff}(M)$ of diffeomorphisms of $M$ and the $C^{\infty}$-manifold ${\rm Emb}(M, N)$ of $C^{\infty}$-embeddings of $M$ into $N$ are defined as open submanifolds of $\cfra^\infty(M,M)$ and $\cfra^{\infty}(M, N)$ respectively (\cite[Theorems 43.1 and 44.1]{KM}).
\par\indent
Suppose further that $M$ and $N$ are $C^{\omega}$-manifolds with $M$ compact. Then, the $C^{\infty}$-manifold $C^{\omega}(M, N)$ of $C^{\omega}$-maps from $M$ to $N$ is defined (\cite[Theorem 42.6]{KM}), and the Lie group ${\rm Diff}^{\omega}(M)$ of $C^{\omega}$-diffeomorphisms of $M$ and the $C^{\infty}$-manifold ${\rm Emb}^{\omega}(M, N)$ of $C^{\omega}$-embeddings of $M$ into $N$ are defined as open submanifolds of $C^{\omega}(M, M)$ and $C^{\omega}(M, N)$ respectively (\cite[Theorems 43.4 and 44.3]{KM}).\\
\begin{cor}\label{goodmfds in ga}
	Let $M$ and $N$ be second countable finite dimensional $C^{\infty}$-manifolds.
	\begin{itemize}
		\item[{\rm (1)}] Suppose that $M$ is compact. Then, every submanifold of $\cfra^{\infty}(M, N)$ $($e.g. ${\rm Diff}(M)\ \text{and} \ {\rm Emb}(M, N))$ is hereditarily Lindel\"{o}f, $C^{\infty}$-regular, and classical, and hence is in $\wcal_{\dcal\, 0}$.
		\item[{\rm (2)}] Each connected component of the $C^{\infty}$-manifolds $\cfra^{\infty}(M, N)$, ${\rm Diff}(M)$, and ${\rm Emb}(M, N)$ is hereditarily Lindel\"{o}f, $C^{\infty}$-regular, and classical, and hence is in $\wcal_{\dcal\, 0}$.
		\item[{\rm (3)}] Suppose that $M$ and $N$ are $C^\omega$-manifolds with $M$ compact. Then, every submanifold of $C^{\omega}(M, N)$ $($e.g. ${\rm Diff}^{\omega}(M)$ and ${\rm Emb}^{\omega}(M, N))$ is hereditarily Lindel\"{o}f, $C^{\infty}$-regular, and classical, and hence is in $\wcal_{\dcal\, 0}$.
	\end{itemize}
\end{cor}

We see from \cite[Theorem 42.14]{KM} that $\cfra^\infty(M,N)$ coincides with the hom-object $C^\infty(M,N)\,(=\dcal(M,N))$ as a diffeological space if and only if $M$ is compact or $N$ is discrete. For Corollary \ref{goodmfds in ga}(2), note the obvious fact that a diffeological space $X$ is in $\wcal_{\dcal}$ if and only if each connected component of $X$ is in $\wcal_{\dcal}$.\par
The $C^\infty$-manifold ${\rm Emb}_{\rm prop}(M, N)$ of proper (equivalently closed) embeddings is defined as an open submanifold of ${\rm Emb}(M,N)$ (\cite[Theorem 44.1]{KM}). We can observe from \cite[Theorem 42.14]{KM} that ${\rm Emb}_{\rm prop}(M,N)$ is a disjoint union of connected components of ${\rm Emb}(M,N)$. Hence, Corollary \ref{goodmfds in ga}(2) also holds true for the $C^\infty$-manifold ${\rm Emb}_{\rm prop}(M,N)$.\par
For the proof of Corollary \ref{goodmfds in ga}, we need the following facts on the $C^{\infty}$-manifolds $\cfra^{\infty}(M, N)$ and $C^\omega(M,N)$; Part 1 is a corrected version of \cite[Proposition 42.3]{KM} (see Remark \ref{cpt noncpt}) and its proof complements that of \cite[Proposition 42.3]{KM} (see Remark \ref{correction}).
\begin{lem}\label{mfds of maps} Let $M$ and $N$ be second countable finite dimensional $C^{\infty}$-manifolds.
	\begin{itemize}
		\item[$(1)$] If $M$ is compact, then $\cfra^\infty(M,N)$ is Lindel\"of and regular and is modeled on nuclear Fr\'echet spaces. If $M$ is noncompact, then each connected component of $\cfra^\infty(M,N)$ is Lindel\"of and regular and is modeled on strict inductive limits of sequences of nuclear Fr\'echet spaces.
		\item[$(2)$] Suppose that $M$ and $N$ are $C^\omega$-manifolds with $M$ compact. Then, $C^\omega(M,N)$ is Lindel\"of and regular and is modeled on nuclear Silva spaces.
	\end{itemize}
\end{lem}
\begin{proof} 
	(1) {\sl Model vector spaces  } See \cite[Theorem 42.1 and the proof of Lemma 30.4]{KM}.
	\vspace{5pt}\\
	{\sl Lindel\"of property  }  The proof of \cite[Proposition 42.3]{KM} applies to each connected component of $\cfra^{\infty}(M, N)$, showing its Lindel\"{o}f property (see Remark \ref{cpt noncpt}).
	\par\indent
	 If $M$ is compact, then $\cfra^\infty(M,N)=C^\infty(M,N)$ holds in $\dcal$, and hence $\pi_{0}\ \cfra^{\infty}(M, N)$ is isomorphic to $[\widetilde{M}, \widetilde{N}]_{\czero}$ (see Theorem \ref{mapsmoothing}). Since $\widetilde{M}$ and $\widetilde{N}$ have the homotopy types of a finite complex and a countable complex respectively, $\pi_{0}\ \cfra^{\infty}(M, N)$ is countable. Thus, $\cfra^{\infty}(M, N)$ is Lindel\"{o}f.
	\vspace{5pt}\\
	{\sl Regularity  }  Throughout this proof, we use the notation of \cite[the proof of Theorem 42.1]{KM}.\par
	Note that $\cfra^\infty(M,N)$ is Hausdorff (\cite[Theorem 42.1]{KM}) and that its model vector spaces are $C^\infty$-regular (\cite[Theorem 16.10]{KM}). Also recall the charts $(U_f,u_f)\,(f\in \cfra^\infty(M,N))$ constructed in \cite[the proof of Theorem 42.1]{KM}. Then, we have only to show that for any $f\in \mathfrak{C}^\infty(M,N)$, there exist disjoint open sets $W_f$ and $V_f$ of $\mathfrak{C}^\infty(M,N)$ such that
	\[
	f \in W_f \text{ and } U_f^c \subset V_f
	\]
	(see Lemma \ref{regularity} and Remark \ref{modelreg}). Consider the open embedding
	\[
	TN \supset U \xhookrightarrow[]{(\pi_N, exp)} N\times N
	\]
	constructed in \cite[the proof of Theorem 42.1]{KM}. To simplify the argument below, we choose a new metric on the tangent bundle $TN$ such that the unit disk bundle of $TN$ is contained in $U$ and redefine $U$ as the open unit disk bundle of $TN$ (with respect to the new metric); explicitly,
	\begin{equation*}
	\begin{split}
	U=\{v\in TN \,|\, |v| < 1 \} \subset TN,\\
	V=(\pi_N, exp)(U)\subset N \times N.
	\end{split}
	\end{equation*}
	Recall that
	\[
	U_f = \{g\in \mathfrak{C}^\infty(M,N)\,|\, (f(x),g(x))\in V \text{ for all } x\in M, g\sim f \}.
	\]
	We define the subsets $W_f$ and $K_f$ of $\mathfrak{C}^\infty(M,N)$ by
	\begin{equation*}
	\begin{split}
	W_f = \{g\in \mathfrak{C}^\infty(M,N)\,|\,(f(x),g(x))\in V_{\frac{1}{3}} \text{ for all } x\in M, g\sim f \},\\
	K_f = \{g\in\mathfrak{C}^\infty(M,N)\,|\,(f(x), g(x))\in \overline{V_{\frac{2}{3}}} \text{ for all } x \in M, g\sim f \},
	\end{split}
	\end{equation*}
	where we set
	\begin{equation*}
	\begin{split}
	U_r = \{v\in TN \,|\, |v| < r \},\\
	V_r = (\pi_N, exp)(U_r)
	\end{split}
	\end{equation*}
	for $0 < r < 1$. We then have the obvious inclusion relations
	\[
	W_f \subset K_f \subset U_f.
	\]
	Note that $(f\times 1_N)^{-1}(V_r)$ is an open neighborhood of the graph $\Gamma_f$ in $M\times N$ and that the equivalence
	\[
		(f(x), g(x)) \in V_r \text{ for all } x\in M \Longleftrightarrow \Gamma_g \subset (f\times 1_N)^{-1} (V_r)
	\]
	holds. Recalling \cite[Remark 42.2]{KM}, we thus see that $W_f$ is open. Similarly, we also see that $K_f$ is closed. Thus, $W_f$ and $V_f := \mathfrak{C}^\infty(M,N)\backslash K_f$ are disjoint open sets such that
	\[
	f \in W_f \text{ and } U^c_f \subset V_f.
	\]
	\par
	(2) Choose a real analytic exponential map of $N$. Then, the atlas $\{ (U^{\omega}_{f}, u^{\omega}_{f}) \}_{f \in C^{\omega}(M, N)}$ of $C^{\omega}(M, N)$ is constructed similarly to the atlas $\{ (U_{f}, u_{f}) \}_{f \in C^{\infty}(M, N)}$ of $\cfra^{\infty}(M, N)$ (see the proofs of \cite[Theorem 8.2]{KMActa} and \cite[Theorems 42.1 and 42.6]{KM}). By Part 1 and \cite[the proof of Theorem 42.8]{KM}, we can choose a countable atlas $\{ (U_{f_{n}}, u_{f_{n}}) \}_{n \in \nbb}$ of $\cfra^{\infty}(M, N)$ such that each $f_{n}$ is in $C^{\omega}(M, N)$. Hence, $\{ (U^{\omega}_{f_{n}}, u^{\omega}_{f_{n}}) \}_{n \in \nbb}$ is an atlas of $C^{\omega}(M, N)$. Since the model vector spaces $C^{\omega}(M \leftarrow f^{\ast}TN)$ are nuclear Silva spaces (\cite[Lemma 30.6 and Theorem 11.4]{KM}), their open sets are Lindel\"of (see the proof of Corollary \ref{nuclear}), and hence $C^\omega(M,N)$ is Lindel\"of. Last, we show that $C^\omega(M,N)$ is regular. Since $C^\omega (M,N)$ is Hausdorff (\cite[Corollary 42.11]{KM}) and its model vector spaces are $C^\infty$-regular (\cite[Theorem 16.10]{KM}), we have only to show that for any $f\in C^\omega(M,N)$, there exists a coordinate open neighborhood satisfying condition (b) in Lemma \ref{regularity} (see Remark \ref{modelreg}). Recalling that the inclusion $C^\omega(M,N)\longhookrightarrow \cfra^\infty(M,N)$ is smooth (\cite[Corollary 42.11]{KM}), we can obtain the desired coordinate open neighborhood of $f$ in $C^\omega(M,N)$ by restricting that in $\cfra^\infty(M,N)$ constructed in Part 1.
\end{proof}
\begin{rem}\label{cpt noncpt}
	In \cite[Proposition 42.3]{KM}, it is stated that $\cfra^{\infty} (M, N)$ is Lindel\"{o}f. However, this is incorrect. In fact, suppose that $M$ is noncompact and that $\dim N \geq 1.$ Then, $\cfra^{\infty}(M, N)$ has uncountably many constant maps, which represent different connected components (see \cite[Theorem 42.14]{KM}). Hence, $\cfra^{\infty}(M, N)$ is not Lindel\"{o}f.
\end{rem}
\begin{proof}[Proof of Corollary \ref{goodmfds in ga}] (1) The result follows from Theorem \ref{goodmfds}(2) and Lemma \ref{mfds of maps}(1).\par
(2) By Part 1, we may assume that $M$ is noncompact. By \cite[Sections 10, 11, and 13]{Michor}, the $C^{\infty}$-manifolds in question are classical (see Lemma \ref{Kclassical}). Thus, we see from Theorem \ref{goodmfds}(3) and Lemma \ref{mfds of maps}(1) that the result holds for $\cfra^\infty(M,N)$, and hence that each connected component of the two open submanifolds is also Lindel\"of. Thus, the result also holds for the two open submanifolds.\par
(3) The result follows from Theorem \ref{goodmfds}(2) and Lemma \ref{mfds of maps}(2).
\end{proof}
Let us investigate important submanifolds of $\cfra^{\infty} (M, N)$ and $C^{\omega}(M, N)$ other than those mentioned in Corollary \ref{goodmfds in ga}.
\begin{exa}\label{moreex} Let $M$ and $N$ be second countable finite dimensional $C^{\infty}$-manifolds; $M$ and $N$ are assumed to be $C^\omega$-manifolds with $M$ compact, whenever we consider $C^\omega$-maps between them. Let $X$ be one of the following $C^{\infty}$-manifolds:
	\begin{itemize}
		\item The $C^{\infty}$-manifold $\cfra^{\infty}(q)$ of smooth sections of a given smooth map $q: N \longrightarrow M$ and the $C^{\infty}$-manifold $C^{\omega}(q)$ of analytic sections of a given $C^\omega$-map $q: N \longrightarrow M$ (see \cite[Thorem 42.20]{KM}).
		\item The Lie group ${\rm Diff}(M, \sigma)$ of symplectic diffeomorphisms of a symplectic manifold $(M, \sigma)$ and the Lie group ${\rm Diff}^{\omega}(M, \sigma)$ of $C^{\omega}$-symplectic diffeomorphisms of a $C^{\omega}$-symplectic manifold $(M, \sigma)$ (see \cite[Theorem 43.12]{KM}).
	\end{itemize}
If $M$ is compact, then every submanifold of $X$ is hereditarily Lindel\"{o}f, $C^{\infty}$-regular, and classical, and hence is in $\wcal_{\dcal\, 0}$ (see Corollary \ref{goodmfds in ga}(1) and (3)). Even if $M$ is noncompact, $X\,(=\cfra^\infty(q)$ or ${\rm Diff}(M,\sigma))$ is classical (see \cite[Sections 10 and 14]{Michor} and Lemma \ref{Kclassical}), and thus each connected component of $X$ is hereditarily Lindel\"of, $C^\infty$-regular, and classical, and hence is in $\wcal_{\dcal\,0}$ (see Corollary \ref{goodmfds in ga}(2) and Proposition \ref{countable}).
\par\indent
This result applies to the following $C^\infty$-manifolds that can be formulated as $C^\infty$-manifolds of smooth sections:
\begin{quote}
	The $C^{\infty}$-manifold $\bcal(M)$ of nondegenerate bilinear structures on $M$, the $C^{\infty}$-manifold 
	$\mcal^{q}(M)$ of pseudo Riemannian metrics on $M$ of signature $q$, the $C^{\infty}$-manifold $\Omega^{2}_{\rm nd}(M)$ of nondegenerate 2-forms on $M$, and the $C^{\infty}$-manifold $\hcal(M)$ of almost Hermitian structures on $M$ (see \cite[Section 45]{KM} and \cite[Section 1]{GM}).
\end{quote}
Hence, if $M$ is compact, the $C^{\infty}$-manifold ${\rm Symp}(M)$ of symplectic structures on $M$ is in $\wcal_{\dcal\,0}$. Even if $M$ is noncompact, each connected component of ${\rm Symp}(M)$ is in $\wcal_{\dcal\,0}$ (see \cite[45.21]{KM}).

\end{exa}
Next, we apply Theorem \ref{goodmfds} to the $C^{\infty}$-manifold $B(M, N)$ of submanifolds of $N$ of type $M$ (\cite[p. 474]{KM}) and the $C^{\infty}$-manifold $B^{\omega}(M, N)$ of real analytic submanifolds of $N$ of type $M$ (\cite[p. 477]{KM}).
\begin{cor}\label{goodmfds of submfds} Let $M$ and $N$ be a second countable, connected $C^{\infty}$-manifolds with $\dim M \leq \dim N < \infty$.
	\begin{itemize}
		\item[(1)] Suppose that $M$ is compact. Then, every submanifold of $B(M, N)$ is hereditarily Lindel\"{o}f, $C^{\infty}$-regular, and classical, and hence is in $\wcal_{\dcal\, 0}$.
		\item[(2)] Suppose that $M$ is noncompact. Then, each connected component of $B(M, N)$ is hereditarily Lindel\"{o}f, $C^{\infty}$-regular, and classical, and hence is in $\wcal_{\dcal\, 0}$.
		\item[(3)] Suppose that $M$ and $N$ are $C^{\omega}$-manifolds with $M$ compact. Then, every submanifold of $B^{\omega}(M, N)$ is hereditarily Lindel\"{o}f, $C^{\infty}$-regular, and classical, and hence is in $\wcal_{\dcal\, 0}$.
	\end{itemize}
\end{cor}
For the proof of Corollary \ref{goodmfds of submfds}, we need basic facts on $C^\infty$-manifolds $B(M,N)$ and $B^\omega(M,N)$ (Lemma \ref{BMN}). The following lemma is needed to prove Lemma \ref{BMN}. See \cite[41.5 and 41.10]{KM} for the $CO$-topology and the $WO^\infty$-topology.
\begin{lem}\label{noncpt}
	Let $M$ and $N$ be second countable finite dimensional $C^\infty$-manifolds.
	\begin{itemize}
		\item[{\rm (1)}] For any compact subset $K$ of $M$, there exists a compact submanifold $K_0$ with boundary such that $K\subset K_0$.
		\item[{\rm (2)}] If $f\in \cfra^\infty(M,N)$ is in the connected component of a diffeomorphism, then $f$ is surjective.
		\item[{\rm (3)}] The canonical maps
		\[
			\:\:\:\:\:\:\:\:\:\:\:\:\:\:\:\:\:\: \cfra^\infty(M,N) \xrightarrow{\ id \ } C^\infty(M,N)_{WO^\infty} \text{  and  } C^\infty(M,N)_{WO^\infty} \longhookrightarrow C^0(M,N)_{CO}
		\]
		are continuous, where $C^\infty(M,N)_{WO^\infty}$ (resp. $C^0(M,N)_{CO}$) is the set of smooth (resp. continuous) maps from $M$ to $N$ endowed with the $WO^\infty$-(resp. CO-)topology.
	\end{itemize}
	\begin{proof}
		Set $m=\dim M$.\par
		{\rm (1)} Choose a closed embedding $\iota:M\longhookrightarrow \mathbb{R}^{2m+1}$ (\cite[(7.10)]{BJ}) and define the smooth function $\varphi:\mathbb{R}^{2m+1}\longrightarrow \mathbb{R}$ by $\varphi(x_1,\cdots, x_{2m+1})=x^2_1 + \cdots + x^2_{2m+1}$. Then, the composite
		\[
			M \xhookrightarrow{\ \iota \ } \mathbb{R}^{2m+1} \xrightarrow{\ \varphi \ } \mathbb{R}
		\]
		is a proper smooth map. Set $K(r)=(\varphi \circ \iota)^{-1}[0,r]$ for $r\in \mathbb{R}$. Choose $r_0 > 0$ such that $(\varphi\circ \iota)(K) \subset [0, r_0]$ and $r_0$ is a regular value of $\varphi \circ \iota$ (\cite[(6.1)]{BJ}). Then, $K_0=K(r_0)$ is the desired compact submanifold with boundary.\par
		{\rm (2)} For simplicity, we assume that $M$ and $N$ are connected. Then, we can choose a diffeomorphism $g$ and a connected compact submanifold $K$ (with boundary) of $M$ such that $f$ is smoothly homotopic to $g$ relative to $M-\mathring{K}$ (see \cite[Lemma 42.5]{KM} and Part 1). Then, we see that
		\[
			f_\ast = g_\ast : H_m(M,M-\mathring{K};\zbb/2)\longrightarrow H_m(N,N-g(\mathring{K});\zbb/2),
		\]
		and hence that $f_\ast$ is an isomorphism. Since
		\[
			H_m(M,M-\mathring{K};\zbb/2) \cong H_m(K,\partial K;\zbb/2) \cong \zbb/2
		\]
		and
		\[
			H_m(M-\{x\}, M-\mathring{K};\zbb/2) \cong H_m(K-\{x\}, \partial K;\zbb/2)=0
		\]
		for $x\in \mathring{K}$ (see \cite[p. 304 and p. 307]{Dold2}), $f:M\longrightarrow N$ is surjective.\par
		{\rm (3)} See \cite[Section 41 and Remark 42.2]{KM}.
	\end{proof}
\end{lem}
\begin{lem}\label{BMN}
	Let $M$ and $N$ be a second countable, connected $C^{\infty}$-manifolds with $\dim M \leq \dim N < \infty$.
	\begin{itemize}
		\item[$(1)$] If $M$ is compact, then $B(M,N)$ is Lindel\"of and regular and is modeled on nuclear Fr\'echet spaces. If $M$ is noncompact, then each connected component of $B(M,N)$ is Lindel\"of and regular and is modeled on strict inductive limits of sequences of nuclear Fr\'echet spaces.
		\item[$(2)$] Suppose that $M$ and $N$ are $C^\omega$-manifolds with $M$ compact. Then, $B^\omega(M,N)$ is Lindel\"of and regular and is modeled on nuclear Silva spaces.
	\end{itemize}
\begin{proof}
	For the proof, recall the following fact: Since an arc-generated space $X$ is locally arcwise connected, the connected components of $X$ are just the arcwise connected components of $X$ and they are open sets.\par
	(1) {\sl Lindel\"of property and Model vector spaces     } Consider the principal ${\rm Diff}(M)$-bundle $\pi: {\rm Emb}(M, N) \longrightarrow B(M, N)$ (\cite[Theorem 44.1]{KM}). First, suppose that $M$ is compact. Then, ${\rm Emb}(M, N)$, and hence $B(M,N)$ is Lindel\"{o}f (see Corollary \ref{goodmfds in ga}(1)). We see from \cite[the proof of Theorem 44.1 and Section 30.1]{KM} that the model vector spaces of $B(M,N)$ are nuclear Fr\'echet spaces. Next, suppose that $M$ is noncompact. Then, for any $f \in {\rm Emb}(M, N)$, the restriction $\pi: {\rm Emb}(M, N)_{f} \longrightarrow B(M, N)_{\pi(f)}$ is a principal $G$-bundle for some open Lie subgroup $G$ of ${\rm Diff}(M)$, where ${\rm Emb}(M, N)_{f}$ and $B(M, N)_{\pi(f)}$ are the connected components of ${\rm Emb}(M, N)$ and $B(M, N)$ containing $f$ and $\pi(f)$ respectively. Since ${\rm Emb}(M, N)_{f}$ is Lindel\"{o}f (Corollary \ref{goodmfds in ga}(2)), $B(M, N)_{\pi(f)}$ is also Lindel\"{o}f. We see from \cite[the proof of Theorem 44.1 and Section 30.4]{KM} that the model vector spaces of $B(M, N)$ are strict inductive limits of sequences of nuclear Fr\'{e}chet spaces.\vspace{0.2cm}\\
	{\sl Regularity     } In the case of $\dim M =\dim N$, $B(M,N)$ is $0$-dimensional (see \cite[the proof of Theorem 44.1]{KM}), and hence discrete. Thus, we assume that $\dim M < \dim N$. Since $B(M,N)$ is Hausdorff (\cite[Theorem 44.1]{KM}) and its model vector spaces are $C^\infty$-regular (\cite[Theorem 16.10]{KM}), we have only to construct, for any $\widehat{i}=\pi(i)\in B(M,N)$, a coordinate open neighborhood satisfying condition (b) in Lemma \ref{regularity} (see also Remark \ref{modelreg}).\par
	In this proof, given a map $p:E\longrightarrow B$ and a subset $S$ of $B$, we write $E|_S$ for the inverse image of $S$ under $p$.\par
	We first recall the construction of charts of $B(M,N)$ from \cite[Section 13]{Michor}. Let us fix an embedding $i\in {\rm Emb}(M,N)$. Since $L=i(M)$ is a submanifold of $N$, there exist a tubular neighborhood $W_L$ of $L$; more precisely, we have a submersion $p_L:W_L \longrightarrow L$, a vector bundle $q_L:V_L\longrightarrow L$, and a diffeomorphism over $L$
	\[
		V_L \xrightarrow[\cong]{\ \ \phi_{L}\ \ } W_L.
	\]
	(We construct $p_L$, $q_L$, and $\phi_L$ by choosing a Riemannian metric on $N$.) Using these, a chart around $\widehat{i}\,(=\pi(i))$
	\[
		B(M,N) \supset \widehat{Q_i} \xrightarrow[\cong]{\ \ \ \ } C^\infty_c (M \longleftarrow i^* V_L)
	\]
	is constructed.\par
	Next, consider the space $J^1(i^* V_L)$ of $1$-jets of local sections of $i^* V_L$ and the canonical epimorphism $\pi^1_0:J^1(i^\ast V_L) \longrightarrow i^\ast V_L$ of vector bundles over $M$ (see \cite[Section 41]{KM} for the basic facts and notation on jets). We fix a metric on the jet bundle $J^1(i^* V_L)$. By identifying the orthogonal complement of ${\rm Ker}\,\pi^1_0$ with $i^\ast V_L$ via $\pi^1_0$, we endow $i^\ast V_L$ with the metric induced by that of $J^1(i^*V_L)$. Then, we have the inequality
	\begin{equation}
		|\sigma| \ge |\pi^1_0(\sigma)| \ \ \ \ \text{for $\sigma\in J^1(i^\ast V_L)$}.
		\tag{$11.1$}
	\end{equation}
	Given a finite rank vector bundle $E\longrightarrow X$ with metric, $D_r(E)$ (resp. $\mathring{D}_r (E)$) denotes the subbundle of $E$ whose fiber is the disk (resp. open disk) of radius $r$.\par
	We thus define the subsets $\wcal_i$ and $\kcal_i$ of $\widehat{Q_i}$ as follows using the chart described above:
	\[
		\wcal_i \cong \{s\in C^\infty_c (M \longleftarrow i^* V_L)\:|\: j^1 s(M) \subset \mathring{D}_{\frac{1}{3}}(J^1(i^* V_L))\}, 
	\]
	\[
		\kcal_i \cong \{s\in C^\infty_c (M \longleftarrow i^* V_L)\:|\: j^1 s(M) \subset D_{\frac{2}{3}}(J^1(i^* V_L))\}.
	\]
	Since the relations
	\[
		\widehat{i} = \pi(i) \in \wcal_i \subset \kcal_i \subset \widehat{Q_i}
	\]
	hold and $\wcal_i$ is open in $B(M,N)$ (\cite[41.13]{KM}), we have only to show that $\kcal_i$ is closed in $B(M,N)$.\par
	Noticing that $\pi:{\rm Emb}(M,N)\longrightarrow B(M,N)$ is a topological quotient map (see Lemma \ref{bdlequotient}), we show that $\pi^{-1}\kcal_i$ is closed in ${\rm Emb}(M,N)$. Since $\pi^{-1} \kcal_i$ is closed in $\pi^{-1} \widehat{Q_i}$, it suffices to show that every $j_0 \in (\pi^{-1}\widehat{Q_i})^c$ has a neighborhood disjoint from $\pi^{-1}\kcal_i$.\par
	We can easily see that $j_0 \in (\pi^{-1}\widehat{Q_i})^c$ satisfies one of the following conditions:
	\begin{itemize}
		\item[{\rm (i)}] $j_0(M) \not\subset W_L$.
		\item[{\rm (ii)}] $j_0(M) \subset W_L$ and $p_L\circ j_0:M\longrightarrow L$ is not a diffeomorphism.
	\end{itemize}\par
	Suppose that condition (i) is satisfied. We first deal with the case where $j_0(M)\not\subset \overline{D_\frac{2}{3}(W_L)}$, where $D_\frac{2}{3}(W_L)$ is the image of $D_\frac{2}{3}(i^\ast V_L)$ under the composite
	\[
		i^\ast V_L \underset{\cong}{\longrightarrow} V_L \xrightarrow[\cong]{\phi_L} W_L \longhookrightarrow N.
	\]
	In this case, we see from Lemma \ref{noncpt}(3) and (11.1) that the subset
	\[
		\{j \in {\rm Emb}(M,N)\:|\: j(M)\not\subset \overline{D_\frac{2}{3}(W_L)} \}
	\]
	is an open neighborhood of $j_0$ disjoint from $\pi^{-1}\kcal_i$. Next, we deal with the case where $j_0(M)\subset \overline{D_\frac{2}{3}(W_L)}$. Then, there exists $x_0\in M$ such that $j_0(x_0)\in \overline{D_\frac{2}{3}(W_L)} \backslash D_\frac{2}{3}(W_L)$; such a situation must be considered when $i$ is a nonclosed embedding (see Fig 11.1).
	\def\DefLeftWidth{3}
	\begin{center}
		\begin{tikzpicture}[thick, baseline=0pt]
\path [fill=gray!40] (-\DefLeftWidth-0.5-0.1,2) rectangle (\DefLeftWidth+0.6,-2);
\draw [black, thick] (-\DefLeftWidth-0.5-0.1,2) circle [radius=0.05];
\draw [thick, -](-\DefLeftWidth-0.5,2)--(\DefLeftWidth+0.5,2)  node [above]{$D_{\frac{2}{3}}(W_L)$};
\draw [black, thick] (\DefLeftWidth+0.5+0.1,2) circle [radius=0.05];
\draw [black, thick] (-\DefLeftWidth-0.5-0.1,0) circle [radius=0.05];
\draw [thick, -](-\DefLeftWidth-0.5,0)--(\DefLeftWidth+0.5,0)  node [below]{$L=i(M)$};
\draw [black, thick] (\DefLeftWidth+0.5+0.1,0) circle [radius=0.05];
\draw [black, thick] (-\DefLeftWidth-0.5-0.1,-2) circle [radius=0.05];
\draw [thick, -](-\DefLeftWidth-0.5,-2)--(\DefLeftWidth+0.5,-2);
\draw [black, thick] (\DefLeftWidth+0.5+0.1,-2) circle [radius=0.05];

\draw [dotted](-\DefLeftWidth-0.5-0.1,-1.9)--(-\DefLeftWidth-0.5-0.1,-0.1);
\draw [dotted](-\DefLeftWidth-0.5-0.1,0.1)--(-\DefLeftWidth-0.5-0.1,1.9);

\draw [dotted](\DefLeftWidth+0.5+0.1,-1.9)--(\DefLeftWidth+0.5+0.1,-0.1);
\draw [dotted](\DefLeftWidth+0.5+0.1,0.1)--(\DefLeftWidth+0.5+0.1,1.9);

\draw [thick, blue](-\DefLeftWidth-0.5,0)--(-1,0);
\draw [thick, blue] (-1,0) to [out= 30, in=-20] (-1, 0.2);
\draw [thick, blue] (-3.3, 0.7) to [out= -7, in=152] (-1,0.2);
\draw [thick, blue] (-3.3,1.3)  arc (90:270:0.3); 

\draw[thick, blue] (-3.3, 1.3) to [out=0,in=130] (1,0.2);
\draw [thick, blue](1.5,0)--(\DefLeftWidth+0.5+0.1, 0);
\draw[thick, blue] (1, 0.2) to [out=-60,in=0] (1.5,0);

\node at (1.5,0.5) [above] {$j_0(M)$};
		\end{tikzpicture}\\
		Fig 11.1
	\end{center}
	If $\widehat{j_0}$ is not in the connected component of $\widehat{i}$, then the connected component of $j_0$ is an open neighborhood of $j_0$ disjoint from $\pi^{-1}\kcal_i$. Thus, we assume that $\widehat{j_0}$ is in the connected component of $\widehat{i}$. Then, there exists $h\in {\rm Diff}(M)$ such that $\{x\in M\:|\: j_0(x)\neq i\circ h(x)\}$ is relatively compact in $M$ (see \cite[Lemma 42.5]{KM}). Set $K=\overline{\{x\in M \:|\: j_0(x)\neq i\circ h(x)\}}$ and choose compact submanifolds $K_0$ and $K_1$ (with boundary) of $M$ such that $K\subset K_0 \subset \mathring{K}_1$ (see the proof of Lemma \ref{noncpt}(1)). Further, we can choose a tubular neighborhood $W_{j_0(M)}$ of $j_0(M)$ and a submersion $p_{j_0(M)}:W_{j_0(M)}\longrightarrow j_0(M)$ such that $W_{j_0(M)}|_{j_0(M)-j_0(K_0)} \subset W_L|_{L-i\circ h(K_0)}$ and the diagram
	\[
	\begin{tikzcd}
		W_{j_0(M)}|_{j_0(M)-j_0(K_0)} \arrow[hook]{r} \arrow[swap]{d}{p_{j_0(M)}} & W_L|_{L-i\circ h(K_0)} \arrow{d}{p_L}\\
		j_0(M)-j_0(K_0) \arrow{r}{=} & L-i\circ h(K_0)
	\end{tikzcd}
	\]
	commutes.\par
	Define the open sets $\vcal_1$, $\vcal_2$, and $\vcal_3$ of ${\rm Emb}(M,N)$ by
	$$
		\vcal_1 = \{j\in {\rm Emb}(M,N)\:|\: j(x_0)\in N- D_\frac{2}{3}(W_L)|_{i\circ h(K_1)}\},
	$$
	$$
		\vcal_2 = \{j\in {\rm Emb}(M,N)\:|\: j(K_0)\subset W_{j_0(M)}|_{j_0(\mathring{K}_1)}\},
	$$
	$$
		\vcal_3 = \{j \in {\rm Emb}(M,N)\:|\: j(M) \subset W_{j_0(M)}\}
	$$
	(see Lemma \ref{noncpt}(3)) and define $\vcal$ to be the connected component of $j_0$ of $\vcal_1 \cap \vcal_2 \cap \vcal_3$. Since $\vcal$ is obviously an open neighborhood of $j_0$, we have only to show that $\vcal$ is disjoint from $\pi^{-1} \kcal_i$. We thus show that every $j\in \vcal$ satisfies one of the following conditions:
	\begin{itemize}
		\item $j(M) \not\subset D_\frac{2}{3}(W_L)$.
		\item $j(M)\subset D_\frac{2}{3}(W_L)$ and $p_L\circ j:M\longrightarrow L$ is not injective.
	\end{itemize}
	Observe from Lemma \ref{noncpt}(2) that the image of $M-K_0$ under the composite
	\[
		M \xhookrightarrow{\ \ \ j\ \ \ } W_{j_0(M)} \xrightarrow{\ p_{j_0(M)}\ } j_0(M)
	\]
	contains $j_0(M)-j_0(\mathring{K}_1) = L-i\circ h(\mathring{K}_1)$. Now suppose that $j(M)\subset D_\frac{2}{3}(W_L)$. Noticing that $x_0\in K$ and that $p_L\circ j(x_0)\in L-i\circ h(K_1)$, we then see that the composite $M \overset{j}{\longhookrightarrow} W_L \overset{p_L}{\longrightarrow} L$ is not injective.\par
	Suppose that condition (ii) is satisfied. We first deal with the case of compact $M$. If the differential of $p_L \circ j_0:M\longrightarrow L$ is invertible at every point $x\in M$, then $p_L \circ j_0$ is a covering (see \cite[Theorem 8.12]{BJ}). Since $p_L\circ j_0$ is not a diffeomorphism, it is an $d$-fold cover with $d>1$. Considering the mapping degree, we thus see that a sufficiently small neighborhood of $j_0$ is disjoint from $\pi^{-1}\kcal_i$ (see Lemma \ref{noncpt}(3)). Hence, we may assume that the differential of $p_L\circ j_0$ is not invertible at some $x_0\in M$. Then, we can choose a nonzero vector $v_0 \in T_{x_0}M$ with ${p_L}_\ast {j_0}_\ast (v_0)=0$. Noticing that ${p_L}_\ast j_\ast(v_0)$ is small for $j$ near $j_0$, we see that a sufficiently small neighborhood of $j_0$ is disjoint from $\pi^{-1}\kcal_i$ (see Fig. 11.2). We next deal with the case of noncompact $M$. If the connected component of $j_0$ does not contain any element of $\pi^{-1}\kcal_i$, then it is an open neighborhood of $j_0$ disjoint from $\pi^{-1}\kcal_i$. Thus, we assume that the connected component of $j_0$ contains an element of $\pi^{-1}\kcal_i$. Then, $p_L\circ j_0$ restricts to a diffeomorphism between $M \backslash K$ and $L \backslash p_L \circ j_0 (K)$ for some compact subset $K$ of $M$ (see \cite[Lemma 42.5]{KM}). If the differential of $p_L \circ j_0$ is invertible at every point $x\in M$, then $p_L \circ j_0$ is a diffeomorphism (see \cite[Theorem 8.12]{BJ}), which is a contradiction. Thus, the differential of $p_L\circ j_0$ is not invertible at some $x_0\in M$, which shows that a sufficiently small neighborhood of $j_0$ is disjoint from $\pi^{-1} \kcal_i$ (see the argument in the case of compact $M$).
\def\DefWidth{5}
\begin{center}
	\begin{tikzpicture}
	\draw [thick, -](-\DefWidth-0.5,0)--(\DefWidth+0.5,0) node [right]{$L=i(M)$};
	\draw [thick, domain=-3*pi/2:-pi/2+0.5] plot(\x, {sin(\x r)-1}); 
	\draw [thick, -] (-1.6+0.51,-1.88)  arc (298:361:2.1); 
	\draw [thick, domain=pi/2 - 0.3  :3*pi/2] plot(\x, {sin(\x r)+1}); 
	\draw [thick, -] (0.03,0)  arc (180:114:2.15); 
	\node at (\DefWidth/2,0.5) [above] {$j_0(M)$};
	\node at (\DefWidth,1) [above] {$W_L$};	
	\end{tikzpicture}\\
	Fig 11.2
\end{center}\par
	(2) The $C^{\infty}$-structure of $B^{\omega}(M, N)$ is introduced similarly to that of $B(M, N)$, though the model vector spaces are nuclear Silva spaces unlike the case of $B(M, N)$ (\cite[Theorem 44.3]{KM}). Since the principal ${\rm Diff}^\omega(M)$-bundle $\pi^\omega : {\rm Emb}^\omega (M,N) \longrightarrow B^\omega (M,N)$ exists (\cite[Theorem 44.3]{KM}) and ${\rm Emb}^{\omega}(M, N)$ is Lindel\"{o}f (Corollary \ref{goodmfds in ga}(3)), $B^{\omega}(M, N)$ is Lindel\"{o}f.\par
	Last, we show that $B^\omega(M,N)$ is regular. Since $B^\omega(M,N)$ is Hausdorff (\cite[Theorem 44.3]{KM}) and its model vector spaces are $C^\infty$-regular (\cite[Theorem 16.10]{KM}), we have only to show that for any $\widehat{i}=\pi^\omega(i) \in B^\omega(M,N)$, there exists a coordinate open neighborhood satisfying condition (b) in Lemma \ref{regularity} (see also Remark \ref{modelreg}).\par
	Noticing that $\pi^\omega : {\rm Emb}^\omega (M,N) \longrightarrow B^\omega (M,N)$ is a $\dcal$-quotient map (Lemma \ref{bdlequotient}), we have from \cite[Corollary 42.11]{KM} the commutative diagram in $C^\infty$
	\[
		\begin{tikzcd}
			{\rm Emb}^\omega(M,N) \arrow[hook]{r} \arrow[swap]{d}{\pi^\omega} & {\rm Emb}(M,N) \arrow{d}{\pi} \\
			B^\omega(M,N) \arrow[hook]{r} & B(M,N).
		\end{tikzcd}
	\]
	Thus, we can obtain the desired coordinate open neighborhood of $\widehat{i}$ in $B^\omega(M,N)$ by restricting that in $B(M,N)$ constructed in Part 1 (see the proof of \cite[Theorem 44.3]{KM}).
\end{proof}
\end{lem}

\begin{proof}[Proof of Corolarry \ref{goodmfds of submfds}.] (1) The result follows from Theorem \ref{goodmfds}(2) and Lemma \ref{BMN}(1).\par
(2) Since $B(M, N)$ is classical (\cite[Section 13]{Michor} and Lemma \ref{Kclassical}), the result follows from Theorem \ref{goodmfds}(3) and Lemma \ref{BMN}(1).
\par\indent
	(3) The result follows from Theorem \ref{goodmfds}(2) and Lemma \ref{BMN}(2).
\end{proof}
Let $B_{\rm closed}(M, N)$ be the open submanifold of $B(M,N)$ consisting of closed submanifolds (\cite[p. 474]{KM}). We can observe that $B_{\rm closed}(M,N)$ is a disjoint union of connected components of $B(M,N)$ (see the comment on ${\rm Emb}_{\rm prop}(M,N)$ after Corollary \ref{goodmfds in ga}). Hence, Corollary \ref{goodmfds of submfds}(2) remains true even if we replace $B(M,N)$ with $B_{\rm closed}(M,N)$.\par
Theorem \ref{goodmfds} also applies to direct limits of finite dimensional $C^{\infty}$-manifolds, which are extensively studied by Gl\"{o}ckner and coauthors (\cite{Gl}, \cite{Gl07}, \cite{Gl11}). Many infinite dimensional $C^{\infty}$-manifolds important in algebraic topology are obtained as such direct limits.
\par\indent
More precisely, we consider an ascending sequence $\{ M_{n} \}$ of second countable finite dimensional $C^{\infty}$-manifolds whose bonding maps are closed embeddings. Then, the direct limit $M$ exists in the category $C^{\infty}$, and if $\lim\limits_{n \rightarrow \infty}\ \dim M_{n} = \infty$, then $M$ is modeled on the Silva space $\rbb^{\infty}$ (\cite[Proposition 3.6]{Gl}). For $\fbb = \rbb, \cbb$, the infinite classical groups such as $GL(\infty, \fbb)$, $SL(\infty, \fbb)$, and $U(\infty)$, and the infinite Grassmannians $G(k, \infty; \fbb)$ are described as such direct limits; these $C^{\infty}$-manifolds are introduced in a different manner in \cite[Section 47]{KM}. The infinite Grassmannian $G(\infty, \infty; \fbb)$, which is not introduced in \cite[Section 47]{KM}, can be defined to be the direct limit $\lim\limits_{\rightarrow} \ G(k, \infty; \fbb)$, since 
\[\lim\limits_{\rightarrow}\ G(k, \infty; \fbb) = \lim\limits_{\underset{k}{\rightarrow}} \lim\limits_{\underset{l}{\rightarrow}} G(k, l; \fbb) = \lim\limits_{\rightarrow}\ G(k, k;\fbb),\]
where $G(k, l;\fbb)$ is the finite Grassmannian consisting of $k$-dimensional linear subspaces of $\fbb^{k+l}$.

\begin{cor}\label{goodmfds in at} Let $\{ M_{n} \}$ be an ascending sequence of second countable finite dimensional $C^{\infty}$-manifolds whose bonding maps are closed embeddings. Then, every submanifold of the direct limit $M=\displaystyle \lim_\rightarrow M_n$ is hereditarily Lindel\"{o}f, $C^{\infty}$-regular, and classical, and hence is in $\wcal_{\dcal\, 0}$. In particular, the infinite classical groups such as $GL(\infty, \fbb)$, $SL(\infty, \fbb)$, and $U(\infty)$, and the infinite Grassmannians $G(k, \infty; \fbb)$ $(0 \leq k \leq \infty)$ are in $\wcal_{\dcal\, 0}$.
\end{cor}
\begin{proof}
	Note that the direct limit $M$ is Lindel\"{o}f and regular and is modeled on $\rbb^{\infty}$ (\cite[Proposition 3.6 and its proof]{Gl}) and that $\rbb^{\infty}$ is a nuclear Silva space (\cite[Corollary 21.2.3]{Jar}). Then, the result follows from Theorem \ref{goodmfds}(2).
\end{proof}


\appendix

\renewcommand{\thesection}{Appendix A}
\section{Pathological diffeological spaces}
\renewcommand{\thesection}{A}
The subclasses $\wcal_\dcal$ and $\vcal_\dcal$ of $\dcal$ were introduced in Section 1.3 as subclasses of $\dcal$-homotopically good objects.
Setting $ \widetilde{}\wcal_{\czero} = \{ A \in \dcal \ | \ \widetilde{A} \in \wcal_{\czero} \}$, we obtain from Corollary \ref{W} the following
Venn diagram:
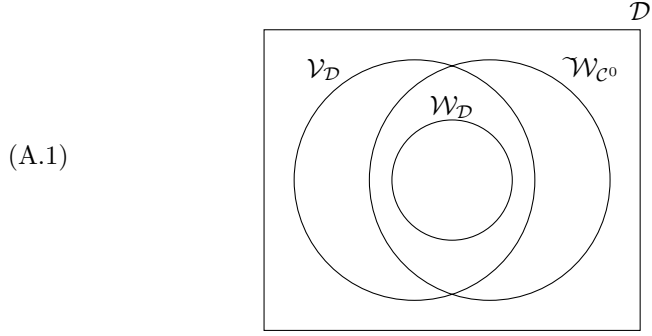
\[
\begin{tikzpicture}[fill=white]
\scope
\clip (-2,-2) rectangle (2,2)
(1,0) circle (1);
\fill (0,0) circle (1);
\endscope
\scope
\clip (-2,1) rectangle (2,2)
(1,0) circle (1);
\fill (0,0) circle (1);
\endscope
\scope
\clip (-2,-2) rectangle (2,2)
(0,0) circle (1);
\fill (1,0) circle (1);
\endscope
\draw (-5,0) node [text=black,above] {(A.1)};
\draw (15,0) node [text=black,above] {};
\draw (-5,0) node [text=black,above] {};
\draw (0,0) circle (1.6) (-1.2,1.2)  node [text=black,above] {$\vcal_{\dcal}$}
(1,0) circle (1.6) (2.4,1.2)  node [text=black,above] {$\widetilde{}\wcal_{\czero}$}
(0.5,0) circle (0.8) (0.5,0.7)  node [text=black,above] {$\wcal_{\dcal}$}
(-2,-2) rectangle (3,2) node [text=black,above] {$\dcal$};
\end{tikzpicture}
\]
In this appendix, we show that all the inclusions shown in Venn diagram (A.1) are proper by providing examples of diffeological spaces in the classes $\dcal \backslash (\vcal_{\dcal} \cup \ \widetilde{}\wcal_{^\czero}),\ \widetilde{}\,\wcal_{\czero} \backslash \vcal_{\dcal},\ \vcal_{\dcal} \backslash \ \widetilde{}\,\wcal_{\czero},$ and $(\vcal_{\dcal} \cap\ \widetilde{}\wcal_{\czero}) \backslash \wcal_{\dcal}$ (Examples \ref{Hawaiian}, \ref{quotient}, and \ref{RX}). Our study in this appendix organizes the study of pathological diffeological spaces presented by Iglesias-Zemmour \cite[8.38]{IZ} and Christensen-Wu \cite[Examples 3.12 and 3.20]{CW}.
\par\indent
We begin with the following lemma on topological spaces.
\begin{lem}\label{nonW}
	Let $X$ be a topological space having the homotopy type of a $CW$-complex. Then, the following are equivalent:
	\begin{itemize}
		\item[{\rm (i)}] $X$ is homotopy equivalent to a Lindel\"{o}f space.
		\item[{\rm (ii)}] $X$ is homotopy equivalent to a countable $CW$-complex.
		\item[{\rm (iii)}] $H_i(X)$ and $\pi_i(X,x)$ are countable for any $i\ge 0$ and any $x\in X$.
		\item[{\rm (iv)}] $\pi_i(X,x)$ is countable for any $i\ge 0$ and $x\in X$.
	\end{itemize}
	\begin{proof}
			For ${\rm (i)}\Longleftrightarrow{\rm (ii)}$, see \cite[Proposition 2]{Mi}. For ${\rm (ii)} \Longleftrightarrow {\rm (iii)} \Longleftrightarrow {\rm (iv)}$, see Remark \ref{refinement}(1) and (2).
\end{proof}
\end{lem}
\begin{rem}\label{compact}
	We can prove the following analogue of Lemma \ref{nonW} using the results in \cite[pp. 75-76]{MayAT} and \cite[Theorem 4.5.2]{MP}:
	\begin{quote}
		Let $X$ be a topological space having the homotopy type of a $CW$-complex. If $X$ is homotopy equivalent to a compact space, then $H_\ast(X)=\underset{i\ge 0}{\oplus} H_i(X)$ and $\pi_1(X,x)$ are finitely generated for any $x\in X$. Further if $X$ is $1$-connected, then $\pi_i(X,x)$ is finitely generated for any $i\ge 1$ and any $x\in X$.
	\end{quote}
\end{rem}
\begin{exa}\label{Hawaiian}
	The $n$-dimensional {\it Hawaiian earring} $E^n$ is defined to be the diffeological subspace
	$\overset{\infty}{\underset{k=1}{\cup}} E^{n}_{k}$ of $\rbb^{n+1}$, where
	\[
	E^{n}_{k} = \{ (x_{0}, \ldots, x_{n}) \in \rbb^{n+1}\ |\ (x_{0} - \frac{1}{\sqrt[n]{k}})^{2} + x^{2}_{1} + \cdots + x^{2}_{n} = (\frac{1}{\sqrt[n]{k}})^{2} \}.
	\]
	\begin{itemize}
		\item[(1)] $\widetilde{E^n}$ does not have the homotopy type of a $CW$-complex.
		\item[(2)] The natural homomorphism
		$$
		\pi_n^\dcal (E^n) \longrightarrow \pi_n (\widetilde{E^n})
		$$
		is not surjective.
		\item[(3)] $E^{n}$ is in $\dcal \backslash ( \vcal_{\dcal} \cup\ \widetilde{}\wcal_{\czero} ).$
	\end{itemize}
	\begin{proof} (1) By the argument of \cite[Example 3.12]{CW} (see also \cite[p. 18]{KM}), we see that the underlying topology of $E^n$ is just the sub-topology of $\rbb^{n+1}$. Observe that the canonical epimorphism $\pi_{n}(\widetilde{E^n}) \longrightarrow {\prod}_{k \geq 1} \pi_{n}(\widetilde{E^{n}_{k}}) = {\prod}_{k \geq 1} \zbb$ exists (cf. \cite[Example 1 in Section 71]{Munk}). Then, the result follows from Lemma \ref{nonW}.

		(2) We can prove the result by the argument in the proof of Part 1 and an argument similar to that in \cite[Example 3.12]{CW}.

		(3) The result follows from Parts 1-2 and Corollary \ref{W}(3).
	\end{proof}
\end{exa}
\begin{rem}\label{earring} This remark relates to the $n$-dimensional Hawaiian earring $E^{n}$. We can easily see that $\widetilde{E^{n}}$ is homeomorphic to the topological subspace $E^{n}_{\rm top}$ of $\rbb^{n+1}$ defined by
	$$
E^{n}_{\rm top} = \overset{\infty}{\underset{k=1}{\cup}} \ \{ (x_0, \ldots , x_n) \in \rbb^{n+1} |
	\begin{array}{l}
	(x_0 - \frac{1}{k})^2 + x_1^2 + \cdots + x_n^2 = (\frac{1}{k})^2
	\end{array}
\},
$$
	which is called the $n$-dimensional {\sl topological Hawaiian earring}. See \cite{EKK,EKK2} for the exact calculations of $\pi_{n}(E^{n}_{\rm top})$.
\end{rem}
Next, we investigate the quotient of a diffeological group by a dense subgroup.
\begin{lem}\label{homog}
	Let $G$ be a diffeological group, and $H$ a diffeological subgroup of $G$. If $H$ is dense in $\widetilde{G}$ and the inclusion $H \longhookrightarrow G$ is not a weak equivalence in $\dcal$, then $G / H$ is in $\ \widetilde{}\wcal_{\czero} \backslash \vcal_{\dcal}$.
	\begin{proof}
		By the assumption, $\widetilde{G / H}$ is an indiscrete space (\cite[Lemma 2.11]{origin}), and hence, contractible. On the other hand, we see that $\pi_\ast^\dcal (G/H) \neq 0$ from the homotopy exact sequence of the sequence
		$$
		H \longhookrightarrow G \longrightarrow G/H
		$$
		(\cite[8.15 and 8.21]{IZ}). Thus, we obtain the result by Corollary \ref{W}(3).
	\end{proof}
\end{lem}
For an irrational number $\theta$, the {\it irrational torus} $T_\theta^2$ {\it of slope} $\theta$ is defined to be the quotient diffeological group $T^2 / \rbb_\theta$, where $T^2 = \rbb^2 / \zbb^2$ is the usual $2$-torus and $\rbb_\theta$ is the image of the injective homomorphism $\rbb \longrightarrow T^2$ sending $x$ to $[x, \theta x]$ (\cite[8.38]{IZ}).
\begin{exa}\label{quotient}
	The quotient diffeological groups $\rbb / \qbb$ and the irrational torus $T_\theta$ are in $\widetilde{}\wcal_{\czero} \backslash \vcal_{\dcal}$.
	\begin{proof}
		The result is immediate from Lemma \ref{homog}.
	\end{proof}
\end{exa}
Lastly, we show the following lemma and then identify conditions under which a diffeological space of the form $RX$ is in $\vcal_{\dcal}$,\, $\widetilde{}\wcal_{\czero}$ and $\wcal_{\dcal}$ (see Proposition \ref{conven}(3) and Remark \ref{convenrem}(2) for the functor $R$).
\begin{lem}\label{R}
	Let $X$ be an arc-generated space and $Z$ a cofibrant diffeological space. Then, any smooth map $f : RX \longrightarrow Z$ is locally constant.
\end{lem}
\begin{proof}
	Let $\gamma:\emptyset \longrightarrow Z$ be the canonical map from the initial object $\emptyset$ to $Z$. Consider the cofibration-trivial fibration factorization
	\[
	\begin{tikzcd}
		\emptyset \arrow{r}{i_\infty} \arrow[swap]{rd}{\gamma}& G^\infty(\mathcal{I},\gamma) \arrow{d}{p_\infty}\\
		& Z
	\end{tikzcd}
	\]
	as in \cite[the proof of Theorem 1.3]{origin} and recall that $i_\infty$ is a sequential relative $\ical$-cell complex. Then, by solving the lifting problem in $\dcal$
	\[
	\begin{tikzcd}
		\emptyset \arrow{r}{i_\infty} \arrow{d}& G^\infty(\mathcal{I},\gamma) \arrow{d}{p_\infty}\\
		Z \arrow{r}{1} \arrow[dashed]{ru} & Z,
	\end{tikzcd}
	\]
	we obtain the retract diagram in $\dcal$
	\[
		Z \longhookrightarrow G^\infty(\mathcal{I},\gamma) \longrightarrow Z.
	\]
	Thus, we may assume that $Z$ is the colimit of a sequence in $\dcal$
	\[
		Z_0 \overset{i_1}{\longhookrightarrow} Z_1 \overset{i_2}{\longhookrightarrow} Z_2 \longhookrightarrow \cdots
	\]
	such that $Z_0$ is a coproduct of copies of $\Delta^0$ and each $i_n$ fits into a pushout diagram of the form
	\[
	\begin{tikzcd}
		\underset{\lambda \in \Lambda_n}{\coprod} \dot{\Delta}^{p_\lambda} \arrow{d} \arrow[hook]{r} & \underset{\lambda \in \Lambda_n}{\coprod} \Delta^{p_\lambda} \arrow{d}\\
		Z_{n-1} \arrow[hook]{r}{i_n} & Z_n.
	\end{tikzcd}
	\]
	
	Since $X$ is locally arcwise connected, we may also assume that $X$ is arcwise connected. Suppose that there exists a nonconstant smooth map $f : RX \longrightarrow Z$. Then, we can choose $x_0, x_1 \in X$ and a continuous curve $d : \rbb \longrightarrow X$ satisfying the following conditions:
	\begin{itemize}
		\item $f(x_0) \neq f(x_1)$.
		\item $d((-\infty, 0]) = x_0$, $d([1, \infty)) = x_1$.
	\end{itemize}
	We see from Proposition \ref{conven}(3) that the composite
	$$
	\rbb \xrightarrow{\ \ d\ \ } RX \xrightarrow{\ \ f\ \ } Z
	$$
	is a nonconstant smooth curve.
	\par\indent
	Recall \cite[Lemma 9.9]{origin} and set
	$$
	m = \text{min}\ \{ n\ | \ (f \circ d)(\rbb) \subset Z_n \}.
	$$
	Then, we obtain the nonconstant smooth map
	$$
	\rbb \xrightarrow{\ \ f \circ d\ \ } Z_m.
	$$
	Since the inclusion $\mathring{\Delta}_\lambda^{p_\lambda} \longrightarrow Z_m$ is a $\dcal$-embedding for $\lambda \in \Lambda_m$, we regard $\mathring{\Delta}_\lambda^{p_\lambda}$ as a diffeological subspace of $Z_m$. Noticing that $(f \circ d)^{-1} (\mathring{\Delta}_\lambda^{p_\lambda})$ is a nonempty open set of $\rbb$ for some $\lambda \in \Lambda_m$, we may assume that the nonconstant smooth map $ f \circ d : \rbb \longrightarrow Z_m$ corestricts to $\mathring{\Delta}_\lambda^{p_\lambda}$. Then, we can choose a linear map $\ell:\mathbb{R}^{p_\lambda+1} \longrightarrow \mathbb{R}$ such that the composite
	\[
		\mathbb{R} \xrightarrow{\ \ f \circ d\ \ } \mathring{\Delta}^{p_\lambda}_\lambda \xrightarrow{\ \ \ell \ \ }\mathbb{R}
	\]
	is a nonconstant smooth map (see Lemma \ref{simplex}(3)). By the definition of the functor $R$, the composite
	$$
	\rbb \xrightarrow{\ \ c\ \ } \rbb \xrightarrow{\ \ \ell \circ f \circ d \ \ } \mathbb{R}
	$$
	is smooth for any continuous curve $c$, which is a contradiction.
\end{proof}
\begin{rem}\label{mfdtarget}
	Lemma \ref{R} remains true even if a cofibrant target $Z$ is replaced by a $C^\infty$-manifold $M$. This is proved as follows. Suppose that there exists a smooth map $f:RX\longrightarrow M$ which is not locally constant. Then, we can choose a continuous curve $d:\mathbb{R}\longrightarrow X$, a chart $E_\alpha \supset u_\alpha(U_\alpha) \xleftarrow[\cong]{u_\alpha} U_\alpha \longhookrightarrow M$, and a continuous linear functional $\ell:E_\alpha \longrightarrow \mathbb{R}$ such that
	\begin{itemize}
		\item $f\circ d:\mathbb{R}\longrightarrow M$ corestricts to $U_\alpha$.
		\item The composite
		\[
			\mathbb{R} \xrightarrow{f\circ d} U_\alpha \xrightarrow[\cong]{u_\alpha} u_\alpha(U_\alpha) \longhookrightarrow E_\alpha \xrightarrow{\ \ \ell\ \ } \mathbb{R}
		\]
		is a nonconstant smooth map.
	\end{itemize}
	(See \cite[p. 127]{Jar} and \cite[Corollary 2.11]{KM}). This yields a contradiction (see the proof of Lemma \ref{R}).
\end{rem}
We prove the following result; Lemma \ref{R} is used for the proof of Part 3.
\begin{lem}\label{RM}
	Let $X$ be an arc-generated space.
	\begin{itemize}
		\item[$(1)$] $RX$ is in $\vcal_\dcal$.
		\item[$(2)$] $RX$ is in $\widetilde{}\wcal_{\czero}$ if and only if $X$ is in $\wcal_{\czero}$.
		\item[$(3)$] $RX$ is in $\wcal_{\dcal}$ if and only if $X$ has the homotopy type of a discrete set.
	\end{itemize}
\end{lem}
\begin{proof}
	Parts 1 and 2 are immediate from the equality $\tilde{\cdot} \circ R = Id_{\ccal^0}$ (Remark \ref{convenrem}(2)).\par
	(3) \noindent($\Leftarrow$) Note that if $f, g : X \longrightarrow Y$ are $\ccal^0$-homotopic, then $Rf, Rg : RX \longrightarrow RY$ are $\dcal$-homotopic (see the argument in the proof of Corollary \ref{W}(2)). Then, the implication is obvious.\\
	($\Rightarrow$) Note that every arc-generated space is locally arcwise connected. Then, we see that $X=\underset{i}{\coprod} X_i$ in $\czero$, and hence that $RX= \underset{i}{\coprod}RX_i$ in $\dcal$, where $\{X_i\}$ is the set of arcwise connected components of $X$. Thus, we may assume that $X$ is arcwise connected.\par
	If $RX$ is in $\wcal_\dcal$, then $1_{RX}:RX\longrightarrow RX$ is $\dcal$-homotopic to a constant map (Lemma \ref{R}), and hence $RX\simeq_\dcal \ast$ holds.
\end{proof}
\begin{exa}\label{RX}
	\begin{itemize}
	\item[{\rm (1)}] $RE^{n}_{\rm top}$ is in $\vcal_{\dcal} \backslash\ \widetilde{}\wcal_{\czero}$ for $n \geq 1$, where $E^{n}_{\rm top}$ is the $n$-dimensional topological Hawaiian earring.
	\item[{\rm (2)}] $RS^{n}_{\mathrm{top}}$ is in $(\vcal_{\dcal} \cap\ \widetilde{}\wcal_{\czero}) \backslash \wcal_\dcal$ for $n \geq 1$, where $S^{n}_{\mathrm{top}}$ is the $n$-dimensional topological sphere.
	\end{itemize}
\end{exa}
\begin{proof} (1) Recall that $E^{n}_{\rm top}$ is an arc-generated space which is not in $\wcal_{\czero}$ (see Example \ref{Hawaiian}(1) and Remark \ref{earring}). Then, the result follows from Lemma \ref{RM}(1)-(2).\par
	(2) The result follows from Lemma \ref{RM}(1)-(3).
\end{proof}
	By Examples \ref{Hawaiian}, \ref{quotient}, and \ref{RX}, we have shown that all the inclusions in Venn diagram (A.1) are proper.
\begin{rem}\label{}
	Since $S^n$ is in $\wcal_\dcal$ (Theorem \ref{mfdhcofibrancy}), $[RS^n_{\rm top},S^n]_\dcal=\ast$ by Lemma \ref{R}. From Example \ref{RX}(2), we thus see that even if $A\in \vcal_\dcal \cap\ \widetilde{}\wcal_\czero$ and $X \in \wcal_\dcal$, $[A, X]_\dcal \longrightarrow [\widetilde{A}, \widetilde{X}]_\czero$ need not be bijective (cf. Theorem \ref{dmapsmoothing}).
\end{rem}
\renewcommand{\thesection}{Appendix B}
\section{Keller's $C^{\infty}_{c}$-theory and diffeological spaces}
\renewcommand{\thesection}{B}
Keller's $C^{\infty}_{c}$-theory is a representative example of classical infinite dimensional calculi (Remark \ref{classical}) and is widely used by many authors (cf. \cite{Michor} and the articles of Gl\"{o}ckner, Neeb, and their coauthors); see \cite[1.3]{Gl} and references therein for the basics of Keller's $C^{\infty}_{c}$-theory. In this appendix, we discuss the relation between $C^{\infty}$-manifolds in Keller's $C^{\infty}_{c}$-theory, $C^{\infty}$-manifolds in convenient calculus, and diffeological spaces. In particular, we show that the $C^{\infty}$-manifolds in Keller's $C^{\infty}_{c}$-theory cannot be fully faithfully embedded into the category $\dcal$ (cf. Proposition \ref{dmfd}).

$C^{\infty}$-manifolds in Keller's $C^{\infty}_{c}$-theory are defined by gluing open sets of locally convex vector spaces via diffeomorphisms; we require that $C^{\infty}$-manifolds in Keller's $C^{\infty}_{c}$-theory are Hausdorff. Let $C^{\infty}_{c}$ denote the category of $C^{\infty}$-manifolds in Keller's $C^{\infty}_{c}$-theory. We call an object and a morphism of $C^{\infty}_{c}$ a $C^{\infty}_{c}$-manifold and a $C^{\infty}_{c}$-map to distinguish them from convenient ones.\par
Let $C^{\infty}_{c\ conv}$ denote the full subcategory of $C^{\infty}_{c}$ consisting of $C^{\infty}_{c}$-manifolds modeled on convenient vector spaces. We can then define the faithful functor
\[
J : C^{\infty}_{c\ conv} \longrightarrow C^{\infty}
\]
by assigning to $M \in C^{\infty}_{c\ conv}$ the object $JM$ of $C^{\infty}$ having the same underlying set and atlas as $M$.
\begin{rem}\label{classical}
In this paper, a classical infinite dimensional calculus is a word opposed to the word ``convenient calculus". A classical infinite dimensional calculus is assumed to satisfy the following conditions:
\begin{itemize}
	\item A smooth map is a continuous map between open sets of locally convex vector spaces (in some specified class).
	\item The smooth curves into a locally convex vector space are just the usual ones (see \cite[1.2]{KM}).
	\item Smooth maps are closed under composites.
\end{itemize}
Keller's $C^\infty_c$-theory is a classical infinite dimensional calculus in this sense. The definition of $C^{\infty}$-manifolds and that of the faithful functor $J$ apply to every classical infinite dimensional calculus.
\end{rem}
\begin{prop}\label{notfull}
	The faithful functor
	\[
	J : C^{\infty}_{c \ conv} \longrightarrow C^{\infty}
	\]
	is not full.
\end{prop}
\begin{proof}
	Choose a nonbornological convenient vector space $E$ (see \cite{Val}) and consider the bornologification $id : E_{\rm born} \longrightarrow E$ (\cite[Lemma 4.2]{KM}). Then, $J\ id: JE_{\rm born} \longrightarrow JE$ is an isomorphism in $C^\infty$ (see \cite[Corollary 2.11]{KM}). Noticing that $id: E \longrightarrow E_{\rm born}$ is not even continuous, we see that
	\[
		J:C^\infty_{c\ conv}(E,E_{\rm born}) \longrightarrow C^\infty(JE,JE_{\rm born})
	\]
	is not surjective.
\end{proof}
We define the faithful functor $I : C^{\infty}_{c} \longrightarrow \dcal$ similarly to the faithful functor $I : C^{\infty} \longrightarrow \dcal$ (see Section 2.2); explicitly, $I$ assigns to $M \in C^{\infty}_{c}$ the set $M$ endowed with the diffeology
\[
D_{IM} = \{ \text{$C^{\infty}_{c}$-parametrizations of $M$} \}.
\]
\begin{cor}\label{Inotfull}
	The faithful functor
	\[
	I : C^{\infty}_{c} \longrightarrow \dcal
	\]
	is not full
\end{cor}
\begin{proof}
	We prove the result in two steps.\vspace{0.1cm}\\
	Step 1. We show that the diagram of faithful functors
	\[
	\begin{tikzcd}
	C^{\infty}_{c \ conv} \arrow{r}{J} \arrow[hook']{d} & C^{\infty} \arrow[hook']{d}{I}\\
	C^{\infty}_{c} \arrow{r}{I} & \dcal
	\end{tikzcd}
	\]
	is commutative. For this, we have only to show that for a domain $U$ of $\rbb^{n}$ and a locally convex vector space $E$, the following equivalence holds:
	\[
	f: U \longrightarrow E \ \text{is a $C^{\infty}_{c}$-map}\ \Leftrightarrow f: U \longrightarrow E\ \text{is a $C^{\infty}$-map}.
	\]
	The implication $(\Rightarrow)$ is obvious. For $(\Leftarrow)$, use \cite[1.3]{Gl} and \cite[Definition 3.17, Theorem 3.18, the proof of Corollary 5.11, and Theorem 4.11(1)]{KM}.\vspace{0.1cm}\\
	Step 2. Consider the commutative diagram in Step 1. Since the two vertical functors are full (Proposition \ref{dmfd}), $I : C^{\infty}_{c} \longrightarrow \dcal$ is not full by Proposition \ref{notfull}.
\end{proof}
\begin{rem}\label{full}
	Let $C^\infty_{c\,F}$ (resp. $C^\infty_{c\,S}$) denote the full subcategory of $C^\infty_c$ consisting of $C^\infty_c$-manifolds modeled on Fr\'echet spaces (resp. Silva spaces). Similarly, let $C^\infty_F$ (resp. $C^\infty_S$) denote the full subcategory of $C^\infty$ consisting of $C^\infty$-manifolds modeled on Fr\'echet spaces (resp. Silva spaces). Since the class of Fr\'{e}chet spaces and that of Silva spaces are closed under finite products (\cite[Proposition 3]{Yo}), we see from \cite[Theorem 4.11(1)-(2)]{KM} that if $M$ is in $C^\infty_F$ or $C^\infty_S$, then $M$ is necessarily Hausdorff. Further, using an argument similar to that in the proof of Corollary \ref{Inotfull}, we can show that the functor $J: C^{\infty}_{c\ conv} \longrightarrow C^{\infty}$ restricts to the isomorphisms of categories $C^{\infty}_{c\ F} \xrightarrow[\cong]{} C^{\infty}_{F}$ and $C^\infty_{c\ S} \xrightarrow[\cong]{} C^\infty_S$ (see \cite{CGA} for details).
\end{rem}
Proposition \ref{dmfd} and Corollary \ref{Inotfull} show that convenient calculus also provides a good categorical setting for the purpose of investigating infinite dimensional $C^{\infty}$-manifolds and smooth maps via the smooth homotopy theory of diffeological spaces. However, Keller's $C^{\infty}_{c}$-theory plays an important role in supplying classical $C^{\infty}$-manifolds in convenient calculus (Lemma \ref{Kclassical}). 

We end this appendix with the following remark on convenient categories of smooth spaces.
\begin{rem}
	Kriegl-Michor \cite{KM} used the category $\fcal$ of Fr\"{o}licher spaces as a convenient category of smooth spaces. Using the adjoint pair $F: \dcal \rightleftarrows \fcal : I$ with I fully faithful (\cite{St}), we can construct a model structure on $\fcal$ for which $(F, I)$ is a pair of Quillen equivalences (see a forthcoming paper). However, we do not adopt $\fcal$ as a convenient category of smooth spaces since it cannot be ensured that $\fcal$ contains $C^{\infty}$ (see \cite[Lemma 27.5]{KM}).
\end{rem}
\renewcommand{\thesection}{Appendix C}
\section{Smooth regularity and smooth paracompactness}
\renewcommand{\thesection}{C}
In this appendix, we discuss smooth regularity and smooth paracompactness, clarifying points often misunderstood and overlooked in the literature (Remarks \ref{T-reg} and \ref{correction}). The results in this appendix are used in Sections 11.3 and 11.4 to establish the smooth paracompactness of all submanifolds of $\cfra^\infty(M,N)$, $C^\omega(M,N)$, $B(M,N)$, and $B^\omega(M,N)$ (see Section 11.4 for these $C^\infty$-manifolds).\par
Some results are discussed in a rather general context using the notation of Section 5; $\ccal$ denotes one of the categories $C^\infty$, $\dcal$, $\czero$ and $\tcal$, and $U:\ccal \longrightarrow \tcal$ denotes the underlying topological space functor for $\ccal$.\par
First, we prove the following result on $\ccal$-regularity and apply it to topological manifolds and $C^\infty$-manifolds. An object $X$ of $\ccal$ is called {\sl locally $\ccal$-regular} if any $x\in X$ has a $\ccal$-regular open neighborhood.
\begin{lem}\label{C-regular}
	For $X\in \ccal$ with $UX$ $T_1$-space, the following are equivalent:
	\begin{itemize}
		\item[{\rm (i)}] $X$ is $\ccal$-regular.
		\item[{\rm (ii)}] $UX$ is regular and $X$ is locally $\ccal$-regular.
	\end{itemize}
\begin{proof}
	$(\rm i) \Longrightarrow (\rm ii)$ See Remark \ref{regular}.\\
	$(\rm ii) \Longrightarrow (\rm i)$ Assume given a point $x\in X$ and an open neighborhood $U$ of $x$. We construct a $\ccal$-function $\phi:X \longrightarrow \mathbb{R}$ with $\phi(x)=1$, ${\rm carr}\,\phi \subset U$, and $\phi(X)\subset[0,1]$.\par
	First, we choose a $\ccal$-regular open neighborhood $U'$ of $x$ such that $U' \subset U$ (see Remark \ref{regular}). Then, we can choose disjoint open sets $W'$ and $V'$ of $X$ such that
	\[
	x \in W' \text{ and } U'^c \subset V'
	\]
	and a $\ccal$-function $\varphi:U'\longrightarrow \mathbb{R}$ such that $\varphi(x)=1$, carr $\varphi \subset W'$, and $\varphi(U')\subset [0,1]$.\par
	Since the function $\varphi:U'\longrightarrow \mathbb{R}$ and the zero function $0:V'\longrightarrow \mathbb{R}$ coincide on $U'\cap V'$, $\phi:= \varphi+0:X=U' \cup V' \longrightarrow \mathbb{R}$ is a $\ccal$-function on $X$, which is the desired bump function.	
\end{proof}
\end{lem}
\begin{rem}\label{T-reg}
	In this remark, we specialize Lemma \ref{C-regular} to topological manifolds (see Remark \ref{topmfd} for the definition of a topological manifold and recall from Remark \ref{regular} that $\tcal$-regularity is just complete regularity). Since locally convex vector spaces are completely regular (\cite[Sections 2.1 and 2.9]{Jar}), every topological manifold is locally $\tcal$-regular. Hence, Lemma \ref{C-regular} specializes to the following statement:
	\begin{quote}
		A topological manifold is completely regular if and only if it is regular.
	\end{quote}
	This is a corrected version of Eells' statement that a topological manifold is completely regular \cite[p. 765]{Eells}. In fact, a separation condition strictly stronger than Hausdorff property is necessary to derive complete regularity of a topological manifold (see \cite[the example in Section 27.6]{KM}).
\end{rem}
Now we establish the relation between $C^\infty$-regularity of a $C^\infty$-manifold and $C^\infty$-regularity of its model vector spaces.
\begin{prop}\label{smregular}
	For a $C^\infty$-manifold $M$, the following are equivalent:
	\begin{itemize}
		\item[{\rm (i)}] $M$ is $C^\infty$-regular.
		\item[{\rm (ii)}] $M$ is regular as a topological space and is modeled on $C^\infty$-regular convenient vector spaces.
		\item[{\rm (iii)}] $M$ is regular as a topological space and admits an atlas consisting of charts onto $C^\infty$-regular convenient vector spaces.
	\end{itemize}
\begin{proof}
	${\rm (iii)}\Longrightarrow {\rm (ii)}$ Obvious.\\
	${\rm (ii)}\Longrightarrow {\rm (i)}$ The implication follows from Lemma \ref{C-regular}.\\
	${\rm (i)} \Longrightarrow {\rm (iii)}$ Since $M$ is obviously regular as a topological space, we have only to construct a chart onto a $C^\infty$-regular convenient vector space around any $x\in M$.\par
	Take a chart of $M$ around $x$
	\[
		M \supset U \xrightarrow[\cong]{\ \ u\ \ } V \subset E
	\]
	with $u(x)=0$, where $E$ is a convenient vector space and $V$ is a $c^\infty$-open set of $E$. We may assume that $V$ is radial by replacing $V$ with its star (see \cite[Theorem 16.21]{KM}). Set $V_{\frac{1}{2}}= \frac{1}{2}V$ and $U_{\frac{1}{2}}=u^{-1}(V_{\frac{1}{2}})$. By the $C^\infty$-regularity of $M$, we can choose a smooth function $\varphi:M \longrightarrow \mathbb{R}$ such that $\varphi(x)=1$, ${\rm supp}\,\varphi \subset U_{\frac{1}{2}}$ and $\varphi(M)\subset [0,1]$. Define the smooth function $\psi:V \longrightarrow \mathbb{R}$ to be the composite
	\[
		V \xrightarrow[\cong]{\ \ u^{-1}\ \ } U \xrightarrow[]{\ \ \varphi|_U\ \ } \mathbb{R}
	\]
	and set $V'={\rm carr}\, \psi$. Then, the $c^\infty$-open set $V'$ of $E$ and the smooth function $\psi|_{V'}$ satisfy the assumption of \cite[Theorem 16.21]{KM}, and hence $V'':={\rm star}\, V'$ is a $C^\infty$-regular $c^\infty$-open set of $E$ which is diffeomorphic to $E$. Thus, the composite
	\[
		M \supset U'' \xrightarrow[\cong]{\ \ u\ \ } V'' \xrightarrow[\ \ \cong\ \ ]{} E
	\]
	is the desired chart around $x$, where $U''= u^{-1}(V'')$.
\end{proof}
\end{prop}
Next, we discuss $\ccal$-paracompactness of a Lindel\"of $\ccal$-space, especially $C^\infty$-paracompactness of a Lindol\"of $C^\infty$-manifold. Recall the following result from \cite{KM}.
\begin{prop}\label{Lindelof}
	Let $X$ be an object of $\ccal$ with $UX$ Lindel\"of $T_1$-space. Then, $X$ is $\ccal$-paracompact if and only if $X$ is $\ccal$-regular.
	\begin{proof}
		See Remark \ref{regular}, the comment after it, and \cite[Theorem 16.10]{KM}.
	\end{proof}
\end{prop}
The following is a refinement of the first statement of \cite[Corollary 27.4]{KM}.
\begin{thm}\label{smparacompact}
	Let $M$ be a Lindol\"of $C^\infty$-manifold. Then, the following conditions are equivalent:
	\begin{itemize}
		\item[{\rm (i)}] $M$ is $C^\infty$-paracompact.
		\item[{\rm (ii)}] $M$ is $C^\infty$-regular.
		\item[{\rm (iii)}] $M$ is regular as a topological space and is modeled on $C^\infty$-regular convenient vector spaces.
	\end{itemize}
\begin{proof}
	The equivalences ${\rm (i)} \Longleftrightarrow {\rm (ii)}$ and ${\rm (ii)} \Longleftrightarrow {\rm (iii)}$ follow from Propositions \ref{Lindelof} and \ref{smregular} respectively.
\end{proof}
\end{thm}
\begin{rem}\label{correction}
	We can observe from \cite[Section 27.6]{KM} that in Theorem \ref{smparacompact}, the regularity condition on $M$ in condition {\rm (iii)} is necessary for the implication $(\rm iii)\Longrightarrow (\rm i)$ to hold.\par
	At first glance, the first statement of \cite[Corollary 27.4]{KM} does not require the regularity condition. However, Kriegl replied my email that they use the definition of Lindel\"of (e.g., \cite[3.8]{Engelking}) where regularity is assumed. Thus, the implication $(\rm iii)\Longrightarrow(\rm i)$ in Theorem \ref{smparacompact} is a variant of the first statement of \cite[Corollary 27.4]{KM}. I wrote down the details of the proof since they are not so trivial and contain points often misunderstood.\par
	As mentioned above, the regularity of a $C^\infty$-manifold $X$ must be verified to derive the $C^\infty$-paracompactness of $X$ from its Lindel\"{o}f property and the $C^\infty$-regularity of its model vector spaces. However, Kriegl informed me that they forgot to mention how to obtain the regularity of $\cfra^\infty(M,N)$ in the proof of $C^\infty$-paracompactness of $\cfra^\infty(M,N)$ (\cite[Proposition 42.3]{KM}). We give a regularity criterion for a manifold (Lemma \ref{regularity}) and apply it to establish the (hereditary) $C^\infty$-paracompactness not only of $\cfra^\infty(M,N)$ but also of $C^\omega(M,N),B(M,N),$ and $B^\omega(M,N)$ (see Section 11.4).
\end{rem}
We end this appendix with the following useful criterion for a manifold to be regular as a topological space (see the proofs of Lemmas \ref{mfds of maps} and \ref{BMN}).
\begin{lem}\label{regularity}
	For a $T_1$-space $X$, the following are equivalent.
	\begin{itemize}
		\item[{\rm (i)}] $X$ is regular.
		\item[{\rm (ii)}] For any $x\in X$, there exists an open neighborhood $U_x$ of $x$ satisfying following conditions:
		\begin{itemize}
			\item[{\rm (a)}] $U_x$ is regular.
			\item[{\rm (b)}] There exist disjoint open sets $W_x$ and $V_x$ such that
			\[
			x \in W_x \text{ and } U^c_x \subset V_x.
			\]
		\end{itemize}
	\end{itemize}
	\begin{proof}
		${\rm (i)}\Rightarrow {\rm (ii)}$ Obvious.\\
		${\rm (ii)\Rightarrow {\rm (i)}}$ Assume given a point $x\in X$ and an open neighborhood $U$ of $x$. We choose an open neighborhood $U_x$ of $x$ satisfying conditions (a) and (b), and then choose disjoint open sets $W_x$ and $V_x$ such that
		\[
		x \in W_x \text{ and } U^c_x \subset V_x.
		\]
		Next, we use condition (a) to choose disjoint open sets $W'_x$ and $V'_x$ of $U_x$ such that
		\[
		x \in W_x' \text{ and } U_x\backslash U \cap U_x \subset V'_x.
		\]
		Then, we have
		\begin{equation*}
		\begin{split}
		x \in W_x \cap W'_x,\: U^c \subset V_x \cup V'_x,\\
		(W_x \cap W'_x)\cap (V_x \cup V'_x) = \emptyset.
		\end{split}
		\end{equation*}
	\end{proof}
\end{lem}
\begin{rem}\label{modelreg}
As mentioned above, we apply Lemma \ref{regularity} mainly to manifolds, in which case $U_x$ is a specified coordinate open neighborhood of $x$. If $X$ is a topological manifold in the sense of Remark \ref{topmfd}, then condition (a) is automatically satisfied (see \cite[Section 2.1]{Jar}). If $X$ is a $C^\infty$-manifold,
then the underlying topological space $\widetilde{E}$ of its model vector space $E$ is a vector space in $\czero$ (see Proposition \ref{dmfd} and Remark \ref{suitable}) but need not be a topological vector space (\cite[4.16-4.26]{KM}). Thus, we do not know whether $\widetilde{E}$ is always regular. However, we are mainly concerned with the case where $E$ is $C^\infty$-regular, in which case $\widetilde{E}$ is regular. We also see that if $E$ is a Fr\'echet space or a Silva space, then $\widetilde{E}$ is regular (see \cite[Theorem 4.11(1)-(2)]{KM}).
\end{rem}

\bibliographystyle{amsalpha}

\samepage


\end{document}